\documentclass[12pt,final]{article}   

\usepackage{amsmath}
\usepackage{amsthm}
\usepackage{amsgen}
\usepackage[mathscr]{eucal}
\usepackage{mathrsfs}
\usepackage{enumerate}
\usepackage{newtxtext}
\usepackage[varvw]{newtxmath}
\usepackage{color}
\usepackage{fancybox}
\usepackage{graphicx}

\usepackage{showkeys}
\usepackage{here}
\usepackage{ulem}

\usepackage{bm}

\usepackage{braket}
\usepackage{tikz}
\usepackage{gnuplot-lua-tikz}

\newtheorem{lem}{Lemma}[section]
\newtheorem{thm}[lem]{Theorem}
\newtheorem{pro}[lem]{Proposition}

\newtheorem{rem}[lem]{Remark}
\newtheorem{asp}[lem]{Assumption}

\newcommand{\bqn}{\begin{equation}}
\newcommand{\eqn}{\end{equation}}
\newcommand{\beqx}{\begin{equation*}}
\newcommand{\eeqx}{\end{equation*}}
\newcommand{\barr}{\begin{array}}
\newcommand{\earr}{\end{array}}
\newcommand{\beqn}{\begin{eqnarray}}
\newcommand{\eeqn}{\end{eqnarray}}
\newcommand{\beqnx}{\begin{eqnarray*}}
\newcommand{\eeqnx}{\end{eqnarray*}}
\newcommand{\bmt}{\begin{multline}}
\newcommand{\emt}{\end{multline}}

\numberwithin{equation}{section}

\newcommand{\D}{\partial}

\def\det{\mathop{\rm det}\nolimits}

\def\dim{\mathop{\rm dim}\nolimits}

\newcommand{\sA}{\mathscr{A}}
\newcommand{\sB}{\mathscr{B}}

\newcommand{\sC}{\mathscr{C}}

\newcommand{\ve}{\varepsilon}

\newcommand{\vphi}{{\varphi}}

\newcommand{\er}{\eqref}
\newcommand{\lb}{\label}

\newcommand{\jump}[1]{\ensuremath{[\![#1]\!]} }

\def\disp{\displaystyle}

\newcommand{\hitem}{\addtocounter{enumi}{1}\quad\labelenumi\hspace{\labelsep}}

\title{Asymptotic stability of transonic steady nozzle flows}

\author{%
{\large\sc Masashi Ohnawa}${}^1$
\thanks{corresponding author} \
{\normalsize and}
{\large\sc Masahiro Suzuki}${}^2$
}

\date{%
\normalsize
${}^1$%
Department of Ocean Sciences,
\\
Tokyo University of Marine Science and Technology,
\\
Konan 4-5-7, Minato-ku, Tokyo, 108-8477, Japan
\\
ohnawa@m.kaiyodai.ac.jp
\\ [7pt]
${}^2$%
Department of Computer Science and Engineering, 
\\
Nagoya Institute of Technology,
\\
Gokiso-cho, Showa-ku, Nagoya, 466-8555, Japan
\\
masahiro@nitech.ac.jp
}

\begin{document}

\maketitle

\begin{abstract}
In the present article,  
we prove the nonlinear asymptotic stability of stationary solutions
to the quasi-one-dimensional barotropic Euler equations in a finite-length nozzle
with variable cross-sectional area $A(x)$.
We consider all stationary flow configurations arising from
subsonic inflow: flows that remain subsonic throughout the nozzle,
flows that pass continuously through a sonic state and remain
supersonic up to the exit, and flows that subsequently return to
the subsonic regime through a stationary shock.
We prove that all these stationary solutions are asymptotically
stable with respect to sufficiently small perturbations.
No smallness restrictions are imposed on the mass flux, 
the magnitude of $A_x$, or the prescribed exit pressure, 
provided that the corresponding stationary solution exists. 
The only geometrical assumptions are that $A_{xx}>0$ 
at a smooth subsonic--supersonic transition, 
and $A_x>0$ at a stationary supersonic--subsonic shock.
To the best of our knowledge, this gives the first nonlinear
asymptotic stability result for a stationary flow passing smoothly
from the subsonic to the supersonic regime in a finite-length nozzle.
The principal analytical difficulty is the degeneracy of the
characteristic structure at the sonic point.
We overcome this difficulty by coupling energy structures adapted
respectively to the subsonic and supersonic regimes.
This coupling provides the dissipation needed to control the smooth
sonic transition, and when combined with the dynamics of the shock,
yields a unified nonlinear stability framework for smooth and discontinuous transonic nozzle flows.
\end{abstract}

\begin{description}

\item[{\it Keywords:}]
quasi-one-dimensional compressible Euler equations; transonic flow; shock wave; nonlinear stability.

\item[{\it 2020 Mathematics Subject Classification:}]
35B35, 
35L65, 
35L67, 
35Q31, 
76H05, 
76N15 

\end{description}

\tableofcontents

\section{Introduction} \label{Sec01}

This study aims to establish the asymptotic stability of steady flows in nozzles
such as jet and rocket engines and steam turbines
that steadily accelerate gas stream above sound speeds 
through nozzles with variable cross-sectional area.
The nozzle is assumed to have finite length, and the flow is modeled 
by the quasi-one-dimensional compressible Euler equations.
The mathematical theory of steady compressible nozzle flows has been developed for more than seventy years since the classical works of Courant and Friedrichs \cite{CF1948}. 
Although the existence and qualitative classification of steady transonic flows are well understood, the nonlinear asymptotic stability of continuously accelerating transonic steady flows has remained open. 

Compressible flow fields are commonly classified into three distinct regimes-subsonic, transonic, and supersonic-depending on whether the local flow velocity at each spatial point is less than, equal to, or greater than the local speed of sound, respectively. Among these regimes, transonic flows are especially important, since they involve a mixture of subsonic and supersonic regions and often give rise to nonlinear phenomena such as shock formation, expansion waves, and complex wave interactions. These features make transonic flows particularly challenging to analyze from both mathematical and physical perspectives. 

Transonic flows manifest in a variety of natural and engineering contexts.
The technological importance of transonic flow became widely recognized in the late nineteenth century. A Swedish engineer, Gustav de Laval, invented the convergent-divergent nozzle, whose cross-sectional area first decreases and then increases, in order to enhance the efficiency of steam turbines. The operating principle of this device is based on the fact that a compressible gas accelerates in a converging section when the flow is subsonic and continues to accelerate in a diverging section once the flow becomes supersonic. This mechanism enables continuous acceleration from subsonic to supersonic speeds within a single device. Theoretical foundation of this phenomenon was described by Osborne Reynolds \cite{Reynolds1886} in 1886. 
The usefulness of this simple mechanism has been widely recognized, and it remains in broad use today in a variety of engineering systems, including rocket propulsion, supersonic wind tunnels, and high-performance turbine technologies.
Owing to its practical significance, Reynolds' theory has been refined over time, and the analysis of flow within nozzles is now systematically presented in standard textbooks on gas dynamics such as \cite{LR1957,ZH1976}. 
In particular, the long-time behavior of compressible flows in nozzles has attracted considerable attention because steady flows often correspond to physically realizable operating states in engineering devices. In particular, understanding whether such steady flows are dynamically stable is essential for predicting the performance and reliability of systems involving high-speed gas dynamics.

From a mathematical viewpoint, the quasi-one-dimensional compressible Euler equations provide a fundamental model for studying these phenomena. In this framework, variations in cross-sectional area are incorporated into the governing equations, while physical quantities such as density, velocity, and pressure are assumed to be uniform across each cross section along the flow direction. 
Furthermore, a barotropic equation of state, in which the pressure depends solely on the density, is often employed to reduce the complexity of thermodynamic relations.
These approximations greatly simplify the analysis while preserving the essential characters of compressible flow. 
Even such a simplified model provides valuable physical insight and enables evaluation of engineering performance. At the same time, it offers a mathematically tractable framework for investigating shock formation, choking, and stability of steady flows. 

Motivated by these considerations, this study focuses on the asymptotic stability of all types of stationary solutions in a finite-length nozzle including transonic flows. Establishing such stability results contributes to a deeper understanding of the long-time dynamics of compressible flows, and provides a rigorous mathematical foundation for engineering applications.

In a typical convergent-divergent nozzle, if the incoming flow is subsonic, three fundamentally different types of steady states may arise depending on boundary conditions. This will be illustrated in Figure \ref{Fig012} after precise formulation,
or refer to \cite{LR1957,ZH1976}:
\vspace{-2mm}
\begin{enumerate}[(A)]
\item
A purely subsonic flow throughout the nozzle.
\vspace{-2mm}
\item
A transonic flow that accelerates from subsonic to sonic at the throat (the point of minimal cross-sectional area), becomes supersonic downstream, and then undergoes a discontinuous transition (shock) back to a subsonic state before exiting.
\vspace{-2mm}
\item
A fully transonic flow that accelerates through the throat and remains supersonic all the way to the exit.
\end{enumerate}
\vspace{-2mm}
The existence and qualitative structure of these steady configurations, as well as their dependence on the exit pressure, were already described in the classical work of Courant-Friedrichs \cite{CF1948}.
However, despite their fundamental importance in both fluid mechanics and the theory of nonlinear hyperbolic equations, the asymptotic stability of such steady flows, in particular the transonic configurations of types (B) and (C), has remained a longstanding open problem.
From a mathematical perspective, the principal difficulty arises from 
the presence of the sonic point.
More precisely, standard energy estimates valid in the subsonic region \cite{RXX2013} lose coercivity near the sonic transition, 
whereas the weighted energy method used in the analysis for supersonic flows 
cannot be directly extended across the sonic point.
This situation creates serious analytical obstacles that have resisted a complete treatment for decades.

To overcome this difficulty, we investigate the detailed structure 
of the stationary flow around the sonic point,  and develop an energy framework
that couples the subsonic and supersonic regimes.
This coupling reveals a dissipative mechanism that enables us to control
the perturbation across the sonic transition.
As a result, we obtain a unified nonlinear stability theory for the steady nozzle  flows described above.
Moreover, the stability results do not require any smallness assumptions 
on the mass flux, the magnitude of $A_x$, or the prescribed exit pressure,
provided that the corresponding stationary solution exists.

\paragraph{Mathematical formulation.}
The quasi-one-dimensional model consists of the conservation laws of mass and momentum:
\vspace{-5mm}
\begin{subequations} \label{eq101}
\begin{gather} 
( A \rho )_t + ( A m )_x = 0,  \label{eq101a}
\\
( A m )_t + \left( A ( \dfrac{m^2}{\rho} + p(\rho) ) \right)_x = A_x p(\rho),  \label{eq101b}
\end{gather}
\end{subequations}
contained in a bounded interval $I := (x_1,x_2) \subset \mathbb{R}$ with $x_1 < 0 < x_2$, 
where $x=x_{1}$ and $x=x_{2}$ are the inlet and exit of the nozzle, respectively.
The unknown variables {$\rho=\rho(t,x)$ and $m=m(t,x)$} represent density and momentum of the gas respectively, 
and {$A=A(x)$} is the area of cross section of the nozzle at $x$.
In this paper, it is supposed that $\inf_{x \in \bar{I}}A(x)>0$.
The system \eqref{eq101} can be regarded as compressible Euler equations 
with source terms arising from the variation in the nozzle cross-sectional area.
For the equation of state, we assume 
\begin{subequations} \label{eq102}
\begin{gather} 
p \in C^5(0,\infty), \ \ p'(\rho) > 0, \ \ p''(\rho) \ge 0 \quad \text{for} \ \rho>0,
 \label{eq102b}
\\
\lim_{\rho \to +0} \rho^{2} \int^{\rho}_{\rho_\infty} \frac{p'(\zeta)}{\zeta} d \zeta  = 0,
\label{eq102c}
\end{gather}
\end{subequations}
where $\rho_\infty(>0)$ is an arbitrarily fixed reference density.
The familiar power law $p(\rho) \propto \rho^\gamma \ (\gamma \ge 1)$ meets these requirements.
In the stability analysis of type (B), we deal with flows that are piecewise smooth.
If a discontinuous point is located at $x=d(t)$ at time $t$,
the conservation laws \eqref{eq101} imply the Rankine-Hugoniot conditions of
\begin{subequations} \label{eq106}
\begin{gather} 
\dot{d}(t) \jump{ A \rho } (t,d(t)) = \jump{ A m } (t,d(t)), \label{eq106a}
\\
\dot{d}(t) \jump{ A m } (t,d(t)) = \Bigl[\!\Bigl[ A \Bigl( \dfrac{m^2}{\rho} + p(\rho) \Bigr) \Bigr]\!\Bigr] (t,d(t)), \label{eq106b}
\end{gather}
\end{subequations}
where the jump notation stands for $\disp \jump{f}(t,x) := f(t,x+0) - f(t,x-0)$.
{At the discontinuity,  we assume the Lax entropy condition}
\begin{equation} \label{eq107}
 u(t,d(t)-0) - a(t,d(t)-0) > \dot{d}(t) > u(t,d(t)+0) - a(t,d(t)+0),
\end{equation}
where the velocity $u$ and the speed of sound $a$ are defined by 
\begin{gather*} 
u = u(t,x):= \frac{m(t,x)}{\rho(t,x)}, \ \ \text{and} \ \ a=a(t,x) := \sqrt{p'(\rho(t,x))},
\end{gather*}
respectively.
The condition \eqref{eq107} means that the flow is supersonic on the upstream side, and subsonic on the downstream side.

Now we formulate the stationary problem associated with the quasi-one-dimensional model  \eqref{eq101}.
A stationary solution $(\bar{\rho}, \bar{m})=(\bar{\rho}, \bar{m})(x)$ to \eqref{eq101} satisfies 
\begin{subequations} \label{eq151}
\begin{gather} 
 ( A \bar{m} )_x = 0,
\label{eq151a}
\\
\left( A  \Bigl( \dfrac{\bar{m}^2}{\bar{\rho}} + p(\bar{\rho}) \Bigr) \right)_x = A_x p(\bar{\rho}).
\label{eq151b}
\end{gather}
\end{subequations}
Concerning the boundary conditions, 
one natural way is to fix the mass flux $A \bar{m}$ at the inlet by
\begin{gather} \label{eq153}
( A \bar{m} ) (x_1) = Q  > 0,
\end{gather}
where $Q$ is an arbitrary positive value. 
At the exit, the pressure is usually prescribed. 
This, in the barotropic model, means that the density is fixed: 
\begin{gather} \label{eq155}
\bar{\rho}(x_2) = \rho_+.
\end{gather}
If the stationary solution has {a discontinuous point  $x = \bar{d} \in I$ as in type (B)}, 
the Rankine-Hugoniot conditions are
\begin{subequations} \label{eq157}
\begin{gather} 
\jump{ A \bar{m} } (\bar{d}) = 0,
\label{eq157a}
\\
\Bigl[\!\Bigl[ A \Bigl( \dfrac{\bar{m}^2}{\bar{\rho}} + p(\bar{\rho}) \Bigr) \Bigr]\!\Bigr] (\bar{d}) = 0,
\label{eq157b}
\end{gather}
\end{subequations}
where $\disp \jump{f}(x) := f(x+0) - f(x-0)$, and the Lax entropy condition of \eqref{eq107} reads
\begin{gather} \label{eq159*}
 \bar{u}(\bar{d}-0) - \bar{a}(\bar{d}-0) > 0 > \bar{u}(\bar{d}+0) - \bar{a}(\bar{d}+0),
\end{gather}
where $\bar{u}$ and $\bar{a}$ are defined by
\begin{equation*} 
\bar{u}(x) := \dfrac{\bar{m}(x)}{\bar{\rho}(x)}
\ \ \text{and} \ \ 
\bar{a}(x) := \sqrt{p'(\bar{\rho}(x))},
\end{equation*}
respectively. The condition \eqref{eq159*} is equivalent to 
\begin{gather} \label{eq159}
\jump{ \bar{u} } (\bar{d}) < 0
\end{gather}
under \eqref{eq157}.
Here we give a remark on the boundary conditions \eqref{eq153} and \eqref{eq155} for the stationary problem.
In the present work, we consider the asymptotic stability of steady flows 
in which the fluid enters subsonically, and is continuously accelerated to supersonic speeds. 
In the time-dependent problem, only one boundary condition can be imposed at the inlet. In view of this, we give one boundary condition \eqref{eq153} at the inlet for the steady problem.
However, solutions of types (B) and (C) coincide up to at least the throat, 
and are indistinguishable based solely on the inlet conditions.
So for the steady problem we impose another boundary condition \eqref{eq155}
at the exit, thereby allowing us to observe the emergence of three types of stationary solutions (A), (B), and (C) distinguished by the differences in the exit pressure.

Now we observe that steady flows in a convergent-divergent nozzle 
change their types from (A) to (B) and then to (C) 
as the pressure at the exit $p(\rho_{+})$ decreases.
The graphs in Figure \ref{Fig012} illustrate schematic profiles of density 
and the Mach number for each type, which is defined by
\begin{equation*}
\bar{M} = \bar{M}(x) := \frac{|\bar{u}(x)|}{\bar{a}(x)}.
\end{equation*}
The flow of stationary solutions is subsonic if $\bar{M}(x)<1$, is transonic 
if $\bar{M}(x)=1$, and is supersonic if $\bar{M}(x)>1$.
The results on the existence of stationary solutions for all the states will be summarized in Theorem \ref{thm2.1}.
\\
\begin{minipage}{1.0\hsize}
\begin{figure}[H] 
\centering
\includegraphics[scale=0.75]{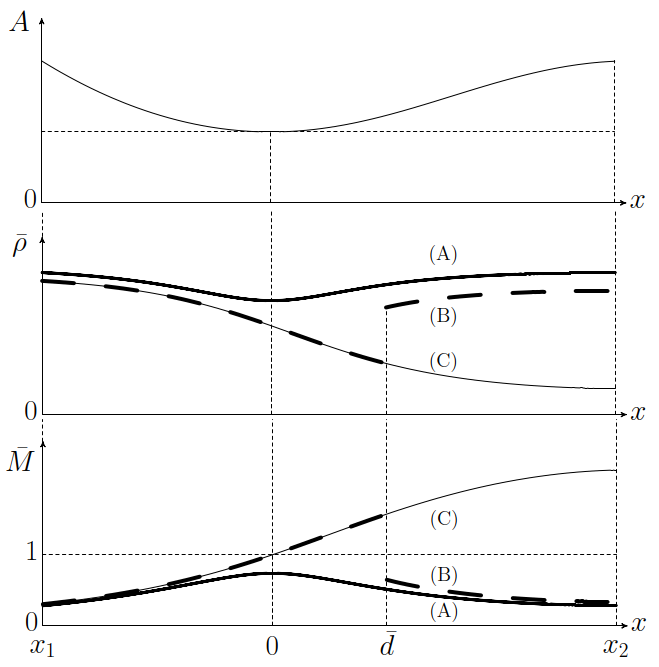} 
\caption{Graphs of the sectional area $A$ (upper) for a typical convergent-divergent nozzle, 
the corresponding density $\bar{\rho}$ (middle), and the Mach number $\bar{M}$ (lower) 
as functions of $x$. 
The case when $A(x)$ takes its minimum value only at $x=0$ is considered 
as an example in the top panel.
In the lower two panels, stationary solutions of types (A), (B), and (C) 
which appear according to the prescribed density at the exit are illustrated.
Thick solid lines, dashed lines, and thin solid lines represent 
stationary solutions of types (A), (B), and (C), respectively for a fixed $Q$.}
\label{Fig012}
\end{figure}
\end{minipage}

\bigskip

To investigate the asymptotic stability of stationary solutions,
we formulate the initial--boundary value (or mixed) problem to the quasi-one-dimensional model  \eqref{eq101}.
Consistent with each state of stationary solutions, we impose 
boundary conditions for the time-dependent problem \eqref{eq101} as follows.
For all types (A), (B), and (C), 
the boundary condition at the inlet $x=x_{1}$ is given for the mass flux $A m$ by
\begin{gather} \label{eq104}
A m(t, x_1) = Q.
\end{gather}
For types (A) and (B), where the flow is subsonic at the exit, 
the density is prescribed at the exit $x=x_{2}$ by 
\begin{gather} \label{eq105}
\rho(t, x_2) = \rho_+.
\end{gather}
While for type (C), where the flow is supersonic at the exit, 
no boundary condition is needed at the exit.

At a discontinuity, both of the Rankine-Hugoniot conditions \eqref{eq106a} and \eqref{eq106b} provide expressions for $\dot{d}$, and the condition under which they are compatible is obtained by eliminating $\dot{d}$ as
\begin{equation} \label{eq110}
\jump{ \rho } \Bigl[\!\Bigl[ \dfrac{m^2}{\rho} + p(\rho) \Bigr]\!\Bigr] (t, d(t))
= \jump{ m } (t, d(t))^2.
\end{equation}
This equation serves as a boundary condition.
Indeed, under the entropy condition \eqref{eq107}, 
there are two characteristic curves reaching $x=d(t)-0$, 
and one reaching $x=d(t)+0$,
and the remaining necessary information is given by \eqref{eq110}.
Then $\dot{d}$ can be determined from \eqref{eq106a} alone, 
without contradicting \eqref{eq106b}.
Hereafer we employ \eqref{eq106a} and \eqref{eq110} 
instead of \eqref{eq106a} and \eqref{eq106b},
and regard \eqref{eq106a} as an evolution equation for $d(t)$, 
and \eqref{eq110} as a boundary condition at $x=d(t)$, respectively.

The initial value of \eqref{eq101} is prescribed by
\begin{gather} \label{eq108}
(\rho, m)(0,x) = (\rho^0, m^0)(x).
\end{gather}
%
{
For the analysis of type (B) only, the initial location of the discontinuity $x=d(t)$ is given by
}
\begin{gather} \label{eq109}
d(0) = d^0.
\end{gather}
We require that initial values are compatible with boundary conditions
\eqref{eq104}, \eqref{eq105}, and  \eqref{eq110}.
The precise conditions following \cite{RM1974} will be given in Sections 3, 4, and 5.

The main assertions {of asymptotic stability} for each state are briefly stated in Theorems \ref{Thm11} (type (A)), \ref{Thm13} (type (C)), and \ref{Thm15} (type (B)), with reference to the structure of the corresponding stationary solutions. 
Precise formulations, including the function spaces for initial values and solutions, are provided in Sections 3, 4, and 5, respectively.

\begin{thm}[Type (A): entirely subsonic stationary solutions] \label{Thm11}
Assume $A\in {C^3(\bar{I}) }$, 
and that the stationary solution $(\bar{\rho}, \bar{m})  \in {C^3(\bar{I}) \times C^3(\bar{I})} $ to (\ref{eq151})--(\ref{eq155}) exists and satisfies
\begin{gather*}
\inf_{x \in \bar{I}} \bar{\rho}(x) >0, \quad
\bar{M}(x) < 1 \quad \text{for} \ x \in \bar{I}.
\end{gather*}
Then the stationary solution is asymptotically stable to small perturbations
$\disp ( \rho^0 - \bar{\rho}, m^0 - \bar{m} ) \in H^2(I) \times H^2(I)$.
\end{thm}

\begin{thm}[Type (C): transonic stationary solutions without discontinuities] \label{Thm13} 
Assume $A \in  C^4(\bar{I})$ with $A_{x}(0) = 0$ and $A_{xx}(0) > 0$, 
and that the stationary solution $(\bar{\rho}, \bar{m})  \in C^3(\bar{I}) \times C^4(\bar{I})$ to (\ref{eq151})--(\ref{eq155}) exists and satisfies
\begin{gather*}
\inf_{x \in \bar{I}} \bar{\rho}(x) >0,  \quad
\begin{cases}
\bar{M}(x) < 1 & \text{for} \ x \in [x_1, 0),
\\
\bar{M}(x) = 1 & \text{for} \ x = 0,
\\
\bar{M}(x) > 1 & \text{for} \ x \in (0, x_2].
\end{cases}
\end{gather*}
Then the stationary solution is asymptotically stable to small perturbations
$\disp ( \rho^0 - \bar{\rho}, m^0 - \bar{m} ) \in H^2(I) \times H^2(I)$.
\end{thm}

To state the asymptotic stability of a stationary solution which is smooth 
except for a single discontinuity at $x = \bar{d}$,
we note that for a certain positive constant $\varepsilon_0$,
the stationary solution of type (B) above, 
when restricted on $[x_1, \bar{d})$, can be extended smoothly to 
$[x_1, \bar{d}+\varepsilon_0)$ which we call $(\bar{\rho}_1, \bar{m}_1)$.
Note that the extended solution satisfies \eqref{eq151}, but ignores the boundary condition 
at $x=x_2$.
Here we mention that it is because $\lim_{x \to \bar{d}-0} \bar{M}(x) \ne 1$,
and the equation \eqref{eq221} is solvable in a certain neiborhood of $x=\bar{d}$.
Similarly, when it is restricted on $(\bar{d}, x_2]$, it can be extended smoothly to 
$(\bar{d}-\varepsilon_0, x_2]$ which we call $(\bar{\rho}_2, \bar{m}_2)$.

\begin{thm}[Type (B): transonic stationary solutions with a discontinuity] \label{Thm15} 
Assume $A \in  C^4(\bar{I})$ with $A_{x}(0) = 0$ and $A_{xx}(0) > 0$.
Suppose that there exists a constant $\bar{d} \in (0, x_2)$ such that 
$A_x(\bar{d}) > 0$ holds, and that a stationary solution 
$(\bar{\rho}, \bar{m})  \in C^3(\bar{I} \setminus \{ \bar{d} \} ) \times C^4(\bar{I})$ 
to (\ref{eq151})--(\ref{eq159*}) exists and satisfies
\begin{gather*}
\inf_{x \in \bar{I}} \bar{\rho}(x) >0,  \quad
\begin{cases}
\bar{M}(x) < 1 & \text{for} \ x \in [x_1, 0),
\\
\bar{M}(x) = 1 & \text{for} \ x = 0,
\\
\bar{M}(x) > 1 & \text{for} \ x \in (0, \bar{d}),
\\
\bar{M}(x) < 1 & \text{for} \ x \in (\bar{d}, x_2].
\end{cases}
\end{gather*}
Note that $\bar{u} \in C^3(\bar{I} \setminus \{ \bar{d} \} )$ holds under these assumptions.
If $\disp ( \rho^0, m^0 )$ is smooth except for a single point
$x=d^0 \in (\bar{d}-\varepsilon_0, \bar{d}+\varepsilon_0)$,
and
the restriction of $\disp \rho^0 - \bar{\rho}_1, m^0 - \bar{m}_1$ on $\disp I_1^0 := (x_1, d^0)$,
and that of $\disp \rho^0 - \bar{\rho}_2, m^0 - \bar{m}_2$ on $\disp I_2^0 := (d^0, x_2)$
belong to $H^2(I_1^0)$ and $H^2(I_2^0)$ respectively 
with the sum of their norms plus $|d^0-\bar{d}|$ sufficiently small,
then a global solution exists, and it converges to the stationary solution.
\end{thm}

To the best of our knowledge, Theorems \ref{Thm13} and \ref{Thm15} provide the first mathematical results establishing the asymptotic stability of stationary solutions that exhibit a {\it continuous} transition from subsonic to supersonic flow in nozzles of finite length.
Furthermore, we impose no restrictions on the magnitude of the mass flux $Q$, the derivative of the cross-sectional area $A_x(x)$, or the exit pressure, provided that a stationary solution exists.

\paragraph{Literature review.}
%
The study of steady gas flows through ducts of varying cross-sectional area has a long history. As early as 1886, Reynolds \cite{Reynolds1886} investigated the flow of compressible gases through ducts and identified the fundamental relation between changes in the cross-sectional area and the acceleration or deceleration of the flow. These early observations provided a basis for the subsequent quasi-one-dimensional theory of nozzle flows. A systematic description was later given by Courant and Friedrichs \cite{CF1948}. Within the quasi-one-dimensional approximation, they described the principal steady-flow patterns in a convergent-divergent nozzle and, in particular, explained how varying the receiver pressure gives rise to 
the three flow patterns (A), (B), and (C) as shown in Figure \ref{Fig012}. 

By the 1950s, the mathematical theory of subsonic and transonic gas dynamics had developed far beyond the quasi-one-dimensional setting. A systematic account of this theory was given by Bers \cite{Bers1958}, who treated two-dimensional potential flows using, among other tools, the hodograph transformation and the theory of equations of mixed elliptic-hyperbolic type. In this framework, the sonic state is no longer merely a critical point of an ordinary differential equation, as in the quasi-one-dimensional nozzle model, but forms part of the boundary between regions in which the governing equation has different types. This mixed-type character constitutes one of the fundamental mathematical difficulties of multidimensional transonic flow.
A particularly striking manifestation of this difficulty was discovered by Morawetz. Beginning in the 1950s, Morawetz (\cite{Mrwt1956} and subsequent papers) showed that shock-free continuous transonic flows past a two-dimensional airfoil are structurally unstable with respect to arbitrarily small perturbations of the profile. Her work revealed the delicate nature of smooth multidimensional transonic flows and motivated the subsequent study 
of physically admissible weak solutions containing shocks.

Since then, substantial progress has been made in the mathematical theory of multidimensional steady transonic flows.
The existence, uniqueness, and structural stability of multidimensional transonic shocks, which separate a strictly supersonic upstream flow from a subsonic downstream flow without a continuous passage through a sonic point, have been extensively studied; see, for example, \cite{CF2003,CF2022,XY2005,XYY2009} and the references therein.
In contrast, important progress has also been made for shock-free continuous transonic flows, including the analysis of smooth transonic flows of Meyer type in two-dimensional de Laval nozzles by Wang and Xin \cite{WangXin2019} and the existence, uniqueness, and structural stability theory developed by Weng and Xin \cite{WengXin2024} for quasi-one- and quasi-two-dimensional steady Euler models.

For the time-dependent problems, 
we first review studies on global weak solutions
for {\it infinitely long} nozzles, i.e. $I=\mathbb R$.
A major difficulty in these studies arises from the possibility that spatial derivatives of solutions may blow up in finite time, see e. g. \cite{CPZ2017,Lax1964}.
This inspired the construction of solution in the weak sense.
In the case of constant cross-sectional area, DiPerna \cite{DiPerna1983} employed the vanishing viscosity method together with compensated compactness to construct global entropy solutions for compressible gas dynamics with a special power-law equation of state.
%
When the cross-sectional area varies, Liu \cite{Liu1979} constructed global solutions to general quasilinear hyperbolic systems, including \eqref{eq101}, 
using Glimm's scheme, and analyzed their asymptotic behavior.
However, this analysis excluded cases involving sonic transition points, and the total variation of $A(x)$ was assumed to be small.
Recently, Tsuge \cite{Tsuge2012} established the global-in-time existence of entropy weak solutions that may include sonic points by constructing an invariant region, assuming that $A_{xx}(x)$ vanishes at the throat.
Subsequent works \cite{Tsuge2013,Tsuge2017} extended these results to various nozzle configurations under smallness assumptions on the integral of $|A_x(x)|$.

The dynamical and asymptotic behavior of one-dimensional transonic flows was investigated by Liu in a series of works. In \cite{Liu1982b}, he analyzed the asymptotic wave patterns for the Euler equations in ducts whose cross-sectional area varies over a bounded interval and is constant outside it. In the transonic regime, he showed that the admissible asymptotic states may undergo bifurcation and need not depend uniquely and smoothly on the far-field states. A particularly important consequence of this analysis is the geometric distinction between contracting and expanding portions of the duct: standing shocks occurring in an expanding portion belong to a stable branch of asymptotic states, whereas those occurring in a contracting portion are associated with an unstable branch. This led to the principle that a stable standing shock can occur only in an expanding portion of a convergent-divergent nozzle. Liu subsequently developed this observation in a genuinely dynamical direction in \cite{Liu1982a}, where he studied the nonlinear stability and instability of stationary discontinuous transonic configurations and established the dynamical stability of standing shocks located in an expanding portion of the nozzle, in contrast to the instability of those located in a contracting portion. These results identified the nozzle geometry as a decisive mechanism for the dynamical selection of stationary transonic shocks. The relevance of such a selection principle was further emphasized by Embid, Goodman and Majda \cite{EGM1984}, who showed that the quasi-one-dimensional model may admit multiple stationary solutions in convergent-divergent nozzles, in particular when a supersonic inflow is connected discontinuously to a subsonic downstream state through the Rankine-Hugoniot conditions. Their analysis demonstrates that existence alone does not single out a physically relevant steady configuration and hence underlines the role of dynamical stability in selecting asymptotic states among multiple admissible transonic states.

However, the above studies did not address the stability of stationary solutions that undergo a continuous transition from subsonic to supersonic flow through a sonic point. For nozzles of finite length, Rauch, Xie and Xin \cite{RXX2013} established the global nonlinear asymptotic stability of stationary transonic shocks for the quasi-one-dimensional Euler equations under suitable boundary conditions. In their setting, the flow enters the nozzle supersonically, undergoes a discontinuous transition to a subsonic state through a shock, and exits subsonically. Their proof combines energy estimates with spectral theory. An important feature of their result is that no smallness assumption is imposed on the variation of the cross-sectional area or on the strength of the stationary shock. On the other hand, since the upstream flow is already supersonic, the stationary flow considered there does not undergo a continuous transition through a sonic point.

Related stability problems have also been studied for the one-dimensional compressible Navier--Stokes equations in a half space. Kawashima, Nishibata and Zhu \cite{KNZ2003} studied stationary outflow solutions with supersonic or transonic far-field states and proved their asymptotic stability by an $L^2$-energy method combined with a Poincar\'{e}-type inequality. Nakamura, Nishibata and Yuge \cite{NNY2007} subsequently investigated the convergence rates toward these stationary solutions by means of a weighted energy method. We emphasize that the transonic stationary solutions considered in these works do not connect a subsonic state to a supersonic state through an interior sonic point. Thus, despite the occurrence of a transonic far-field state, the degeneracy arising in these problems is different from the continuous sonic transition considered in the present paper.

A stability problem involving a continuous sonic transition was studied by Hong, Hsu, Lin and Liu \cite{HHLL2013}. They considered a stationary solution in an infinitely long convergent-divergent nozzle that passes continuously from the subsonic to the supersonic regime. Under additional assumptions, including that the Mach number is greater than $1/\sqrt{2}$ everywhere and that $A_{xx}>0$ at the throat, they proved that solutions to the linearized problem decay to zero in $L^\infty$ on every compact spatial interval. Their argument is based on energy estimates for perturbations up to the $H^1$-level. This provides an important stability result for continuous transonic stationary flows, but the analysis is restricted to the linearized problem in an infinitely long nozzle.

Consequently, nonlinear asymptotic stability of a stationary flow that passes continuously from the subsonic to the supersonic regime through an interior sonic point has remained largely open, particularly for nozzles of finite length. 
The finite-length setting is also natural from the classical viewpoint of Courant and Friedrichs \cite{CF1948}, where they classified steady flows in a convergent-divergent nozzle in terms of the receiver pressure at the exit, whose variation selects among flow patterns (A), (B), and (C) mentioned earlier.

In the present paper, we establish a nonlinear asymptotic stability theory for steady flows in a finite convergent-divergent nozzle, including those that undergo a continuous transition from subsonic to supersonic flow. The main analytical difficulty lies in controlling perturbations across the sonic point, where one of the characteristic speeds degenerates. Exploiting the precise structure of the stationary solution near the transonic point, we first extend the weighted energy analysis from the supersonic side upstream across the sonic point, so that the resulting estimate is valid on a region containing the transonic point in its interior. The complementary region is then strictly subsonic, where we combine an energy argument closely related to that of \cite{RXX2013} with a characteristic method. 
A crucial feature of this decomposition is that the two energy estimates generate different fluxes at their common interface. The resulting gap between these interface fluxes produces a coercive dissipative term, which in turn yields dissipation of the perturbation energy over the entire nozzle by applying the characteristic method directly and time-weighted energy method. Combining these mechanisms, we obtain the desired nonlinear asymptotic stability without imposing smallness assumptions on the mass flux or the variation of the nozzle cross-sectional area, nor additional restrictions on the exit pressure once the corresponding stationary solution exists.

\paragraph{Main ideas of the proof.}
We briefly describe the main analytical mechanisms 
underlying our stability results.  
A fundamental difficulty in the transonic problem is the change 
of characteristic regime across the sonic point.  
In the strictly subsonic region, the system admits a coercive energy structure 
for which the boundaries provide the dissipation needed to obtain asymptotic stability.  
This mechanism, however, ceases to be uniform 
as the flow approaches the sonic state, where one of the characteristic speeds vanishes.  
On the other hand, the supersonic region admits a different, 
spatially weighted energy estimate.  
Thus, no single energy estimate used in the strictly subsonic or supersonic regime alone is sufficient to control a perturbation across the entire transonic flow.

Our approach is to couple these two energy structures 
rather than to seek a single energy functional that remains uniformly coercive across the sonic point.  
We apply the subsonic energy estimate on a region terminating slightly upstream of the sonic point, while the weighted estimate is applied to a region that contains the sonic point and extends into the supersonic regime.  
The two estimates are coupled across an artificial interface 
located slightly upstream of the sonic point.
At the artificial interface separating the corresponding energy estimates, 
the outgoing flux from the subsonic estimate and the incoming flux 
from the weighted estimate do not cancel.  
The resulting flux gap has a favorable sign and supplies the dissipation 
required to close the global nonlinear energy estimate.  
In this sense, the degeneracy at the sonic point is controlled not by a uniformly coercive energy across the transition, but by coupling energy structures adapted to the two different characteristic regimes.  
The local structure of the stationary transonic solution near the sonic point, 
and in particular the explicit information on its first derivatives obtained in Section~2, is essential for establishing the required sign.

A second dissipative mechanism appears when the stationary transonic flow contains a shock.  
The region upstream of the shock is controlled by the transonic framework 
just described, whereas the downstream subsonic region is treated 
by the subsonic energy method.  
These estimates are coupled through the Rankine--Hugoniot conditions 
and the evolution equation for the shock location.  
The geometric condition that the stationary shock lie in an expanding portion 
of the nozzle, namely $A_x(\bar d)>0$, 
yields a favorable contribution controlling perturbations of the shock position.  

There is also an important difference between the present argument 
and earlier approaches to stability in finite-length nozzles.  
In the subsonic region, we combine the energy method with a characteristic analysis to identify directly the dissipation generated by the physical boundaries.  
This avoids the spectral contradiction argument used in previous work and allows substantially lower regularity of the initial perturbations.  
More importantly, the same viewpoint makes it possible to pass from the purely subsonic problem to the transonic problem: physical boundary dissipation in the subsonic problem is replaced, at the level of the coupled estimates, by the favorable flux gap at an artificial interface. 
The resulting framework therefore treats boundary dissipation, sonic transition, and shock motion within a common nonlinear energy argument.

\paragraph{Organization of the article.}
In Section~2, 
we classify the stationary solutions in a convergent--divergent nozzle 
according to the prescribed exit density $\rho_+$.  
We identify the threshold values separating the three types of stationary flows described above: purely subsonic solutions of type (A), 
transonic solutions of type (B) containing a stationary shock 
in the divergent part of the nozzle,
and smooth transonic solutions of type (C) that pass through the sonic state 
at the throat and remain supersonic downstream.
We also investigate the regularity of the transonic stationary solutions 
at the sonic point, and explicitly determine their first derivatives there, 
which are used in the subsequent stability analysis.

In Section~3, we establish the nonlinear asymptotic stability of
the purely subsonic stationary solutions of type (A).

In Section~4, we prove the nonlinear asymptotic stability 
of the smooth transonic stationary solutions of type (C), 
which pass continuously from the subsonic to the supersonic regime 
through the sonic state. 
This section contains the main analysis required 
to overcome the degeneracy at the sonic point.

Finally, in Section~5, we establish the nonlinear asymptotic stability 
of the transonic stationary solutions of type (B)
containing a single stationary shock in the divergent part of the nozzle.  
The proof combines the transonic analysis of Section~4
with the subsonic stability framework of Section~3, 
and the dynamics of the shock location.

\paragraph{Notation.}
$c$ and $C$ denote generic positive values which may differ from line to line.
Letters with subscritps such as $c_i$ and $C_i$ $(i=0,1,2,\cdots)$ play the same roles.
LHS and RHS stand for left hand side and right hand side respectively. 
Finally, $\mathbb{R}_+:= (0,\infty)$ is used.

\section{Construction of stationary solutions} \label{Sec02}

The goal of this section is to construct stationary solutions
for convergent-divergent nozzles which are widely used in industries. 
The results are summarized in Theorem \ref{thm2.1} below, where
the stationary solutions corresponding to the types (A)--(C) in Section \ref{Sec01} appear depending on the values of pressure at the exit.

To state the theorem, we introduce notation.
First we note that due to the equations \eqref{eq151a} and \eqref{eq157a} and the boundary condition \eqref{eq104}, the stationary solution satisfies
\begin{equation} \label{eq209}
A \bar{m} (x) = Q \ \text{for an arbitrary} \ x \in I 
\end{equation}
irrespectively of the presence of discontinuities.
We define
\begin{equation} \label{eq215}
F(\rho, A) := \frac{Q^2}{2A^2 {\rho}^2} + \int_{\rho_\infty}^{\rho} \frac{p'(\zeta)}{\zeta} d \zeta,
\end{equation}
where $\rho_\infty$ is the reference density in \eqref{eq102c}.
The important thing is that if $A(x)$ and $\bar{\rho}(x)>0$ are smooth and 
satisfies $F(\bar{\rho}(x), A(x))=c$ for a certain constant $c$, 
then $(\bar{\rho},\bar{m})=(\bar{\rho},Q/A)$ solves the equations \eqref{eq151}.
Obviously, $F(\rho, A)$ increases as $A$ decreases, and it holds that
\begin{equation} \label{eq217}
\frac{\partial}{\partial {\rho}} F({\rho}, A) 
= \frac{1}{{\rho}^3} \left( {\rho}^2 p'({\rho}) - \frac{Q^2}{A^2} \right).
\end{equation}
For each fixed $A>0$, the function $F(\cdot, A)$ attains its minimum at $\disp \rho = \rho_T(A) >0$, where 
$\rho_T(A)$ is defined by
\begin{equation} \label{eq219}
\rho_T(A):= h^{-1}(Q^2/A^2)
\ \ \text{with} \ \
h(\rho):=\rho^2 p'(\rho).
\end{equation}
Here $h(\rho)$ is a monotonically increasing function whose range is $[0,\infty)$ due to \eqref{eq102}, 
and thus $h^{-1}$ is well-defined.
It is seen that $F(\rho_T(A), A)$ increases as $A$ decreases.
Indeed, if $0<A_{1}<A_{2}$, then $F(\rho_{T}(A_{1}), A_{1})>F(\rho_{T}(A_{1}), A_{2})>F(\rho_{T}(A_{2}), A_{2})$.
Furthermore, $\disp \bar{\rho} \lesseqgtr \rho_T(A)$ and \eqref{eq209} imply
$\disp \bar{\rho}^2 p'(\bar{\rho}) \lesseqgtr Q^2 / A^2 = \bar{\rho}^2 \bar{u}^2 $,
and hence $\disp p'(\bar{\rho}) \lesseqgtr \bar{u}^2$.
In particular, the flow is transonic if and only if $\disp \bar{\rho} = \rho_T(A)$ holds.
Conditions \eqref{eq102} imply
\begin{gather}\label{defF2}
\lim_{\rho \to +0} F(\rho, A) = \lim_{\rho \to \infty} F(\rho, A) = \infty.
\end{gather}
Therefore, for any fixed $A$ and $\disp c > F(\rho_T(A), A)$, the equation
$\disp F(\bar{\rho}, A) = c$ for $\bar{\rho}$ has one subsonic root and one supersonic root.
If $\disp c = F(\rho_T(A), A)$, then $\disp F(\bar{\rho}, A) = c$ has a unique transonic root $\bar{\rho} = \rho_T(A)$.

We also define
\begin{gather}\label{defG1}
G(\rho, A):= \dfrac{Q^2}{A^2 \rho} + p(\rho).
\end{gather}
It is noteworthy that the Rankine-Hugoniot condition \eqref{eq157b} can be rewritten as $\jump{ G(\bar{\rho}, A) } (\bar{d}) = 0$ by using \eqref{eq209} and the continuity of $A$.
From \eqref{eq102}, it holds that 
\begin{equation*}
\frac{\partial}{\partial \rho} G(\rho, A)
= - \dfrac{Q^2}{A^2 \rho^2} + p'(\rho),
\quad
\lim_{\rho \to +0} G(\rho, A) = \lim_{\rho \to +\infty} G(\rho, A)= + \infty.
\end{equation*}
We note that $G(\cdot, A)$ also attains its minimum at $\disp \rho = \rho_T(A)$.
For an arbitrary $A>0$ and a density $\disp \rho_L (\ne \rho_T(A))$, 
there exists a unique density $\disp \rho_R = \rho_R(\rho_L, A)$
such that 
\begin{equation} \label{eq227}
G(\rho_{L}, A)=G(\rho_{R}, A)
\ \text{with} \ \rho_L \ne \rho_R. 
\end{equation}
The entropy condition \eqref{eq159} and the condition \eqref{eq209} require 
\begin{equation} \label{eq229}
\rho_L < \rho_T(A) < \rho_R.
\end{equation}

Now we set
\begin{gather*}
\disp A_- := A(x_1), \quad \disp A_+ := A(x_2), \quad \disp A_0 := \min_{x \in \bar{I}} A(x), \quad \disp F_0 := F(\rho_T(A_0), A_0).
\end{gather*}
Furthermore, we denote by $\rho_1,\rho_2$ the roots of $\rho$ to $F(\rho, A_+) = F_0$ 
with $\rho_1 > \rho_2$, and set $\disp \rho_3 := \rho_R(\rho_2, A_+)$,
where $\rho_R$ is defined in \eqref{eq227}.
Figure \ref{Fig01} below will help readers understand the definitions.

\begin{minipage}{0.9\hsize}
\begin{figure}[H] 
\centering
\includegraphics[scale=0.6]{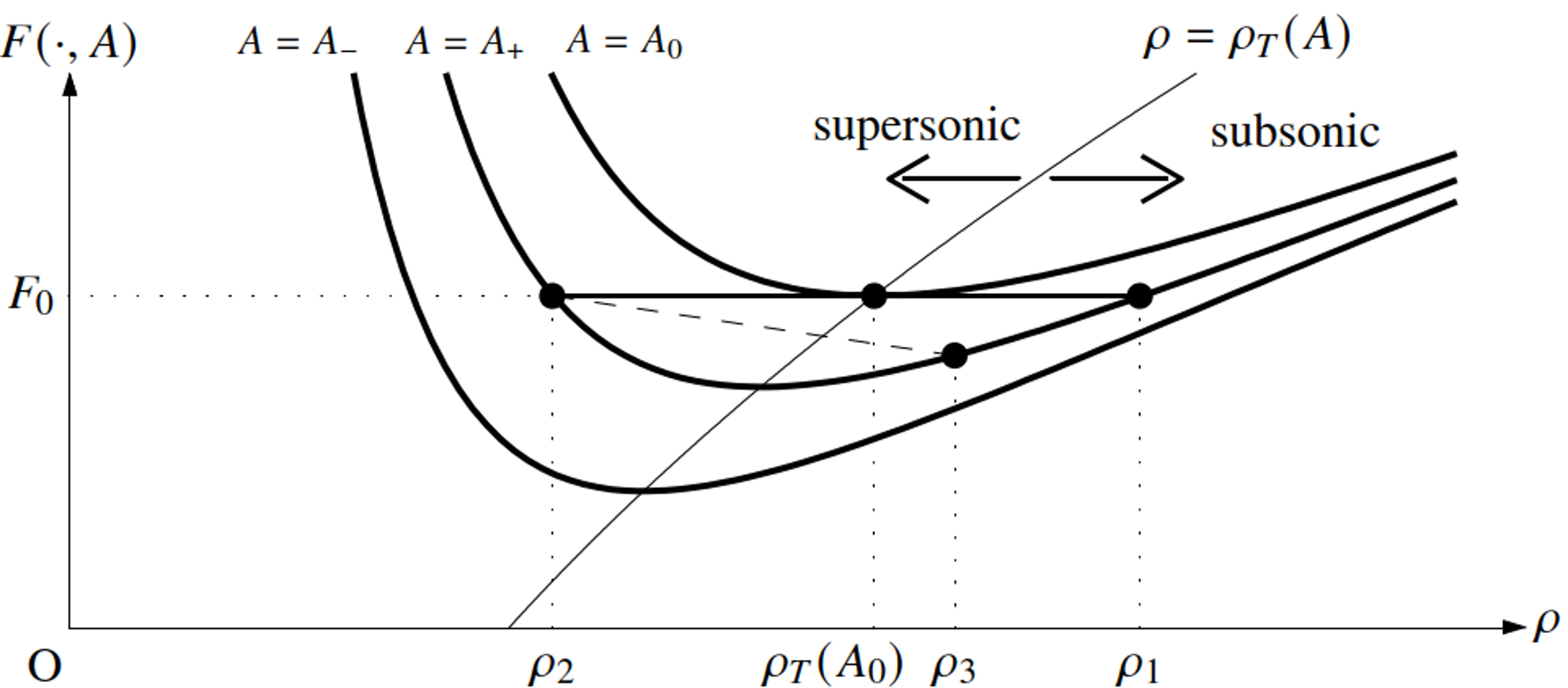} 
\vspace{-5mm}
\caption{Graphs of $F(\cdot,A)$ for representative cross-sectional areas $A=A(x)$.
See the text for the definitions of $\rho_1$, $\rho_2$, $\rho_3$, and $F_0$.}
\label{Fig01}
\end{figure}
\end{minipage}

\begin{thm} \label{thm2.1}
Suppose a nozzle configuration satisfies 
\begin{gather} \label{eq231}
A \in C^4(\bar{I}), \quad 
A_{x}(x) \lessgtr 0 \ \text{for} \ x \lessgtr 0, \quad
A_{xx}(0)>0.
\end{gather}
For an arbitrary positive $Q$,
there exists a unique piecewise smooth solution $(\bar{\rho},\bar{m})$ 
to \eqref{eq151}--\eqref{eq159} 
that satisfies $\inf_{x \in \bar{I}} \bar{\rho}(x) >0$ and $\bar{M}(x_{1})<1$
if and only if the boundary value $\rho_ +$ meets either one of the case (i)--(iii) below.
\begin{enumerate}[(i)]
\item
If $\rho_+ > \rho_1$, the solution is subsonic entirely in $\bar{I}$.
It belongs to $C^4(\bar{I}) \times C^4(\bar{I})$.

\item
If $\rho_+ \in (\rho_3, \rho_1)$, the solution is subsonic in $x<0$, exactly transonic at $x=0$, 
supersonic in $(0, \bar{d})$ for a certain  $\bar{d} \in (0, x_2)$ uniquely determined by $\rho_+$, and subsonic again in $(\bar{d}, x_2]$.
It belongs to $C^3(\bar{I} \setminus \{ \bar{d} \} ) \times C^4(\bar{I})$.

\item
If $\rho_+ = \rho_2$, the solution is subsonic in $x<0$, exactly transonic at $x=0$, 
and supersonic in $x>0$. It belongs to $C^3(\bar{I}) \times C^4(\bar{I})$.
\end{enumerate}
Furthermore, in the cases (ii) and (iii), there holds that
\begin{gather} \label{eq237}
\bar{\rho}_x(0)
=
- \bar{\rho} \sqrt{ \frac{A_{xx}}{A ( K(\bar{\rho}) + 2 )} } (0) < 0,
\quad
\bar{u}_x(0)
=
\bar{u} \sqrt{ \frac{A_{xx}}{A ( K(\bar{\rho}) + 2 )} } (0) > 0,
\end{gather}
where $K(\bar{\rho}) := {\bar{\rho} p''(\bar{\rho})} /{p'(\bar{\rho})}$ and $\bar{\rho}(0)=\rho_{T}(A(0))$.
\end{thm}
\begin{rem}
The stationary solutions in (i)--(iii) correspond to the flows of types (A)--(C) in Section \ref{Sec01}, respectively.
Representative profiles of $\bar{\rho}$ classified as case (i), (ii), and (iii) 
depending on the value of $\rho_+$ are illustrated in Figure \ref{Fig212}.

\begin{minipage}{0.9\hsize}
\begin{figure}[H] 
\centering
\includegraphics[scale=0.8]{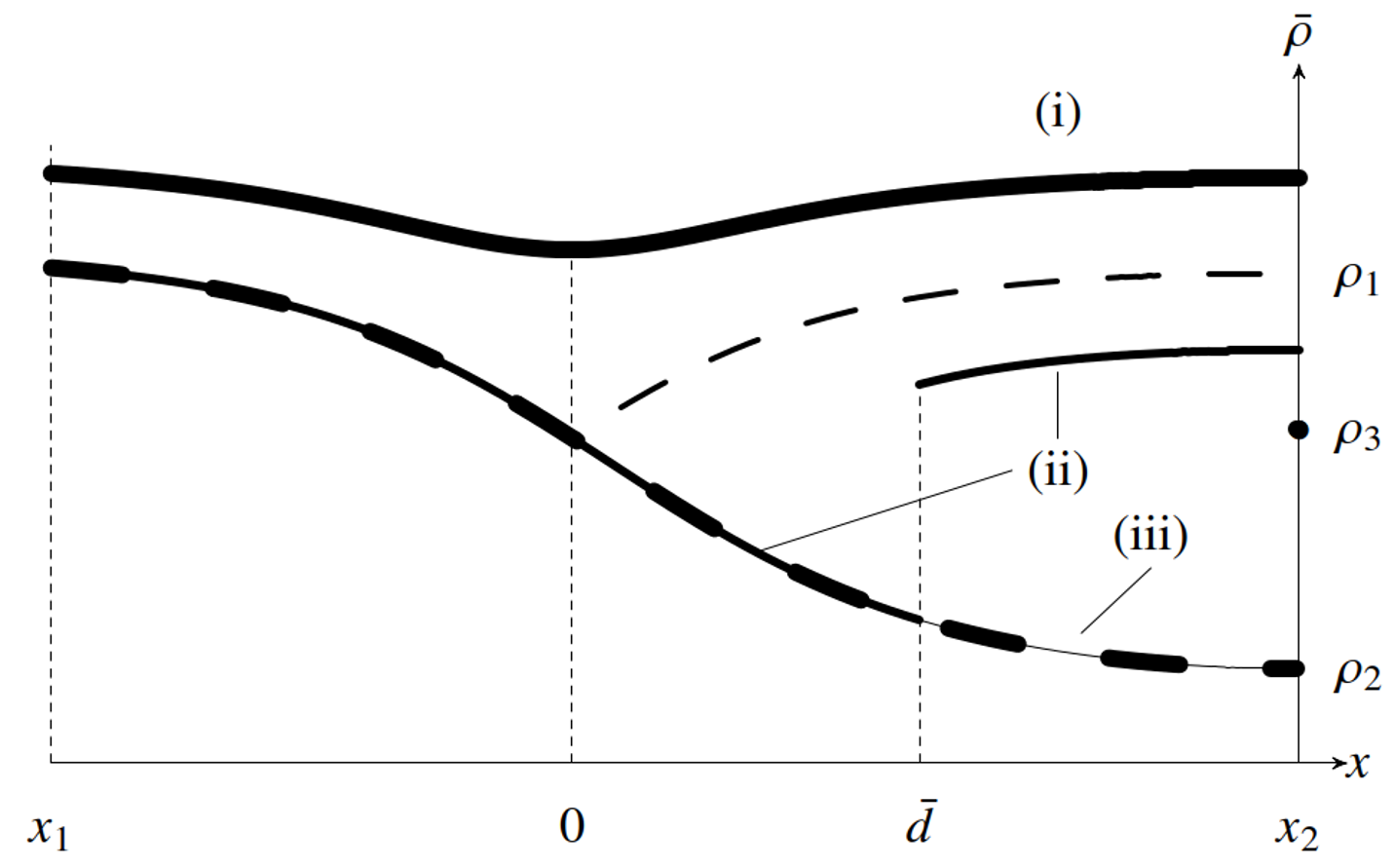} 
\vspace{-5mm}
\caption{Graphs of $\bar{\rho}(x)$ for different values of $\rho_+$.
Typical solutions for cases (i), (ii), and (iii) are depicted by thick, medium (discontinuous at $x=\bar{d}$), 
and thin solid lines, respectively.
In the figure, $\rho_1$ and $\rho_3$ are the infimums of $\rho_+$ of stationary solutions 
classified as case (i) and (ii), respectively.
The thin dash-dotted line shows the limiting profile of $\bar{\rho}$ as $\rho_+ \to \rho_1$,
whereas the thick dashed line shows the limiting profile of $\bar{\rho}$ as $\rho_+ \to \rho_3+0$.
In the latter limit, the shock location $\bar{d}$ converges to $x_2-0$.}
\label{Fig212}
\end{figure}
\end{minipage}

\end{rem}

\begin{rem}
If we take into account stationary solutions which are supersonic at the entrance,
other types of stationary solutions appear, and the uniqueness is violated \cite{EGM1984}.
Since nozzles are used to accelerate subsonic flows into supernonic state in industries, 
we only deal with flows that are subsonic at the entrance. 
\end{rem}


For the stationary solution $(\bar{\rho},\bar{m})$ in Theorem \ref{thm2.1}, the following holds:
\begin{gather}\label{DF0}
\begin{cases}
\disp \frac{d}{dx} F(\bar{\rho}(x), A(x)) = 0 \quad \text{for $x \in \bar{I}$} & \text{in Cases (i), (iii)},
\\[10pt]
\disp \frac{d}{dx} F(\bar{\rho}(x), A(x)) = 0 \quad \text{for $x \in \bar{I} \backslash \{\bar{d}\} $} & \text{in Case (ii)}.
\end{cases}
\end{gather}
This means that $F(\bar{\rho}(x), A(x))$ is piecewise constant on $\bar{I}$.
We draw the behavior of $\bar{\rho}(x)$ in the $\rho$-$F$ plane in Figures \ref{Fig02a}--\ref{Fig02c} below,
where the value of $\bar{\rho}$ varies according to $A(x)$ along the line (curve) with arrows.

First we outline the proof.  Recall $\bar{m}$ must be equal to $Q/A$ due to \eqref{eq209}.
In order to construct the stationary solution $(\bar{\rho},\bar{m})$,
it is sufficient to determine $\bar{\rho}=\bar{\rho}(x)$ by choosing a proper root of the equation $F(\bar{\rho},A(x))=c$ with suitable constants $c$ according to Cases (i)--(iii).
In Case (i), we determine $\bar{\rho}(x)$ by solving the equation $F(\bar{\rho},A(x))=F(\rho_{+},A_{+})$ with $\bar{\rho}> \rho_{T}(A(x))$. The root $\bar{\rho}(x)$ behaves as in Figure \ref{Fig02a}. 


\begin{minipage}{0.9\hsize}
\begin{figure}[H] 
\centering
\includegraphics[scale=0.6]{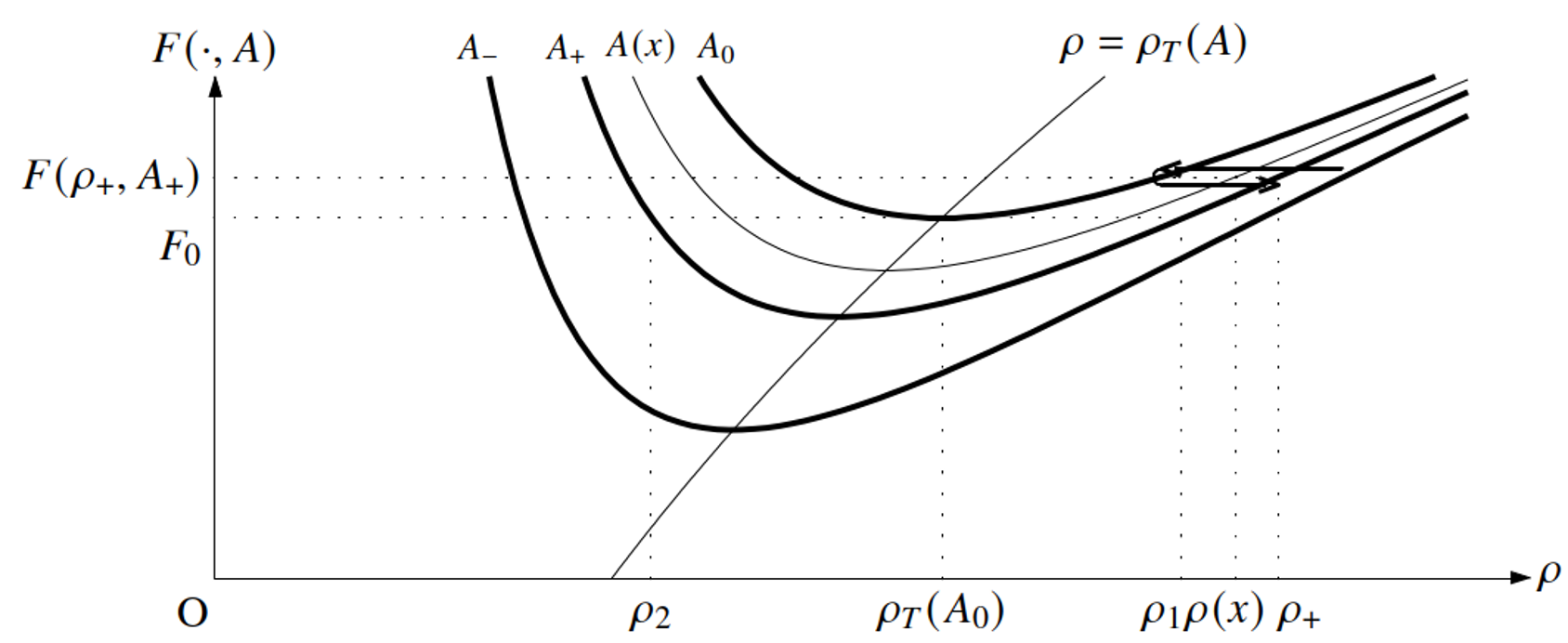} 
\caption{Case (i) in Theorem \ref{thm2.1}, i.e. $\rho_+ > \rho_1$.}
\label{Fig02a}
\end{figure}
\end{minipage}

\bigskip

In Case (ii), we can uniquely find $x_{3} \in\left(0, x_{2}\right)$ such that $F\left(\rho_{T}\left(A\left(x_{3}\right)\right), A\left(x_{3}\right)\right)=F\left(\rho_{+}, A_{+}\right)$ holds.
For an arbitrary $\bar{d} \in\left(x_{3}, x_{2}\right)$, we define $\bar{\rho}(x)$ by solving
$$
\left\{\begin{array}{l}
F(\bar{\rho}(x), A(x))=F_{0} \ \text { with } \ \bar{\rho}(x) \geq \rho_{T}(A(x)) \text { for } x \in\left[x_{1}, 0\right], \\
F(\bar{\rho}(x), A(x))=F_{0} \ \text { with } \ \bar{\rho}(x)<\rho_{T}(A(x)) \text { for } x \in(0, \bar{d}), \\
F(\bar{\rho}(x), A(x))=F(\rho_{+}, A_{+}) \ \text { with } \ \bar{\rho}(x) \geq \rho_{T}(A(x)) \text { for } x \in\left(\bar{d}, x_{2}\right].
\end{array}\right.
$$
Then $\bar{d}$ can be determined uniquely so that the Rankine-Hugoniot conditions \eqref{eq157} are satisfied at $x=\bar{d}$.
In Figure \ref{Fig02b} below, the discontinuity is illustrated by the dashed line.

\begin{minipage}{0.9\hsize}
\begin{figure}[H] 
\centering
\includegraphics[scale=0.5]{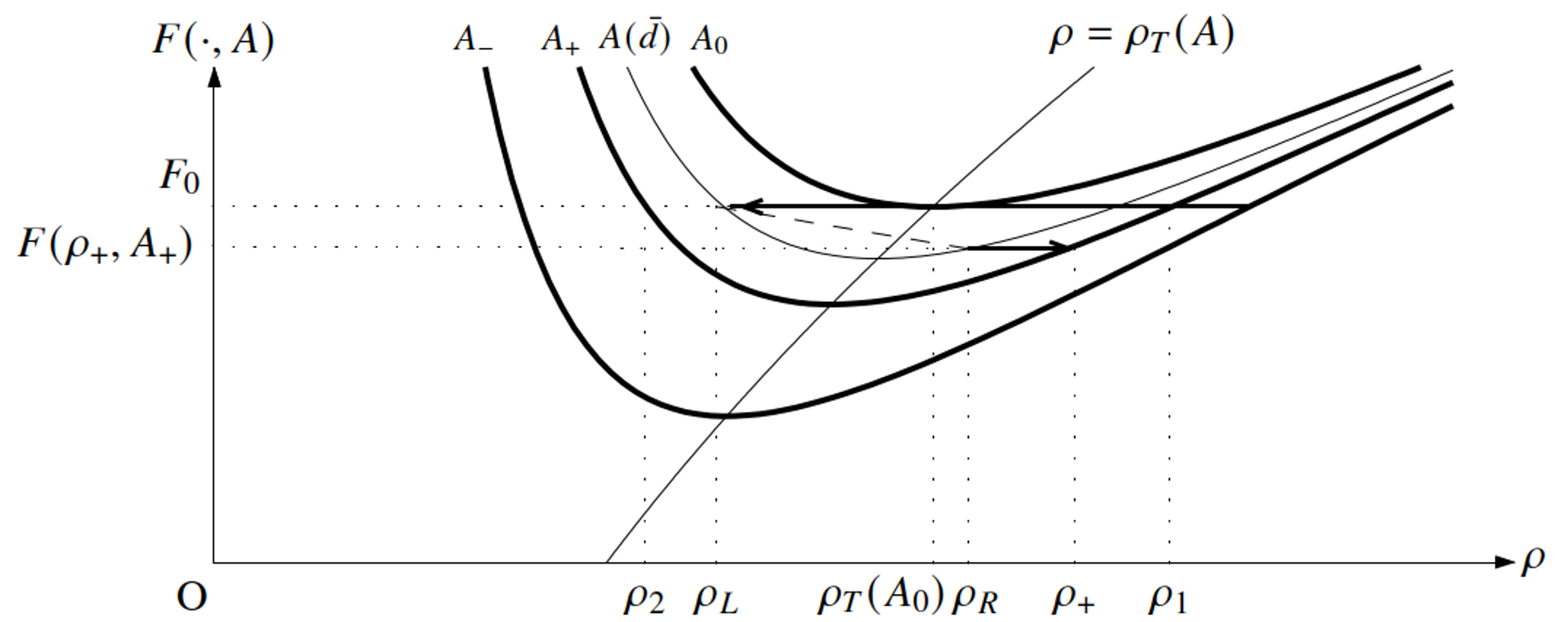}
\caption{Case  (ii) in Theorem \ref{thm2.1}, i.e. $\rho_+ \in (\rho_3, \rho_1)$.}
\label{Fig02b}
\end{figure}
\end{minipage}

\bigskip

In Case (iii), the state $\bar{\rho}(x)$ for $x \in [x_1,0]$ is prescribed 
by solving the equation $F(\bar{\rho},A(x))=F_{0}$ 
with $\bar{\rho} \ge \rho_T(A(x))$, while for $x \in (0,x_2]$, 
the same equation is solved with $\bar{\rho} < \rho_T(A(x))$.

\begin{minipage}{0.9\hsize}
\begin{figure}[H] 
\centering
\includegraphics[scale=0.6]{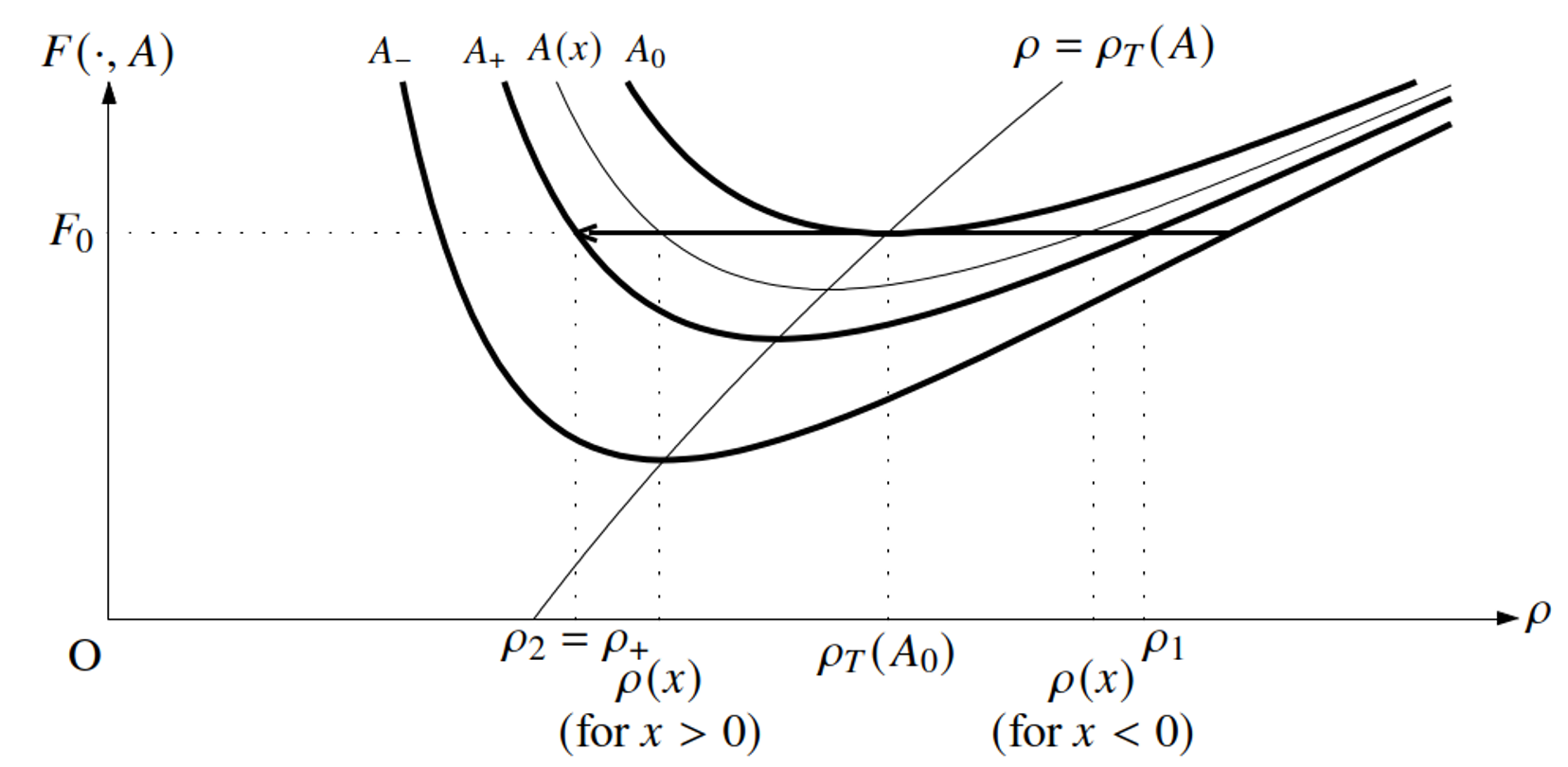}
\caption{Case  (iii) in Theorem \ref{thm2.1}, i.e.  $\rho_+ = \rho_2$.}
\label{Fig02c}
\end{figure}
\end{minipage}

\bigskip

To complete the proof of Theorem \ref{thm2.1}, we begin with deriving several useful equations.
From \eqref{eq151}, we have
\begin{subequations} \label{eq221}
\begin{gather}
\bar{\rho}_x = \frac{\bar{M}^2}{1-\bar{M}^2} \frac{A_x}{A} {\bar{\rho}} = \frac{Q^2 \bar{\rho}}{A^2 \bar{\rho}^2 p'(\bar{\rho}) - Q^2} \frac{A_x}{A},
\label{eq221a}\\
\bar{u}_x = - \frac{1}{1-\bar{M}^2} \frac{A_x}{A} {\bar{u}},
\\
\bar{m}_x = - \frac{A_x}{A} \bar{m},
\end{gather}
\end{subequations}
where $\bar{M}$ is the Mach number of the steady flow. 
%
By a direct computation, it is also seen that 
\begin{gather}  \label{eq225}
\bar{M}_x = - \frac{2 + K(\bar{\rho})}{2} \frac{\bar{M}}{1-\bar{M}^2} \frac{A_x}{A}.
\end{gather}
This equation implies that
the stationary solution cannot undergo {\it continuous} transition 
from a subsonic state ($\bar{M} < 1$) to a supersonic state ($\bar{M} > 1$) or vice versa
except where $A_x$ vanishes.

Next we show Lemmas \ref{Lem20}--\ref{Lem24} below before proving Theorem \ref{thm2.1} in the end of this section. We use Lemma \ref{Lem20} to construct the stationary solution $(\bar{\rho}, \bar{m})$ as in Theorem \ref{thm2.1} (i); we use Lemmas \ref{Pro21}--\ref{Pro22} to construct $(\bar{\rho}, \bar{m})$ as in Theorem \ref{thm2.1} (iii); we use all Lemmas \ref{Lem20}--\ref{Lem24} to construct $(\bar{\rho}, \bar{m})$ as in Theorem \ref{thm2.1} (ii).

\begin{lem} \label{Lem20}
Let $n=2,3,4$ and $A \in C^n(\bar{I})$.
Set $J:=(x_-,x_+) \subset \bar{I}$, $A_m := \inf_{x \in J} A(x)$, and $F_{m}:=F(\rho_T(A_m),A_m)$.
For any $\disp c > F_{m}$, the equation ${\cal F}(\bar{\rho}, x):=  F(\bar{\rho}, A(x)) = c$ has a unique subsonic (resp. supersonic) and positive root $\bar{\rho} \in C^n(\bar{J})$ 
with $\bar{\rho}(x) > \rho_T(A(x))$ (resp. $0<\bar{\rho}(x) < \rho_T(A(x))$) for any $x \in J$. 
Furthermore, the pair $(\bar{\rho},\bar{m})(x):=(\bar{\rho},Q/A)(x)$ solves \eqref{eq151} for $x \in J$.

\end{lem}
\begin{proof}
As mentioned just after \eqref{defF2}, $F_{m} \geq F(\rho_{T}(A(x)),A(x))$ holds, and if $c>F_{m}$, 
then for any $x \in J$, the equation ${\cal F}(\bar{\rho}, x)=c$ has a unique subsonic (resp. supersonic) root $\bar{\rho}=\bar{\rho}(x)$ with $\bar{\rho} > \rho_T(A(x))$ (resp. $0<\bar{\rho} < \rho_T(A(x))$).
Next we  show the regularity $\bar{\rho} \in C^n(\bar{J})$. It follows from $A \in C^n$ and \eqref{eq102} that $\cal F$ is a $C^{n}$ function of $\bar{\rho}$ and $x$. Furthermore, $({\partial {\cal F}}/{\partial {\rho}}) (\bar{\rho}(x), x) \neq 0$ holds for $x \in J$ owing to \eqref{eq217} and \eqref{eq219}. Now the application of the implicit function theorem concludes $\bar{\rho} \in C^n(\bar{J})$. Finally, it is clear that $(\bar{\rho}(x),\bar{m}(x))=(\bar{\rho}(x),Q/A(x))$ solves \eqref{eq151} for $x \in J$.
\end{proof}

\begin{lem} \label{Pro21}
Set $J:=(x_-,x_+) \subset \bar{I}$, $A_m := \inf_{x \in J} A(x)$, and $F_{m}:=F(\rho_T(A_m),A_m)$.
Suppose $A \in C^2(\bar{I})$, and 
there exists a certain $x_m \in J$ such that
$A_{x}(x) < 0$ for $x \in (x_-, x_m)$, $A_{x}(x) > 0$ for $x \in (x_m, x_+)$, and $A_{xx}(x_{m}) > 0$ hold.
Then the equation ${\cal F}(\bar{\rho}, x):=  F(\bar{\rho}, A(x)) = F_{m}$ has a unique positive root $\bar{\rho} \in C^1(\bar{J})$ with 
\begin{gather}\label{STS1}
\begin{cases}
\bar{\rho}(x) > \rho_T(A(x)) \ \ \text{for} \ \ x_- < x < x_m,
\\
\bar{\rho}(x) = \rho_T(A(x)) \ \ \text{for} \ \ x = x_m,
\\
\bar{\rho}(x) < \rho_T(A(x)) \ \ \text{for} \ \ x_m < x < x_+.
\end{cases}
\end{gather}
Furthermore, $(\bar{\rho},\bar{m})(x):=(\bar{\rho},Q/A)(x)$ solves \eqref{eq151} for $x \in J$, and also satisfies
\begin{gather} 
\bar{\rho}_x(x_m)
=
- \bar{\rho} \sqrt{ \frac{A_{xx}}{A ( K(\bar{\rho}) + 2 )} } (x_m) < 0,
\quad
\bar{u}_x(x_m)
=
\bar{u} \sqrt{ \frac{A_{xx}}{A ( K(\bar{\rho}) + 2 )} } (x_m)  > 0,
\label{eq237'}
\end{gather}
where $K(\bar{\rho})$ is defined in Theorem \ref{thm2.1}.
\end{lem}
\begin{proof}
Recall that $\disp F(\bar{\rho}, A) = c$ has a unique transonic root $\bar{\rho} = \rho_T(A)$ if $\disp c = F(\rho_T(A), A)$.
Therefore, in the same way as in the proof of Lemma \ref{Lem20}, we see that 
the equation ${\cal F}(\bar{\rho}, x)= F_m$ is solved uniquely so that {the root} satisfies $\bar{\rho} \in C^0(\bar{J}) \cap C^2(\bar{J}\backslash \{x_{m}\})$ with \eqref{STS1}.
Then $(\bar{\rho},\bar{m})(x)=(\bar{\rho},Q/A)(x)$ solves \eqref{eq151} for $x \in J\backslash \{x_{m}\}$.
Hence, \eqref{eq221}--\eqref{eq225} also hold on $J\backslash \{x_{m}\}$.

We show the regularity of $\bar{\rho}$ at $x=x_{m}$ and the equalities \eqref{eq237'}
to complete the proof.
To this end, it suffices to show 
\begin{gather}\label{1Drho}
\lim_{x\to x_{m}}\bar{\rho}_{x}(x)= - \bar{\rho} \sqrt{ \frac{A_{xx}}{A ( K(\bar{\rho}) + 2 )} } (x_m) < 0.
\end{gather}
Let us define $\bar{S}(x)$ by
\begin{gather} \label{eq239}
\bar{S}(x) := \disp p'(\bar{\rho}(x)) - \bar{u}(x)^2 = p'(\bar{\rho}(x)) - \frac{Q^2}{A(x)^2 \bar{\rho}(x)^2},
\end{gather}
so that \eqref{eq221a} is simplified to
\begin{gather} \label{eq241}
\disp \bar{\rho}_x 
= \frac{Q^2}{A^3 \bar{\rho}} \frac{A_x}{\bar{S}}.
\end{gather}
The RHS is an indeterminate form due to $A_x(x_m) = \bar{S}(x_m) = 0$.
To know the limit of the RHS, we begin with computing $\disp \lim_{x \to x_m} \left( A_x / \bar{S} \right)^2 (x) $ by applying the l'H\^{o}pital rule.
From \eqref{eq221a}, we have
\begin{gather} \label{eq242}
\bar{S}_x 
= p''(\bar{\rho}) \bar{\rho}_x 
  + \frac{2Q^2}{A^2 \bar{\rho}^2} \left( \frac{A_x}{A} + \frac{\bar{\rho}_x}{\bar{\rho}} \right)
= \frac{A_x}{\bar{S}} H,
\end{gather}
where $H$ is a positive $C^0$-function defined by
\begin{gather*}
H(x) 
:= \left( \bar{\rho} p''(\bar{\rho}) + 2 p'(\bar{\rho}) \right) \frac{Q^2}{A^3\bar{\rho}^2} (x).
\end{gather*}
We recall the assumption $A_{xx}(x_m) > 0$, 
and note that $\bar{\rho}(x_m)$ satisfies $A_m^2 \, \bar{\rho}(x_m)^2 p'(\bar{\rho}(x_m)) = Q^2$ equivalent to $\bar{\rho}(x_{m}) = \rho_T(A(x_{m})) $ to see 
\begin{align*}
 \lim_{x \to x_m} \frac{(A_x^2)_{x}}{(\bar{S}^2)_{x}}(x)
=
\lim_{x \to x_m} \frac{ A_x A_{xx} } { \bar{S} \bar{S}_x } (x)
&=
\frac{A_{xx}} {H} (x_m)
\\
&=
\frac{A_{xx} A^3\bar{\rho}^2} { Q^2 ( \bar{\rho} p''(\bar{\rho}) + 2 p'(\bar{\rho}) ) } (x_m)
=
\frac{A_{xx} A^5 \bar{\rho}^4} { Q^4 ( K(\bar{\rho}) + 2 ) } (x_m),
\end{align*}
where the rightmost is a finite positive value and $K(\bar{\rho})$ is defined in Theorem \ref{thm2.1}.
Applying the l'H\^{o}pital rule, we have
\begin{gather*}
\lim_{x \to x_m} \left( \frac{A_x(x)}{\bar{S}(x)} \right)^2
=
\frac{A_{xx} A^5 \bar{\rho}^4} { Q^4 ( K(\bar{\rho}) + 2 ) } (x_m).
\end{gather*}
By the choice of the root branch, it holds $A_x(x) < 0$ and $\bar{S}(x) > 0$ for $x_- < x < x_m$,
while $A_x(x) > 0$ and $\bar{S}(x) < 0$ for $x_m < x < x_+$.
Thus we have
\begin{gather} \label{eq243}
\lim_{x \to x_m} \frac{A_x} {\bar{S}} (x)
=
- \sqrt{ \frac{A^5 A_{xx} \bar{\rho}^4} { Q^4 ( K(\bar{\rho}) + 2 ) } } (x_m)  < 0. 
\end{gather}
Substituting \eqref{eq243} into \eqref{eq241} concludes \eqref{1Drho}. 
The proof is complete.
%
\end{proof}

\begin{lem} \label{Pro22}
Under the same assumptions as in Proposition \ref{Pro21} but with $A \in C^2(\bar{I})$ replaced by $A \in C^4(\bar{I})$, the pair $(\bar{\rho},\bar{m})$ in Lemma \ref{Pro21} satisfies 
$(\bar{\rho},\bar{m}) \in C^3(\bar{J}) \times C^4(\bar{J})$.
\end{lem}
\begin{proof}
It is enough to show $\bar{\rho} \in C^{3}(\bar{J})$, since $\bar{m}=Q/A\bar{\rho}$ belongs to $C^{4}(J)$.
By differentiating \eqref{eq241} in $x$, we have
\begin{gather} \label{eq249}
\bar{\rho}_{xx}
= \left( \frac{Q^2}{A^3 \bar{\rho}} \frac{A_x}{\bar{S}} \right)_x
= - \left( \frac{3 A_x}{A} + \frac{\bar{\rho}_x}{\bar{\rho}} \right) \bar{\rho}_x
+ \frac{Q^2}{A^3 \bar{\rho}} \left( \frac{A_x}{\bar{S}} \right)_x
\ \ \text{for} \ \ x \ne x_m.
\end{gather}
In view of \eqref{eq237'} and \eqref{eq243},
in order to compute $\disp \lim_{x \to x_m} \bar{\rho}_{xx}(x)$,
it remains to consider the following limit:
\begin{equation} \label{eq251}
L(x):= \left( \frac{A_x}{\bar{S}} \right)_x (x)
= \frac{ \bar{S} A_{xx} - A_x \bar{S}_x } {\bar{S}^2} (x)
= \frac{ \bar{S}^2 A_{xx} - A_x^2 H } {\bar{S}^3} (x)
\quad \text{as $x \to x_{m}$}.
\end{equation}
We observe that
\begin{align} \label{eq252}
\lim_{x \to x_m} 
\frac{ ( \bar{S}^2 A_{xx} - A_x^2 H )_x } { (\bar{S}^3)_x } (x)
&=
\lim_{x \to x_m} 
\frac{ \bar{S}^2 A_{xxx} - A_x^2 H_x } { 3 A_x \bar{S} H }(x)
\notag \\
&=
\frac{1}{3} \frac{A_{xxx}}{H}(x_m) \lim_{x \to x_m} \frac{\bar{S}}{A_x}(x)
- \frac{1}{3} \frac{H_{x}}{H}(x_m) \lim_{x \to x_m} \frac{A_x}{\bar{S}}(x),
\end{align}
where the rightmost is a finite value due to \eqref{eq243}.
From this and the l'H\^{o}pital rule,
we obtain the convergence of $L(x)$, 
and hence of $\bar{\rho}_{xx}$ as $x \to x_m$.
Thus $\bar{\rho} \in C^{2}(\bar{J})$.
Accordingly, $H(x)$ is also $C^2$ in a certain neighborhood of $x=x_m$.

Finally, differentiating \eqref{eq249} by $x$ results in 
\begin{gather*} 
\bar{\rho}_{xxx} (x)
= - \left( \frac{3 A_x}{A} + \frac{\bar{\rho}_x}{\bar{\rho}} \right) \bar{\rho}_{xx}
  - \left( 3 \frac{A_{xx}A - A_x^2}{A^2} + \frac{\bar{\rho}_{xx}\bar{\rho} - \bar{\rho}_x^2}{\bar{\rho}^2}
    \right) \bar{\rho}_x
  - \left( \frac{3 A_x}{A} + \frac{\bar{\rho}_x}{\bar{\rho}} \right) \frac{Q^2}{A^3 \bar{\rho}} L
  + \frac{Q^2}{A^3 \bar{\rho}} L_x
\end{gather*}
for \ $x \ne x_m$.
Recall that terms other than $L_x(x)$ have been shown to converge as $x \to x_m$.
Let us investigate $\disp \lim_{x \to x_m} L_x(x)$.
From \eqref{eq251}, we have
\begin{gather*} 
L_x = \left( \frac{A_{xx}} {\bar{S}} \right)_x - \left( \frac{ A_x^2 H } {\bar{S}^3} \right)_x
= \frac{L_1} {\bar{S}^5}, \quad
L_1(x) := \left[
 A_{xxx} \bar{S}^4 - H_x A_{x}^2 \bar{S}^2 + 3 H A_x \left( H A_x^2 - A_{xx} \bar{S}^2 \right) 
\right] (x). \notag 
\end{gather*}
Note that $L_1(x_m) =  \bar{S}(x_m)= 0$.
To apply l'H\^{o}pital rule, we compute $\disp \lim_{x \to x_m} (L_1)_{x}(x) / ( \bar{S}^5 )_x (x)$.
From \eqref{eq242} and \eqref{eq243}, it follows that $\bar{S}_x (x_{m}) \neq 0$ and then
\begin{gather*}
( \bar{S}^5 )_{x} = 5 \bar{S}^4 \bar{S}_x = {\cal O}(\bar{S}^4) \ne {\cal O}(\bar{S}^5)
\quad \text{as} \ x \to x_m.
\end{gather*}
In view of this, we decompose 
$(L_1)_{x}$ as $(L_1)_{x} = L_2 + L_3$, where
\begin{align*}
L_{2}(x)&:=\left(A_{xxxx} \bar{S}^4 - H_{xx} A_{x}^2 \bar{S}^2 \right)(x),
\\
L_3(x)&:= \left( H A_{xxx} A_x \bar{S}^2 - 5 H_{x} A_{xx} A_x \bar{S}^2 + 4 H_x H A_x^3
      + 3 H^2 A_{xx} A_x^2 - 3 H A_{xx}^2 \bar{S}^2 \right) (x).
\end{align*}
From $A \in C^4(\bar{I})$, \eqref{eq102}, and \eqref{eq243}, we see that 
$\disp L_{2} / ( \bar{S}^5 )_x $ converges as $x \to x_m$. 
It suffices to show the convergence of $\disp \lim_{x \to x_m} L_3(x) / ( \bar{S}^5 )_x (x)$, where both $L_{3}$ and $( \bar{S}^5 )_x$ vanish at $x=x_{m}$.
To apply the l'H\^{o}pital rule, we compute $\disp \lim_{x \to x_m} (L_{3})_{x}(x) / ( \bar{S}^5 )_{xx} (x)$.
By differentiating \eqref{eq242} in $x$, we have
\begin{equation} \label{eq261}
\bar{S}_{xx} = L H + \frac{A_x} {\bar{S}} H_x,
\end{equation} 
for which the limit as $x \to x_m$ exists due to \eqref{eq243}, \eqref{eq251}, and \eqref{eq252}.
Thus 
\begin{gather*}
( \bar{S}^5 )_{xx} = 5 \bar{S}^3 ( \bar{S} \bar{S}_{xx} + 4 \bar{S}_x^2 ) 
= {\cal O}(\bar{S}^3) \ne {\cal O}(\bar{S}^4)
\quad \text{as} \ x \to x_m.
\end{gather*}
In view of this, we decompose 
$(L_{3})_{x}$ as $(L_{3})_{x}=L_{4}+L_{5}$, where
\begin{align*} 
L_{4}(x) &:= \left[( H A_{xxx} )_x A_x \bar{S}^2 - 5 ( H_{x} A_{xx} )_x A_x \bar{S}^2 
 + 4 ( H_x H )_x A_x^3 \right](x),
\\
L_{5}(x) &:= \left[ H A_{xxx} \left( A_x \bar{S}^2 \right)_x 
 - 5 H_{x} A_{xx} \left( A_x \bar{S}^2 \right)_x
 + 4 H_x H \left( A_x^3 \right)_x
 + \left( 3 H^2 A_{xx} A_x^2 - 3 H A_{xx}^2 \bar{S}^2 \right)_x \right](x).
\end{align*}
It is easy to see that $\disp L_{4}/ ( \bar{S}^5 )_{xx} $ congerges as $x \to x_m$. 
What remains is to prove the convergence of $\disp L_{5}(x) / ( \bar{S}^5 )_{xx} (x)$ 
as $x \to x_m$.
Noting $L_{5} = - ( 5 H A_{xxx} + 8 H_x A_{xx} )  ( A_{xx} \bar{S}^2 - H A_x^2 )$,
we factorize $\disp L_{5}(x) / ( \bar{S}^5 )_{xx} $ as
\begin{equation*}
\frac{5L_{5}} { ( \bar{S}^5 )_{xx} } = - \frac{ 5 H A_{xxx} + 8 H_x A_{xx} } { \bar{S} \bar{S}_{xx} + 4 \bar{S}_x^2 }
\frac{ A_{xx} \bar{S}^2 - H A_x^2 } { \bar{S}^3 }
= - \frac{ 5 H A_{xxx} + 8 H_x A_{xx} } { \bar{S} \bar{S}_{xx} + 4 \bar{S}_x^2 }L.
\end{equation*}
The RHS converges as $x \to x_m$ due to \eqref{eq242}, \eqref{eq243}, and \eqref{eq251}--\eqref{eq261}.
Consequently, $\bar{\rho} \in C^{3}(\bar{J})$ holds. The proof is complete.
\end{proof}

\begin{lem} \label{Lem23}
For an arbitrary section $A(> 0)$ and a density $\rho_L (< \rho_T(A))$,
the inequality
$\disp F(\rho_L, A) > F(\rho_R(\rho_L,A), A)$ holds, 
where $F$ and $\rho_R$ are defined in \eqref{eq215} and \eqref{eq227}, respectively. 
\end{lem}
\begin{proof}
Rewriting each side of \eqref{eq227}, we have
\begin{gather*}
\dfrac{Q^2}{A^2 \rho_L} + p(\rho_L) = \dfrac{Q^2}{A^2 \rho_T(A)} + p(\rho_T(A))
+ \int_{\rho_T(A)}^{\rho_L} \left( p'(\rho) - \dfrac{Q^2}{A^2 \rho^2} \right) d \rho,
\\
\dfrac{Q^2}{A^2 \rho_R} + p(\rho_R) = \dfrac{Q^2}{A^2 \rho_T(A)} + p(\rho_T(A))
+ \int_{\rho_T(A)}^{\rho_R} \left( p'(\rho) - \dfrac{Q^2}{A^2 \rho^2} \right) d \rho,
\end{gather*}
where $\rho_T$ is defined in \eqref{eq219}. 
These give
\begin{gather} \label{eq271}
\dfrac{1}{\rho_T(A)} \int^{\rho_T(A)}_{\rho_L} \left( \dfrac{Q^2}{A^2 \rho^2} - p'(\rho) \right) d \rho
=
\dfrac{1}{\rho_T(A)} \int_{\rho_T(A)}^{\rho_R} \left( p'(\rho) - \dfrac{Q^2}{A^2 \rho^2} \right) d \rho.
\end{gather}
Now we estimate both hand sides of \eqref{eq271}
noting $\disp \rho_L < \rho_T(A) < \rho_R = \rho_R(\rho_L, A)$.
From the LHS, we have
\begin{multline} \label{eq273}
\dfrac{1}{\rho_T(A)} 
\int^{\rho_T(A)}_{\rho_L} \left( \dfrac{Q^2}{A^2 \rho^2} - p'(\rho) \right) d \rho
< \int^{\rho_T(A)}_{\rho_L} \dfrac{1}{\rho} \left( \dfrac{Q^2}{A^2 \rho^2} - p'(\rho) \right)d \rho  
\\
= -\int^{\rho_T(A)}_{\rho_L} \dfrac{\partial F}{\partial \rho} (\rho, A) d \rho
= F(\rho_L, A) - F(\rho_T(A), A).
\end{multline}
The result from the RHS is
\begin{multline} \label{eq275}
\dfrac{1}{\rho_T(A)} \int_{\rho_T(A)}^{\rho_R} \left( p'(\rho) - \dfrac{Q^2}{A^2 \rho^2} \right) d \rho
> \int_{\rho_T(A)}^{\rho_R} \dfrac{1}{\rho} \left( p'(\rho) - \dfrac{Q^2}{A^2 \rho^2} \right) d \rho  
\\
= \int_{\rho_T(A)}^{\rho_R} \dfrac{\partial F}{\partial \rho} (\rho, A) d \rho
= F(\rho_R, A) - F(\rho_T(A), A).
\end{multline}
Substituting \eqref{eq273} and \eqref{eq275} in \eqref{eq271}, 
we have the desired inequality. 
\end{proof}
\begin{lem} \label{Lem24}
Assume \eqref{eq231}.
For a given $x_J \in (0, x_2)$, assume $\bar{\rho} \in C^1(\bar{I} \setminus \{ x_J \} )$ 
and $\bar{m}=Q/A$ satisfy \eqref{eq151}--\eqref{eq153}, \eqref{eq157}--\eqref{eq159} with $\bar{d}=x_{J}$, and the condition that the flow is subsonic in $x<0$, exactly transonic at $x=0$, supersonic in $(0, x_J)$, and subsonic again in $(x_J, x_2]$.
Set $\tilde{\rho}_+:=\bar{\rho}(x_2)$.
Then $\tilde{\rho}_+$ is a $C^1$-function of $A_J := A(x_J)$ 
satisfying 
\begin{gather}\label{eqlem24}
\frac{d \tilde{\rho}_+}{d A_J}(A_J) < 0, \quad 
\lim_{x_J \to + 0} \tilde{\rho}_+ (A(x_J)) = \rho_1 - 0, \quad 
\lim_{x_J \to x_{2}-0} \tilde{\rho}_+ (A(x_J)) = \rho_3 + 0.
\end{gather}
\end{lem}
\begin{proof}
First we investigate the regularity of $\disp \rho_L := \lim_{x \to x_J-0} \bar{\rho}(x)$ with respect to $A_{J}$.
We note that $\bar{\rho}$ satisfies $F(\bar{\rho}(x), A(x)) = F_0 =F(\rho_T(A_0), A_0)$ for $x < x_J$.
It is also seen that
$$\disp \frac{\partial F}{\partial {\rho}} (\rho_L, A_J) 
= \frac{1}{\rho_L} \left( p'({\rho_L}) - \frac{Q^2}{( A_J \rho_L )^2} \right)<0,
$$
because the flow is supersonic in $(0,x_J)$.
Then the implicit function theorem assures that $\rho_L(A_J)$ is a $C^1$-function with
\begin{gather} \label{eq277}
\frac{d \rho_L} {d A_J} = - \frac{\partial F / \partial A_J } {\partial F / \partial \rho_J }
= \frac{\rho_L} {A_J} \frac{\bar{M}_L^2} {1-\bar{M}_L^2} < 0,
\quad \text{where} \
\bar{M}_L := \frac{Q} {A_J \rho_L}  \frac{1}{\sqrt{p'({\rho_L})}}.
\end{gather}

Next we investigate the regularity of $\rho_R$, which is determined as a function of $\rho_L$ and $A_J$ through \eqref{eq227}. 
We see from $\bar{M}_L >1$ and the definition of $\rho_R$ that
\begin{gather*}
\frac{\partial}{\partial \rho_R} ( \text{RHS of } \eqref{eq227} )
= p'(\rho_R) - \frac{Q^2}{( A_J \rho_R )^2} = p'(\rho_R) - u_R^2 >0,
\quad \text{where} \ u_R := \frac{Q}{A_J \rho_R}.
\end{gather*}
The application of the implicit function theorem to \eqref{eq227} deduces that $\rho_R = \rho_R(\rho_L, A_J)$ is a $C^1$-function of $\rho_{L}$, and further
$\rho_R = \rho_R(\rho_L(A_J),A_J)$ is a $C^1$-function of $A_J$.
Then we differentiate \eqref{eq227} in $A_J$.
With the help of \eqref{eq277}, the LHS of  \eqref{eq227} yields 
\begin{align} \label{eq279}
\frac{d}{d A_J} ( \text{LHS of } (\ref{eq227}) )
&= \frac{\partial}{\partial A_J} ( \text{LHS of } (\ref{eq227}) )
+ \frac{\partial}{\partial \rho_L} ( \text{LHS of } (\ref{eq227}) ) \frac{d \rho_L} {d A_J}
\notag \\
&= - \dfrac{2 Q^2}{A_J^3 \rho_L} + \left( p'(\rho_L) - \dfrac{Q^2}{A_J^2 \rho_L^2} \right) 
    \frac{\rho_L} {A_J} \frac{\bar{M}_L^2} {1-\bar{M}_L^2}
 = - \frac{Q^2} {A_J^3 \rho_L}.
\end{align}
The RHS yields
\begin{align} \label{eq281}
\frac{d}{d A_J} ( \text{RHS of } (\ref{eq227}) )
&= - \dfrac{2 Q^2}{A_J^3 \rho_R} + \left( p'(\rho_R) - \dfrac{Q^2}{A_J^2 \rho_R^2} \right) 
    \frac{d \rho_R} {d A_J}.
\end{align}
Equating \eqref{eq279} and \eqref{eq281}, we obtain
\begin{equation} \label{eq283}
\frac{d \rho_R} {d A_J} 
= \left( \dfrac{2 Q^2}{A_J^3 \rho_R} - \dfrac{Q^2}{A_J^3 \rho_L} \right)  \frac{1}{ p'(\rho_R) - u_R^2 }. 
\end{equation}

From now on we show that $\tilde{\rho}_+$ is a $C^{1}$-function of $A_J$ and the inequality ${d \tilde{\rho}_+}/{d A_J}< 0$ in \eqref{eqlem24} holds.
We note that 
$\tilde{\rho}_+$, $\rho_R$, and $A_J$ are related as
\begin{equation} \label{eq285}
F(\rho_R, A_J) = F(\tilde{\rho}_+, A_+).
\end{equation}
It is seen from $A_+ > A_J$ that $\tilde{\rho}_+ > \rho_R$ and thus $1 > \bar{M}_R > \bar{M}(x_2)$.
From this and $\tilde{\rho}_+=\bar{\rho}(x_2)$, we observe that
$$\disp \frac{\partial F} {\partial \rho} (\tilde{\rho}_+, A_+) 
= \frac{1}{\tilde{\rho}_+} \left( p'(\tilde{\rho}_+) - \dfrac{Q^2}{A_+^2 \tilde{\rho}_+^2} \right) >0.
$$
Then applying the implicit function theorem to \eqref{eq285}, we deduce that
$\disp \tilde{\rho}_+ = \hat{\rho}_+(\rho_R, A_J)$ is a $C^1$-function with respect to $(\rho_R, A_J)$.
Therefore, the composition $\disp \tilde{\rho}_+ = \tilde{\rho}_+(\rho_R, A_J) = \hat{\rho}_+(\rho_R(\rho_L(A_J),A_J), A_J)$ 
is a $C^1$-function with respect to $A_J$.

To show the first equality in \eqref{eqlem24}, we differentiate \eqref{eq285} in $A_J$.
Using \eqref{eq283}, the result from the LHS of \eqref{eq285} is
\begin{align} \label{eq287}
\frac{d}{d A_J} ( \text{LHS of } (\ref{eq285}) )
&= \frac{\partial}{\partial A_J} ( \text{LHS of } (\ref{eq285}) )
+ \frac{\partial}{\partial \rho_R} ( \text{LHS of } (\ref{eq285}) ) \frac{d \rho_R} {d A_J}
\notag \\
&= - \dfrac{Q^2}{A_J^3 \rho_R^2} + \left( \frac{p'(\rho_R)} {\rho_R} - \dfrac{Q^2}{A_J^2 \rho_R^3} \right) 
\left( \dfrac{2 Q^2}{A_J^3 \rho_R} - \dfrac{Q^2}{A_J^3 \rho_L} \right)  \frac{1}{ p'(\rho_R) - u_R^2}
\notag \\
&= \frac{Q} {A_J^2 \rho_R} ( u_R - u_L ),
\quad \text{where} \  u_L = \frac{Q}{ A_J \rho_L }.
\end{align}
From the RHS, we have
\begin{align} \label{eq289}
\frac{d}{d A_J} ( \text{RHS of } (\ref{eq285}) )
= \left( \frac{p'(\tilde{\rho}_+)} {\tilde{\rho}_+} - \dfrac{Q^2}{A_J^2 \tilde{\rho}_+^3} \right) 
\frac{d \tilde{\rho}_+} {d A_J}.
\end{align}
It follows from \eqref{eq287} and \eqref{eq289} that 
\begin{equation*} 
\frac{d \tilde{\rho}_+} {d A_J} 
= \frac{Q \tilde{\rho}_+} {A_J^2 \rho_R}  \frac{u_R - u_L} { p'(\tilde{\rho}_+) - u_+^2 },
\quad \text{where} \  u_+ := \frac{Q}{A_+ \tilde{\rho}_+}.
\end{equation*}
Now using the entropy condition \eqref{eq159} and the fact that the flow is subsonic at $x = x_2$,
we arrive at ${d \tilde{\rho}_+}/{d A_J}< 0$.

We show the first equality in \eqref{eqlem24} 
by proving  the following three limits:
\begin{enumerate}[(i)]
\item $\disp \lim_{A_J \to A_0} \rho_L(A_J) = \rho_T(A_0) - 0$;
\hitem $\disp \lim_{A_J \to A_0} \rho_R(\rho_L(A_J), A_J) = \rho_T(A_0) + 0$;
\item $\disp \lim_{A_J \to A_0} F(\tilde{\rho}_+(A_J), A_+) = F( \rho_1, A_+ )$.
\end{enumerate}
Recall that  $F(\rho, A)$ defined in \eqref{eq215} is continuous in $(\rho, A)$,
and $F(\cdot, A_0)$ attains its minimum value $F_0$ only at $\rho_T(A_0)$.
From these and $F(\rho_L(A_J), A_J) = F_0$ for any $A_J$, we obtain (i).
Next recall that $G(\rho, A)$ defined in \eqref{defG1} is continuous in $(\rho, A)$, 
and $G(\rho, A)$ attains its minimum 
only at $\rho = \rho_T(A)$ for each $A$.
From these, \eqref{eq227}, and (i), we deduce (ii).
Furthermore, the continuity of $F(\rho,A)$ and (ii) imply (iii) as follows:
\[
\lim_{A_J \to A_0} F(\tilde{\rho}_+(A_J), A_+) = \lim_{A_J \to A_0} F(\rho_R(\rho_L(A_J), A_J), A_J)
= F(\rho_T(A_0), A_0) = F( \rho_1, A_+ ).
\]
Thus we conclude the second equality in \eqref{eqlem24}. Similarly, the last equality in \eqref{eqlem24} also holds.
The proof is complete.
\end{proof}

We are in a position to prove Theorem \ref{thm2.1}.

\begin{proof}[Proof of Theorem \ref{thm2.1}]
We consider the stationary solutions to \eqref{eq151}--\eqref{eq159} 
which are subsonic at $x=x_{1}$.
From \eqref{eq225}, we know that 
a stationary solution can undergo continuous transition 
from a subsonic state to a supersonic state only where $A_x$ vanishes.
%
%
If a piecewise smooth solution possesses a discontinuity,
the flow is supersonic on the upsteram side and is subsonic on the downstream side
due to the entropy condition \eqref{eq157}  (see also \eqref{eq229}).
From these considerations, the stationary solutions are classified into only three types.
\begin{enumerate}[(a)]
\item
The solution is subsonic entirely in $I$.

\item
The solution is subsonic in $x<0$, exactly transonic at $x=0$, and supersonic in $x>0$. 

\item
The solution is subsonic in $x<0$, exactly transonic at $x=0$, 
supersonic in $(0, x_J)$ for a certain $x_J \in (0, x_2)$, and subsonic again in $(x_J, x_2]$.
\end{enumerate}

We claim that the following is the necessary condition of the boundary value $\rho_{+}$ for the existence of the stationary solutions:
\begin{gather}\label{necessarybc1}
\rho_{+} \in \{\rho_{2}\} \cup (\rho_{3},\rho_{1}) \cup (\rho_{1},\infty),
\end{gather}
where $\rho_{1}$, $\rho_{2}$, and $\rho_{3}$ are defined just above Figure \ref{Fig01}, and satisfies $\rho_2 < \rho_3 < \rho_1$.
Let us verify the claim.
Recall that $A_{0}=A(0)$ and $F_0 = F(\rho_T(A_0), A_0)=F(\rho_{1},A_{+})=F(\rho_{2},A_{+})$.
The function $F(\bar{\rho}(x),A(x))$ remains constant where the solution is smooth as shown in \eqref{DF0}.
In the case (a), it must hold that
$F(\bar{\rho}(x),A(x)) \equiv F(\rho_+,A_+) > F_0$, and so $\bar{\rho}(x_2)=\rho_{+} > \rho_1$.
Similarly in the case (b), it is required that $F(\bar{\rho}(x),A(x)) \equiv F_0$, and so $\bar{\rho}(x_2) =\rho_{+}= \rho_2$.
In the case (c), there must hold that $F(\bar{\rho}(x),A(x)) = F_0$ for $x < x_J$, and 
$F(\bar{\rho}(x),A(x)) = F(\rho_+,A_+) $ for $x > x_J$.
Here $F_{0}=F(\bar{\rho}(x_J-0), A_J) > F(\bar{\rho}(x_J+0), A_J)=F(\rho_{+},A_{+})$ holds due to Lemma \ref{Lem23}.
From Lemma \ref{Lem24}, we have $\bar{\rho}(x_2)=\rho_{+} \in (\rho_3, \rho_1)$.
Consequently, we arrive at the necessary condition \eqref{necessarybc1}.

Since $\rho_2 < \rho_3 < \rho_1$,  it is clear that the solution is unique if ever exists.
What is left is to show the existence of stationary solutions satisfying each case of (i)--(iii) in Theorem \ref{thm2.1}.
First we treat the case $\rho_+ > \rho_1$ as in (i). We note that $F(\rho_+, A_+)>F(\rho_T(A_0), A_0)=F_{0}$.
Applying Lemma \ref{Lem20} with $J=I$, $A_{m}=A_{0}$, and $c=F(\rho_+, A_+)$, 
we see that the pair $(\bar{\rho},\bar{m}) \in C^{4}(\bar{I}) \times C^{4}(\bar{I})$ is the desired stationary solution to \eqref{eq151}--\eqref{eq155} as in (i). 
We next consider the case $\rho_+ = \rho_2$ as in (iii). 
By the definition, $F(\rho_+, A_+)=F(\rho_T(A_0), A_0)$ holds. 
Applying Lemmas \ref{Pro21} and \ref{Pro22} with $J=I$, $x_{m}=0$, $A_m=A_{0}$, and $c=F(\rho_+, A_+)$,
we have the stationary solution $(\bar{\rho},\bar{m}) \in C^{3}(\bar{I}) \times C^{4}(\bar{I})$ to \eqref{eq151}--\eqref{eq155}  as in (iii).
Finally, let us treat the case $\rho_+ \in (\rho_3, \rho_1)$ as in (ii).
We denote by $(\bar{\rho}^{*},\bar{m}^{*}) \in C^{3}(\bar{I}) \times C^{4}(\bar{I})$ the stationary solution as in (iii).
For any $x_J \in (0, x_2)$, we define a pair $(\bar{\rho}^{\#},\bar{m}^{\#}) \in C^{4}(\bar{J}) \times C^{4}(\bar{J})$, which  solves \eqref{eq151} and is subsonic on $(x_J,x_{2})$,
by applying Lemma \ref{Lem20} with $J=(x_J,x_{2})$, $A_{m}=A_{J}=A(x_{J})$, and $c=F(\rho_R(\bar{\rho}^{*}(x_{J}), A_J), A_J)$,
where $\rho_R$ is defined just above \eqref{eq227}, and $F(\rho_R(\bar{\rho}^{*}(x_{J}), A_J), A_J)>F(\rho_{T}(A_J), A_J)$ holds. 
Now we set
\begin{gather*}
(\bar{\rho},\bar{m})(x) =
\begin{cases}
(\bar{\rho}^{*},\bar{m}^{*})(x) \quad \text{for} \ x < x_{J},
\\
(\bar{\rho}^{\#},\bar{m}^{\#})(x) \quad \text{for} \ x > x_J.
\end{cases}
\end{gather*}
Recall that $\rho_R$ is defined in \eqref{eq227} so that the Rankine-Hugoniot condition \eqref{eq157} holds.
Furthermore, the entropy condition \eqref{eq159} also holds at $x=x_{J}$.
Hence, $(\bar{\rho},\bar{m})$ satisfies all the assumptions in Lemma \ref{Lem24}.
Using \eqref{eqlem24}, we can find a unique $x_J \in (0, x_2)$ such that $\bar{\rho}(x_{2})=\rho_{+} \in (\rho_{3},\rho_{1})$. Here $F(\rho_{+},A_{+})=F(\rho_R(\bar{\rho}^{*}(x_{J}), A_J), A_J)$ holds.
Consequently, we have the desired stationary solution $(\bar{\rho},\bar{m})$ as in (ii).
\end{proof}

\section{Subsonic problem} \label{Sec_sub}

In this section, we prove the asymptotic stability of stationary solutions $(\bar{\rho},\bar{m})$
which are entirely subsonic, i.e. $p'(\bar{\rho}(x)) > \bar{u}(x)^2$ for arbitrary $x \in I$, 
namely Theorem \ref{Thm11}.
%

%
If $A(x)$ is a smooth function, the system  \eqref{eq101} and \eqref{eq151} can be rewritten into 
\begin{subequations} \label{eq301}
\begin{gather} 
\rho_t + m_x + \dfrac{A_x}{A} m = 0,  
\\
m_t + \left( \dfrac{m^2}{\rho} + p(\rho) \right)_x + \dfrac{A_x}{A} \dfrac{m^2}{\rho} = 0,  
\end{gather}
\end{subequations}
and
\begin{subequations} \label{eq302}
\begin{gather} 
\bar{m}_x + \dfrac{A_x}{A} \bar{m} = 0,  
\\
\left( \dfrac{\bar{m}^2}{\bar{\rho}} + p(\bar{\rho}) \right)_x + \dfrac{A_x}{A} \dfrac{\bar{m}^2}{\bar{\rho}} = 0,  
\end{gather}
\end{subequations}
respectively.
We denote the perturbation from the stationary solution by
\begin{gather*}
\phi=\phi(t,x) := \rho(t,x) - \bar{\rho}(x), 
\quad
\omega=\omega(t,x) := m(t,x) - \bar{m}(x).
\end{gather*}
From the equations \eqref{eq301} and \eqref{eq302} and the conditions \eqref{eq104}, \eqref{eq105}, \eqref{eq108}, \eqref{eq153}, and \eqref{eq155}, we see that $f:=(\phi,\omega)^{\top}$ solves the system
\begin{subequations}\label{eq311*}
\begin{gather}\label{eq317}
\partial_t f + \hat{A}^1 \partial_x f + \hat{B} f = 0
\end{gather}
with the initial--boundary value
\begin{gather} \label{eq318}
f(0,x) = f^0(x) = 
\begin{bmatrix}
\phi^0(x)
\\ 
\omega^0(x)
\end{bmatrix}
:=
\begin{bmatrix}
\rho^0(x)-\bar{\rho}(x)
\\ 
m^0(x)-\bar{m}(x)
\end{bmatrix},
\\
\omega(t,x_{1})=0, \quad \phi(t,x_{2})=0, \quad t>0,
\label{eq314a}
\end{gather}
\end{subequations}
where
\begin{align*}
\hat{A}^{1}  (\rho,m) &:=
\begin{bmatrix}
0  &  1
\\[10pt]
S & 2m / \rho
\end{bmatrix},
\quad
\hat{B}(x,\rho,m) :=
\begin{bmatrix}
0 & A_x/A
\\  
b_{21} & b_{22}
\end{bmatrix},
\\
 S(\rho,m) &:= p'(\rho) - \frac{m^{2}}{\rho^{2}}, 
\quad
b_{22}(x,\rho,m):=
\frac{m-\bar{m}}{\rho} \frac{A_x}{A}
- \frac{ m + \bar{m} }{\rho^2} \bar{\rho}_x,
\\
b_{21}(x,\rho,m)&:=
\frac{\bar{m}^2}{\rho \bar{\rho}} \frac{A_x}{A}
+ \frac{\bar{m}^2}{\rho^2 \bar{\rho}^2} ( \rho + \bar{\rho} ) \bar{\rho}_x
+ \bar{\rho}_x \int_0^1 p'' ( \bar{\rho} + \theta (\rho-\bar{\rho}) ) d \theta.
\end{align*}
We require the compatibility conditions, which can be formally written as
$\disp \partial_t^j \omega(0,x_1) = 0$ and $\disp \partial_t^j \phi(0,x_2) = 0$ for $j=0,1$.
The precise formulations in terms of the initial value are
\begin{subequations} \label{comp1}
\begin{gather}
\omega^0(x_1) = 0, \quad 
(0 \ 1) \left( \hat{A}^1(\rho^0,m^0) \partial_x f^0(x_1) + \hat{B}(x_1,\rho^0,m^0) f^0(x_1) \right) = 0,
\label{comp1a}\\
\phi^0(x_2) = 0, \quad 
(1 \ 0) \left( \hat{A}^1(\rho^0,m^0) \partial_x f^0(x_2) + \hat{B}(x_2,\rho^0,m^0) f^0(x_2) \right) = 0.
\end{gather}
\end{subequations}
We seek the solution $f$ of the initial--boundary value problem \eqref{eq311*} in the following function space:
\begin{align*}
\mathfrak{X} (T):=  \left[ \bigcap_{j=0}^{2} C^j([0,T]; H^{2-j}(I)) \right]^{2} \ \text{for $T>0$}, \quad
\mathfrak{X} (\infty):= \left[ \bigcap_{j=0}^{2} C^j([0,\infty); H^{2-j}(I)) \right]^{2}. 
\end{align*}

We also use the following symmetrized form of the system \eqref{eq317}: 
\begin{gather}\label{eq319}
A^0 \partial_t f + A^1 \partial_x f + B f = 0,
\end{gather}
where 
\begin{align*}
A^0(x,\rho,m) &:= \dfrac{A}{\bar{\rho}} \, {\rm diag} (S,1), \quad
A^1(x,\rho,m) := A^0 \hat{A}^1
=
\dfrac{A}{\bar{\rho}}
\begin{bmatrix}
0  &  S
\\[10pt]
S & 2m / \rho
\end{bmatrix},
\\
B(x,\rho,m) &:= A^0 \hat{B}
=
\dfrac{A}{\bar{\rho}}
\begin{bmatrix}
0 & S A_x / A
\\  
b_{21} & b_{22}
\end{bmatrix}.
\end{align*}

Now we are ready to state the main theorem in this section, which immediately ensures Theorem \ref{Thm11}. 

\begin{thm} \label{Thm30}
Assume $A\in {C^3(\bar{I})}$, 
and that a stationary solution $(\bar{\rho}, \bar{m})  \in  {C^3(\bar{I}) \times C^3(\bar{I})}$ to (\ref{eq151})--(\ref{eq155}) exists and satisfies
\begin{gather}\label{spro1}
\inf_{x \in \bar{I}} \bar{\rho}(x) >0, \quad
\sup_{x \in \bar{I}} \bar{M}(x) < 1.
\end{gather}
Then there exists a constant $\delta_{0}>0$ such that if the initial value $f^{0} \in H^2(I) \times H^2(I)$ satisfies the compatibility condition \eqref{comp1} and $\| f^{0} \|_{H^2} \leq \delta_{0}$, 
then the initial--boundary value problem \eqref{eq311*} admits a unique time-global solution 
$f \in \mathfrak{X} (\infty)$.
Moreover, the solution $f$ decays in $H^{2}(I) \times H^2(I)$ exponentially fast  as $t \to \infty$.
\end{thm}
\begin{rem}
We remark that the assumption on $A$ in Theorem \ref{Thm30} (see also Theorem \ref{Thm11}) is much weaker than the condition (\ref{eq231}) on $A$ in Theorem \ref{thm2.1}, and thus Theorem \ref{Thm30} covers the case of more general configurations of nozzles. It is seen that the stationary solution in Theorem \ref{Thm30} also satisfies
\begin{gather}\label{spro2}
\inf_{x \in \bar{I}} \bar{u}(x) = \inf_{x \in \bar{I}} \frac{Q}{A\bar{\rho}}(x) >0, \quad
\inf_{x \in \bar{I}} \bar{S}(x) > 0.
\end{gather}
It is worth pointing out that we do not require any restrictions on the magnitude of $|A'(x)|$ and $Q$ nor on the boundary value $\rho_{+}$ as long as the assumptions of Theorem \ref{Thm30} hold. Accordinigly, the derivatives of the stationary solutions can be large.
\end{rem}

Theorem \ref{Thm30} is an immediate consequence of the time-local solvability in Lemma \ref{local3}
and an a-priori estimate in Proposition \ref{Pro31}.
\begin{lem}\label{local3}
Assume the same conditions as in Theorem \ref{Thm30} for $A$ and $(\bar{\rho},\bar{m})$.
Then there exist constants $\delta_{\sharp}>0$ and $T_{\sharp}>0$ such that if the initial value $f^{0} \in H^2(I) \times H^2(I)$ satisfies the compatibility condition \eqref{comp1} and $\| f^{0} \|_{H^2} \leq \delta_{\sharp}$, 
then the initial--boundary value problem \eqref{eq311*} admits a unique solution 
$f \in \mathfrak{X} (T_{\sharp})$.
Furthermore, if $f \in \mathfrak{X} (T)$ is a solution of \eqref{eq311*} {and the flow is subsonic at $x=x_{2}$ for all $t \in [0,T]$}, then the following time-integral of the trace of the highest order derivative makes sense{\rm:}
\begin{gather}\label{trace3}
\int_{0}^{T} |\partial_{t}^{2} f|^{2}(t,x_{2}) dt.
\end{gather}
\end{lem}
\begin{proof}
We will show this lemma in Appendix.
\end{proof}
In general, the time-integral in \eqref{trace3} is not valid for functions in $\mathfrak{X} (T)$. However, if $f$ solves the initial--boundary value problem \eqref{eq311*},
we can verify this integral. 

To present an a-priori estimate, we introduce 
\begin{gather*} 
N(T) := \sup_{t \in[0, T]} \| f(t,\cdot) \|_{H^2(I)}.
\end{gather*}
\begin{pro} \label{Pro31}
Assume the same conditions as in Theorem \ref{Thm30} for $A$ and $(\bar{\rho},\bar{m})$.
Suppose that the initial-boundary value problem \eqref{eq311*} has a solution 
$f \in \mathfrak{X} (T)$ for some $T>0$.
Then there exist positive constants $\delta_{1}$, $C_{1}$, and $\nu$ independent of $T$ such that if $N(T) \leq \delta_{1}$, 
then the following estimate holds{\rm :}
\begin{equation} \label{eqPro31}
\| f (t,\cdot) \|_{H^2} \le C_{1} e^{-\nu t} \| f^{0} \|_{H^2} \quad \text{for} \  t \in [0,T].
\end{equation}
\end{pro}

Assuming Proposition \ref{Pro31} is true, now we prove Theorem \ref{Thm30}.

\begin{proof}[Proof of Theorem \ref{Thm30}]
We set 
\begin{align*}
\delta_{0}&:=\frac{\tilde{\delta}_{0}}{\max\{1,C_{1}\}}, \quad
\tilde{\delta}_{0}:= \frac{1}{2} \min\{\delta_{\sharp},\delta_{1}\} , 
\\
T_{sup}&:=\sup\left\{ t>0 ; \text{the solution $f$ exists and satisfies $N(t) < 2\tilde{\delta}_{0} $} \right\}, 
\end{align*}
where $\delta_{\sharp}$, $C_{1}$, and $\delta_{1}$ are the same constants as in Lemma \ref{local3} and Proposition \ref{Pro31}.

First we show the time-global solvability of the problem \eqref{eq311*} for the initial data $f^{0}$ with $\|f^{0}\|_{H^{2}} \leq \delta_{0}$.
Lemma \ref{local3} ensures $T_{sup} >0$ because of $\| f^{0} \|_{H^2} \leq \delta_{0} \leq \delta_{\sharp}/2$.
It is enough to show $T_{sup}=\infty$.
Suppose otherwise, i.e. that $T_{sup}<\infty$ holds. 
We can apply Lemma~\ref{local3} by regarding $\tilde{t}_{0}:=T_{sup}-\min\{T_{\sharp},T_{sup}\}/2$ as an initial time and $f(\tilde{t}_{0},\cdot)$ as an initial value with $\|f(\tilde{t}_{0},\cdot)\|_{L^{2}}<2\tilde{\delta}_{0}\leq \delta_{\sharp}$, where $T_{\sharp}$ is the same constant as in Lemma~\ref{local3}. 
Therefore the solution $f$ exists until  $t=\tilde{t}_{0}+T_{\sharp}>T_{sup}$ and also $N(T_{sup}) = 2 \tilde{\delta}_{0} \leq \delta_{1}$ holds.
So we can use Proposition \ref{Pro31}, and then see from \eqref{eqPro31} that 
\begin{gather*}
 \|f(t,\cdot)\|_{H^{2}} \leq C_{1} \| f^{0} \|_{H^2} \leq \tilde{\delta}_{0}, \quad  t \in [0,T_{sup}].
\end{gather*}
We arrive at a contradiction, i.e. $2\tilde{\delta}_{0}=N(T_{sup})  \leq \tilde{\delta}_{0}$.
Thus $T_{sup}=\infty$ holds.

It follows from $T_{sup}=\infty$ that $N(T) = 2\tilde{\delta}_{0} \leq \delta_{1}$ for any $T>0$. 
Using Proposition \ref{Pro31}, we see that the solution $f$ decays in $H^{2}(I) \times H^2(I)$ exponentially fast  as $t \to \infty$. 
The proof is complete.
\end{proof}

It remains only to prove Proposition \ref{Pro31}. 
It is seen from \eqref{spro1}, \eqref{spro2}, and $N(T) \ll 1$ that the symmetric matrix $A^{0}$ is positive definite on $[0,T] \times \bar{I}$.
Then we define energy and dissipative functionals for $j=0,1,2$ as 
\begin{subequations}\label{eq321}
\begin{gather}
E(t) := ( E^0 + E^1 + E^2) (t), \quad D(t) := ( D^0 + D^1 + D^2) (t),
\\
E^j(t) := \int_{x_1}^{x_2}  ( \partial_t^j f )^\top A^0 \partial_t^j f (t,x) dx, 
\quad 
D^{j}(t):=|\partial_t^j f |^{2}(t,x_{2}).
\end{gather}
\end{subequations}
We also set
\begin{gather} \label{eq322}
v_- := \frac{1}{2} \sup_{x \in \bar{I}} \left( \bar{u}(x) - \bar{a}(x) \right)<0,
\quad
v_+ := \frac{1}{2} \inf_{x \in  \bar{I}} \left( \bar{u}(x) + \bar{a}(x) \right)>0,
\\ \label{eq323}
T_{-} : = \frac{x_2 - x_1}{| v_- |},
\quad
T_{+} : = \frac{x_2 - x_1}{| v_+ |},
\quad
T_0 : = T_{-} +  T_{+},
\end{gather}
where $\bar{a}$ is the speed of sound. 
Furthermore, we define the following sets $\varOmega_{\pm},\varOmega \subset \overline{\mathbb R_{+}\times I}$  to obtain dissipation through boundaries:
\begin{subequations} \label{eq324}
\begin{align}
\varOmega_- {(\tilde{t}_{0})} &:= \{ (t, x) \, | \, \tilde{t}_{0} \le t \le \tilde{t}_{0}+T_{-}, \ x_2 + v_- (t-\tilde{t}_{0}) \le x \le x_2 \},
\\
\varOmega_+ {(\tilde{t}_{0})} &:= \{ (t, x) \, | \,   \tilde{t}_{0} + T_{-} \le t \le \tilde{t}_{0}+T_{0}, \ x_1 + v_+ (  t  -\tilde{t}_{0}- T_{-} ) \le x \le x_2 \},
\\
\varOmega {(\tilde{t}_{0})} &:= \varOmega_-(\tilde{t}_{0}) \cup \varOmega_+(\tilde{t}_{0}),
\end{align}
\end{subequations}
where $\tilde{t}_{0}$ is any nonnegative constant.

We derive several fundamental inequalities on $N(T)$.
By Sobolev's inequality, it holds that
\begin{subequations}\label{norm-es1}
\begin{gather}
|f(t,x)|+|\partial_{x} f(t,x)| \leq CN(T) \quad \text{for} \ (t,x) \in [0,T] \times \bar{I}. 
\end{gather}
Furthermore, we can also estimate $\partial_{t} f$ by using \eqref{eq317} with the aid of  \eqref{spro1}, $(\bar{\rho},\bar{m})  \in C^2(\bar{I}) \times C^2(\bar{I})$, and $N(T) \ll 1$ as follows:
\begin{gather}
|\partial_{t} f(t,x)| \leq  C (|f(t,x)| + |\partial_x f(t,x)|)  \leq CN(T) \quad \text{for} \ (t,x) \in [0,T] \times \bar{I}.
\end{gather}
\end{subequations}
Here $C$ is a positive constant independent of $(t,x)$ and $T$.

We will prove the a priori estimate in Proposition \ref{Pro31} by combining the methods used in articles \cite{RT1974,RXX2013}. 
Rauch--Xie--Xin \cite{RXX2013} studied the stability of subsonic stationary solutions of the system \eqref{eq101} in a bounded interval. They introduced the {\it total mass from the inlet}  $\varPsi(t,x) := \int_{x_1}^x \rho(t,y) A(y) dy$ as the sole unknown variable,  and rewrote the system into a type of wave equation to apply standard techniques for the wave equation. They also proposed a nice energy functional essentially the same as $E(t)$ above. Indeed, multiplying \eqref{eq319} by $2 f^\top$ from the left, we have
\begin{gather} \label{eq325*}
\partial_t ( f^\top A^0 f ) + \partial_x ( f^\top A^1 f ) + f^\top R f = 0,
\quad  \text{where} \ R := - \partial_t A^0 - \partial_x A^1 +  B + B^\top.
\end{gather}
We will show in Proposition \ref{Pro30} that the matrix $R$ is bounded from above by $CN(T)$, which enables us to establish the asymptotic stability of stationary solutions for nozzles of arbitrary shapes. Furthermore, \cite{RXX2013} pointed out that a dissipative structure is obtained from boundaries. Indeed, integrating the second term on the LHS of \er{eq325*}  over $\bar{I}$ yields a dissipative term. However,  \cite{RXX2013} used resolvent estimates based on a proof by contradiction in applying the techniques for wave equations. To clarify the mechanism of energy dissipation through the boundary terms, we utilize an energy method developed by Rauch--Taylor \cite{RT1974}, which studied {\it linear} hyperbolic systems in bounded domains. They directly estimated an energy functional using boundary terms by tracing characteristic curves of hyperbolic systems. One of the benefits of adopting this method is that we assume only $H^2$-regularity for the initial perturbation $f^0$, whereas \cite{RXX2013} required $H^{17}$-regularity to apply the techniques for  wave equations. Our approach will also play a prominent role in dealing with transonic flows, as discussed in Sections 4 and 5.

\begin{pro} \label{Pro30}
Under the same assumptions as in Proposition \ref{Pro31}, if $N(T)$ is sufficiently small,
the following estimate holds for $(t,x) \in [0,T] \times \bar{I}${\rm :}
\begin{gather} \label{eq326}
|R(t,x)| \leq C N(T),
\end{gather}
where $C$ is a positive constant independent of $(t,x)$ and $T$.
\end{pro}
\begin{proof}
Each entry of the symmetric matrix $R$ is rewritten as follows: 
\begin{gather*}
R_{11} = - \frac{A}{\bar{\rho}} \partial_t S = \mathcal{O} ( |f_t| ), \quad
R_{12} = R_{21} = {\cal R}_{21} + \mathcal{O} (|f|+|f_x|), \quad
R_{22} = {\cal R}_{22}  + \mathcal{O} (|f|+|f_x|),
\end{gather*}
where
\begin{align*}
{\cal R}_{21}&:=- \partial_x \left( \frac{A}{\bar{\rho}} \bar{S} \right)
+ \frac{A}{\bar{\rho}} 
                              \left\{ \bar{S} \frac{A_x}{A} 
                                       + \left(   \frac{\bar{m}^2}{\bar{\rho}^2} \frac{A_x}{A}
                                       + \frac{2 \bar{m}^2}{\bar{\rho}^4} \bar{\rho} \bar{\rho}_x
                                       + p''(\bar{\rho}) \bar{\rho}_x  \right)
                              \right\},
\\
{\cal R}_{22}&:= - \partial_x \left( \frac{2A \bar{u}}{\bar{\rho}} \right)
- \frac{2A}{\bar{\rho}} \frac{2 \bar{m} \bar{\rho}_x} {\bar{\rho}^2}.
\end{align*}
It follows from \eqref{eq151}, \eqref{eq221}, and \eqref{eq242} that ${\cal R}_{21}={\cal R}_{22}=0$.
Therefore, using \eqref{norm-es1}, we deduce \eqref{eq326}.
\end{proof}

\subsection{Basic estimates}\label{Sec31}
In this subsection, we estimate the $L^2$-norm of $f = (\phi, \omega)^\top$.
The main results are Lemmas \ref{Lem32} and \ref{Lem35}, 
which imply the exponential decay of  the $L^2$ norm. 
%
\begin{lem} \label{Lem32}
Under the same assumptions as in Proposition \ref{Pro31}, if $N(T)$ is sufficiently small,
the following estimate holds for $t_{0} \in [0,T]$ and $t \in [t_{0},T]${\rm :}
\begin{gather} \label{eqLem32}
E^0(t) + c \int_{t_{0}}^t {D^{0}(\tau) }d\tau \le E^0(t_{0}) + C N(T) \int_{ t_{0}}^t E^0(\tau) d\tau,
\end{gather}
where $c$ and $C$ are positive constants independent of $t_{0}$ and $T$.
\end{lem}
\begin{proof}
We integrate \eqref{eq325*}  over $[t_{0},t] \times \bar{I}$ and use \eqref{eq326} to obtain
\begin{align} \label{eq327}
&\int_{x_1}^{x_2} f^\top A^0 f (t,x) dx
- \int_{t_{0}}^t  f^\top A^1 f (\tau, x_1) d\tau
+ \int_{t_{0}}^t  f^\top A^1 f (\tau, x_2) d\tau
\notag \\
& =\int_{x_1}^{x_2} f^\top A^0 f (t_{0},x) dx 
-  \int_{t_{0}}^t \!\! \int_{x_1}^{x_2} f^\top R f(\tau, x) dx d\tau
\notag \\
&\le
\int_{x_1}^{x_2} f^\top A^0 f (t_{0},x) dx 
+ C N(T) \int_{t_{0}}^t \!\! \int_{x_1}^{x_2} | f |^2 (\tau, x) dx d\tau.
\end{align}
Due to the boundary condition \eqref{eq314a} at $x=x_{1}$, the integrand of the second term on the LHS of \eqref{eq327} vanishes.
%
Using \eqref{eq314a}, \eqref{spro1}, \eqref{spro2}, and \eqref{norm-es1}, we estimate 
the integrand of the third term on the LHS of \eqref{eq327} from below as 
\begin{align} \label{eq329}
f^\top A^1 f (\tau, x_2) 
=  
\frac{2A}{\bar{\rho}}(x_2) u \omega^2 (\tau, x_2)
\ge 
c | f |^2 (\tau, x_2) =cD^{0}(\tau).
\end{align}
Substitution of \eqref{eq329} into \eqref{eq327} concludes \eqref{eqLem32}.
\end{proof}

As mentioned above, the time-integral of the boundary term at $x=x_{2}$ in \eqref{eqLem32} leads to energy dissipation over the entire interval $I$.
Lemmas \ref{Lem33}--\ref{Lem35} below detail this dissipative mechanism.
For that purpose, we use another symmetrized form of the system \eqref{eq317}.
We see from \eqref{spro1}, \eqref{spro2}, and $N(T) \ll 1$ that $\hat{A}^1$ is regular and
\begin{gather*}
(\hat{A}^1)^{-1} = \frac{1}{S} 
\begin{bmatrix}
-2u & 1
\\
S & 0
\end{bmatrix}.
\end{gather*}
In view of this, set $M^{1}:={\rm diag} (S,1)$ and $M^{0} :=  M^{1} (\hat{A}^1)^{-1}$,
where $M^{1}$ is positive definite.
We multiply \eqref{eq317} by $M^{0}$ from the left to have
\begin{gather}\label{eq333*}
M^0 \partial_t f + M^1 \partial_x f + M^{0} \hat{B} f=0.
\end{gather}

%
\begin{lem} \label{Lem33}
Under the same assumptions as in Proposition \ref{Pro31} with $T \geq T_{0}$, 
if $N(T)$ is sufficiently small,
the following estimate holds for $\tilde{t}_{0} \in [0,T-T_{0}]${\rm :}
\begin{gather} \label{eqLem33}
\iint_{\varOmega(\tilde{t}_{0})} | f |^2(\tau, x) dx d\tau
\le 
C \int_{\tilde{t}_{0}}^{\tilde{t}_{0}+T_0} { D^{0}(\tau) } d\tau,
\end{gather}
where $T_{0}$ and $\varOmega(\tilde{t}_{0})$ are defined in \eqref{eq323} and \eqref{eq324}, respectively, and $C$ is a positive constant independent of $\tilde{t}_{0}$ and $T$.
\end{lem}
\begin{proof}
Multiplying \eqref{eq333*} by $2 f^\top$ from the left yields
\begin{gather}  \label{eq333}
\partial_t \left( f^\top M^{0} f \right)
+\partial_x \left( f^\top M^{1} f  \right)
=
f^\top M^{2} f, \quad  \text{where} \  M^{2}:= \partial_t M^{0} + \partial_x M^{1}  - 2 M^{0} \hat{B}.
\end{gather}
From \eqref{spro1}, \eqref{spro2}, $(\bar{\rho},\bar{m})  \in C^2(\bar{I}) \times C^2(\bar{I})$, and $N(T) \ll 1$, it is seen that  
$\partial_x M^{1}$, $\partial_t M^{0}$, $M^{0}$, and $\hat{B}$ are all bounded uniformly in $(t,x)$, and so is $M^{2}$.
%

We integrate the product of $e^{kx}$ and \eqref{eq333} over $\varOmega(\tilde{t}_{0})$ defined in \eqref{eq324}, where $k>0$ is a constant to be determined later, and then apply Green's theorem to obtain
\begin{align} \label{eq335}
&\int_{\tilde{t}_{0}}^{\tilde{t}_{0}+T_0} e^{k x_2}  f^\top M^{1} f  (\tau, x_2) d\tau 
+ \int_{\tilde{t}_{0}}^{\tilde{t}_{0}+T_{-}} e^{k (x_{2}+v_- (\tau-\tilde{t}_{0}))} f^\top (  v_- M^{0} -M^{1} ) f (\tau, x_{2}+v_- (\tau-\tilde{t}_{0})) d\tau
\notag \\
& \quad + \int_{\tilde{t}_{0}+T_{-}}^{\tilde{t}_{0}+T_0} e^{k (x_{1}+v_+ (\tau -\tilde{t}_{0}- T_0) )} 
     f^\top ( v_+ M^{0} -M^{1} ) f (\tau, x_{1}+v_+ ( \tau - \tilde{t}_{0} - T_0 )  ) d\tau
\notag \\
&= k  \iint_{\varOmega(\tilde{t}_{0})} e^{kx} f^\top M^{1} f  (\tau, x) d\tau dx+\iint_{\varOmega(\tilde{t}_{0})} e^{kx} f^\top M^{2} f (\tau,x) d\tau dx 
\notag \\
&\geq C^{-1} \iint_{\varOmega(\tilde{t}_{0})} e^{kx_{1}} |f|^{2} (\tau,x) d\tau dx.
\end{align}
The last inequality holds for a sufficiently large $k$ because $M^1$ is positive definite, 
and $M^2$ is bounded.

We claim that the second and third terms on the LHS of \eqref{eq335} are nonpositive.
To verify this claim, it suffices to show that $ M^{1} - v M^{0} $ is positive definite for $v \in [v_-, v_+ ]$ if $N(T) \ll 1$.
We rewrite it as $ M^{1} - v M^{0} =  M^{3}[v] + \mathcal{O}(N(T))$, where
\begin{gather*}
M^{3}[v](x) := 
\begin{bmatrix}
\bar{S}(x) + 2 v \bar{u}(x) & -v
\\
-v & 1
\end{bmatrix}.
\end{gather*}
It is seen that $M_{22}^{3}[v](x)=1$ and 
\begin{gather*}
{\rm det} \ M^{3}[v](x) 
= \bar{a}(x)^2 - ( \bar{u}(x) - v )^2 > 0, \quad x \in [x_{1},x_{2}], \ v \in [v_-, v_+ ],
\end{gather*}
where we have used the definitions of $v_-$ and $v_+$ in \eqref{eq322} in deriving the last inequality.
Hence, $M^{3}$ is positive definite and so is $ M^{1} - v M^{0}$. The claim is valid. 
From this claim and \eqref{eq335}, we arrive at the desired estimate \eqref{eqLem33}.
\end{proof}

We rewrite \eqref{eq317} into a diagonalized form as follows.
The eigenvalues and left eigenvectors of the matrix $\hat{A}^1$ are $\lambda_\pm = u \pm a$ and $( S, u \pm a )$, respectively, 
%
where $a$ is a speed of sound. 
We note that $\lambda_{-}<0<\lambda_{+}$ holds for the case of the subsonic flow.
Set
\begin{gather*}
e_\pm:=( \pm a - u,  1) \parallel ( S, u \pm a ), \quad r_\pm := e_\pm f = ( \pm a - u ) \phi + \omega.
\end{gather*}
Multiplying \eqref{eq317} by $e_\pm$ from the left,
we have
\begin{gather} \label{eq339}
\partial_t r_\pm + \lambda_\pm \partial_x r_\pm
= {\cal H}_{\pm} f, \quad 
{\cal H}_{\pm} f :=-  e_\pm \hat{B} f + ( \partial_t e_\pm ) f + \lambda_\pm ( \partial_x e_\pm ) f.
\end{gather}
\begin{lem} \label{Lem34}
Under the same assumptions as in Proposition \ref{Pro31} with $T \geq T_{0}$, 
if $N(T)$ is sufficiently small, 
the following estimate holds for $\tilde{t}_{0} \in [0,T-T_{0}]${\rm :}
\begin{gather} \label{eqLem34}
E^0(\tilde{t}_{0}+T_{-})
\le 
C \int_{\tilde{t}_{0}}^{\tilde{t}_{0}+T_0} { D^{0}(\tau) } d\tau,
\end{gather}
where $T_{-}$ and $T_{0}$ are defined in \eqref{eq323}, and $C$ is a positive constant independent of $\tilde{t}_{0}$ and $T$.
\end{lem}
\begin{proof}
Multiplying \eqref{eq339} by $2 r_\pm$ gives
\begin{gather} \label{eq341}
   \partial_t \left( r_\pm^2 \right) 
+ \partial_x \left( \lambda_\pm r_\pm^2 \right)
= ( \partial_x \lambda_\pm ) r_\pm^2
+ 2 r_\pm {\cal H}_{\pm} f.
\end{gather}
We integrate \eqref{eq341} with minus subscripts over $\varOmega_-(\tilde{t}_{0})$ and then apply Green's theorem to obtain
\begin{align} 
& \int_{x_1}^{x_2} r_-^2 (\tilde{t}_{0}+T_{-},x) dx
+ \int_{\tilde{t}_{0}}^{\tilde{t}_{0}+T_{-}} ( v_- - \lambda_- )  r_-^2 ( \tau, x_2 + v_- (\tau -\tilde{t}_{0}) ) d\tau
\notag \\
&= - \int_{\tilde{t}_{0}}^{\tilde{t}_{0}+T_{-}} \lambda_- r_-^2 (\tau, x_2) d\tau + \iint_{\varOmega_-(\tilde{t}_{0})} \left\{ ( \partial_x \lambda_- ) r_-^2
+ 2 r_- {\cal H}_{-} f \right\}(\tau, x) dx d\tau
\notag \\
&\leq 
C \int_{\tilde{t}_{0}}^{\tilde{t}_{0}+T_{-}} r_-^2(\tau, x_2) d\tau 
+C \iint_{\varOmega_-(\tilde{t}_{0})} |f|^{2}(\tau, x) dx d\tau,
\label{eq343}
\end{align}
where we have also used \eqref{spro1}, \eqref{spro2}, $(\bar{\rho},\bar{m})  \in C^2(\bar{I}) \times C^2(\bar{I})$, and $N(T) \ll 1$ in deriving the last inequality. 
Recalling the definition of $v_-$ in \eqref{eq322},
we see that the second term on the LHS is nonnegative if $N(T)\ll 1$.

%
Similarly, integrating \eqref{eq341} with plus subscripts over $\varOmega_+$ leads to
\begin{align}\label{eq345}
&\int_{x_1}^{x_2} r_+^2 (\tilde{t}_{0}+T_{-},x) dx
+ \int_{\tilde{t}_{0}+T_{-}}^{\tilde{t}_{0}+T_0} ( \lambda_+ - v_+ )  r_+^2 ( \tau, x_1 + v_+ ( \tau - \tilde{t}_{0} - T_{-} ) ) d\tau
\notag\\
&= \int_{\tilde{t}_{0}+T_{-}}^{\tilde{t}_{0}+T_0} \lambda_+ r_-^2 (\tau, x_2) d\tau
- \iint_{\varOmega_+(\tilde{t}_{0})} \left\{( \partial_x \lambda_+ ) r_+^2 + 2 r_+ {\cal H}_{+} f \right\}(\tau, x) dx d\tau
\notag\\
& \leq C \int_{\tilde{t}_{0}+T_{-}}^{\tilde{t}_{0}+T_0} r_+^2(\tau, x_2) d\tau + C \iint_{\varOmega_+(\tilde{t}_{0})} | f |^2(\tau, x) dx d\tau.
\end{align}
%
%
We note that the second term on the LHS is nonnegative if $N(T)\ll 1$.

If $N(T)$ is small enough, 
there exist constants $c,C>0$ independent of $\tilde{t}_{0}$ and $T$ such that
\begin{gather*}
c (r_{-}^{2}+r_{+}^{2})(t,x) \leq |f|^{2}(t,x) \leq C (r_{-}^{2}+r_{+}^{2})(t,x) \quad \text{for $(t,x) \in [0,T] \times \bar{I}$}
\end{gather*}
holds. Now adding \eqref{eq343} and \eqref{eq345} implies
\begin{gather} \label{eq347}
E^0(\tilde{t}_{0}+T_{-})
\le 
C \int_{\tilde{t}_{0}}^{\tilde{t}_{0}+T_0} { D^{0} (\tau)} d\tau
+
C \iint_{\varOmega(\tilde{t}_{0})} | f |^2(\tau, x) dx d\tau.
\end{gather}
We can deduce \eqref{eqLem34} by multiplying \eqref{eq347} by $\varepsilon >0$, adding the result to \eqref{eqLem33}, and then letting $\ve \ll 1$.
\end{proof}

\begin{lem} \label{Lem35}
Under the same assumptions as in Proposition \ref{Pro31} with $T \geq T_{0}$, 
if $N(T)$ is sufficiently small, 
the following estimate holds for $\tilde{t}_{0} \in [0,T-T_{0}]$ and $t \in [\tilde{t}_{0}+T_{0},T]${\rm :}
\begin{gather} \label{eqLem35}
(1+c) E^0(t) + c \int_{\tilde{t}_{0}}^t D^{0}(\tau) d\tau
\le E^0(\tilde{t}_{0}) + C N(T) \int_{\tilde{t}_{0}}^t E^0(\tau) d\tau,
\end{gather}
where $c$ and $C$ are positive constants independent of $\tilde{t}_{0}$ and $T$.
\end{lem}
\begin{proof}
Using \eqref{eqLem32} with $t_{0}=\tilde{t}_{0}+T_{-}$ and $t \in [\tilde{t}_{0}+T_{0},T]$, we have
\begin{gather*} 
E^0(t) + c \int_{\tilde{t}_{0}+T_{-}}^t  { D^{0}(\tau)} d\tau 
\le E^0(\tilde{t}_{0}+T_{-}) + C N(T) \int_{\tilde{t}_{0}+T_{-}}^t E^0(\tau) d\tau.
\end{gather*}
From this and \eqref{eqLem34}, it holds that
\begin{gather} \label{eq353}
E^0(t) \le C \int_{\tilde{t}_{0}}^t { D^{0}(\tau)}  d\tau + C N(T) \int_{\tilde{t}_{0}+T_{-}}^t E^0(\tau) d\tau.
\end{gather}
Multiplying \eqref{eq353} by $\varepsilon>0$, adding the result to \eqref{eqLem32} with $t_{0}=\tilde{t}_{0}$ and $t \in [\tilde{t}_{0}+T_0, T]$, and then letting $\varepsilon \ll 1$, we conclude \eqref{eqLem35}.
\end{proof}

\subsection{Estimates on the higher order derivatives} \label{Sec32}
In this subsection, we will derive the $L^{2}$-estimate of the higher order derivatives of $f$ in a similar spirit of the $L^{2}$-estimate of $f$. To this end, we first derive the estimates for the time-derivatives of the solution. 
The reason why we work with the time-derivative is that it is more convenient when evaluating the boundary terms that arise during the course of our energy method.
After establishing the time-derivative estimates, we will derive estimates of the spatial-derivatives of the solution using the relation 
\begin{gather}
c \sum_{i=0}^j |\partial_x^i f|(t,x)
\le \sum_{i=0}^j |\partial_t^i f|(t,x)
\le C \sum_{i=0}^j |\partial_x^i f|(t,x), 
\label{ajs2}
\end{gather}
for $(t,x) \in [0,T] \times \bar{I}$ and $j=1,2$,  
where $c$ and $C$ are positive constants independent of $(t,x)$ and $T$.
%
%
%
%
%
%
This estimate will be shown in Lemma \ref{Lem3.55} below.

To make the computation in this subsection rigorous, 
we need to use a mollifier with respect to the time variable $t$ 
due to the insufficiency of the regularity of the solution.
However, we omit that argument since it is standard.
We apply the operator $\partial_{t}^{j}$ for $j=1,2$ to \eqref{eq317}  to obtain
\begin{gather}\label{eq361*}
\partial_{t} (\partial_{t}^{j} f) + \hat{A}^{1} \partial_{x} (\partial_{t}^{j} f) + \hat{B} (\partial_{t}^{j} f)
={\cal F}^{j}:=- [\partial_{t}^{j},\hat{A}^{1}\partial_{x}] f - [\partial_{t}^{j},\hat{B}] f,
\end{gather}
where $[\cal T,\cal S]:=\cal T\cal S-\cal S\cal T$ is a commutator.
Multiplying this by $A^{0}$ from the left gives
\begin{gather}
\label{eq361}
A^{0} \partial_{t} (\partial_{t}^{j} f) + A^{1} \partial_{x} (\partial_{t}^{j} f) + B (\partial_{t}^{j} f)
=A^{0}{\cal F}^{j}. 
\end{gather}
It also follows from \eqref{eq314a} that
\begin{gather} \label{eq367}
 \partial_t^j \omega ( t, x_1 ) = 0,
\quad
 \partial_t^j \phi ( t, x_2 ) = 0.
\end{gather}

Now we derive the relation \eqref{ajs2} and also estimate the remaining term ${\cal F}^{j}$.
\begin{lem}\label{Lem3.55}
Under the same assumptions as in Proposition \ref{Pro31}, 
if $N(T)$ is sufficiently small, 
the inequality \eqref{ajs2} and the following estimates hold for $j=1,2${\rm :}
\begin{gather}\label{calF1}
|{\cal F}^{j}|(t,x)  \leq CN(T) \sum_{i=0}^j |\partial_t^i f|(t,x),
\end{gather}
where $C$ is a positive constant independent of $(t,x)$ and $T$.
\end{lem}
\begin{proof}
We show \eqref{ajs2} with $j=1$.
It follows from \eqref{spro1}, \eqref{spro2}, $(\bar{\rho},\bar{m})  \in C^2(\bar{I}) \times C^2(\bar{I})$, and $N(T) \ll 1$ that $A^1$ is regular and also $A^1$, $(A^1)^{-1}$, and $B$ are bounded uniformly in $(t,x)$. 
Multiplying \eqref{eq319} by $(A^{1})^{-1}$ from left, and then estimating $\partial_{x} f$ in the resultant equality, we arrive at the first inequality in \eqref{ajs2} with $j=1$. Similarly, the second inequality in \eqref{ajs2} holds for  $j=1$. 

Let us consider \eqref{ajs2} with $j=2$.
By a direct computation using  \eqref{norm-es1} and \eqref{ajs2} with $j=1$,
we see that $A^{0}{\cal F}^{1}$ in \eqref{eq361}  is estimated as $|A^{0}{\cal F}^{1}| \leq C(|f|+|\partial_{t}f|)$.
We multiply \eqref{eq361} with $j=1$ by $(A^{1})^{-1}$ from left, and then estimate $\partial_{xt} f$ as $|\partial_{xt} f|\leq C(|f|+|\partial_{t} f|+|\partial_{tt} f|)$.
On the other hand, applying $\partial_{x}$ to \eqref{eq319}, multiplying the resultant equality by $(A^{1})^{-1}$, and estimating $\partial_{xx} f$ similarly above, we can also obtain $|\partial_{xx} f| \leq C(|f|+|\partial_{x} f|+|\partial_{xt} f|)$.
Combining these inequalities and \eqref{ajs2} with $j=1$, we conclude the first inequality in \eqref{ajs2} with $j=2$.
Similarly, the second inequality in \eqref{ajs2} holds for  $j=2$. 

Finally, it is straightforward to see \eqref{calF1} by a direct computation with the aid of  \eqref{spro1}, \eqref{spro2}, \eqref{norm-es1}, \eqref{ajs2}, $(\bar{\rho},\bar{m})  \in C^2(\bar{I}) \times C^2(\bar{I})$, and $N(T) \ll 1$.
\end{proof}

In Lemmas \ref{Lem36}--\ref{Lem39} below,
we will estimate the $L^2$ norm of $\partial_{t}^{j} f$ in much the same way as in Lemmas \ref{Lem32}--\ref{Lem35}, but with \eqref{eq361*}, \eqref{eq361}, and \eqref{eq367} in place of \eqref{eq317}, \eqref{eq319}, and \eqref{eq314a}.
Lemmas \ref{Lem36}--\ref{Lem39} are parallel to Lemmas \ref{Lem32}--\ref{Lem35}, respectively. 

%
\begin{lem} \label{Lem36}
Under the same assumptions as in Proposition \ref{Pro31}, 
if $N(T)$ is sufficiently small, 
the following estimate holds for $t_{0} \in [0,T]$, $t \in [t_{0},T]$, and $j=1,2${\rm :}
\begin{gather} \label{eqLem36}
E^{j}(t) + c \int_{t_{0}}^t { D^{j}(\tau) } d\tau 
\le E^{j}(t_{0}) + C N(T) \int_{t_{0}}^t E(\tau) d\tau,
\end{gather}
where $c$ and $C$ are positive constants independent of $t_{0}$ and $T$, and 
the time-integral of the trace $D^{2}(\tau)=|\partial_{t}^{2}f(\tau)|^{2}$ is validated in Lemma \ref{local3}.
\end{lem}

\begin{proof}
Multiplying \eqref{eq361} by $2 \partial_{t}^{j} f^\top$ from the left, we have
\begin{gather} \label{eq325}
\partial_t ( \partial_{t}^{j} f^\top A^0 \partial_{t}^{j} f ) + \partial_x ( \partial_{t}^{j} f^\top A^1 \partial_{t}^{j} f ) + \partial_{t}^{j} f^\top R \partial_{t}^{j} f = 2 \partial_{t}^{j} f^\top A^{0}{\cal F}^{j},
\end{gather}
which is parallel to \eqref{eq325*}. Recall the estimate \eqref{calF1} of ${\cal F}^{j}$ and the boundary value \eqref{eq367}.
It follows from \eqref{spro1}, \eqref{spro2}, $(\bar{\rho},\bar{m})  \in C^2(\bar{I}) \times C^2(\bar{I})$, and $N(T) \ll 1$ that  $A^0$ is bounded uniformly in $(t,x)$. 
Following the proof of Lemma \ref{Lem32}, we can conclude \eqref{eqLem36}.
\end{proof}

\begin{lem} \label{Lem37}
Under the same assumptions as in Proposition \ref{Pro31} with $T \geq T_{0}$, 
if $N(T)$ is sufficiently small, 
the following estimate holds for $\tilde{t}_{0} \in [0,T-T_{0}]$ and $j=1,2${\rm :}
\begin{gather} \label{eqLem37}
\iint_{\varOmega(\tilde{t}_{0})} | \partial_t^j f |^2 (\tau,x) d\tau dx 
\le 
C \int_{\tilde{t}_{0}}^{\tilde{t}_{0}+T_0} { D^{j}(\tau)}  d\tau
 + C N(T) \int_{\tilde{t}_{0}}^{\tilde{t}_{0}+T_{0}} E(\tau) d\tau,
\end{gather}
where $T_{0}$ and $\varOmega(\tilde{t}_{0})$ are defined in \eqref{eq323} and \eqref{eq324}, respectively, and $C$ is a positive constant independent of $\tilde{t}_{0}$ and $T$.
\end{lem}
\begin{proof}
%
We multiply \eqref{eq361*} by $ M^{0}$ defined in the proof of Lemma \ref{Lem33} from the left to have
\begin{gather*}
M^{0} \partial_{t} (\partial_{t}^{j} f) + M^{1} \partial_{x} (\partial_{t}^{j} f) + M^{0}\hat{B} (\partial_{t}^{j} f)
=M^{0}{\cal F}^{j},
\end{gather*}
which corresponds to \eqref{eq333*}.
Multiplying this by $2 \partial_{t}^{j} f^\top $ from the left gives
\begin{gather}  \label{eq373}
\partial_t \left( \partial_{t}^{j} f^\top M^{0} \partial_{t}^{j} f \right)
+\partial_x \left( \partial_{t}^{j} f^\top M^{1}  \partial_{t}^{j} f  \right)
=
\partial_{t}^{j} f^\top M^{2} \partial_{t}^{j}f + 2 \partial_{t}^{j} f ^\top M^{0} {\cal F}^{j},
\end{gather}
which is parallel to \eqref{eq333}.
Recall \eqref{eq367}, \eqref{calF1}, and that $M^{0}$ is bounded uniformly in $(t,x)$.
Then following the proof of Lemma \ref{Lem33}, we can arrive at \eqref{eqLem37}.
\end{proof}

\begin{lem} \label{Lem38}
Under the same assumptions as in Proposition \ref{Pro31} with $T \geq T_{0}$, 
if $N(T)$ is sufficiently small, 
the following estimate holds for $\tilde{t}_{0} \in [0,T-T_{0}]$ and $j=1,2${\rm :}
\begin{gather} \label{eqLem38}
E^{j}(\tilde{t}_{0}+T_{-})
\le 
C \int_{\tilde{t}_{0}}^{\tilde{t}_{0}+T_0} { D^{j}(\tau)}  d\tau
+ C N(T) \int_{\tilde{t}_{0}}^{\tilde{t}_{0}+T_{0}} E(\tau) d\tau,
\end{gather}
where $T_{-}$ and $T_{0}$ are defined in \eqref{eq323}, and $C$ is a positive constant independent of $\tilde{t}_{0}$ and $T$.
\end{lem}
\begin{proof}
Let $\disp r_{\pm}^{j} := e_\pm \, \partial_t^j f$, where $e_\pm = (\pm a -u , 1)$.
We multiply \eqref{eq361*} by $e_\pm$ from the left to obtain
$\partial_t r_{\pm}^{j} + \lambda_\pm \partial_x r_{\pm}^{j}
= {\cal H}_{\pm} \partial_{t}^{j} f + e_\pm {\cal F}^{j}$  
corresponding to \eqref{eq339}. 
Multiplying this by $2 r_\pm^{j}$ gives
\begin{gather}\label{eq391}
   \partial_t \left( (r_\pm^{j})^2 \right) 
+ \partial_x \left( \lambda_\pm (r_\pm^{j})^2 \right)
= ( \partial_x \lambda_\pm ) (r_\pm^{j})^2
+ 2 r_\pm^{j} {\cal H}_{\pm} \partial_{t}^{j} f +  2 r_\pm^{j} e_\pm {\cal F}^{j},
\end{gather}
which is parallel to \eqref{eq341}.
By the same method as in Lemma \ref{Lem34} with the aid of \eqref{calF1}, we can obtain \eqref{eqLem38}.
\end{proof}
\begin{lem} \label{Lem39}
Under the same assumptions as in Proposition \ref{Pro31} with $T \geq T_{0}$, 
if $N(T)$ is sufficiently small, 
the following estimate holds for $\tilde{t}_{0} \in [0,T-T_{0}]$, $t \in [\tilde{t}_{0}+T_{0},T]$, and $j=1,2${\rm :}
\begin{gather} \label{eqLem39}
(1+c) E^{j}(t) + c \int_{\tilde{t}_{0}}^t D^{j} (\tau) d\tau
\le E^{j}(\tilde{t}_{0}) + C N(T) \int_{\tilde{t}_{0}}^t E(\tau) d\tau,
\end{gather}
where $c$ and $C$ are positive constants independent of $\tilde{t}_{0}$ and $T$.
\end{lem}
\begin{proof}
Repeating the proof of Lemma \ref{Lem35}, but replacing \eqref{eqLem32} and \eqref{eqLem34} by \eqref{eqLem36} and \eqref{eqLem38} respectively, we can conclude  \eqref{eqLem39}.
\end{proof}

\subsection{Completion of the a priori estimate} 
This subsection is devoted to the proof of Proposition \ref{Pro31}.
\begin{proof}[Proof of Proposition \ref{Pro31}]
For $t_{0} \in [0,T]$ and $t \in [t_{0},T]$, 
we sum up \eqref{eqLem32} and \eqref{eqLem36} and then apply Gronwall's inequality to the resultant inequality to obtain
 \begin{gather}\label{eq355}
E(t) \leq E(t_{0})e^{C_{1}N(T)(t-t_{0})}  \quad \text{for} \  t \in [t_{0},T],
\end{gather}
where  $C_{1}>0$ is a constant independent of $t_{0}$ and $T$.
Furthermore, for $\tilde{t}_{0} \in [0,T-T_0]$ and $t \in [\tilde{t}_{0}+T_0,T]$, 
we sum up \eqref{eqLem35} and \eqref{eqLem39} and then apply Gronwall's inequality to the result to obtain
\begin{gather} \label{eq356}
E (t) 
\le \frac{1}{ 1 + c_1 } \left( E (\tilde{t}_{0}) + C_{1}N(T) \int_{\tilde{t}_{0}}^{\tilde{t}_{0}+T_0} \!\! E (\tau) d\tau \right) \ 
e^{C_{1}N(T)(t-\tilde{t}_{0}-T_{0})}
\quad \text{for} \ t \in [\tilde{t}_{0}+T_0, T],
\end{gather}
where $C_{1}>0$ can be the same constant as in \eqref{eq355}, 
and $c_1>0$ is a constant  independent of $\tilde{t}_{0}$ and $T$.

We focus on $t \in [0, \min\{ T_0, T \}] \subset [0,T_{0}]$. Using \eqref{eq355} with $t_{0}=0$ and letting $N(T)$ be sufficiently small according to $C_{1}$ and $T_{0}$ but independent of $T$, we have
\begin{gather}\label{eq355*}
E (t) 
\le 2 E(0) \quad \text{for} \ t \in [0, \min\{ T_0, T \}].
\end{gather}

If $T>T_{0}$, we next treat $t \in [T_0, \min\{ 2T_0, T \}] \subset [T_{0},2T_{0}]$.
Substituting \eqref{eq355} with $t_{0}=0$ into the integral in \eqref{eq356} with $\tilde{t}_{0}=0$ leads to
\begin{gather} \label{eq357}
E (t) 
\le \frac{E(0)}{ 1 + c_1 } e^{2C_{1}N(T)T_{0}}
\quad \text{for} \ t \in [T_0, \min\{ 2T_0, T \}] .
\end{gather}
Letting $N(T) \leq \tilde{\delta}_{0}:=(2C_{1}T_{0})^{-1} \log(1+c_{1}/2)$, 
we arrive at
\begin{gather*} 
E (t) 
\le \frac{ 1 + c_1 / 2 }{ 1 + c_1 } E (0)
\quad  \text{for} \ t \in [T_0, \min\{ 2T_0, T \}].
\end{gather*}
If $T > 2T_{0}$, we can repeat the same argument as in the case $t \in [T_0, \min\{ 2T_0, T \}]$.
Indeed, in the same way as in the derivation of \eqref{eq357} but replacing $t_{0}=\tilde{t}_{0}=0$ by $t_{0}=\tilde{t}_{0}=T_{0}$, we deduce
\begin{gather*} 
E (t) 
\le \frac{E(T_{0})}{ 1 + c_1 } e^{2C_{1}N(T)T_{0}}
\quad \text{for} \ t \in [2T_0, \min\{ 3T_0, T \}],
\end{gather*}
where $c_{1}$ and $C_{1}$ are the same constants as in \eqref{eq357}.
From this and $N(T) \leq \tilde{\delta}_{0}$, it holds that
\begin{gather*} 
E (t) 
\le \frac{ 1 + c_1 / 2 }{ 1 + c_1 } E (T_{0})
\quad  \text{for} \ t \in [2T_0, \min\{ 3T_0, T \}].
\end{gather*}
Repeating the same argument $\lfloor T/T_0 \rfloor -2$ more times, we conclude that  there exists a certain constant $c>0$ such that
\begin{equation} \label{eq359}
E(t) \le (1+c)^{-n} E(0) \quad \text{for} \  t \in [n T_0, \min\{ (n+1) T_0, T \} ], \ \ n=1, \cdots, \lfloor T/T_0 \rfloor,
\end{equation}
where $\lfloor z \rfloor$ denotes the largest integer not exceeding $z \in \mathbb R$.
It is straightforward to obtain the desired inequality \eqref{eqPro31} from \eqref{ajs2}, \eqref{eq355*}, and \eqref{eq359}.
\end{proof}


\section{Transonic problem 1: continuous transition}\label{Sec5}

In this section, we prove Theorem \ref{Thm13} claiming the asymptotic stability of stationary solutions 
which are subsonic in $x<0$, exactly transonic at $x=0$, and supersonic in $x>0$.
%
%

The unknown function $f=(\phi,\omega)^{\top}=(\rho - \bar{\rho},m - \bar{m})^{\top}$ follows \eqref{eq311*},
but no boundary condition is imposed at $x = x_2$
{since the stationary solution is supersonic at $x=x_{2}$.
So we consider the following initial--boundary value problem: 
\begin{subequations}\label{eq411*}
\begin{gather}\label{eq417}
\partial_t f + \hat{A}^1 \partial_x f + \hat{B}_{b} f = 0,
\\
%
%
\label{eq418}
f(0,x) = f^0(x), 
\\
\omega(t,x_{1})=0, \quad t>0,
\label{eq414a}
\end{gather}
\end{subequations}
where $\hat{A}^{1}$ and $\hat{B}_b$ are defined just after \eqref{eq311*}.
We require the compatibility conditions \eqref{comp1a} only at $x=x_1$.
}
%
We seek the solution $f$ of the initial--boundary value problem \eqref{eq411*} in the following function space:
\begin{align*}
\mathfrak{X} (T)=  \left[ \bigcap_{j=0}^{2} C^j([0,T]; H^{2-j}(I)) \right]^{2} \ \text{for $T>0$}, \quad
\mathfrak{X} (\infty)= \left[ \bigcap_{j=0}^{2} C^j([0,\infty); H^{2-j}(I)) \right]^{2}. 
\end{align*}

We state the main theorem in this section, which immediately ensures Theorem \ref{Thm13}.

\begin{thm} \label{Thm50}
Assume $A \in  C^4(\bar{I})$ with $A_{x}(0) = 0$ and $A_{xx}(0) > 0$, 
and that a stationary solution $(\bar{\rho}, \bar{m})  \in C^3(\bar{I}) \times C^4(\bar{I})$ to (\ref{eq151})--(\ref{eq155}) exists and satisfies
\begin{gather}\label{spro51}
\inf_{x \in \bar{I}} \bar{\rho}(x) >0,  \quad
\begin{cases}
\bar{M}(x) < 1 & \text{for} \ x \in [x_1, 0),
\\
\bar{M}(x) = 1 & \text{for} \ x = 0,
\\
\bar{M}(x) > 1 & \text{for} \ x \in (0, x_2].
\end{cases}
\end{gather}
Then there exists a constant $\delta_{0}>0$ such that if an initial value $f^{0} \in H^2(I) \times H^2(I)$ satisfies the compatibility condition \eqref{comp1a} and $\| f^{0} \|_{H^2} \leq \delta_{0}$, 
then the initial--boundary value problem \eqref{eq411*} admits a unique time-global solution 
$f \in \mathfrak{X} (\infty)$.
Moreover, the solution $f$ decays in $H^{2}(I) \times H^2(I)$ exponentially fast  as $t \to \infty$.
\end{thm}

\begin{rem}
We remark that the assumption on $A$ in Theorem \ref{Thm50} (see also Theorem \ref{Thm13}) is weaker than the condition (\ref{eq231}) on $A$ in Theorem \ref{thm2.1}. 
It is seen that the stationary solution in Theorem \ref{Thm50} also satisfies
\begin{gather}\label{spro52}
\inf_{x \in \bar{I}} \bar{u}(x) = \inf_{x \in \bar{I}} \frac{Q}{A\bar{\rho}}(x) >0, \quad
\bar{\rho}_{x}(0)<0, \quad \bar{u}_{x}(0)>0.
\end{gather}
The last two equalities are obtained by applying Lemma \ref{Pro21} with $x_{m}=0$ and a suitable choice of $J$, since $\bar{\rho}$ satisfies  \eqref{spro51} and $F(\bar{\rho}(x),A(x))=F(\bar{\rho}(0),A(0))=F(\rho_{T}(A(0)),A(0))$.
It is worth pointing out that we do not require any restrictions on the magnitude of $|A'(x)|$ and $Q$, and accordingly the derivatives of the stationary solutions can be large.
\end{rem}

Theorem \ref{Thm50} is an immediate consequence of the time-local solvability in Lemma \ref{local5}
and an a-priori estimate in Proposition \ref{Pro50}.
\begin{lem}\label{local5}
Assume the same conditions as in Theorem \ref{Thm50} for $A$ and $(\bar{\rho}, \bar{m})$.
Then there exist constants $\delta_{\sharp}>0$ and $T_{\sharp}>0$ such that if the initial value $f^{0} \in H^2(I) \times H^2(I)$ satisfies the compatibility condition \eqref{comp1a} and $\| f^{0} \|_{H^2} \leq \delta_{\sharp}$, 
then the initial--boundary value problem \eqref{eq411*} admits a unique solution 
$f \in \mathfrak{X} (T_{\sharp})$.
Furthermore, if $f \in \mathfrak{X} (T)$
is a solution of \eqref{eq411*}, then the following time-integral of the trace of the highest order derivative $\partial_{t}^{j}\partial_{x}^{i} f$ makes sense for $i,j=0,1,2$ with $i+j=2$ and the point $x \in \bar{I}$ at which none of $u \pm \sqrt{p'(\rho)}$, characteristic speed of \eqref{eq417}, is zero for all $t \in [0,T]${\rm :}
\begin{gather}\label{trace5}
\int_{0}^{T} |\partial_{t}^{j}\partial_{x}^{i} f|^{2}(t,x) dt.  
\end{gather}
\end{lem}
\begin{proof}
We will show this lemma in Appendix.
\end{proof}

In general, the time-integral in \eqref{trace5} does not make sense for functions in $\mathfrak{X} (T)$. However, if $f$ solves the initial--boundary value problem \eqref{eq411*},
we can justify this integral. 

To present an a-priori estimate, we introduce 
\begin{gather*} 
N(T) := \sup_{t \in [0,T]} \| f(t,\cdot) \|_{H^2}.
\end{gather*}
\begin{pro} \label{Pro50}
Assume the same condition as in Theorem \ref{Thm50} for $A$ and $(\bar{\rho}, \bar{m})$.
Suppose that the initial-boundary value problem \eqref{eq411*} has a solution 
$f \in \mathfrak{X} (T)$ for some $T>0$.
Then there exist positive constants $\delta_{0}$, $C$, and $\nu$ independent of $T$ such that if $N(T) \leq \delta_{0}$, 
then the following estimate holds$:$
\begin{equation} \label{eqPro50}
\| f(t,\cdot) \|_{H^2} \le C e^{-\nu t} \| f^{0} \|_{H^2} \quad \text{for $t \in [0,T]$.}
\end{equation}
\end{pro}

Combining the time-local solvability in Lemma \ref{local5} and the a-priori estimate in Proposition \ref{Pro50}, we can prove Theorem \ref{Thm50} in a way similar to that of Theorem \ref{Thm30}.

\bigskip

In the remaining part of this section, we show only Proposition \ref{Pro50}. 
The analysis in the previous section is applicable only to where the flow is subsonic.
In view of this, we divide the entire region $I$ into the subsonic region $I_{b}$ 
and the {\it extended supersonic region} $I_{p}$ as
\begin{gather}\label{decomI}
I = I_{b} \cup I_{p}, \quad I_{b}:=(x_{1},x_{*}), \quad I_{p}:=[x_{*},x_{2}), 
\end{gather}
where $x_{*} \in (x_{1},0)$ is specified in Proposition \ref{Pro51} below.
In the subsonic region $I_{b}$, both the time-dependent and stationary solutions are strictly subsonic,
while $I_{p}$ contains the supersonic region, the transonic point, and a portion of subsonic region.

{
One of the keys of the proof is to use the fundamental values $(\rho,m)$ within $I_{b}$ and $(\rho,u)$ within $I_{p}$.
We denote the velocity perturbation from the stationary solution by
\begin{gather*}
\psi=\psi(t,x) := u(t,x) - \bar{u}(x), \quad \psi^{0}=\psi^{0}(x) := u^{0}(x) - \bar{u}(x). 
\end{gather*}
We also use the subscripts $b$ and $p$ to distinguish between quantities utilized in $I_{b}$ and $I_{p}$, respectively.
For instance, the unknown functions and initial data are expressed by
\begin{gather*}
f_{b} := (\phi,\omega)^{\top}, \   f_{b}^{0} := (\phi^{0},\omega^{0})^{\top}
\ \ \text{and} \ \ 
f_{p} := (\phi,\psi)^{\top}, \  f_{p}^{0} := (\phi^{0},\psi^{0})^{\top}.
\end{gather*}
These unknown functions are  related each other by
\begin{gather}\label{eq528}
f_{b}
=
\begin{bmatrix}
1 & 0 
\\ 
\bar{u} & \rho
\end{bmatrix}
f_{p}
\ \ \text{or} \ \
f_{p} = 
\dfrac{1}{\rho}
\begin{bmatrix}
\rho & 0 
\\ 
-\bar{u} & 1
\end{bmatrix}
f_{b}.
\end{gather}
The unknown functions $f_{b}$ satisfies the system \eqref{eq319} which is a symmetrization of \eqref{eq317} and \eqref{eq417}. So we denote the matrices $A^0$, $A^1$, and $B$ in \eqref{eq319}  by $A^0_b$, $A^1_b$, and $B_b$, respectively, i.e.,  
\begin{gather*}
A_{b}^0 \partial_t f_{b} + A_{b}^1 \partial_x f_{b} + B_{b} f_{b} =A^0 \partial_t f + A^1 \partial_x f + B f= 0.
\end{gather*}
Note that $A_{b}^{0}$ is positive definite in $I_{b}$ where the flow is subsonic.
Let us derive the symmetric system for the unknown functions $f_{p}$.
The equations \eqref{eq101} and \eqref{eq151} can be transformed into 
}
\begin{subequations} \label{eq401}
\begin{gather} 
\rho_t + u \rho_x + \rho u_x + \dfrac{A_x}{A} \rho u = 0,  
\\
u_t + u u_x + \frac{1}{\rho} p(\rho)_x = 0,  
\end{gather}
\end{subequations}
and
\begin{subequations} \label{eq402}
\begin{gather} 
\bar{u} \bar{\rho}_x + \bar{\rho} \bar{u}_x + \dfrac{A_x}{A} \bar{\rho} \bar{u} = 0,  
\\
\bar{u} \bar{u}_x + \frac{1}{\bar{\rho}} p(\bar{\rho})_x = 0,  
\end{gather}
\end{subequations}
respectively.
From the equations \eqref{eq401} and \eqref{eq402}, 
{
we see that $f_p$ satisfies
\begin{gather}\label{eq512}
A^0_p \partial_t f_{p} + A^1_p \partial_x f_{p} + B_p f_{p} = 0,
\end{gather}
}
where
\begin{align*} 
A^0_p(x,\rho,u) &:=
\begin{bmatrix}
\dfrac{A p'(\rho)}{\rho}   & 0
\\[10pt]
0 & A\rho
\end{bmatrix},
\quad
A^1_p(x,\rho,u) := A^0_p
\begin{bmatrix}
u  & \rho
\\[10pt]
\dfrac{p'(\rho)}{\rho} & u
\end{bmatrix},
\\
{B}_p(x,\rho,u) &:= A^0_p
\begin{bmatrix}
\bar{u}_x + \dfrac{A_x}{A} u   & \bar{\rho}_x +  \dfrac{A_x}{A} \bar{\rho}
\\  
b_{p} & \bar{u}_x
\end{bmatrix},
\\
b_{p}(x,\rho,u)  &:= \bar{\rho}_x
\int_0^1 \left(\dfrac{d}{d \rho} \dfrac{p'(\rho)}{\rho} \right)\Big|_{\rho = \bar{\rho} + \theta (\rho-\bar{\rho}) } d \theta.
\notag
\end{align*}

{Another key of the proof} is to find an appropriate point $x_{*} \in (x_{1},0)$ to decompose $I$ as in \eqref{decomI}. 
We mention briefly our idea to find such a point $x_{*}$.
{We employ the weight function
\begin{gather} \label{eq522}
W=W(x) := e^{-k_* (x-x_*)},
\end{gather}
where $k_{*}>0$ will also be determined in Proposition \ref{Pro51}.
}
Multiplying \eqref{eq512} by $2W f^{\top}$ from the left gives
\begin{gather}\label{eq523}
\partial_t \left( W f_{p}^\top A^0_p f_{p} \right) 
+ \partial_x \left( W f_{p}^\top A^1_p f_{p} \right) 
+ W f_{p}^\top (R_{p} +k_{*}A^{1}_{p}) f_{p} = 0,
\end{gather}
where
\begin{gather}
R_p := - \partial_t A^0_p - \partial_x A^1_p +  B_p + B_p^\top.
\label{eq514}
\end{gather}
We note that $A^{1}_{p}$ is positive definite on $[\ve,x_{2}]$ for $\ve>0$ due to \eqref{spro51}, \eqref{spro52}, and $N(T) \ll 1$.
The matrix $k_{*}A^{1}_{p}$ in \eqref{eq523}, associated with the use of the weight function $W$, provides a dissipative structure on $[\ve,x_{2}]$. 
Although this strategy works for purely supersonic problems as in \cite{OS2020}, 
this alone is not enough to obtain a dissipative structure around the transonic point 
because $A^{1}_{p}$ is not positive definite there.
Alternatively, we observe another dissipative structure near the throat of the nozzle where the flow is
transonic.
Indeed, the matrix $R_{p}$ in \eqref{eq523} which includes derivatives of the stationary solution is shown to be positive definite around the transonic point due to \eqref{eq237}.
Consequently, we can choose $x_{*} \in (x_{1},0) $ and $k_{*}>0$ so that $R_{p}+k_{*}A^{1}_{p}$ is positive definite on $\bar{I}_{p}=[x_{*},x_{2}]$.

{The last key} is that at the interface between $I_b$ and $I_p$ i.e. at $x=x_*$, 
a dissipative term arises due to the difference in the flux from both sides, 
which implies the dissipation of the energy over the entire domain $I$.
These insights provide a mathematical description of the mechanism that the convergent-divergent nozzle stably produces a supersonic flow at the exit.

Observing that symmetric matrices $A^{0}_{b}$ and $A^{0}_{p}$ are positive definite on
$\overline{I_b}$ and $\overline{I_b}$ respectively,
we define the following energy and dissipative functionals for $j=0,1,2$:
\begin{subequations}\label{EF5}
\begin{gather}
E^{j}_{bp}(t):=(E^{j}_{b}+E^{j}_{p})(t), \quad D^{j}_{bp}(t):=(D^{j}_{b}+D^{j}_{p})(t), 
\\
E^{j}_b(t) := \int_{x_1}^{x_*}  ( \partial_t^j f_{b} )^\top A^0_b \partial_t^j f_{b} (t, x) dx, 
\quad
E^{j}_p(t) := \int_{x_*}^{x_2} W  ( \partial_t^j f_{p} )^\top A^0_p \partial_t^j f_{p} (t, x) dx,
\\
D^{j}_b(t) := |\partial_{t}^{j} f_{b}|^2 (\tau, x_*),
\quad
D^{j}_p(t) := |\partial_{t}^{j} f_{p}|^2 (\tau, x_2).
\end{gather}
\end{subequations}
Similarly to the derivation of \eqref{norm-es1}, using \eqref{eq512}, \eqref{spro51}, $(\bar{\rho},\bar{m})  \in C^3(\bar{I}) \times C^4(\bar{I})$, and $N(T) \ll 1$, we obtain
\begin{align}\label{norm-es51}
&|f_{b}(t,x)|+|\partial_{x} f_{b}(t,x)|+|\partial_{t} f_{b}(t,x)|  \leq CN(T) \quad \text{for} \ (t,x) \in [0,T] \times \bar{I},
\\
&|f_{p}(t,x)|+|\partial_{x} f_{p}(t,x)|+|\partial_{t} f_{p}(t,x)|  \leq CN(T) \quad \text{for} \ (t,x) \in [0,T] \times \bar{I},
\end{align}
where $C$ is a positive constant independent of $(t,x)$ and $T$.

Now we prove Proposition \ref{Pro51} to determine $x_{*} \in (x_{1},0)$ and $k_{*}>0$.

\begin{pro} \label{Pro51}
Under the same assumptions as in Proposition \ref{Pro50}, 
if $N(T)$ is sufficiently small, 
there exist constants $c, k_* > 0$ and $x_* \in (x_1,0)$ such that 
\begin{gather} \label{eqPro51}
f^\top R^j_{p}(t,x) f \ge c | f |^2 \quad \text{for $f \in \mathbb{R}^2, \ (t,x) \in [0,T] \times [x_*, x_2], \ j=0,1,2$}
\end{gather}
hold, where the $2 \times 2$ symmetric matrices $R^j_{p}$ are defined by
\begin{align*}
R^j_{p} &:= R_p + k_*  A^1_p + j \times \left( 2 \partial_x A^1_p -  (\partial_x A_p^0) (A_p^0)^{-1} A_p^1 - \{(\partial_x A_p^0) (A_p^0)^{-1} A_p^1\}^\top  \right)
\end{align*}
using $R_p$ in \eqref{eq514}.
\end{pro}
\begin{proof}
It suffices to show positive definiteness of $R^0_{p}$ and $R^2_{p}$, since it implies that of $R^1_{p} = ( R^0_{p} + R^2_{p} ) / 2$.
Each entry of $R_{p}^0$ is  
\begin{align*}
R^0_{p11} 
&= - \partial_t \left( \frac{A p'}{\rho} \right) 
- \partial_x \left( \frac{A p'}{\rho} u \right) 
+ \frac{2p'}{\rho}  \left( A \bar{u}_x + A_x u \right)
+ k_* \dfrac{A u p'}{\rho}
=  {\cal R}^{0}_{11} + \mathcal{O} (N(T)),
\\
R^0_{p12} &= R^0_{p21} 
= - \partial_x \left( A p' \right)
+ \frac{p'}{\rho} ( A \bar{\rho} )_x 
+ \rho A b_{21} 
= {\cal R}^{0}_{12}+ \mathcal{O} (N(T)),
\\
R^0_{p22} 
&= 
- \partial_t \left( A \rho \right) 
- \partial_x \left( A \rho u \right)
+ 2A \rho \bar{u}_x + k_* A \rho u
={\cal R}^{0}_{22}+ \mathcal{O} (N(T)),
\end{align*}
where we have used \eqref{spro51}, \eqref{spro52}, \eqref{norm-es51}, and $(\bar{\rho}, \bar{m})  \in C^3(\bar{I}) \times C^4(\bar{I})$, and ${\cal R}^{0}$ is defined by
\begin{gather*}
{\cal R}^{0}_{11}:=\dfrac{A \bar{u} p'(\bar{\rho})}{\bar{\rho}} \left( k_* - K(\bar{\rho}) \dfrac{\bar{\rho}_x}{\bar{\rho}} \right), \quad
{\cal R}^{0}_{12}={\cal R}^{0}_{21}:=k_* A p'(\bar{\rho}), \quad
{\cal R}^{0}_{22}:=A \bar{\rho} \bar{u} \left( k_* + 2 \dfrac{\bar{u}_x}{\bar{u}} \right).
\end{gather*}
Here $K$ is defined in Theorem \ref{thm2.1}.
Likewise, the elements of $R^2_{p}$ are
\begin{align*}
R^2_{p11} = 
{\cal R}^{2}_{11}+ \mathcal{O} (N(T)), \quad
R^2_{p12} = R^2_{p21} = 
{\cal R}^{2}_{12}+ \mathcal{O} (N(T)), \quad
R^2_{p22} = {\cal R}^{2}_{22}+ \mathcal{O} (N(T)),
\end{align*}
where
\begin{align*}
{\cal R}^{2}_{11}&:=\dfrac{A \bar{u} p'(\bar{\rho})} {\bar{\rho}} 
 \left( k_* -  K(\bar{\rho}) \dfrac{\bar{\rho}_x}{\bar{\rho}} + 4 \dfrac{\bar{u}_x}{\bar{u}} \right), \quad
{\cal R}^{2}_{12}={\cal R}^{2}_{21}:=A p'(\bar{\rho}) \left( k_* + 2 K (\bar{\rho}) \dfrac{\bar{\rho}_x}{\bar{\rho}} \right),
\\
{\cal R}^{2}_{22}&:=A \bar{\rho} \bar{u} \left( k_* + 6 \dfrac{\bar{u}_x} {\bar{u}} \right).
\end{align*}
We note that ${\cal R}^{j}$ for $j=0,2$ depends on $x$ but independent of $t$.
We claim that there exist constants $c_{1}>0$ and $x_{*} \in (x_{1},0)$ such that ${\cal R}^{j}_{11}, {\rm det} \, {\cal R}^{j} \geq c_{1}$ holds for $x \in [x_{*},x_{2}]$. If the claim is true, then the desired inequality  \eqref{eqPro51} immediately follows from taking $N(T) \ll 1$ according to $c_{1}>0$. 

Let us show the claim.
Taking $k_{*}$ large gives $ \inf_{x \in \bar{I}} {\cal R}^{j}_{11}(x) >0$ and  $ \inf_{x \in \bar{I}} {\cal R}^{j}_{12}(x) >0$.
Therefore, in order to find $c_{1}$ and $x_{*}$ in the claim, it is sufficient to show that 
\begin{gather} \label{eq519}
Q^j(x) := \dfrac{{\cal R}^j_{11} {\cal R}^j_{22}} { \left( {\cal R}^j_{12} \right)^{2}} 
\ge 1 + c_2, \quad x \in [x_{*},x_{2}]
\end{gather}
for certain constants $c_{2}>0$ and $x_{*} \in (x_{1},0)$.
We see that as $\ k_* \to \infty$,
\begin{align*}
&Q^0(x) 
=\bar{M}^2 \left[ 1 + \left( - K(\bar{\rho}) \dfrac{\bar{\rho}_x} {\bar{\rho}} + 2 \dfrac{\bar{u}_x} {\bar{u}} \right) k_*^{-1} + \mathcal{O}(k_*^{-2}) \right],
\\
&Q^2(x) 
=
\bar{M}^2
\left[
 1 + \left( - 5 K(\bar{\rho}) \dfrac{\bar{\rho}_x} {\bar{\rho}} 
 + 10 \dfrac{\bar{u}_x}{\bar{u}} \right) k_*^{-1} + \mathcal{O} ( k_*^{-2} )
\right].
\end{align*}
Due to \eqref{spro52} and $(\bar{\rho}, \bar{m})  \in C^3(\bar{I}) \times C^4(\bar{I})$,
there exists a constant $\ell>0$ such that
\begin{gather*} 
\sup_{x\in[-\ell,\ell]}\bar{\rho}_x(x)<0, \quad \inf_{x\in[-\ell,\ell]}\bar{u}_x(x) >0. 
\end{gather*}
On the other hand, it also holds from \eqref{spro51} that $\inf_{x \in [0,x_{2}]} \bar{M} (x) \geq 1$ 
and $\inf_{x \in [\ell,x_{2}]} \bar{M} (x) > 1$.
Using these and letting $k_*>0$ be sufficiently large, we obtain
\begin{gather*}
Q^j(x) \ge 1 + 2c_{2}
\quad \text{for} \  (t,x) \in [0,T] \times [0, x_2], \ j = 0, 2
\end{gather*}
for a certain constant $c_2>0$.
By the continuity of $Q^j$, we can find $x_* \in (x_1,0)$ such that \eqref{eq519} holds.
The proof is complete.
\end{proof}

\subsection{Basic estimates} \label{Sec41}
In this subsection,
%
we derive the most fundamental estimate of $\| f_p \|_{L^2}$
in Lemma \ref{Lem54}. 
In the derivation, we divide the entire region $I$ into the subsonic region $I_{b}=(x_{1},x_{*})$, and
the extended supersonic region $I_{p}=[x_{*},x_{2})$, where $x_{*} \in (x_{1},0)$ is the constant 
specified in Proposition \ref{Pro51}.

\begin{lem} \label{Lem54}
Under the same assumptions as in Proposition \ref{Pro50}, 
if $N(T)$ is sufficiently small, the following estimate holds for $t_{0} \in [0,T]$ and $t \in [t_{0},T]${\rm :}
\begin{gather} \label{eqLem54}
E^0_{bp}(t)
+ c \int_{t_{0}}^t  { D^{0}_{bp} (\tau)} d\tau 
\le
E^0_{bp}(t_{0}) 
+ C N(T) \int_{t_{0}}^t E^0_{bp}(\tau) d\tau, 
\end{gather}
where $c$ and $C$ are positive constants independent of $t_{0}$ and $T$, and $E^{0}$ and $D^{0}$ are defined in \eqref{EF5}.
\end{lem}
\begin{proof}
We claim that the following inequalities hold:
\begin{gather}
E_b^0(t) 
+ \int_{t_{0}}^t f_{b}^\top A^1_b f_{b} (\tau, x_*) d\tau
\le
E_b^0(t_{0})  
+ C N(T) \int_{t_{0}}^t E_b^0(\tau) d\tau,
\label{eqLem53}
\\
E_p^0(t) 
- \int_{t_{0}}^t f_{p}^\top  A^1_p  f_{p} (\tau, x_*) d \tau
 + c \int_{t_{0}}^t  D_{p}^{0} (\tau) d\tau 
\le 
E_p^0(t_{0}),
\label{eqLem52}
\end{gather}
where $E^0_b$, $E^0_p$, $D^0_b$, and $D^0_p$ are defined in \eqref{EF5}.

Let us first show \eqref{eqLem53}.
The proof is analogous to that of \eqref{eqLem32} in Lemma \ref{Lem32}.
Repeating the proof there but with $\bar{I}_{b}=[x_{1},x_{*}]$ in place of $\bar{I}=[x_{1},x_{2}]$, 
we obtain the estimate corresponding to \eqref{eq327} as
\begin{align} \label{eq531}
&\int_{x_1}^{x_*} f_b^\top A_b^0 f_b (t,x) dx
- \int_{t_{0}}^t  f_b^\top A_b^1 f_b (\tau, x_1) d\tau
+ \int_{t_{0}}^t  f_b^\top A_b^1 f_b (\tau, x_*) d\tau
\notag \\
& =\int_{x_1}^{x_*} f_b^\top A_b^0 f_b (t_{0},x) dx 
-  \int_{t_{0}}^t \!\! \int_{x_1}^{x_*} f_b^\top R_b f_b(\tau, x) dx d\tau
\notag \\
&\le
\int_{x_1}^{x_*} f_b^\top A_b^0 f_b(t_{0},x) dx 
+ C N(T) \int_{t_{0}}^t \!\! \int_{x_1}^{x_*} | f_b |^2 (\tau, x) dx d\tau,
\end{align}
where $R_b := - \partial_t A_b^0 - \partial_x A_b^1 +  B_b + B_b^\top$.
We note that the integrand of the second term on the LHS of \eqref{eq531} vanishes due to the boundary condition $\omega(t,x_{1})=0$. 
Therefore the estimate \eqref{eq531} is identical to \eqref{eqLem53}.

We next show \eqref{eqLem52}.
Integrating the transport equation \eqref{eq523} of energy $E^0_p$ over $[t_{0},t] \times \bar{I}_{p}$ gives
\begin{multline} \label{eq524}
E_p^0(t) 
- \int_{t_{0}}^t f_{p}^\top A^1_p  f_{p} (\tau, x_*) d\tau
+ W(x_2)  \int_{t_{0}}^t f_{p}^\top A^1_p f_{p} (\tau, x_2) d\tau
\\
+ \int_{t_{0}}^t\!\!\int_{x_{*}}^{x_{2}} W f_{p}^\top R^{0}_{p} f_{p} (\tau, x) dxd\tau
=
E_p^0(t_{0}), 
\end{multline}
where $R^0_{p}$ is positive definite on $[0,T] \times \bar{I}_{p}$ as shown in Proposition \ref{Pro51}.
Furthermore, it follows from  \eqref{spro51}, \eqref{spro52}, \eqref{norm-es51}, $(\bar{\rho}, \bar{m})  \in C^3(\bar{I}) \times C^4(\bar{I})$, and $N(T) \ll 1$ that $A^1_p(t,x_2)$ is also positive definite for any $t \in [0,T]$.
From these, we see that the last term on the LHS of \eqref{eq524} is nonnegative, 
{and that the integrant of the third term is bounded from below as
\begin{gather*}
f_{p}^\top A^1_p f_{p} (\tau, x_2)  \geq c  |f_{p}|^{2} (\tau, x_{2})   = c D_{p}^{0} (\tau).
\end{gather*}
}
Thus \eqref{eqLem52} holds.

We complete the proof by showing \eqref{eqLem54}.
Adding \eqref{eqLem53} and \eqref{eqLem52} yields
\begin{multline} \label{eq532}
E^0_{bp}(t)
+ \int_{t_{0}}^t \left( f_{b}^\top A^1_b f_{b} - f_{p}^\top A^1_p f_{p} \right) (\tau, x_*) d\tau
+ c \int_{t_0}^t D_p^0(\tau) d\tau \\
\le
E^0_{bp}(t_{0}) 
+ C N(T) \int_{t_{0}}^t E^0_{bp}(\tau) d\tau. 
\end{multline}
Using \eqref{eq528}, we rewrite the integrand of the second term on the LHS of \eqref{eq532} as
\begin{align*}
\left( f_{b}^\top A^1_b f_{b} - f_{p}^\top A^1_p f_{p} \right) (\tau, x_*) 
=
f_{b}^\top \left( \dfrac{A\bar{u}}{\bar{\rho}}
\mathrm{diag} \left( \bar{S}, 1 \right) + \mathcal{O}(N(T)) \right) f_{b} (\tau, x_*),
\end{align*}
where $\bar{S}$ is defined in \eqref{eq239} and $\bar{S}(x_{*})>0$ holds due to $\bar{M}(x_{*})<1$ in \eqref{spro51}.
Hence there exists a positive constant $c$ independent $T$ such that
\begin{gather} \label{eq533}
\left( f_{b}^\top A^1_b f_{b} - f_{p}^\top A^1_p f_{p} \right) (\tau, x_*) 
\ge
c | f_{b} |^2 (\tau, x_*) = c D_{b}^{0} (\tau) 
\end{gather}
as long as $N(T)$ is small enough. 
Substituting \eqref{eq533} into \eqref{eq532} concludes \eqref{eqLem54}.
\end{proof}

In the Lemma \ref{Lem55} below, we show that the second term on the LHS of \eqref{eqLem54} 
in Lemma \ref{Lem54} gives rise to the energy decay over the entire domain.
Combination of Lemma \ref{Lem54} with Lemma \ref{Lem55} will imply 
the exponential decay of $E^0_{bp}(t)$.

First, noting that the flow is subsonic over $I_b$ if $N(T)$ is sufficiently small,
we develop the same arguments as in Lemmas \ref{Lem33}--\ref{Lem35}
regarding $\int_{t_{0}}^t  |f_{b}|^2 (\tau, x_*) d\tau$ in \eqref{eqLem54} 
as the boundary term in \eqref{eqLem32}.
Next, we apply the time-weighted energy method to \eqref{eq523} over $I_p$.

For this purpose, we use the following notation:
%
%
\begin{gather} \label{eq525}
v_-^{*} := \frac{1}{2} \sup_{x \in \bar{I}_{b}} \left( \bar{u}(x) - \bar{a}(x) \right)< 0,
\quad
v_+^{*} := \frac{1}{2} \inf_{x \in \bar{I}_{b}} \left( \bar{u}(x) + \bar{a}(x) \right)> 0,
\\
\label{eq526}
T_{-}^{*} : = \frac{x_* - x_1}{| v^*_- |},
\quad
T_{+}^{*} : = \frac{x_* - x_1}{| v^*_+ |},
\quad
T_0^{*} : = T_{-}^{*} +  T_{+}^{*},
\end{gather}
where we have just replaced $\bar{I}$ and $x_{2}$ by $\bar{I}_{b}$ and $x_{*}$, respectively, in \eqref{eq322} and \eqref{eq323}.
Furthermore, we also define $\varOmega_{-}^{*} {(\tilde{t}_{0})}$, $\varOmega_{+}^{*} {(\tilde{t}_{0})}$, 
and $\varOmega^{*} {(\tilde{t}_{0})}$ by 
\begin{align*}
\varOmega_-^{*} {(\tilde{t}_{0})} &:= \{ (t, x) \, | \, \tilde{t}_{0} \le t \le \tilde{t}_{0}+T_{-}^{*}, \ 
x_* + v_-^{*} (t-\tilde{t}_{0}) \le x \le x_* \},
\\
\varOmega_+^{*} {(\tilde{t}_{0})} &:= \{ (t, x) \, | \, \tilde{t}_{0} + T_{-}^{*} \le t \le \tilde{t}_{0}+T_{0}^{*}, \ 
x_1 + v_+^{*} (  t  -\tilde{t}_{0}- T_{-}^{*} ) \le x \le x_* \},
\\
\varOmega^{*} {(\tilde{t}_{0})} &:= \varOmega_-^{*}(\tilde{t}_{0}) \cup \varOmega_+^{*}(\tilde{t}_{0})
\end{align*}
for any nonnegative constant $\tilde{t}_{0}$.

\begin{lem} \label{Lem55}
Under the same assumptions as in Proposition \ref{Pro50} with $T \geq T_{0}^{*}$, 
if $N(T)$ is sufficiently small, 
the following estimate holds for $\tilde{t}_{0} \in [0,T-T_{0}^{*}]$ and $t \in [\tilde{t}_{0}+T_{0}^{*},T]${\rm :}
\begin{gather} \label{eqLem56}
 ( 1 + c ) E^{0}_{bp}(t)
+ c \int_{\tilde{t}_{0}}^t D_{bp}^0 (\tau) d\tau
\le 
E_{bp}^{0}(\tilde{t}_{0}) + C N(T) \int_{\tilde{t}_{0}}^t E_{bp}^{0}(\tau) d\tau,
\end{gather}
where $c$ and $C$ are positive constants independent of $\tilde{t}_{0}$ and $T$.
\end{lem}
\begin{proof}
We claim that the following estimates hold:
 \begin{gather} 
E_b^0(t) 
\le 
C \int_{\tilde{t}_{0}}^t { D_{b}^{0}(\tau) }d \tau 
+ C N(T) \int_{\tilde{t}_{0}}^t E_b^0(\tau) d\tau,
\label{eqLem55} 
\\
( 1 + c ) E_p^0(t) 
\le 
E_p^0(\tilde{t}_{0}) + C \int_{\tilde{t}_{0}}^t {D_{b}^{0}(\tau) } d\tau,
\label{eq547}
\end{gather}
where $c$ and $C$ are positive constants independent of $\tilde{t}_{0}$ and $T$.

We first show \eqref{eqLem55} by repeating the proof of Lemma \ref{Lem33} 
but replacing $[x_{1},x_{2}]$ by $\bar{I}_{b}=[x_{1},x_{*}]$. 
Then we can have
\begin{gather} \label{eq535}
\iint_{\varOmega^{*}(\tilde{t}_{0})} | f_{b} |^2(\tau, x) dx d\tau
\le 
C \int_{\tilde{t}_{0}}^{\tilde{t}_{0}+T_0^{*}} { D_{b}^{0}(\tau) } d\tau,
\end{gather}
which is parallel to \eqref{eqLem33}.
In much the same way as in the proof of Lemma \ref{Lem34} using \eqref{eq535} instead of \eqref{eqLem33}, we can obtain
\begin{gather} \label{eq537}
E_b^{0}(\tilde{t}_{0}+T_{-}^{*})
\le 
C \int_{\tilde{t}_{0}}^{\tilde{t}_{0}+T_0^{*}} { D_{b}^{0}(\tau) } d\tau,
\end{gather}
which is parallel to \eqref{eqLem34}.
Then we substitute \eqref{eq537} into \eqref{eqLem53} with $t_{0}=\tilde{t}_{0}+T_{-}^{*}$ 
and $t \in [\tilde{t}_{0}+T^*_0,T]$, and move the second term on the LHS of the resultant inequality 
to the RHS to obtain \eqref{eqLem55}. 


To prove \eqref{eq547}, 
we integrate the product of $e^{q t}$ and \eqref{eq523} over $[\tilde{t}_{0}, t] \times \bar{I}_{p}$ 
for a positve constant $q$ to obtain
\begin{align}
& E_p^0(t) 
+ W(x_2) \int_{\tilde{t}_{0}}^t e^{-q (t-\tau)}   f_{p}^\top A^1_p f_{p} (\tau, x_2) d\tau
\notag\\
& \quad 
+ \int_{\tilde{t}_{0}}^t\!\!\int_{x_{*}}^{x_{2}} e^{-q (t-\tau)} W f_{p}^\top (R^{0}_{p}-q A^{0}_{p}) f_{p} (\tau, x) dxd\tau
\notag \\
\quad
& =
e^{-q(t-\tilde{t}_{0})}E_p^0(\tilde{t}_{0})
+ \int_{\tilde{t}_{0}}^t e^{-q (t-\tau)} f_{p}^\top A^1_p  f_{p} (\tau, x_*) d\tau
\notag\\
& \leq 
e^{-q(t-\tilde{t}_{0})}E_p^0(\tilde{t}_{0})
+  C \int_{\tilde{t}_{0}}^t { D_{b}^{0}(\tau) }d\tau,
 \label{eq545}
\end{align}
where we have used \eqref{spro51}, \eqref{eq528}, $(\bar{\rho}, \bar{m})  \in C^3(\bar{I}) \times C^4(\bar{I})$, and $N(T) \ll 1$ in the derivation.
Note that $R^{0}_{p}$ is positive definite in $\bar{I}_{p}$, 
and $A^{1}_{p}$ is also positive definite at $x=x_{2}$. 
Then letting $q \ll 1$, we see that the second and the last term on the LHS are nonnegative.
On the RHS, $e^{-q (t-\tilde{t}_{0})} \le e^{-q T_0^{*}}<1$ holds for $\disp t \in [\tilde{t}_{0}+T_0^{*}, T]$.
Substituting this into \eqref{eq545}, we have \eqref{eq547}.

Finally, we multiply the sum of twice of $\eqref{eqLem55}$ and $\eqref{eq547}$ by $\ve>0$, 
and add the resultant inequality to \eqref{eqLem54} with $t_0=\tilde{t}_{0}$.
Then letting $\ve \ll 1$ concludes \eqref{eqLem56}.
\end{proof}

\subsection{Estimates on the higher order derivatives} \label{Sec42}
In this subsection, we will derive the $L^{2}$-estimate of the higher order derivatives of the solution.
In the subsonic region $\bar{I}_{b}$, we work with the time-derivatives similarly to Section \ref{Sec_sub}, since we can use the relation \eqref{ajs2} to obtain the $L^{2}$-estimate of the spatial derivatives. 
On the other hand, in the extended supersonic region $\bar{I}_{p}$, the relation \eqref{ajs2} does not hold, 
because neither ${A}^1_b$ nor ${A}^1_p$ is regular entirely on $\bar{I}_{p}$. So some additional care is required to estimate the spatial-derivatives of the solution in $\bar{I}_{p}$.
%
%

To make the computation in this subsection rigorous, we need to use a mollifier due to the insufficiency of the regularity of the solution. 
However, we omit the argument since it is standard. 
We apply the operator $\partial_{t}^{j}$ for $j=1,2$ to \eqref{eq512} to obtain
\begin{gather}\label{eq554}
{A}^{0}_{p}\partial_{t} (\partial_{t}^{j} f_{p}) + {A}^{1}_{p} \partial_{x} (\partial_{t}^{j} f_{p}) + {B}_{p} (\partial_{t}^{j} f_{p})
={\cal F}^{j}_{p},
\\
{\cal F}^{j}_{p}:=- [\partial_{t}^{j},{A}^{0}_{p}\partial_{t}] f_{p} - [\partial_{t}^{j},{A}^{1}_{p}\partial_{x}] f_{p} - [\partial_{t}^{j},{B}_{p}] f_{p},
\notag
\end{gather}
where $[\cal T,\cal S]:=\cal T\cal S-\cal S\cal T$ is a commutator. It is straightforward to see from \eqref{eq528} that
\begin{gather}\label{eq528*}
\left| \partial_t^j f_{p}
-
\dfrac{1} {\rho}
\begin{bmatrix}
\rho & 0 
\\ 
-\bar{u} & 1
\end{bmatrix}
\partial_t^j f_{b}
\right|(t,x)
\leq CN(T) \sum_{i=0}^{j-1} |\partial_t^i f_{b}|(t,x),
\end{gather}
where $C$ is a positive constant independent of $(t,x)$ and $T$.
Let us also consider the spatial-derivative of $f_{p}$.
Applying $\partial_{x}^{j}$ for $j=1,2$ to \eqref{eq512} gives
\begin{gather*}
{A}^{0}_{p}\partial_{t} (\partial_{x}^{j} f_{p}) + {A}^{1}_{p} \partial_{x} (\partial_{x}^{j} f_{p}) + {B}_{p} (\partial_{x}^{j} f_{p})
=- [\partial_{x}^{j},{A}^{0}_{p}\partial_{t}] f_{p} - [\partial_{x}^{j},{A}^{1}_{p}\partial_{x}] f_{p} - [\partial_{x}^{j},{B}_{p}] f_{p}.
\end{gather*}
Then we rewrite the term $(\partial_{x} A^{1}_{p}) \partial_{x}^{j-1} \partial_{t} f_{p}$ on the RHS into some linear combination of $f_{p}$ and its spatial-derivatives of at most $(j-1)$-th order 
by using \eqref{eq512} to obtain
\begin{gather}\label{eq561}
{A}^{0}_{p}\partial_{t} (\partial_{x}^{j} f_{p}) + {A}^{1}_{p} \partial_{x} (\partial_{x}^{j} f_{p}) + {B}_{p} (\partial_{x}^{j} f_{p})
+j\times \left(\partial_x A^1_p -  (\partial_x A_p^0) (A_p^0)^{-1} A_p^1\right)\!(\partial_{x}^{j} f_{p})
={\cal G}^{j}_{p},
\end{gather}
where
\begin{align*}
{\cal G}^{1}_{p}&:= \left( \partial_x A_p^0 \right) \left( A_p^0 \right)^{-1} B_p f_{p} - (\partial_{x} B_{p}) f_{p},
\\
{\cal G}^{2}_{p}&:= 2 \left( \partial_x A_p^0 \right) \left( A_p^0 \right)^{-1} \!\! \left[ B_p \partial_{x} f_p - (\partial_x A_p^0) \left( A_p^0 \right)^{-1}\!\! \left( A_p^1  \partial_{x} f_p + B_p f_p \right) + (\partial_x A_p^1)  \partial_{x} f_p + (\partial_x B_p) f_{p} \right]
\\
& \quad - 2 (\partial_{x} B_{p})  \partial_{x} f_{p} - (\partial _{xx} A_{p}^{0})  \partial_{t} f_{p} - (\partial _{xx} A_{p}^{1})  \partial_{x}f_{p} - (\partial_{xx} B_{p}) f_{p}.
\end{align*}
The equality \eqref{eq561} holds also for $j=0$ by regarding ${\cal G}^{0}_{p} = 0$.

Now we estimate the remaining terms ${\cal F}^{j}_{p}$ and ${\cal G}^{j}_{p}$.

\begin{lem}\label{Lem5.55}
Under the same assumptions as in Proposition \ref{Pro50}, 
if $N(T)$ is sufficiently small, 
the following estimates hold for $j=1,2${\rm :}
\begin{gather}
|{\cal F}^{j}_{p}|(t,x)  \leq CN(T) \sum_{i=0}^j (|\partial_t^i f_{p}| + |\partial_x^i f_{p}|)(t,x),
\label{calF51} \\
|{\cal G}^{j}_{p}|(t,x)  \leq C \sum_{i=0}^{j-1} |\partial_x^i f_{p}|(t,x),
\label{calG51}
\end{gather}
where $C$ is a positive constant independent of $(t,x)$ and $T$.
\end{lem}
\begin{proof}
We note that the derivative $\partial_{xt} f$ is the only term up to the second order derivatives of $f_{p}$ missing in the RHS of \eqref{calF51}.
Let us estimate $\partial_{xt} f$ by $f$ and other derivatives.
It follows from \eqref{spro51}, \eqref{spro52}, $(\bar{\rho}, \bar{m})  \in C^3(\bar{I}) \times C^4(\bar{I})$, and $N(T) \ll 1$ that $A^0_{p}$ is regular, and also $(A^{0}_{p})^{-1}$ and all the matrices in \eqref{eq561} with $j=1$ are bounded uniformly in $(t,x)$. 
We multiply \eqref{eq561} with $j=1$ by $(A^{0})^{-1}$ from left, and then estimate $\partial_{xt} f$ as $|\partial_{xt} f|\leq C(|f|+|\partial_{x} f|+|\partial_{t} f|+|\partial_{xx} f|)$.
Using this and \eqref{norm-es51}, we arrive at \eqref{calF51}.

It is straightforward to see \eqref{calG51} by a direct computation with the aid of  \eqref{spro51}, \eqref{spro52}, \eqref{norm-es51}, $(\bar{\rho}, \bar{m})  \in C^3(\bar{I}) \times C^4(\bar{I})$, and $N(T) \ll 1$.
\end{proof}

Let us prove Lemma \ref{Lem57} below, which corresponds to Lemma \ref{Lem54}. 
We use notation
\begin{subequations}\label{EF5*}
\begin{gather}
E_{bp}(t) := (E_{bp}^{0}+E_{bp}^{1}+E_{bp}^{2}+E^{\rm sp}_{p,\mu})(t), \quad 
D_{bp}(t):=(D_{bp}^{0}+D_{bp}^{1}+D_{bp}^{2})(t),
\\
E^{\rm sp}_{p,\mu} (t) := 
\sum_{j=0}^2 \mu^{j+1} \int_{x_*}^{x_2}  W ( \partial_x^j f_{p} )^\top A^0_p \partial_x^j f_{p} (t, x) dx,
\end{gather}
\end{subequations}
where $\mu \in (0,1)$ is to be determined in Lemma \ref{Lem57}; 
$W$, $E_{bp}^{j}$, and $D_{bp}^{j}$ are defined in \eqref{eq522} and \eqref{EF5}, respectively.

\begin{lem} \label{Lem57}
Under the same assumptions as in Proposition \ref{Pro50}, 
if $N(T)$ is sufficiently small, there exists a constant $\mu \in (0,1)$ 
such that the following estimate holds for $t_{0} \in [0,T]$ and $t \in [t_{0},T]${\rm :}
\begin{gather} \label{eqLem57} 
E_{bp}(t)
+ c \int_{t_{0}}^t { D_{bp} (\tau) } d\tau 
\le
E_{bp}(t_{0}) 
+ C N(T) \int_{t_{0}}^t E_{bp}(\tau) d\tau, 
\end{gather}
where $c$ and $C$ are positive constants independent of $t_{0}$ and $T$,
and the time-integral of the traces of the highest order derivatives $\partial_t^2 f_{b}$ and $\partial_t^2 f_{p}$ are valid by Lemma \ref{local5}.
\end{lem}
\begin{proof}
Firstly we prove the following inequality on $E_{bp}^{j}$ for $j=1,2$ parallel to \eqref{eqLem54}:
\begin{multline} \label{eqLem57*} 
E_{bp}^j(t)
{ + c \int_{t_{0}}^t D_{bp}^{j} (\tau) d \tau }
\\
\le
E_{bp}^j(t_{0}) 
+ C N(T) \int_{t_{0}}^t D_{bp}(\tau) d\tau
+C \mu^{-3} N(T) \int_{t_{0}}^t \Biggl(E^{\rm sp}_{p,\mu}+\sum_{i=0}^{2}E_{bp}^{i}\Biggr) (\tau) d\tau, 
\end{multline}
where $c,C>0$ are constants independent of $\mu$, $t_{0}$, and $T$.
To verify this, we show  
\begin{gather} \label{eq553}
E_b^j(t) 
+ \int_{t_{0}}^t ( \partial_t^j f_{b} )^\top A_{b}^{1} \partial_t^j  f_{b} (\tau, x_*) d\tau
\le
E_b^j(t_{0})  + C N(T)  \sum_{i=0}^{2} \int_{t_{0}}^t E^{i}_{b} (\tau) d\tau,
\end{gather}
and
\begin{multline}
E_p^j(t) 
- \int_{t_{0}}^t ( \partial_t^j f_{p} )^\top A_{p}^{1} \partial_t^j  f_{p} (\tau, x_*) d \tau
+ c \int_{t_{0}}^t D_{p}^{j} (\tau) d \tau
\\ 
\leq E_p^j(t_{0})+ C \mu^{-3} N(T) \int_{t_{0}}^t \Biggl(E^{\rm sp}_{p,\mu}+\sum_{i=0}^{2}E_{p}^{i}\Biggr)(\tau) d\tau,
\label{eq556}
\end{multline}
which are parallel to \eqref{eqLem53} and \eqref{eqLem52}, respectively.

We obtain \eqref{eq553} by repeating the proof of Lemma \ref{Lem36} 
but with $\bar{I}_{b}=[x_{1},x_{*}]$ in place of $\bar{I}=[x_{1},x_{2}]$.  
Indeed, recall that the second term on the LHS of \eqref{eqLem36} in Lemma \ref{Lem36} has been obtained by estimating $( \partial_t^j f_{b} )^\top A_{b}^{1} \partial_t^j  f_{b} (t,x_2)$
from below by $c|\partial_t^j  f_{b} |^{2} (t,x_2)$ using the boundary condition \eqref{eq314a}. 

In order to show \eqref{eq556}, 
we multiply \eqref{eq554} by $2W( \partial_t^j f_{p} )^\top$ from the left to get
\begin{multline}\label{eq523*}
\partial_t \left( W ( \partial_t^j f_{p} )^\top A_{p}^{0} \partial_t^j  f_{p} \right) 
+ \partial_x \left( W ( \partial_t^j f_{p} )^\top A_{p}^{1} \partial_t^j  f_{p}  \right) 
\\
+ W ( \partial_t^j f_{p} )^\top (R_{p} +k_{*}A^{1}_{p})  \partial_t^j  f_{p} 
= 2W(\partial_t^j f_{p} )^\top {\cal F}^{j}_{p},
\end{multline}
which is parallel to \eqref{eq523}.
Recalling the estimate \eqref{calF51} of ${\cal F}^{j}_{p}$, and following the derivation of \eqref{eqLem52}, we conclude \eqref{eq556}.
%

Then we can obtain \eqref{eqLem57*} by following the proof of Lemma \ref{Lem54} 
but replacing \eqref{eq528}, \eqref{eqLem53}, and \eqref{eqLem52} 
by \eqref{eq528*}, \eqref{eq553}, and \eqref{eq556}, respectively.

Secondly we prove the following inequality on $E^{\rm sp}_{p,\mu}$:
\begin{gather} \label{eqLem51} 
E^{\rm sp}_{p,\mu} (t) \leq
E^{\rm sp}_{p,\mu} (t_{0}) 
+ C \mu \int_{t_0}^t { D_{bp}(\tau) } d \tau,
\end{gather}
where $C$ is a positive constant independent of $\mu$, $t_{0}$, and $T$.
For $j=0,1,2$, we multiply \eqref{eq561} by $2 \mu^{j+1}W(\partial_{x}^{j} f_p)^{\top}$ from the left, 
and use the definition of $R^{j}_{p}$ in Proposition \ref{Pro51} to have
\begin{multline}\label{eq523**}
\partial_t \left( \mu^{j+1}W (\partial_{x}^{j} f_p)^{\top} A^0_p \partial_{x}^{j} f_p \right) 
+ \partial_x \left( \mu^{j+1}W (\partial_{x}^{j} f_p)^{\top} A^1_p \partial_{x}^{j} f_p \right) 
+ \mu^{j+1} W (\partial_{x}^{j} f_p)^{\top} R^{j}_{p} \partial_{x}^{j} f_p 
\\
= 2 \mu^{j+1} W (\partial_{x}^{j} f_p)^{\top}{\cal G}^{j}_{p},
\end{multline}
which is parallel to \eqref{eq523}.
From \eqref{eq528}, \eqref{ajs2}, and $x_{*} \in \bar{I}_{b}$, we have the relation
\begin{gather*}
\sum_{j=0}^{2} |\partial_{x}^{j} f_{p}|(t, x_*) \leq C \sum_{j=0}^{2} |\partial_{x}^{j} f_{b}|(t, x_*) 
\leq C D_{bp}(t).
\end{gather*}
Now we follow the derivation of \eqref{eq545} in Lemma \ref{Lem55}.
Namely, integrating the product of $e^{q t}$ and \eqref{eq523**} over $[t_{0}, t] \times \bar{I}_{p}$
for a positive constant $q$, summing up all the resultant equalities, 
and using the above relation, we arrive at  
\begin{multline}\label{Espa1}
E^{\rm sp}_{p,\mu}(t) 
+ \sum_{j=0}^{2} \mu^{j+1} \int_{t_{0}}^t\!\!\int_{x_{*}}^{x_{2}} e^{-q (t-\tau)} W (\partial_{x}^{j} f_p)^{\top}  (R^{j}_{p}-q A^{0}_{p}) \partial_{x}^{j} f_p (\tau, x) dxd\tau
\\
\leq
e^{-q(t-t_0)}E^{\rm sp}_{p,\mu}(t_{0})
+ C \mu \int_{t_0}^t { D_{bp}(\tau)} d \tau + 2 \sum_{j=0}^{2} \mu^{j+1} \int_{t_{0}}^t\!\!\int_{x_{*}}^{x_{2}} e^{-q (t-\tau)}  W (\partial_{x}^{j} f_p)^{\top}{\cal G}^{j}_{p} (\tau, x) dxd\tau,
\end{multline}
where $C$ is a positive constant independent of $\mu$, $t_{0}$, and $T$.

We claim that the rightmost term can be absorbed into the second term on the LHS if $\mu$ is sufficiently small.
Recall $R^{j}_{p}$ is positive definite in $I_{p}$ as shown in Proposition \ref{Pro51}.
Letting $q\ll 1$ yields
\begin{gather*}
\text{(2nd term on LHS)} \geq c\sum_{j=0}^{2} \mu^{j+1} \int_{t_{0}}^t\!\!\int_{x_{*}}^{x_{2}} e^{-q (t-\tau)} W  |\partial_{x}^{j} f_p |^{2} (\tau, x) dxd\tau \geq 0.
\end{gather*}
On the other hand, using \eqref{calG51} and Schwarz's inequality in view of ${\cal G}^{0}_{p} = 0$, we have
\begin{gather*}
\text{(rightmost term)} \leq C \mu^{1/2} \sum_{j=0}^{2} \mu^{j+1} \int_{t_{0}}^t\!\!\int_{x_{*}}^{x_{2}} e^{-q (t-\tau)} W  |\partial_{x}^{j} f |^{2} (\tau, x) dxd\tau.
\end{gather*}
Therefore choosing $\mu$ sufficiently small, we verify the claim.
From this claim and \eqref{Espa1}, we concludes \eqref{eqLem51}.

Finally, summing up \eqref{eqLem54}, \eqref{eqLem57*} for $j=1,2$, and \eqref{eqLem51}, 
choosing $\mu$ sufficiently small, and letting $N(T) \ll 1$, we conclude \eqref{eqLem57}.
\end{proof}

Hereinafter we fix $\mu$ so that \eqref{eqLem57} in Lemma \ref{Lem57} holds. 
Let us derive an estimate corresponding to \eqref{eqLem56} in Lemma \ref{Lem55}.

\begin{lem} \label{Lem59}
Under the same assumptions as in Proposition \ref{Pro50} with $T \geq T_{0}^{*}$, 
if $N(T)$ is sufficiently small, 
the following estimate holds for $\tilde{t}_{0} \in [0,T-T_{0}^{*}]$ and $t \in [\tilde{t}_{0}+T_{0}^{*},T]${\rm :}
\begin{gather} \label{eqLem58}
( 1 + c ) E_{bp} (t)
+ c \int_{\tilde{t}_{0}}^t D_{bp} (\tau) d\tau
\le E_{bp} (\tilde{t}_{0}) + C N(T) \int_{\tilde{t}_{0}}^t E_{bp} (\tau) d\tau,
\end{gather}
where $c$ and $C$ are positive constants independent of $\tilde{t}_{0}$ and $T$.
\end{lem}
\begin{proof}
We claim that the following inequalities hold:
\begin{gather} 
E_b^j(t) 
\le 
C \int_{\tilde{t}_{0}}^t D_b(\tau) d \tau 
+ C N(T) \int_{\tilde{t}_{0}}^t E_b(\tau) d\tau,
\ \ j=1,2,
\label{eqLem59}
\\
( 1 + c ) E_p^j(t) 
\le 
E_p^j(\tilde{t}_{0}) + C \int_{\tilde{t}_{0}}^t D_b(\tau) d\tau + C N(T) \int_{\tilde{t}_{0}}^t E_p (\tau) d\tau,
\ \ j=1,2,
\label{eqLem50}
\\  
(1+c) E^{\rm sp}_{p,\mu}(t) 
\le 
E^{\rm sp}_{p,\mu} (\tilde{t}_{0}) 
+ C \int_{\tilde{t}_{0}}^t D_{bp}(\tau)  d\tau, 
\label{eqLem50*}
\end{gather}
where $c$ and $C$ are positive constants independent of $\tilde{t}_{0}$ and $T$.

To derive \eqref{eqLem59}, we repeat the proofs of Lemmas \ref{Lem37} and \ref{Lem38} 
but replacing $[x_{1},x_{2}]$ by $\bar{I}_{b}=[x_{1},x_{*}]$.
Then we can deduce
\begin{gather} \label{eq537*}
E_b^{j}(\tilde{t}_{0}+T_{-}^{*})
\le 
C \int_{\tilde{t}_{0}}^{\tilde{t}_{0}+T_0^{*}} D_{b}^{j}(\tau) d\tau
+ C N(T) \int_{\tilde{t}_{0}}^{\tilde{t}_{0}+T_0^{*}} E_b(\tau) d\tau,
\end{gather}
which is parallel to \eqref{eqLem38}.
Then we substitute \eqref{eq537*} into \eqref{eq553} with $t_{0}=\tilde{t}_{0}+T_{-}^{*}$, and move the second term on the LHS of the resultant inequality to the RHS to obtain \eqref{eqLem59}. 
The proof of \eqref{eqLem50} is done in much the same way as in the proof of \eqref{eq547} but with \eqref{eq523*} and \eqref{eq528*} in place of \eqref{eq523} and \eqref{eq528} respectively.
Indeed, recalling that the transport equation of energy \eqref{eq523*} corresponds to \eqref{eq523}, 
and that the remaining term ${\cal F}^{j}_{p}$ satisfies \eqref{calF51}, we can conclude \eqref{eqLem50}.
The estimate \eqref{eqLem50*} is obtained simply by replacing $t_{0}$ by $\tilde{t}_{0}$ in \eqref{Espa1}, 
and using $e^{-q (t-\tilde{t}_{0})} \le e^{-q T_0^{*}}<1$.

Finally, we sum up \eqref{eqLem56}, twice of \eqref{eqLem59} for $j=1,2$, \eqref{eqLem50} for $j=1,2$, and \eqref{eqLem50*},
then add the product of the resultant inequality and $\ve>0$ to \eqref{eqLem57} with $t_0=\tilde{t}_{0}$, 
and finally let $\ve \ll 1$ to obtain \eqref{eqLem58}.
The proof is complete.
\end{proof}

\subsection{Completion of the a priori estimate}

This subsection is devoted to the proof of Proposition \ref{Pro50}. 

\begin{proof}[Proof of Proposition \ref{Pro50}]
First it follows from the application of Gronwall's inequality to \eqref{eqLem57} that 
\begin{gather}\label{eq560}
E_{bp}(t) \leq E_{bp}(t_{0})e^{C_{1}N(T)(t-t_{0})}  \quad \text{for} \  t \in [t_{0},T],
\end{gather}
where  $C_{1}>0$ is a constant independent of $t_{0}$ and $T$.

Secondly, applying Gronwall's inequality to \eqref{eqLem58}, 
we have
\begin{gather}\label{eq570}
E_{bp}(t) \le \frac{1}{ 1 + c_1 } \left( E_{bp} (\tilde{t}_{0}) + C_{1}N(T) \int_{\tilde{t}_{0}}^{\tilde{t}_{0}+T_0^{*}} \!\! E_{bp} (\tau) d\tau \right)  
e^{C_{1}N(T)(t-\tilde{t}_{0}-T_{0}^{*})}
\ \ \text{for} \ t \in [\tilde{t}_{0}+T_0^{*}, T],
\end{gather}
where $C_{1}>0$ can be the same constant as in \eqref{eq560}, and $c_1>0$ is a constant  independent of $\tilde{t}_{0}$ and $T$.
We note that \eqref{eq560} and \eqref{eq570} are parallel to \eqref{eq355} and \eqref{eq356}, respectively.
It is also clear that $E(0) \leq C\|f^{0}_{b}\|_{H^{2}(I)}^{2}$.
Therefore, following the proof of Proposition \ref{Pro31}, we can conclude \eqref{eqPro50}.
\end{proof}

\section{Transonic problem 2: continuous and discontinuous transitions}\label{Sec7}

{This section is devoted to the proof of Theorem \ref{Thm15}.
Namely, we prove the asymptotic stability of discontinuous stationary solutions 
in which the flow} undergoes not only a continuous transition 
from a subsonic state to a supersonic state around the throat at $x=0$,
but also a discontinuous transition 
from a supersonic state to a subsonic state at a single point $x = d(t) \in (0,x_2)$.

We rewrite the system of equations \eqref{eq101} 
with a moving discontinuity at $x=d(t)$ following \eqref{eq106a} into a fixed coordinate system.
First, let us denote the restriction of the stationary solution {$(\bar{\rho},\bar{m})$}
on $[x_1, \bar{d})$ by $(\bar{\rho}_1, \bar{m}_1)$, 
and on $(\bar{d}, x_2]$ by $(\bar{\rho}_2, \bar{m}_2)$, respectively. 
{We also set $\bar{u}_{1}:= \bar{m}_{1} / \bar{\rho}_1$ and $\bar{u}_{2}:= \bar{m}_{2} / \bar{\rho}_2$.}
We note there exists a positive constant $\varepsilon_0$ such that {the smooth extension of 
$(\bar{\rho}_1, \bar{m}_1)$ to $[x_1, \bar{d}+\varepsilon_0)$ satisfies the equations \eqref{eq151}, and  the smooth extension of  $(\bar{\rho}_2, \bar{m}_2)$ to $(\bar{d}-\varepsilon_0, x_2]$ satisfies the equations \eqref{eq151}.
}
{It is expected} that the discontinuity $x=d(t)$ always lies in $[\bar{d}-\varepsilon_0, \bar{d}+\varepsilon_0] \subset (0,x_2)$ if the {initial} perturbation is small enough.

We introduce the perturbation from the stationary solution as
\begin{align*}
(\tilde{\phi}_{1}, \tilde{\omega}_{1}, \tilde{\psi}_{1})(t,x) &:= 
(\rho,m,u)(t,x) - (\bar{\rho}_{1},\bar{m}_{1},\bar{u}_{1})(x) \quad \text{for} \ x \in J_{1}(t):=(x_{1},d(t)),
\\
(\tilde{\phi}_{2}, \tilde{\omega}_{2}, \tilde{\psi}_{2})(t,x) &:= 
(\rho,m,u)(t,x) - (\bar{\rho}_{2},\bar{m}_{2},\bar{u}_{2})(x) \quad \text{for} \ x \in J_{2}(t):=(d(t),x_{2}),
\\
\tilde{\sigma}(t) &:= d(t) - \bar{d}.
\end{align*}
We perform a change of coordinates $(t,x)$ into $(s,y)$ defined by
\begin{gather*}
\left\{
\begin{array}{lll}
 s = t ,
\\
 y = x_i + \dfrac{\bar{d} - x_i}{\bar{d}+\sigma(t) - x_i} ( x - x_i ) 
 \end{array}
\right. \quad
\text{for}  \ t \in \mathbb R_{+} \ \text{and} \ x \in J_{i}(t),
\end{gather*}
{where $i =1,2$.}
{We note that}
$\overline{J_{1}(t)}=[x_1,d(t)]$ and $\overline{J_{2}(t)}=[d(t),x_2]$ are mapped onto $\bar{I}_{1}$ and $\bar{I}_{2}$ for all $t>0$, respectively, where
\begin{gather*}
I_{1}:=(x_1,\bar{d}), \quad  I_{2}:=(\bar{d},x_2).
\end{gather*}
Conversely, the original coordinates are represented {as}
\begin{gather*}
\left\{
\begin{array}{lll}
t = t(s,y) = s, 
\\[10pt]
x = x(s,y) = y + \sigma(s) \dfrac{y - x_i}{\bar{d} - x_i} 
\end{array}
\right. \quad
\text{for}  \ (s,y) \in \mathbb R_{+} \times I_{i}, 
\end{gather*}
where $i =1,2$.
For later convenience, let us denote the RHS of the latter transform by $\hat{x}_{i}(\sigma,y)$:
\begin{equation*}
\hat{x}_{i}(\sigma,y) := y + \sigma \dfrac{y - x_i}{\bar{d} - x_i} \quad 
\text{for} \ (\sigma,y) \in (x_1-\bar{d}, x_2-\bar{d}) \times I_i.
\end{equation*}
The differential operators are replaced with
\begin{gather*}
\left\{
\begin{array}{lll}
\partial_t = \partial_s - \dot{\sigma} \dfrac{y - x_i}{\bar{d}+\sigma(s) - x_i} \,  \partial_y,
\\[10pt]
\partial_x = \dfrac{\bar{d} - x_i}{\bar{d}+\sigma(s) - x_i} \, \partial_y
\end{array}
\right. \quad
\text{for}  \ (s,y) \in \mathbb R_{+} \times I_{i}.
\end{gather*}
The perturbations in the new coordinate are defined by
\begin{align*}
(\phi_{i}, \omega_{i}, \psi_{i})(s,y) &:= (\tilde{\phi}_{i}, \tilde{\omega}_{i}, \tilde{\psi}_{i})(t(s,y), x(s,y) ), \quad
\sigma(s) := \tilde{\sigma}(t(s,y)) \quad \text{for} \ (s,y) \in \mathbb R_{+} \times I_{i}.
\end{align*}
Hereafter fundamental functions $\rho_{i}$, $m_{i}$, and $u_{i}$ mean 
\begin{gather*}
(\rho_{i}, m_{i}, u_{i}) = (\rho_{i}, m_{i}, u_{i})(s,y) 
= (\bar{\rho}_i, \bar{m}_{i}, \bar{u}_i) (x(s,y)) + (\phi_{i}, \omega_{i}, \psi_{i})(s,y) \quad \text{for} \ (s,y) \in \mathbb R_{+} \times I_{i}.
\end{gather*}
Then we see from the systems \eqref{eq101}, \eqref{eq151}, and \eqref{eq106a},
the boundary conditions \eqref{eq104}, \eqref{eq105}, \eqref{eq110}, and \eqref{eq153}--\eqref{eq157}, 
and the initial conditions \eqref{eq108}, \eqref{eq109} 
that $f_{i}:=(\phi_{i},\omega_{i})^{\top}$ for $i=1,2$ and $\sigma$ solve the system
\begin{subequations} \label{eq701*}
\begin{gather}
\partial_s f_{i} + \hat{A}_{i}^1 \partial_y f_{i} + \hat{B}_{i} f_{i} = 0 \quad \text{in} \ \mathbb R_{+} \times I_{i},
\label{eq701}\\
 \dot{\sigma}(s) \jump{ \rho }(s,\bar{d}) 
 = \jump{\omega}(s,\bar{d}) \quad \text{for} \ s>0,
\label{eq706a}
\end{gather}
with the initial--boundary value
\begin{gather} 
f_{i}(0,y) = f_{i}^0(y) = 
\begin{bmatrix}
\phi_{i}^0(y)
\label{eq709}\\ 
\omega_{i}^0(y)
\end{bmatrix}
:=
\begin{bmatrix}
\rho^0(x(0,y))-\bar{\rho}_{i}(x(0,y))
\\ 
m^0(x(0,y))-\bar{m}_{i}(x(0,y))
\end{bmatrix} \quad \text{for} \ y \in I_{i},
\\
\sigma(0)=\sigma^{0} := d^0 - \bar{d},
\label{eq709b}
\\
\omega_{1}(s,x_{1})=0, \quad \phi_{2}(s,x_{2})=0 \quad \text{for} \ s>0,
\label{eq705a}
\\
\jump{ \rho } (s, \bar{d}) \Bigl[\!\Bigl[ \dfrac{m^2}{\rho} + p(\rho) \Bigr]\!\Bigr] (s, \bar{d})
= \jump{ \omega } (s, \bar{d})^2 \quad \text{for} \ s>0,
\label{eq706c}
\end{gather}
\end{subequations}
where {the function $\dot{\sigma}$ and the constant $\sigma^{0}$ denote the $s$-derivative and the initial value of $\sigma$, respectively.} Here
\begin{align*}
\jump{g}(s,\bar{d}) &= g_{2}(s,\bar{d}) - g_{1}(s,\bar{d}) \quad \text{for $g_{i} \in C^{0}(\overline{\mathbb R_{+} \times I_{i}})$},
\\
\jump{g}(x) &= g_{2}(x) - g_{1}(x) \quad \text{for $g_{1} \in C^{0}([x_{1},\bar{d}+\ve])$, $g_{2} \in C^{0}([\bar{d}-\ve,x_{2}])$},
\\
\hat{A}_{i}^{1} = \hat{A}_{i}^{1}(y,\rho_i, m_i, \sigma, \dot{\sigma}) &:=
 \frac{\bar{d} - x_i}{\bar{d}+\sigma - x_i} \hat{A}^1(\rho_i, m_i)
- \dot{\sigma} \frac{y - x_i}{\bar{d}+\sigma - x_i} {\rm id},
\\
\hat{B}_i = \hat{B}_i(y,\rho_i, m_i,\sigma) &:= \hat{B}(\hat{x}_{i}(\sigma,y),\rho_i, m_i),
\end{align*}
where $\hat{A}^1$ and $\hat{B}$ are defined just below \eqref{eq311*},
and ${\rm id}$ is an identity matrix.
We require the compatibility conditions at $y=x_1, x_2$ and $y=\bar{d}$. 
At $y=x_1, x_2$, they can be formally written as
$\disp \partial_s^j \omega_1(0,x_1) = 0$ and $\disp \partial_s^j \phi_2(0,x_2) = 0$ for $j=0,1$.
At $y=\bar{d}$, the compatibility conditions can be formally given 
by the $j$-th ($j=0,1$) derivatives of \eqref{eq706c} in $s$ when $s=0$.
The precise formulations in terms of the initial value are
\begin{subequations} \label{comp7}
\begin{gather}
\omega_1^0(x_1) = 0, \\
(0 \ 1) 
\left( \hat{A}^1_1(x_1, \rho_1^0(x_1), m_1^0(x_1), \sigma^0, ``\dot{\sigma}" ) (f_1^0)_y(x_1)
 + \hat{B}_1(x_1, \rho_1^0(x_1), m_1^0(x_1), \sigma^0) f_1^0(x_1) \right) = 0,
\\
\phi_2^0(x_2) = 0, \\ 
(1 \ 0) 
\left( \hat{A}^1_2(x_2, \rho_2^0(x_2), m_2^0(x_2), \sigma^0, ``\dot{\sigma}" ) (f_2^0)_y(x_2)
 + \hat{B}_2(x_2, \rho_2^0(x_2), m_2^0(x_2), \sigma^0) f_2^0(x_2) \right) = 0,
\\
\jump{ \rho^0 } (\bar{d}) \Bigl[\!\Bigl[ \dfrac{(m^0)^2}{\rho^0} + p(\rho^0) \Bigr]\!\Bigr] (\bar{d})
= \jump{ \omega^0 } (\bar{d})^2,
\\
\jump{ ``\partial_s \rho" } (\bar{d}) \Bigl[\!\Bigl[ \dfrac{(m^0)^2}{\rho^0} + p(\rho^0) \Bigr]\!\Bigr (\bar{d})
+
\jump{ \rho^0 } (\bar{d}) 
\Bigl[\!\Bigl[ 
 \left( - \dfrac{(m^0)^2}{(\rho^0)^2} + p'(\rho^0) \right) ``\partial_s \rho" 
 + \dfrac{2m^0}{\rho^0} ``\partial_s m" 
\Bigr]\!\Bigr] (\bar{d})
\notag \\
\hspace{50mm}
= 2 \jump{ \omega^0 } (\bar{d}) \jump{ ``\partial_s \omega'' } (\bar{d}),
\end{gather}
\end{subequations}
where $s$-derivative terms in the last condition should be understood as
\begin{gather*}
``\partial_s \rho_i" (\bar{d}) = ``\dot{\sigma}" (\bar{\rho}_i)_x(\bar{d}+\sigma^0) 
+ ``\partial_s \phi_i"(\bar{d}),
\quad
``\partial_s m_i" (\bar{d}) = ``\dot{\sigma}" (\bar{m}_i)_x(\bar{d}+\sigma^0) 
+ ``\partial_s \omega_i"(\bar{d})
\end{gather*}
with $\disp ``\dot{\sigma}" = \jump{\omega^0}(\bar{d}) / \jump{\rho^0}(\bar{d})$ 
due to \eqref{eq706a}, and
\begin{gather*}
\begin{bmatrix}
``\partial_s \phi_i"
\\ 
``\partial_s \omega_i"
\end{bmatrix} (\bar{d})
= -  \left ( \hat{A}^1_i(\bar{d}, \rho^0_i(\bar{d}), m^0_i(\bar{d}), \sigma^0, ``\dot{\sigma}") (f^0_i)_y(\bar{d}) 
             + \hat{B}_i(\bar{d}, \rho^0_i(\bar{d}), m^0_i(\bar{d}), \sigma^0) f^0_i(\bar{d}) \right).
\end{gather*}
We seek the solution $(f_{1},f_{2},\sigma)$ of the problem \eqref{eq701*} in the function space 
$\mathfrak{X}_{1} (T) \times  \mathfrak{X}_{2} (T) \times  H^{3}(0,T)$ for $T>0$, where
\begin{align*}
\mathfrak{X}_{i} (T):= \left[ \left\{ \left. f_{i} \in  \bigcap_{j=0}^{2} C^j([0,T]; H^{2-j}(I_{i})) \, \right| \, f_{i}(\cdot,\bar{d}) \in H^{2}(0,T) \right\}\right]^{2}.
\end{align*}
We denote by $\mathfrak{X}_{i} (\infty)$ the function space which is replaces $T$ by $\infty$ in $\mathfrak{X}_{i} (T)$.

We are in a position to state the main theorem in this section, which immediately ensures Theorem \ref{Thm15}.

\begin{thm} \label{Thm70}
Assume $A \in  C^4(\bar{I})$ with $A_{x}(0) = 0$ and $A_{xx}(0) > 0$.
Suppose that there exists a constant $\bar{d} \in (0, x_2)$ such that 
$A_x(\bar{d}) > 0$ holds, and that a stationary solution 
$(\bar{\rho}, \bar{m})  \in C^3(\bar{I} \setminus \{ \bar{d} \} ) \times C^4(\bar{I})$ to (\ref{eq151})--(\ref{eq159}) 
exists and satisfies
\begin{gather}\label{spro71}
\inf_{x \in \bar{I}} \bar{\rho}(x) >0,  \quad
\begin{cases}
\bar{M}(x) < 1 & \text{for} \ x \in [x_1, 0),
\\
\bar{M}(x) = 1 & \text{for} \ x = 0,
\\
\bar{M}(x) > 1 & \text{for} \ x \in (0, \bar{d}),
\\
\bar{M}(x) < 1 & \text{for} \ x \in (\bar{d}, x_2].
\end{cases}
\end{gather}
Then there exists a constant $\delta_{0}>0$ such that if an initial value $(f_{1}^{0},f_{2}^{0},\sigma^{0}) \in (H^{2}(I_{1}))^{2} \times (H^{2}(I_{2}))^{2} \times \mathbb R$ satisfies the compatibility condition \eqref{comp7} and $\| f^0_1 \|_{H^2(I_1)} + \| f^0_2 \|_{H^2(I_2)} + | \sigma_0 | \leq \delta_{0}$,
then the problem \eqref{eq701*} admits a unique global solution $(f_{1},f_{2},\sigma) \in  \mathfrak{X}_{1} (\infty) \times  \mathfrak{X}_{2} (\infty) \times  H^{3}(0,\infty)$.
Moreover, the solution $(f_{1},f_{2},\sigma)$  decays in $(H^{2}(I_{1}))^{2} \times (H^{2}(I_{2}))^{2} \times \mathbb R$ exponentially fast  as $s \to \infty$.
\end{thm}
\begin{rem}
We remark that the assumptions on $A$ in Theorem \ref{Thm70} (see also Theorem \ref{Thm15}) is weaker than the condition (\ref{eq231}) on $A$ in Theorem \ref{thm2.1}. 
It is seen that the stationary solution in Theorem \ref{Thm70} also satisfies
\begin{gather}\label{spro72}
\inf_{x \in \bar{I}_{i}} \bar{u}_{i}(x) = \inf_{x \in \bar{I}_{i}} \frac{Q}{A\bar{\rho}_{i}}(x) >0
\ \ \text{for} \ i=1,2, \ \
\bar{\rho}_{1x}(0)<0, \quad \bar{u}_{1x}(0)>0, \quad \jump{\bar{\rho}}(\bar{d})>0.
\end{gather}
\end{rem}

Theorem \ref{Thm70} is an immediate consequence of the time-local solvability in Lemma \ref{local7}
and an a-priori estimate in Proposition \ref{Pro70}.
\begin{lem}\label{local7}
Assume the same conditions as in Theorem \ref{Thm70} for $A$ and $(\bar{\rho}, \bar{m})$.
Then there exist constants $\delta_{\sharp}>0$ and $T_{\sharp}>0$ such that if the initial value $(f_{1}^{0},f_{2}^{0},\sigma^{0}) \in (H^{2}(I_{1}))^{2} \times (H^{2}(I_{2}))^{2} \times \mathbb R$ satisfies the compatibility condition {\eqref{comp7}} and $\| f^0_1 \|_{H^2(I_1)} + \| f^0_2 \|_{H^2(I_2)} + | \sigma_0 |\leq \delta_{\sharp}$, 
then the problem \eqref{eq701*} admits a unique solution 
$(f_{1},f_{2},\sigma) \in  \mathfrak{X}_{1} (T_{\sharp}) \times  \mathfrak{X}_{2} (T_{\sharp}) \times  H^{3}(0,T_{\sharp})$.
Furthermore, if $(f_{1},f_{2},\sigma) \in  \mathfrak{X}_{1} (T) \times  \mathfrak{X}_{2} (T) \times  H^{3}(0,T)$
is a solution of \eqref{eq701*}, then the following $s$-integral of the trace of the highest order derivative $\partial_{s}^{j}\partial_{y}^{k} f_{i}$ makes sense for $i=1,2$, $j,k=0,1,2$ with $j+k=2$, 
and the point $y \in \bar{I}_{i}$ at which none of {$\frac{\bar{d} - x_i}{\bar{d}+\sigma - x_i} (u_{i} \pm \sqrt{p'(\rho_{i})}) - \dot{\sigma} \frac{y - x_i}{\bar{d} - x_i}$}, characteristic speed of \eqref{eq701}, is zero for all $s \in [0,T]${\rm :}
\begin{gather}\label{trace7}
\int_{0}^{T} |\partial_{s}^{j}\partial_{y}^{k} f_{i}|^{2}(s,y) ds.  
\end{gather}
\end{lem}
\begin{proof}
We will show this lemma in Appendix.
\end{proof}
To present an a-priori estimate, we introduce 
\begin{gather*} 
N_\sigma(T) := \|\sigma\|_{C^{2}(0,T)}+\sup_{s \in [0,T]} \left(\| f_{1}(s,\cdot) \|_{H^2(I_{1})}+\| f_{2}(s,\cdot) \|_{H^2(I_{2})}\right).
\end{gather*}
\begin{pro} \label{Pro70}
{
Assume the same conditions as in Theorem \ref{Thm70} for $A$ and $(\bar{\rho}, \bar{m})$.
Suppose that the problem \eqref{eq701*} has a solution 
$(f_{1},f_{2},\sigma) \in  \mathfrak{X}_{1} (T) \times  \mathfrak{X}_{2} (T) \times  H^{3}(0,T)$ for some $T>0$.
Then there exist positive constants $\delta_{0}$, $C$, and $\nu$ independent of $T$ such that if $N_\sigma(T) \leq \delta_{0}$, 
then the following estimate holds for $s \in [0,T]$$:$
}
\begin{gather} \label{eqPro70}
 \| f_{1}(s,\cdot) \|_{H^2(I_{1})}+\| f_{2}(s,\cdot) \|_{H^2(I_{2})} + {|\sigma(s)| }
\le C e^{-\nu t} \left(\| f_{1}^{0} \|_{H^2(I_{1})}+\| f_{2}^{0} \|_{H^2(I_{2})} +|\sigma^{0}| \right).
\end{gather}
\end{pro}

Combining the time-local solvability in Lemma \ref{local7} and the a-priori estimate in Proposition \ref{Pro70}, we can prove Theorem \ref{Thm70}. We omit the proof, since it is similar to that of Theorem \ref{Thm30}.

In the remaining part of this section, we show only Proposition \ref{Pro70}. 
{
We observe that 
\begin{gather}\label{approx1}
|x(s,y)-y|+|\partial_{y} x(s,y)-1|+|\sigma(s)| + |\dot{\sigma}(s)| + |\ddot{\sigma}(s)|  \leq CN_\sigma(T), \ \ (s,y) \in [0,T] \times \bar{I}_{i}
\end{gather}
for $i=1,2$.
This allows the matrices $\hat{A}^1_{i}$ and $\hat{B}_{i}$ in \eqref{eq701} to be approximated as follows:
\begin{subequations}\label{approxM1}
\begin{align}
|\hat{A}^1_{i}(y,\rho_{i}(s,y), m_{i}(s,y), \sigma(s), \dot{\sigma}(s)) - \hat{A}^1(y, \bar{\rho}_{i}(y), \bar{m}_{i}(y))| & \leq CN_\sigma(T),
\\
|\hat{B}_{i}(y,\rho_{i}(s,y), m_{i}(s,y), \sigma(s)) - \hat{B}(y, \bar{\rho}_{i}(y), \bar{m}_{i}(y))| & \leq CN_\sigma(T),
\end{align}
\end{subequations}
where $\hat{A}^1$ and $\hat{B}$ are the same as in \eqref{eq317}.
Recall the assumption \eqref{spro71} and the definition that $(\bar{\rho}_{i},\bar{m}_{i})(y)=(\bar{\rho},\bar{m})(y)$ for $y \in I_{i}$ and $i=1,2$.
Within the interval $I_{1}$, the stationary solution $(\bar{\rho}_{1},\bar{m}_{1})$ is subsonic in $(x_{1},0)$, exactly transonic at $y=0$, and supersonic in $(0,\bar{d})$. The stationary solution $(\bar{\rho}_{2},\bar{m}_{2})$ is entirely subsonic in the interval $I_{2}$.
The situation in $I_1$ is similar to the problem in Section 4, 
while that in $I_2$ resembles the problem in Section 3.
The arguments in these sections are useful here, but spectial care is needed to handle the motion 
of the shock location.
In subsections \ref{S7.1} and \ref{S7.2}, we will derive the a priori estimates in $I_{1}$ and $I_{2}$, respectively.
}

\subsection{A priori estimate in $I_{1}$}\label{S7.1}
In this subsection, we show the a priori estimate in $I_{1}$ using the same energy method as in Section \ref{Sec5}.
Namely, we divide the interval $I_{1}$ into the subsonic region $I_{1b}$ and the extended supersonic region $I_{1p}$ as
\begin{gather*}
I_{1} = I_{1b} \cup I_{1p}, \quad I_{1b}:=(x_{1},y_{*}), \quad I_{1p}:=[y_{*},\bar{d}),
\end{gather*}
where $y_{*} \in (x_{1},0)$ is specified in Proposition \ref{Pro71} below. 

We use the subscripts $b$ and $p$ to distinguish between quantities utilized in $I_{1b}$ and $I_{1p}$, respectively. 
In the subsonic region $I_{1b}$, we adopt the following symmetric system for the unknown function $f_{1b}:=(\phi_{1},\omega_{1})^{\top}$:
\begin{gather}\label{eq702}
A^0_ {1b} \partial_s f_{1b} + A^1_{1b} \partial_y f_{1b} + B_{1b} f_{1b} = 0 \quad \text{in $\mathbb R_{+} \times I_{1b}$},
\end{gather}
where 
\begin{align*}
A^0_{1b}(y,\rho_{1},m_{1},\sigma) &:= \dfrac{A(\hat{x}_{1}(\sigma,y))}{\bar{\rho}(\hat{x}_{1}(\sigma,y))} \, {\rm diag} (S(\rho_{1},m_{1}),1),\\
A^1_{1b}(y,\rho_{1}, m_{1}, \sigma, \dot{\sigma})& := A^0_{1b}(y,\rho_{1},m_{1},\sigma) \hat{A}^1_{1}(y,\rho_{1}, m_{1}, \sigma, \dot{\sigma}), \\
B_{1b}(y,\rho_{1},m_{1},\sigma) &:= A^0_{1b}(y,\rho_{1},m_{1},\sigma) \hat{B}_{1}(y,\rho_{1},m_{1},\sigma).
\end{align*}
Here the function $S$ and the matrices $\hat{A}^1_{1}$ and $\hat{B}_{1}$ 
are defined after \eqref{eq311*} and \eqref{eq701*}.
In the extended supersonic region $I_{1p}$, we adopt the following symmetric system for the unknown function $f_{1p}:=(\phi_{1},\psi_{1})^{\top}$:
\begin{gather}\label{eq704}
A^0_ {1p} \partial_s f_{1p} + A^1_{1p} \partial_y f_{1p} + B_{1p} f_{1p} = 0 \quad \text{in $\mathbb R_{+} \times I_{1p}$},
\end{gather}
where 
\begin{align*}
A^0_{1p}(y,\rho_{1},u_{1},\sigma) &:= A^0_{p}(\hat{x}_{1}(\sigma,y),\rho_{1},u_{1}),\\
A^1_{1p}(y,\rho_{1}, u_{1}, \sigma, \dot{\sigma})& :=  \frac{\bar{d} - x_1}{\bar{d}+\sigma - x_1} 
A_{p}^1(\hat{x}_{1}(\sigma,y),\rho_{1}, u_{1})
- \dot{\sigma} \frac{y - x_1}{\bar{d}+\sigma - x_1} A^0_{p}(\hat{x}_{1}(\sigma,y),\rho_{1},u_{1}), \\
B_{1p}(y,\rho_{1},u_{1},\sigma) &:= B_{p}(\hat{x}_{1}(\sigma,y),\rho_{1},u_{1}).
\end{align*}
Here the matrices $A^0_{p}$, $A^1_{p}$, and $B_{p}$ are defined just after \eqref{eq512},
and $f_{1p}$ is related to $f_{1b}$ by
\begin{equation*}
f_{1p} = 
\dfrac{1}{\rho}
\begin{bmatrix}
\rho & 0 
\\ 
-\bar{u} & 1
\end{bmatrix}
f_{1b}.
\end{equation*}

Now we fix $y_{*} \in (x_{1},0)$ and $k_{*}>0$ in the same way as in Proposition \ref{Pro51}.
\begin{pro} \label{Pro71}
Under the same assumptions as in Proposition \ref{Pro70}, 
if $N_\sigma(T)$ is sufficiently small, 
there exist constants $c, k_*,C > 0$ and $y_* \in (x_1,0)$ such that the following estimates hold$:$
\begin{gather} 
|R_{1b}(s,y)| \leq CN_\sigma(T)  \quad \text{for $(s,y) \in [0,T] \times \bar{I}_{1}$},
\label{eqPro79}
\\
f^\top R^j_{1p}(s,y) f \ge c | f |^2 \quad \text{for $f \in \mathbb{R}^2, \ (s,y) \in [0,T] \times [y_*, \bar{d}], \ j=0,1,2$},
\label{eqPro71}
\end{gather}
where the $2 \times 2$ symmetric matrices $R_{1b}$ and $R^j_{1p}$ are defined by
\begin{align*}
R_{1b} &:= - \partial_s A_{1b}^0 - \partial_y A_{1b}^1 +  B_{1b} + B_{1b}^\top,
\\
R_{1p} &:= - \partial_s A^0_{1p} - \partial_y A^1_{1p} +  B_{1p} + B_{1p}^\top,
\\
R^j_{1p} &:= R_{1p} + k_*  A^1_{1p} + j \times \left( 2 \partial_y A^1_{1p} -  (\partial_y A_{1p}^0) (A_{1p}^0)^{-1} A_{1p}^1 - \{(\partial_y A_{1p}^0) (A_{1p}^0)^{-1} A_{1p}^1\}^\top  \right).
\end{align*}
\end{pro}
\begin{proof}
Recalling \eqref{approx1}, we can approximate $R_{1b}$ and $R_{1p}^{j}$ similarly to \eqref{approxM1} to arrive at \eqref{eqPro79} just following the proofs of Propositions \ref{Pro31}. 
The choice of $y_* = x_*$ ensures that the arguments of the proof of Proposition \ref{Pro51} works here,
leading to \eqref{eqPro71}.
\end{proof}

%
%

Next we will show a lemma similar to Lemma \ref{Lem57}.
For this purpose, 
we apply the operator $\partial_{s}^{j}$ for $j=0,1,2$ to \eqref{eq701}  to obtain
\begin{gather}\label{eq716*}
\partial_{s} (\partial_{s}^{j} f_{1b}) + \hat{A}_{1b}^{1} \partial_{y} (\partial_{s}^{j} f_{1b}) + \hat{B}_{1b} (\partial_{s}^{j} f_{1b})
={\cal F}_{1b}^{j}:=- [\partial_{s}^{j},\hat{A}_{1b}^{1}\partial_{y}] f_{1b} - [\partial_{s}^{j},\hat{B}_{1b}] f_{1b},
\end{gather}
where $\hat{A}_{1b}^{1}:=\hat{A}_{1}^{1}$ and $\hat{B}_{1b}:=\hat{B}_{1}$.
Multiplying this by $A_{1b}^{0}$ from the left gives
\begin{gather}
\label{eq716}
A_{1b}^{0} \partial_{s} (\partial_{s}^{j} f_{1b}) + A_{1b}^{1} \partial_{y} (\partial_{s}^{j} f_{1b}) + B_{1b} (\partial_{s}^{j} f_{1b})
=A_{1b}^{0}{\cal F}_{1b}^{j}. 
\end{gather}
We also apply the operator $\partial_{s}^{j}$ for $j=0,1,2$ to \eqref{eq704} to obtain
\begin{gather}\label{eq717}
{A}^{0}_{1p}\partial_{s} (\partial_{s}^{j} f_{1p}) + {A}^{1}_{1p} \partial_{y} (\partial_{s}^{j} f_{1p}) + {B}_{1p} (\partial_{s}^{j} f_{1p})
={\cal F}_{1p}^{j},
\\
{\cal F}^{j}_{1p}:=- [\partial_{s}^{j},{A}^{0}_{1p}\partial_{s}] f_{1p} - [\partial_{s}^{j},{A}^{1}_{1p}\partial_{y}] f_{1p} - [\partial_{s}^{j},{B}_{1p}] f_{1p}.
\notag
\end{gather}
Let us consider the $y$-derivative of $f_{1p}$.
Applying $\partial_{y}^{j}$ for $j=0,1,2$ to \eqref{eq704} gives
\begin{multline*}
{A}^{0}_{1p}\partial_{s} (\partial_{y}^{j} f_{1p}) + {A}^{1}_{1p} \partial_{y} (\partial_{y}^{j} f_{1p}) + {B}_{1p} (\partial_{y}^{j} f_{1p})
\\
=- [\partial_{y}^{j},{A}^{0}_{1p}\partial_{s}] f_{1p} - [\partial_{y}^{j},{A}^{1}_{1p}\partial_{y}] f_{1p} - [\partial_{y}^{j},{B}_{1p}] f_{1p}.
\end{multline*}
For $j=1,2$, we rewrite the term $(\partial_{y} A^{0}_{1p}) \partial_{y}^{j-1} \partial_{s} f_{1p}$ on the RHS into some linear combination of $f_{1p}$ and its $y$-derivatives of at most $(j-1)$-th order 
by using \eqref{eq704} to obtain
\begin{multline}\label{eq718}
{A}^{0}_{1p}\partial_{s} (\partial_{y}^{j} f_{1p}) + {A}^{1}_{1p} \partial_{y} (\partial_{y}^{j} f_{1p}) + {B}_{1p} (\partial_{y}^{j} f_{1p})
\\
+j\times \left(\partial_y A^1_{1p} -  (\partial_y A_{1p}^0) (A_{1p}^0)^{-1} A_{1p}^1\right)(\partial_{y}^{j} f_{1p}) = {\cal G}^{j}_{1p},
\end{multline}
where 
\begin{align*}
{\cal G}^{1}_{1p}&:= \left( \partial_y A_{1p}^0 \right) \left( A_{1p}^0 \right)^{-1} B_{1p} f_{1p} - (\partial_{y} B_{1p}) f_{1p},
\\
{\cal G}^{2}_{2p}&:= 2 \left( \partial_y A_{1p}^0 \right) \left( A_{1p}^0 \right)^{-1} 
\\
&\quad \times  \left[ B_{1p} \partial_{y} f_{1p} - (\partial_y A_{1p}^0) \left( A_{1p}^0 \right)^{-1} \left( A_{1p}^1  \partial_{y} f_{1p} + B_{1p} f_{1p} \right) + (\partial_y A_{1p}^1)  \partial_{y} f_{1p} + (\partial_y B_{1p}) f_{1p} \right]
\\
& \quad - 2 (\partial_{y} B_{p1})  \partial_{y} f_{1p} - (\partial _{yy} A_{1p}^{0})  \partial_{s} f_{1p} - (\partial _{yy} A_{1p}^{1})  \partial_{y}f_{1p} - (\partial_{yy} B_{1p}) f_{1p}.
\end{align*}
The equality \eqref{eq718} holds also for $j=0$ by regarding ${\cal G}^{0}_{1p}:=0$.
We note that \eqref{eq716*}, \eqref{eq716}, \eqref{eq717}, and \eqref{eq718} correspond to \eqref{eq361*}, \eqref{eq361}, \eqref{eq554}, and \eqref{eq561}, respectively.	

The following estimates hold for $(s,y) \in [0,T] \times I_{1}$ and $j=0,1,2$:
\begin{gather}
\left| \partial_s^j f_{1p}
-
\dfrac{1} {\rho}
\begin{bmatrix}
\rho & 0 
\\ 
-\bar{u} & 1
\end{bmatrix}
\partial_s^j f_{1b}
\right|
(s,y)
\leq 
CN_\sigma(T) \sum_{i=0}^1 |\partial_s^i f_{1b}|(s,y),
\label{eq728*}
\\
c \sum_{i=0}^j |\partial_y^i f_{1b}|(s,y)
\le \sum_{i=0}^j |\partial_s^i f_{1b}|(s,y)
\le C \sum_{i=0}^j |\partial_y^i f_{1b}|(s,y), 
\label{ajs72}
\\
|{\cal F}_{1b}^{j}|(s,y)  \leq CN_\sigma(T) \sum_{i=0}^j |\partial_s^i f_{1b}|(s,y)+CN_\sigma(T)|\dddot{\sigma}|(s),
\label{calF70}\\
|{\cal F}^{j}_{1p}|(s,y)  \leq CN_\sigma(T) \sum_{i=0}^j (|\partial_s^i f_{1p}|+ |\partial_y^i f_{1p}|)(s,y)
+CN_\sigma(T)|\dddot{\sigma}|(s),
\label{calF71} \\
|{\cal G}^{j}_{1p}|(s,y)  \leq C j \sum_{i=0}^{\max\{0,j-1\}} |\partial_y^i f_{1p}|(s,y),
\label{calG571}
\end{gather}
where $c$ and $C$ are positive constants independent of $(s,y)$ and $T$.
These derivations are parallel to those of \eqref{eq528*}, \eqref{ajs2}, \eqref{calF1}, \eqref{calF51}, and \eqref{calG51}, respectively.
The key difference, however, lies only in the presence or absence of $|\dddot{\sigma}|$ appearing in \eqref{calF70} and   \eqref{calF71}.
%
%

We employ the weight function  
\begin{gather} \label{eq715}
W_{1}(y) := e^{-k_* (y-y_*)}
\end{gather}
using $k_* > 0$ and $y_* \in (x_1, 0)$ determined in Proposition \ref{Pro71},
and define the following functionals for $j=0,1,2$:
\begin{subequations}\label{EF7}
\begin{gather}
E_{1bp}(s) := (E^{0}_{1bp}+E^{1}_{1bp}+E^{2}_{1bp}+E^{\rm sp}_{1p,\mu})(s), \quad E_{1bp}^{j}(s):=(E^{j}_{1b}+E^{j}_{1p})(s),
\\
E^{j}_{1b}(s) := \int_{x_1}^{y_*}  ( \partial_s^j f_{1b} )^\top A^0_{1b} \partial_s^j f_{1b} (s, y) dy, 
\\
E^{j}_{1p}(s) := \int_{y_*}^{\bar{d}} W_{1}  ( \partial_s^j f_{1p} )^\top A^0_{1p} \partial_s^j f_{1p} (s, y) dy, 
\\
E^{\rm sp}_{1p,\mu} (s) := 
\sum_{j=0}^2 \mu^{j+1} \int_{x_*}^{\bar{d}}  W_{1} ( \partial_y^j f_{1p} )^\top A^0_{1p} \partial_y^j f_{1p} (s, y) dy,
\\
D_{1bp}(s) := (D^{0}_{1bp}+D^{1}_{1bp}+D^{2}_{1bp})(s), \quad D_{1bp}^{j}(s):=(D^{j}_{1b}+D^{j}_{1p})(s),
\\
D_{1b}^{j}(s):=  | \partial_s^j f_{1b} |^2 (s, y_*), \quad D_{1p}^{j}(s):=  | \partial_s^j f_{1p} |^2 (s, \bar{d}), 
\end{gather}
\end{subequations}
which are analogous with $E_{bp}$, $E^{j}_{bp}$, $E^{j}_{b}$, $E^{j}_{p}$, $E^{\rm sp}_{p,\mu}$, $D_{bp}$, $D_{bp}^{j}$, $D_{b}^{j}$, and $D_{p}^{j}$ in \eqref{EF5} and \eqref{EF5*}.
Here $\mu \in (0,1)$ is a constant to be determined in Lemma \ref{Lem72} below, which is parallel to Lemma \ref{Lem57}.

\begin{lem} \label{Lem72}
Under the same assumptions as in Proposition \ref{Pro70}, 
if $N_\sigma(T)$ is sufficiently small, there exists a constant $\mu \in (0,1)$ 
such that the following estimate holds for $s_{0} \in [0,T]$ and $s \in [s_{0},T]${\rm :}
\begin{gather} \label{eqLem72} 
E_{1bp}(s)
+ c \int_{s_{0}}^s  D_{1bp} (\tau)  d\tau 
\le
E_{1bp}(s_{0}) 
+ C N_\sigma(T) \int_{s_{0}}^s (E_{1bp}+|\dddot{\sigma}|^{2})(\tau) d\tau, 
\end{gather}
where $c$ and $C$ are positive constants independent of $s_{0}$ and $T$.
\end{lem}
\begin{proof} 
Let us follow the proof of Lemma \ref{Lem57}, replacing $I_{b}$, $I_{p}$, and $[t_{0},t]$ by $I_{1b}$, $I_{1p}$, and $[s_{0},s]$, respectively.
Multiplying \eqref{eq716} by $2 \partial_{s}^{j} f_{1b}^\top$ from the left, 
we have the following equalities for $j=0,1,2$:
\begin{gather*} 
\partial_s \left( (\partial_s^j f_{1b})^\top A_{1b}^0 \partial_s^j f_{1b} \right) 
+ \partial_y \left( (\partial_s^j f_{1b})^\top A_{1b}^1 \partial_s^j f_{1b} \right) 
+ (\partial_s^j f_{1b})^\top R_{1b} \partial_s^j f_{1b}
= 2(\partial_s^j f_{1b})^\top A_{1b}^{0}{\cal F}_{1b}^{j},
\end{gather*}
which is parallel to \eqref{eq325} in Lemma \ref{Lem36}.
Then repeating the derivation of \eqref{eq553} with the aid of \eqref{eqPro79} and \eqref{calF70}, one can obtain
\begin{multline} \label{eq723}
E_{1b}^{j}(s) + \int_{s_{0}}^s (\partial_s^j f_{1b})^\top A^1_{1b} \partial_s^j f_{1b} (\tau, y_*) d \tau 
\leq E_{1b}^{j}(s_{0}) + C N_\sigma(T) \int_{s_{0}}^s \left( \sum_{i =0}^{2 }E_{1b}^{i} + |\dddot{\sigma}|^{2} \right) (\tau) d \tau,
\end{multline}
where $C>0$ is a constant independent of $s_{0}$ and $T$.
Next we multiply \eqref{eq717} by $2W_{1}( \partial_s^j f_{1p} )^\top$ from the left,
and use the definition of $R^{0}_{1p}$ in Proposition \ref{Pro71}
to get the following equalities for $j=0,1,2$:
\begin{multline*}
\partial_s \left( W_{1} ( \partial_s^j f_{1p} )^\top A_{1p}^{0} \partial_s^j  f_{1p} \right) 
+ \partial_y \left( W_{1} ( \partial_s^j f_{1p} )^\top A_{1p}^{1} \partial_s^j  f_{1p}  \right) 
\\
+ W_{1} ( \partial_s^j f_{1p} )^\top R^{0}_{1p}  \partial_s^j  f_{1p} 
= 2W_{1}(\partial_s^j f_{1p} )^\top {\cal F}^{j}_{1p},
\end{multline*}
which are parallel to \eqref{eq523*}. 
In the same way as the derivation of \eqref{eq556} with using \eqref{eqPro71} and \eqref{calF71}, one can have
\begin{multline}\label{eq724}
E_{1p}^j(s) 
- \int_{s_{0}}^s ( \partial_s^j f_{1p} )^\top A_{1p}^{1} \partial_s^j  f_{1p} (\tau, y_*) d \tau
+ c \int_{s_{0}}^s D_{1p}^{j} (\tau) d \tau
\\ 
\leq E_{1p}^j(s_{0})+ C {\mu}^{-3} N_\sigma(T) \int_{s_{0}}^s \Biggl(E^{\rm sp}_{1p,\mu}+\sum_{i=0}^{2}E_{1p}^{i}+|\dddot{\sigma}|^{2}\Biggr)(\tau) d\tau,
\end{multline}
where $c,C>0$ are constants independent of $\mu$, $s_{0}$, and $T$.
Now we repeat the derivation of \eqref{eqLem57*} but with \eqref{eq728*}, \eqref{eq723}, and \eqref{eq724} in place of \eqref{eq528*}, \eqref{eq553}, and \eqref{eq556}, respectively, to conclude that 
\begin{multline} \label{eqLem77*} 
E_{1bp}^j(s)
+ c \int_{s_{0}}^s  D_{1bp}^{j} (\tau) d\tau 
\\
\le
E_{1bp}^j(s_{0}) 
+ C N_\sigma(T) \int_{s_{0}}^s D_{1bp}(\tau) d\tau
+C {\mu}^{-3} N_\sigma(T) \int_{s_{0}}^s \Biggl(E^{\rm sp}_{1p,\mu}+\sum_{i=0}^{2}E_{1bp}^{i}+|\dddot{\sigma}|^{2}\Biggr) (\tau) d\tau.
\end{multline}

Furthermore,  for $j=0,1,2$, we multiply \eqref{eq718} by $2 \mu^{j+1}W(\partial_{y}^{j} f_{1p})^{\top}$ from the left, 
and use the definition of $R^{j}_{1p}$ in Proposition \ref{Pro71} to have
\begin{multline}\label{eq723**}
\partial_s \left( \mu^{j+1}W_{1} (\partial_{y}^{j} f_{1p})^{\top} A^0_{1p} \partial_{y}^{j} f_{1p} \right) 
+ \partial_y \left( \mu^{j+1}W_{1} (\partial_{y}^{j} f_{1p})^{\top} A^1_{1p} \partial_{y}^{j} f_{1p} \right) 
\\
+ \mu^{j+1} W_{1} (\partial_{y}^{j} f_{1p})^{\top} R^{j}_{1p} \partial_{y}^{j} f_{1p} 
= 2 \mu^{j+1} W_{1} (\partial_{y}^{j} f_{1p})^{\top}{\cal G}^{j}_{1p},
\end{multline}
which is parallel to \eqref{eq523**}. 
Then following the derivation of \eqref{eqLem51} with the aid of  \eqref{eqPro71} and \eqref{calG571}, we arrive at
\begin{gather} \label{eqLem71} 
E^{\rm sp}_{1p,\mu} (s) \leq
E^{\rm sp}_{1p,\mu} (s_{0}) 
+ C \mu \int_{s_0}^s D_{1bp}(\tau)  d \tau,
\end{gather}
where $C$ is a positive constant independent of $\mu$, $s_{0}$, and $T$.

Finally, summing up \eqref{eqLem77*} for $j=0,1,2$ and \eqref{eqLem71}, 
choosing $\mu$ sufficiently small, and letting $N_\sigma(T) \ll 1$, we conclude \eqref{eqLem72}.
\end{proof} 

Hereinafter we fix $\mu$ so that \eqref{eqLem72} in Lemma \ref{Lem72} holds.
Next let us derive the inequality \eqref{eqLem78} below, which is similar to \eqref{eqLem58} in Lemma \ref{Lem59}.
We set
\begin{gather*}
v_{1-}^{*} := \frac{1}{2} \sup_{y \in \bar{I}_{1b} } \left( \bar{u}(y) - \bar{a}(y) \right) < 0, \quad
v_{1+}^{*} := \frac{1}{2} \inf_{y \in \bar{I}_{1b} } \left( \bar{u}(y) + \bar{a}(y) \right) > 0,
\\
T_{1-}^* : = \frac{y_* - x_1}{| v_{1-}^* |}, \quad 
T_{1+}^{*} : = \frac{y_* - x_1}{| v_{1+}^* |}, \quad 
T_{10}^{*} : = T_{1-}^* +  T_{1+}^*,
\end{gather*}
%
which are analogous to \eqref{eq525} and \eqref{eq526}. 

\begin{lem} \label{Lem79}
Under the same assumptions as in Proposition \ref{Pro70} with $T \geq T_{10}^{*}$, 
if $N_\sigma(T)$ is sufficiently small, 
the following estimate holds for $\tilde{s}_{0} \in [0,T-T_{10}^{*}]$ and $s \in [\tilde{s}_{0}+T_{10}^{*},T]${\rm :}
\begin{gather} \label{eqLem78}
( 1 + c ) E_{1bp} (s) + c\int_{\tilde{s}_{0}}^s  D_{1bp}(\tau) d\tau
\le E_{1bp} (\tilde{s}_{0}) + C N_\sigma(T) \int_{\tilde{s}_{0}}^s (E_{1bp}+|\dddot{\sigma}|^{2}) (\tau) d\tau,
\end{gather}
where $c$ and $C$ are positive constants independent of $\tilde{s}_{0}$ and $T$.
\end{lem}
\begin{proof}
This lemma can be shown in the same way as in the proof of Lemma \ref{Lem59}.
Indeed, we use the equations \eqref{eq716*}, \eqref{eq716}, \eqref{eq717}, and \eqref{eq718} in place of the equations \eqref{eq361*}, \eqref{eq361}, \eqref{eq554}, and \eqref{eq561}, respectively, and replace $I_{b}$, $I_{p}$, and $[\tilde{t}_{0},t]$ by $I_{1b}$, $I_{1p}$, and $[\tilde{s}_{0},s]$, respectively.
Then repeating the derivations of \eqref{eqLem59}--\eqref{eqLem50*}, one can conclude that for $j=0,1,2$,
\begin{gather} 
E_{1b}^j(s) 
\le 
C \int_{\tilde{s}_{0}}^s D_{1bp}(\tau) d \tau 
+ C N_\sigma(T) \int_{\tilde{s}_{0}}^s (E_{1bp}+|\dddot{\sigma}|^{2}) (\tau) d\tau,
\label{eqLem79}
\\
( 1 + c ) E_{1p}^j(s) 
\le 
E_{1p}^j(\tilde{s}_{0}) + C \int_{\tilde{s}_{0}}^s  D_{1bp}(\tau) d\tau+ C N_\sigma(T) \int_{\tilde{s}_{0}}^s (E_{1bp}+|\dddot{\sigma}|^{2})  (\tau) d\tau,
\label{eqLem70}
\\  
(1+c) E^{\rm sp}_{1p,\mu}(s) 
\le 
E^{\rm sp}_{1p,\mu} (\tilde{s}_{0}) 
+ C \int_{\tilde{s}_{0}}^s   D_{1bp}(\tau)  d\tau, 
\label{eqLem70*}
\end{gather}
where $c$ and $C$ are positive constants independent of $\tilde{s}_{0}$ and $T$.
Finally, we sum up twice of \eqref{eqLem79} for $j=0,1,2$, \eqref{eqLem70} for $j=0,1,2$, and \eqref{eqLem70*},
then add the product of the resultant inequality and $\ve>0$ to \eqref{eqLem72} with $s_0=\tilde{s}_{0}$, 
and finally let $\ve \ll 1$ to obtain \eqref{eqLem78}.
\end{proof}

At this stage, we cannot close the a priori estimate due to the presence of $\dddot{\sigma}$ on RHSs of \eqref{eqLem72} and \eqref{eqLem78}.
However, through the derivation of the a priori estimate in $I_{2}$, 
we will be able to have a dissipative term on $\dddot{\sigma}$.

\subsection{A priori estimate in $I_{2}$}\label{S7.2}
In this subsection, we show the a priori estimate in $I_{2}$ applying an energy method similar to one used in Section \ref{Sec_sub}. A primary distinction from Section \ref{Sec_sub} concerns with the handling of the boundary values on $x=x_{1}$ and $y=\bar{d}$. While the boundary value at $x=x_{1}$ vanishes in Section \ref{Sec_sub} due to \eqref{eq314a} in the course of the energy method, in this context, we are required to handle the nonzero boundary value at $y=\bar{d}$ 
subject to the Rankine-Hugoniot conditions \eqref{eq706a} and \eqref{eq706c}. 
In particular, we need to find a dissipative term on $\dddot{\sigma}$. Our focus will be on the treatment of the boundary value at $y=\bar{d}$.

Similarly to Section \ref{Sec_sub}, we use the following symmetrization of \eqref{eq701} with $i=2$:
\begin{gather}\label{eq703}
A^0_ {2} \partial_s f_{2} + A^1_{2} \partial_y f_{2} + B_{2} f_{2} = 0 \quad \text{in $\mathbb R_{+} \times I_{2}$},
\end{gather}
where 
\begin{align*}
A^0_{2}=A^0_{2}(y,\rho_2,m_2,\sigma) &:= 
\dfrac{A(\hat{x}_{2}(\sigma,y))}{\bar{\rho}_{2}(\hat{x}_{2}(\sigma,y))} \, {\rm diag} (S(\rho_2,m_2),1), \\
A^1_{2}=A^1_{2}(y,\rho_2, m_2, \sigma, \dot{\sigma})& := A^0_{2}(y,\rho_2,m_2,\sigma) \hat{A}^1_{2}(y,\rho_2, m_2, \sigma, \dot{\sigma}), \\
B_{2}=B_{2}(y,\rho_2,m_2,\sigma) &:= A^0_{2}(y,\rho_2,m_2,\sigma) \hat{B}_{2}(y,\rho_2,m_2,\sigma).
\end{align*}
Here the function $S$ and the matrices $\hat{A}^1_{2}$ and $\hat{B}_{2}$ are defined after \eqref{eq311*} and \eqref{eq701*}.

Now we show the following proposition which corresponds to Proposition \ref{Pro30}.
\begin{pro} \label{Pro72}
Under the same assumptions as in Proposition \ref{Pro70}, 
if $N_\sigma(T)$ is sufficiently small, 
there exists a constant $C > 0$ such that the following estimate holds$:$
\begin{gather} 
|R_{2}(s,y)| \leq CN_\sigma(T)  \quad \text{for $(s,y) \in [0,T] \times \bar{I}_{2}$},
\label{eqPro72}
\end{gather}
where $R_{2} := - \partial_s A_{2}^0 - \partial_y A_{2}^1 +  B_{2} + B_{2}^\top$.
\end{pro}
\begin{proof}
Recalling \eqref{approx1}, we can approximate $R_{2}$  similarly to \eqref{approxM1}. Then we arrive at \eqref{eqPro72} just by following the proof of Proposition \ref{Pro30}.
\end{proof}

Next we will show a lemma similar to Lemmas \ref{Lem32} and \ref{Lem36}.
For this purpose, 
we apply the operator $\partial_{s}^{j}$ for $j=0,1,2$ to \eqref{eq701} with $i=2$  to obtain
\begin{gather}\label{eq761*}
\partial_{s} (\partial_{s}^{j} f_{2}) + \hat{A}^{1}_{2} \partial_{y} (\partial_{s}^{j} f_{2}) + \hat{B}_{2} (\partial_{s}^{j} f_{2})
={\cal F}^{j}_{2}:=- [\partial_{s}^{j},\hat{A}^{1}_{2}\partial_{y}] f_{2} - [\partial_{s}^{j},\hat{B}_{2}] f_{2}.
\end{gather}
Multiplying this by $A^{0}_{2}$ from the left gives
\begin{gather}
\label{eq761}
A^{0}_{2} \partial_{s} (\partial_{s}^{j} f_{2}) + A^{1}_{2} \partial_{y} (\partial_{s}^{j} f_{2}) + B_{2} (\partial_{s}^{j} f_{2})
=A^{0}_{2}{\cal F}^{j}_{2}.
\end{gather}
Furthermore, applying the operator ${d^{j}}/{ds^{j}}$ for $j=0,1,2$ to the conditions 
\eqref{eq706a}, \eqref{eq706c} and \eqref{eq705a}, we have
\begin{subequations} \label{eq732}
\begin{gather} 
\left| \partial_{s}^{j}\omega_2(s,\bar{d}) 
 -\jump{ \bar{\rho} } ( \bar{d} ) \dfrac{d^{j+1} \sigma}{ds^{j+1}} (s) \right|
  \leq C \sqrt{D_{1bp}}(s)+CN_\sigma(T)(|\sigma|+\sqrt{D_{2}})(s),
\\
\begin{aligned}
 \biggl| \partial_{s}^{j} \phi_{2}(s,\bar{d}) - 
      \frac{Q A_x}{A^2\bar{S}_2} \jump{\bar{u}} (\bar{d}) \dfrac{d^{j} \sigma}{ds^{j}} (s)
    &+ 2 \frac{ \bar{u}_2 }{\bar{S}_2}(\bar{d})  \partial_{s}^{j}  \omega_{2}(s,\bar{d}) \biggr|
\\
  & \leq C \sqrt{D_{1bp}}(s)+CN_\sigma(T)(|\sigma|+\sqrt{D_{2}})(s),
\end{aligned}
\\
\partial_{s}^{j} \phi_{2}(s,x_{2})=0,
 \label{eq732*}
\end{gather}
\end{subequations}
where $\bar{S}_2:=p'(\bar{\rho}_{2}) - \bar{u}_{2}^2$, and $D_{1bp}$ is defined in \eqref{EF7}.
We note that \eqref{eq761*}, \eqref{eq761}, and \eqref{eq732*} correspond to \eqref{eq361*}, \eqref{eq361}, and \eqref{eq367}, respectively.

The following hold for $(s,y) \in [0,T] \times I_{2}$ and $j=0,1,2$:
\begin{gather}
c \sum_{i=0}^j |\partial_y^i f_{2}|(s,y)
\le \sum_{i=0}^j |\partial_s^i f_{2}|(s,y)
\le C \sum_{i=0}^j |\partial_y^i f_{2}|(s,y), 
\label{ajs72*}
\\
|{\cal F}_{2}^{j}|(s,y)  \leq CN_\sigma(T) \sum_{i=0}^j |\partial_s^i f_{2}|(s,y)+CN_\sigma(T)\max\{0,j-1\}|\dddot{\sigma}|(s),
\label{calF70*}
\end{gather}
where $c$ and $C$ are positive constants independent of $(s,y)$ and $T$.
These derivations are parallel to those of \eqref{ajs2} and \eqref{calF1}, respectively.

Now we define the following functionals for $j=0,1,2$:
\begin{subequations}\label{EF7*}
\begin{gather}
E_{2}(s):=(E^{0}_{2}+E^{1}_{2}+E^{2}_{2})(s), \quad D_{2}(s):=(D^{0}_{2}+D^{1}_{2}+D^{2}_{2})(s),
\\
E^{j}_{2}(s) := \int_{\bar{d}}^{x_{2}}  ( \partial_s^j f_{2} )^\top A^0_{2} \partial_s^j f_{2} (s, y) dy 
+ \left( A_x \bar{\rho}_1 \jump{\bar{u}}^2\right) (\bar{d}) \left( \dfrac{d^{j} \sigma}{ds^j} \right)^2 (s),
\\
D_{2}^{j}(s):=   |\partial_s^j f_{2}|^{2} (s, x_{2}) + \left( \dfrac{d^{j+1} \sigma}{ds^{j+1}} \right)^2 (s),
\end{gather}
\end{subequations}
where the condition $A_x(\bar{d})>0$ is assumed in Theorem \ref{Thm70}.
If the terms involving $\sigma$ are ignored in \eqref{EF7*}, 
these functionals are parallel to $E$, $D$, $E^{j}$, and $D^{j}$ in \eqref{eq321}.

Now we show the following lemma, which implies the nonincrease of $E_2$ 
up to the dissiaptive term $D_1$. 
\begin{lem} \label{Lem73}
Under the same assumptions as in Proposition \ref{Pro70}, if $N_\sigma(T)$ is sufficiently small, 
the following estimate holds for $s_{0} \in [0,T]$, $s \in [s_{0},T]$, and $\alpha \in (0,1)${\rm :}
\begin{multline} \label{eqLem73} 
E_{2}(s)
+ c \int_{s_{0}}^s  D_{2} (\tau)  d\tau 
\\
\le
E_{2}(s_{0}) 
+ C \alpha^{-1} \int_{s_{0}}^s D_{1bp}(\tau) d\tau
+ C(N_\sigma(T) +\alpha) \int_{s_{0}}^s (E_{2}+D_{2})(\tau) d\tau,
\end{multline}
where $c$ and $C$ are positive constants independent of $s_{0}$, $T$, and $\alpha$.
\end{lem}
\begin{proof}
We multiply \eqref{eq761} by $2(\partial_{s}^{j} f_{2})^{\top} $ from the left hand and integrate it over $\disp [s_{0},s] \times \bar{I}_{2}$ to have
\begin{align} \label{eq731}
&\int_{\bar{d}}^{x_{2}}  ( \partial_s^j f_{2} )^\top A^0_{2} \partial_s^j f_{2} (s, y) dy
+ \int_{s_{0}}^s (\partial_s^j f_{2})^\top (-A^1_2) \partial_s^j f_{2} (\tau, \bar{d}) d \tau 
 +  \int_{s_{0}}^s (\partial_s^j f_{2})^\top A^1_2 \partial_s^j f_{2} (\tau, x_2) d \tau 
\notag \\
&=
\int_{\bar{d}}^{x_{2}}  ( \partial_s^j f_{2} )^\top A^0_{2} \partial_s^j f_{2} (s_{0}, y) dy
+ \int_{s_{0}}^s \!\! \int_{\bar{d}}^{x_{2}}  (\partial_{s}^{j} f_{2})^\top (2A_{2}^{0}{\cal F}_{2}^{j}-R_{2}) \partial_{s}^{j} f_{2} (\tau,y) dyd \tau
\notag \\
&\leq \int_{\bar{d}}^{x_{2}}  ( \partial_s^j f_{2} )^\top A^0_{2} \partial_s^j f_{2} (s_{0}, y) dy
+C N_\sigma(T) \int_{s_{0}}^s (E_{2}+|\dddot{\sigma}|^{2})(\tau) d\tau,
\end{align}
where we have used \eqref{eqPro72} and \eqref{calF70*} in deriving the last inequality.
Using \eqref{eq732}, \eqref{spro71}, \eqref{spro72}, and the Schwarz inequality, we estimate
the integrand of the second term in \eqref{eq731} from below as
\begin{align*} 
(\partial_s^j f_{2})^\top  (-A^1_2) \partial_s^j f_{2} (\tau, \bar{d}) 
&= - \frac{2 A}{\bar{\rho}_2} (\bar{d}+\sigma) 
   \left( S(\rho_{2},m_{2}) \partial_{s}^{j}\phi_{2} + u_{2}  \partial_{s}^{j} \omega_{2} \right)  \partial_{s}^{j} \omega_{2}(\tau,\bar{d}) 
\\
&\ge \frac{d}{d s} \left[ \left( A_x \bar{\rho}_1 \jump{\bar{u}}^2 \right) (\bar{d}) 
 \left( \dfrac{d^{j} \sigma}{ds^{j}}\right)^2 \right]
 + c \left( \dfrac{d^{j+1} \sigma}{ds^{j+1}} \right)^2 
 \\
& \quad 
 - \alpha^{-1} C D_{1bp}(s) - C(N_\sigma(T)+\alpha) (|\sigma|^{2}+D_{2})(s),
\end{align*}
where $\alpha \in (0,1)$.
Substituting these into \eqref{eq731} and summing up the resultant inequalities for $j=0,1,2$, 
we arrive at \eqref{eqLem73}.
\end{proof}

Next we show Lemma \ref{Lem74} below, which corresponds to Lemmas \ref{Lem34} and \ref{Lem38}.
To this end, we use notation 
\begin{gather*}
v_{2-} := \frac{1}{2} \sup_{y \in \bar{I}_{2} } \left( \bar{u}(y) - \bar{a}(y) \right) < 0, \quad
v_{2+} := \frac{1}{2} \inf_{y \in \bar{I}_{2} } \left( \bar{u}(y) + \bar{a}(y) \right) > 0,
\\
T_{2-} : = \frac{x_{2} - \bar{d}}{| v_{2-} |}, \quad 
T_{2+} : = \frac{x_{2} - \bar{d}}{| v_{2+} |}, \quad 
T_{20} : = T_{2-} +  T_{2+},
\end{gather*}
which are analogous to \eqref{eq322} and \eqref{eq323}. 
We also define the trapezoidal set $\Omega_{2}(\tilde{s}_{0})$ for $\tilde{s}_{0} \geq 0$ as
\begin{gather*}
\Omega_{2}(\tilde{s}_{0}) := \{ (s, y) \, | \, 
  \bar{d} \leq  y \leq x_{2},  \  x_{2} + v_{2-} (s-\tilde{s}_{0}) \leq y,  \ \bar{d} + v_{2+} (s  -\tilde{s}_{0}- T_{2-}-T_{20} ) \le y \},
\end{gather*}
which will be used for the treatment of $\sigma$.
\begin{lem} \label{Lem74}
Under the same assumptions as in Proposition \ref{Pro70}  with $T \geq 2T_{20}$,
if $N_\sigma(T)$ is sufficiently small, 
the following estimate holds for $\tilde{s}_{0} \in [0,T-2T_{20}]${\rm :}
\begin{gather}\label{eqLem74**}
E_{2}(\tilde{s}_{0}+T_{2-})
\le
C \int_{\tilde{s}_{0}}^{\tilde{s}_{0}+2T_{20}}  (D_{1bp}+D_{2}) (\tau)  d\tau
+ C N_\sigma(T) \int_{\tilde{s}_{0}}^{\tilde{s}_{0}+2T_{20}} E_{2}(\tau) d\tau.
\end{gather}
where $c$ and $C$ are positive constants independent of $\tilde{s}_{0}$ and $T$.
\end{lem}
\begin{proof}
It suffices to show that
\begin{gather} \label{eqLem74} 
\int_{\bar{d}}^{x_{2}}  ( \partial_s^j f_{2} )^\top A^0_{2} \partial_s^j f_{2} (\tilde{s}_{0}+T_{2-}, y) dy
\leq 
C \int_{\tilde{s}_{0}}^{\tilde{s}_{0}+T_{20}}  D_{2} (\tau)  d\tau 
+ C N_\sigma(T) \int_{\tilde{s}_{0}}^{\tilde{s}_{0}+T_{20}} E_{2}(\tau) d\tau,
\\
A_x \bar{\rho}_1 \jump{\bar{u}}^2(\bar{d})  \left( \dfrac{d^{j} \sigma}{ds^{j}} \right)^2 (\tilde{s}_{0}+T_{2-}) 
\leq 
C \int_{\tilde{s}_{0}}^{\tilde{s}_{0}+2T_{20}}  (D_{1bp}+D_{2}) (\tau)  d\tau,
 \label{eqLem74*} 
 \end{gather}
Indeed, summing up \eqref{eqLem74} and \eqref{eqLem74*} for $j=0,1,2$ gives \eqref{eqLem74**}. 

Firstly we show \eqref{eqLem74} following the proofs of Lemmas \ref{Lem37} and \ref{Lem38}.
Set $M_{2}^{1}:={\rm diag} (S(\rho_{2},m_{2}),1)$, $M_2^{0} :=  M^{1}_{2} (\hat{A}_{2}^1)^{-1}$,
and $M_{2}^{2}:= \partial_s M_{2}^{0} + \partial_y M_{2}^{1}  - 2 M_{2}^{0} \hat{B}_{2}$.
We multiply \eqref{eq761*} by $2 \partial_{s}^{j} f_{2}^\top M_{2}^{0}$ from the left to have
\begin{gather}  \label{eq733}
\partial_s \left( \partial_{s}^{j} f_{2}^\top M_{2}^{0} \partial_{s}^{j} f_{2} \right)
+\partial_y \left( \partial_{s}^{j} f_{2}^\top M_{2}^{1}  \partial_{s}^{j} f_{2}  \right)
=
\partial_{s}^{j} f_{2}^\top M_{2}^{2} \partial_{s}^{j}f_{2} + 2 \partial_{s}^{j} f_{2} ^\top M_{2}^{0} {\cal F}_{2}^{j},
\end{gather}
which is parallel to \eqref{eq373}.
Furthermore, we denote the eigenvalues and left eigenvectors of the matrix $\hat{A}_{2}^1$ by $\lambda_{2\pm}$ and $e_{2\pm}$, respectively, where $\lambda_{2-}<0<\lambda_{2+}$ holds.
Set $r_{2\pm}^{j} := e_{2\pm} \, \partial_s^j f_{2}$.
Multiplying \eqref{eq761*} by $e_{2\pm}$ from the left to obtain $\partial_s r_{2\pm}^{j} + \lambda_{2\pm} \partial_y r_{2\pm}^{j}
= {\cal H}_{2\pm} \partial_{s}^{j} f_{2} + e_{2\pm} {\cal F}^{j}_{2}$, 
where ${\cal H}_{2\pm} f :=-  e_{2\pm} \hat{B}_{2} f + ( \partial_s e_{2\pm} ) f + \lambda_{2\pm} ( \partial_y e_{2\pm} ) f.$
Multiplying this by $2 r_{2\pm}^{j}$ gives
\begin{gather*}
   \partial_s \left( (r_{2\pm}^{j})^2 \right) 
+ \partial_y \left( \lambda_{2\pm} (r_{2\pm}^{j})^2 \right)
= ( \partial_y \lambda_{2\pm} ) (r_{2\pm}^{j})^2
+ 2 r_{2\pm}^{j} {\cal H}_{2\pm} \partial_{s}^{j} f_{2} +  2 r_{2\pm}^{j} e_{2\pm} {\cal F}^{j}_{2},
\end{gather*}
which is parallel to \eqref{eq391}.
Recalling \eqref{approx1} and \eqref{approxM1}, 
we see that $M_{2}^{0}$, $M_{2}^{1}$, and $M_{2}^{1}$ can be approximated similarly to \eqref{approxM1}, and that $|\lambda_{2\pm} - ( u_{2} \pm \sqrt{p'(\rho_{2})} ) | \leq CN_\sigma(T)$ holds.
Now repeating the proofs of Lemmas \ref{Lem37} and \ref{Lem38} with the aid of \eqref{calF70*}, one can conclude \eqref{eqLem74}.

Secondly we show \eqref{eqLem74*}. 
The case $j=1,2$ immediately follows from standard trace theory
and the definition of $D_2$ in \eqref{EF7*}. 
It suffices to show the case $j=0$.
We integrate the product of $e^{ky}$ and \eqref{eq733} with $j=0$ over $\Omega(\tilde{s}_{0})$, where $k>0$ is a constant to be determined later, and then apply Green's theorem to obtain
\begin{align*} 
& \int_{\tilde{s}_{0}+T_{2-}}^{s_{0}+T_{2-}+T_{20}} e^{k \bar{d}}  f_{2}^\top M_{2}^{1} f_{2} (\tau,\bar{d}) d\tau 
\notag \\
& \quad + \int_{\tilde{s}_{0}}^{\tilde{s}_{0}+T_{2-}} e^{k (x_{2}+v_{2-} (\tau-\tilde{s}_{0}))} f_2^\top ( M_{2}^{1} - v_{2-} M_{2}^{0} ) f_{2} (\tau, x_{2}+v_{2-} (\tau-\tilde{s}_{0})) d\tau
\notag \\
&  \quad + \int_{\tilde{s}_{0}+T_{2-}+T_{20}}^{\tilde{s}_{0}+2T_{20}} e^{k (\bar{d} + v_{2+} (\tau  -\tilde{s}_{0}- T_{2-}-T_{20} ))} 
     f_{2}^\top ( M_{2}^{1} - v_{2+} M_{2}^{0}  ) f_{2} (\tau, \bar{d} + v_{2+} (\tau  -\tilde{s}_{0}- T_{2-}-T_{20} ) ) d\tau
\notag \\
& \quad + k  \iint_{\Omega(\tilde{s}_{0})} e^{ky} f_{2}^\top M_{2}^{1} f_{2}  (\tau, y) d\tau dy
+ \iint_{\Omega(\tilde{s}_{0})} e^{ky} f_{2}^\top M_{2}^{2} f_{2} (\tau,y) d\tau dy
\notag \\
&= \int_{\tilde{s}_{0}}^{\tilde{s}_{0}+2T_{20}} e^{k x_2}  f_{2}^\top M_{2}^{1} f_{2}  (\tau, x_2) d\tau.
\end{align*}
It is seen that $ M_{2}^{1} - v M_{2}^{0} $ is positive definite for $v \in [v_{2-}, v_{2+} ]$ if $N_\sigma(T) \ll 1$ by the same method as in the last paragraph of the proof of Lemma \ref{Lem33}. So the second and third terms on the LHS are nonnegative.
Using the bound of $M^{2}_{2}$ and the positive definiteness of $M^{1}_{2}$, and taking $k$ sufficiently large, we see that the sum of the last two terms on the LHS is nonnegative. 
Consequently, using the positive definiteness of $M^{1}_{2}$ for the first term on the LHS, we obtain
\begin{gather*} 
\int_{\tilde{s}_{0}+T_{2-}}^{s_{0}+T_{2-}+T_{20}} | f_{2} |^2 (\tau, \bar{d}) d\tau
\le
C \int_{\tilde{s}_{0}}^{\tilde{s}_{0}+2T_{20}} | f_{2} | (\tau, x_2) d\tau.
\end{gather*}
From this, \eqref{eq732}, and standard trace theory, we observe that
\begin{gather*} 
\sigma^2 (\tilde{s}_{0}+T_{2-}) \leq C
\int_{\tilde{s}_{0}+T_{2-}}^{s_{0}+T_{2-}+T_{20}}  ( \sigma^2 + \dot{\sigma}^2 ) (\tau) d\tau 
\le
C \int_{\tilde{s}_{0}}^{\tilde{s}_{0}+2T_{20}} \left( D_{1bp}+D_{2}\right)(\tau) d\tau.
\end{gather*}
Thus \eqref{eqLem74*} holds for $j=0$. 
The proof is complete.
\end{proof}

We also show the next lemma which corresponds to Lemmas \ref{Lem34} and \ref{Lem39}.

\begin{lem} \label{Lem77}
Under the same assumptions as in Proposition \ref{Pro70}  with $T \geq 2T_{20}$,
if $N_\sigma(T)$ is sufficiently small, 
the following estimate holds for $\tilde{s}_{0} \in [0,T-2T_{20}]$, $s \in [\tilde{s}_{0}+2T_{20},T]$, and $\alpha \in (0,1)${\rm :}
\begin{multline} \label{eqLem77} 
(1+c)E_{2}(s) + c \int_{\tilde{s}_{0}}^s D_{2}(\tau) d\tau
\\
\le
E_{2}(\tilde{s}_{0}) 
+ C \alpha^{-1} \int_{\tilde{s}_{0}}^s D_{1bp}(\tau) d\tau
+ C (N_\sigma(T) +\alpha)  \int_{\tilde{s}_{0}}^s E_{2}(\tau) d\tau,
\end{multline}
where $c$ and $C$ are positive constants independent of $\tilde{s}_{0}$, $T$, and $\alpha$.
\end{lem}
\begin{proof}
First it follows from \eqref{eqLem74**} and \eqref{eqLem73} with $s_{0}=\tilde{s}_{0}+T_{2-}$ and $s \in [\tilde{s}_{0}+ 2T_{20},T]$ that
\begin{gather*}
E_{2}(s) 
\le
 C \alpha^{-1} \int_{\tilde{s}_{0}}^s D_{1bp}(\tau) d\tau
+ C \int_{\tilde{s}_{0}}^s D_{2}(\tau) d\tau
+ C (N_\sigma(T) +\alpha) \int_{\tilde{s}_{0}}^s (E_{2} + D_2 ) (\tau) d\tau.
\end{gather*}
Multipling this by $\ve>0$, adding the resultant inequality to \eqref{eqLem73} with $s_{0}=\tilde{s}_{0}$ and $s \in [\tilde{s}_{0}+2T_{20},T]$, and letting $\ve \ll 1$, we arrive at \eqref{eqLem77}.
\end{proof}

\subsection{Completion of the a priori estimate}

In this subsection, we complete the proof of Proposition \ref{Pro70}. 
\begin{proof}[Proof of Proposition \ref{Pro70}]
Set
\begin{gather*}
T_{0}:=T_{10}^{*} +2T_{20}, \quad 
E_{\alpha}(s):=(E_{1bp}+\alpha^{2} E_{2})(s) \quad \text{for $\alpha \in (0,1)$}.
\end{gather*}
We add the product of \eqref{eqLem73} and $\alpha^{2}$ to \eqref{eqLem72}, and letting $N_\sigma(T) \leq \alpha^{3} \leq \alpha \ll 1$  to see that for $s_{0} \in [0,T]$ and $s \in [s_{0},T]$,
\begin{gather} \label{eq793} 
E_{\alpha}(s)
\le
E_{\alpha}(s_{0}) 
+ C (N_\sigma(T) +\alpha) \int_{s_{0}}^s E_{\alpha}(\tau) d\tau,
\end{gather}
where  $C>0$ is a constant independent of $s_{0}$, $T$, and $\alpha$.
This enables us to obtain the a priori estimate for  the case $T \leq T_{0}$.
To treat another case $T > T_{0}$, multiplying \eqref{eqLem77} with $\tilde{s}_{0} \in [0,T-T_{0}]$ and $s \in [\tilde{s}_{0}+T_{0},T]$ by $\alpha^{2}$, adding the resultant inequality to \eqref{eqLem78} with $s_{0}=\tilde{s}_{0}$ and $s \in [\tilde{s}_{0}+T_{0},T]$, and letting $N_\sigma(T) \leq \alpha^{3} \leq \alpha \ll 1$, we deduce that
for $\tilde{s}_{0} \in [0,T-T_{0}]$ and $s \in [\tilde{s}_{0}+T_{0},T]$,
\begin{gather}\label{eq794}
(1+c)E_{\alpha}(s)
\le
E_{\alpha}({\tilde{s}_{0}}) 
+ C (N_\sigma(T) +\alpha) \int_{{\tilde{s}_{0}}}^s E_{\alpha}(\tau) d\tau,
\end{gather}
where $c,C>0$ are constants independent of $\tilde{s}_{0}$, $T$,  and $\alpha$.

We complete the proof by showing \eqref{eqPro70}. 
It follows from the application of Gronwall's inequality to \eqref{eq793} that 
\begin{gather}\label{eq795}
E_{\alpha}(s) \leq E_{\alpha}(s_{0})e^{C_{1}(N_\sigma(T)+\alpha)(s-s_{0})}  \quad \text{for} \  s \in [s_{0},T],
\end{gather}
where  $C_{1}>0$ is a constant independent of $s_{0}$ $T$, and $\alpha$.
Applying Gronwall's inequality to \eqref{eq794}, we have
\begin{multline}
E_{\alpha} (s) \le \frac{1}{ 1 + c_1 } \Bigl( E_{\alpha} (\tilde{s}_{0}) + C_{1}(N_\sigma(T)+\alpha)  \int_{\tilde{s}_{0}}^{\tilde{s}_{0}+T_0}  E_{\alpha} (\tau) d\tau \Bigr)  
\\
  \times e^{C_{1}(N_\sigma(T)+\alpha)(s-\tilde{s}_{0}-T_{0})}
\quad \text{for} \ s \in [\tilde{s}_{0}+T_0, T],
\label{eq796}
\end{multline}
where $C_{1}>0$ can be the same constant as in \eqref{eq795}, and $c_1>0$ is a constant  independent of $\tilde{s}_{0}$, $T$, and $\alpha$.
We note that \eqref{eq795} and \eqref{eq796} are parallel to \eqref{eq355} and \eqref{eq356}, respectively.
It also follows from \eqref{ajs72} and \eqref{ajs72*} that $E_{\alpha}(0) \leq C(\|f^{0}_{1}\|_{H^{2}(I_{1})}^{2}+\|f^{0}_{2}\|_{H^{2}(I_{2})}^{2}+|\sigma^{0}|^{2})$.
Therefore, repeating the proof of Proposition \ref{Pro31} with taking both $N_\sigma(T)$ and $\alpha$ small enough, we can conclude \eqref{eqPro70}.
\end{proof}

\begin{appendix}

\section{General IBVPs with nonlinear boundary conditions}\label{SA}
In this section, we study the following general initial--boundary value problem with ODEs and nonlinear boundary conditions:
\begin{subequations}\label{gIBVP0}
\begin{gather}
\begin{aligned}
A^0(t,x,u,\sigma,\dot{\sigma})\D_t u + A^1(t,x,u,\sigma,\dot{\sigma}) \D_{x} u & + B(t,x,u,\sigma,\dot{\sigma}) u 
\\
& =  F^{0}(t,x,\sigma,\dot{\sigma}) + F^{1}(t,x), \quad t>0, \ x \in  I,
\end{aligned}
\label{geq0}\\
u(0,x)=u_0(x), \quad x \in I, 
\label{gID0}\\
Q_{i}({\rm id}-M_{i}(t,P_{i} u(t,L_{i}),\sigma_{i}) P_{i}) u (t,L_{i}) =0, \quad t>0, \ i=1,2
\label{gBD0}
\end{gather}
with ODEs
\begin{gather}\label{gODE0}
\dot{\sigma}_{i} = \alpha_{i}(t,P_{i} u(t,L_{i}),\sigma_{i}) \sigma_{i} + \beta_{i}(t,P_{i} u(t,L_{i})), \quad \sigma_{i}(0)=\sigma_{i0},
\quad t>0, \ i=1,2,
\end{gather}
\end{subequations}
where $I:=(L_{1},L_{2})$ is a bounded open interval; 
$u=u(t,x)$ and $\sigma=(\sigma_{1},\sigma_{2})=(\sigma_{1},\sigma_{2})(t)$ are unknown vector functions with $l$ and $2$ components, respectively; 
${\rm id}$ is an identity matrix of size $l$;
$P_{i}$ and $Q_{i}$ are $l \times l$ diagonal matrices whose components are either 0 or 1;
These matrices serve as orthogonal projectors  on $\mathbb R^{l}$ and satisfy the relation $P_{i} ={\rm id} - Q_{i}$;
$M_{i}$ is a $l \times l$ matrix.
Let ${\cal O}_{u,\sigma,\dot{\sigma}} \subset \mathbb {R}^{l+2+2}$ be an open set containing the origin $(u,\sigma,\dot{\sigma})=(0,0,0)$, and fix  $\ve_{0} \in (0,1)$ so that $\overline{B_{2\ve_{0}}(0)} \subset {\cal O}_{u,\sigma,\dot{\sigma}}$, 
where $B_{2\ve_{0}}(0)$ is a ball centered at $(u,\sigma,\dot{\sigma})=(0,0,0)$ with radius $2\ve_{0}$.
For notational convenience, $\langle u,v \rangle:=u^{\top} v$ denotes the inner product for vectors $u,v$ with $l$ components.

The $l \times l$ matrices $A^0$, $A^1$, $B$, ${M}_{1}$, ${M}_{2}$, the $l$-dimensional vector $F$,
and the functions $\alpha_{1}$, $\alpha_{2}$, $\beta_{1}$, $\beta_{2}$ 
are assumed to satisfy the following conditions on regularities and structures:
\begin{asp}\label{ABasp1} The following hold for some positive constants $c_{0}$, $c_{1}$, and $T_{0}${\rm:} 
\begin{enumerate}[\rm (i)]
\item $A^0, A^1, B, F^{0} \in {\cal B}^2([0,T_0]\times \overline{I} \times {\cal O}_{u,\sigma,\dot{\sigma}})$,
$F^{1} \in H^{2}((0,T_{0})\times I) $, and $M_{1}, M_{2} \in {\cal B}^2([0,T_0]  \times {\cal O}_{u,\sigma,\dot{\sigma}})${\rm;}
\item $\alpha_{1}, \alpha_{2}, \beta_{1}, \beta_{2} \in {\cal B}^2([0,T_0]  \times {\cal O}_{u,\sigma})${\rm;}
\item $A^0$ and $A^1$ are symmetric{\rm;}
\item $\left\langle A^0(t,x,u,\sigma,\dot{\sigma}) \phi, \phi \right\rangle \geq c_0 |\phi|^2$ holds for all $\phi\in{\mathbb R}^l$ and $ (t,x,u,\sigma,\dot{\sigma}) \in [0,T_0] \times \overline{I} \times \overline{{\cal O}_{u,\sigma,\dot{\sigma}}}${\rm;}
\item $|(\det A^1)(t,x,u,\sigma,\dot{\sigma})| \geq c_1$  holds  for all $(t,x,u,\sigma,\dot{\sigma}) \in [0,T_0]\times \{L_{1},L_{2}\} \times \overline{{\cal O}_{u,\sigma,\dot{\sigma}}}${\rm;} 
\item \footnote{We remark that $\dim (\ker Q_{i}({\rm id}-M_{i}(t,P_{i} u,\sigma_{i}) P_{i})) = \dim (\ker Q_{i})$ holds.}$\dim (\ker Q_{i}({\rm id}-M_{i}(t,P_{i} u,\sigma_{i}) P_{i}))$  is equal to the number of nonnegative eigenvalues of $(-1)^{i}A^1(t,L_{i},u,\sigma,\dot{\sigma})$ for all $(t,u,\sigma,\dot{\sigma}) \in [0,T_0] \times \overline{{\cal O}_{u,\sigma,\dot{\sigma}}}$ and $i=1,2${\rm;}
\item $\left\langle (-1)^{i}A^1(t,L_{i},u,\sigma,\dot{\sigma}) \phi, \phi \right\rangle >0$ holds for all $\phi\in \ker Q_{i}({\rm id}-M_{i}(t,P_{i} u,\sigma_{i}) P_{i})$ with $|\phi|=1$, $ (t,u,\sigma,\dot{\sigma}) \in [0,T_0] \times \overline{{\cal O}_{u,\sigma,\dot{\sigma}}}$, and $i=1,2${\rm;}
\item $\beta_{i} (0,P_{i} 0)=0$ holds for $i=1,2$.
\end{enumerate}
Here the condition (vii) is equivalent to the following conditions{\rm:}
\begin{enumerate}
\item[\rm (vii)'] There exists a positive constant $c_{2}$ such that
$\left\langle (-1)^{i}A^1(t,L_{i},u,\sigma,\dot{\sigma}) \phi, \phi \right\rangle \geq c_2 |P_{i} \phi|^2$ holds for all $\phi\in \ker Q_{i}({\rm id}-M_{i}(t,P_{i} u,\sigma_{i}) P_{i})$, $ (t,u,\sigma,\dot{\sigma}) \in [0,T_0] \times \overline{{\cal O}_{u,\sigma,\dot{\sigma}}}$, and $i=1,2${\rm;}
\item[\rm  (vii)\mbox{''}] $\displaystyle \inf_{\substack{\phi\in \ker Q_{i}({\rm id}-M_{i}(t,P_{i} u,\sigma_{i}) P_{i}), \\ |\phi |=1, \\ (t,u,\sigma,\dot{\sigma}) \in [0,T_0] \times \overline{{\cal O}_{u,\sigma,\dot{\sigma}}}}} \left\langle (-1)^{i}A^1(t,L_{i},u,\sigma,\dot{\sigma}) \phi, \phi \right\rangle >0$ holds for $i=1,2$.
\end{enumerate}
\end{asp}

We will construct a time-local solution in the neighborhood of the origin $(u,\sigma,\dot{\sigma})=(0,0,0)$.
Generally speaking, we need to assume the conditions (i)--(v) and (viii) even for linear boundary conditions, 
where we use the condition (viii) to guarantee the smallness of $\dot{\sigma}$.
On the other hand, the \footnote{The condition (vi) can be relaxed by $\left\langle (-1)^{i}A^1(t,L_{i},u,\sigma,\dot{\sigma}) \phi, \phi \right\rangle \geq 0$ for linear boundary conditions.}condition (vii) is assumed for the study of the nonlinear boundary conditions \eqref{gBD0}.

The following are the compatibility conditions of the zeroth and first orders:
\begin{subequations}\label{compa0}
\begin{gather}
Q_{i}({\rm id}-M_{i} (0,P_{i} u_{0}(L_{i}), \sigma_{i0}) P_{i})u_{0} (L_{i}) =0, 
\label{0th}
\\
\begin{aligned}
& Q_{i}({\rm id}-M_{i} (0,P_{i} u_{0}(L_{i}),\sigma_{i0}) P_{i})\mbox{``}\D_{t} u (L_{i}) {\it \mbox{''}} 
- Q_{i} \D_{t} M_{i} (0,P_{i} u_{0}(L_{i}), \sigma_{i0}) P_{i} u_{0} (L_{i}) 
\\ 
&- Q_{i} \{\nabla_{\mathbb P_{i} u} M_{i} (0,P_{i} u_{0}(L_{i}), \sigma_{i0})\cdot P_{i} \mbox{``}\D_{t} u (L_{i})\mbox{''}\} P_{i} u_{0} (L_{i}) 
\\
& - Q_{i} \{\D_{\sigma_{i}} M_{i} (0,P_{i} u_{0} (L_{i}),\sigma_{i0})\mbox{``}\dot{\sigma}_{i}\mbox{''} \} P_{i} u_{0} (L_{i})
=0
\end{aligned}
\label{1st}
\end{gather}
\end{subequations}
for $i=1,2$, where
\begin{align*}
\mbox{``}\dot{\sigma}\mbox{''}&:=(\mbox{``}\dot{\sigma}_{1}\mbox{''},\mbox{``}\dot{\sigma}_{2}\mbox{''}),
\\
\mbox{``}\dot{\sigma}_{i}\mbox{''} &:= \alpha_{i}(0,P_{i} u_{0}(L_{i}),\sigma_{i0}) \sigma_{i0} +\beta_{i}(0,P_{i} u_{0}(L_{i})),
\\
\mbox{``}\D_{t} u (L_{i})\mbox{''}&:= (A^{0}(0,L_{i},u_{0}(L_{i}),\sigma_{0},\mbox{``}\dot{\sigma}\mbox{''}))^{-1} G_{i},
\\
G_{i}&:=-A^{1}(0,L_{i},u_{0}(L_{i}),\sigma_{0},\mbox{``}\dot{\sigma}\mbox{''})\D_{x}u_{0}(L_{i})
-B(0,L_{i}, u_{0}(L_{i}),\sigma_{0},\mbox{``}\dot{\sigma}\mbox{''})u_{0}(L_{i})
\\
& \quad +F^{0}(0,L_{i}, u_{0}(L_{i}),\sigma_{0},\mbox{``}\dot{\sigma}\mbox{''}) + F^{1}(0,L_{i}).
\end{align*}

We denote by $X_{2}(T,\ve,N_{1},N_{2},N_{3},N_{4})$ a set of functions $(u,\sigma)$ which satisfies the following conditions:
\begin{subequations}\label{X2}
\begin{gather}
(u,\sigma) \in \left[\bigcap_{k=0}^{2} C^{k}([0,T];H^{2-k}(I)) \right]^{l} \times \left[H^{3}(0,T)\right]^{2},
\quad 
u(\cdot,L_{i}) \in \left[H^{2}(0,T) \right]^{l}, 
\label{X2r}\\
u(0,x)=u_{0}(x), \quad \sigma(0)=\sigma_{0}, \quad
\D_{t} u(0,L_{i}) =\mbox{``}\D_{t} u(L_{i})\mbox{''}, \quad
\D_{t}\sigma(0)=\mbox{``}\dot{\sigma}\mbox{''},
\label{X2i}\\
\sup_{(t,x)\in [0,T]\times \overline{I}} | u (t,x)|
+\sup_{t\in [0,T] } |(\sigma,\dot{\sigma})(t)|  \leq \ve,
\label{X2ep}\\
\begin{aligned}
\sup_{t \in [0,T]}|\D_{t}^{2}\sigma(t)|^{2}  \leq N_{1},
\end{aligned}
\label{X2N1}\\
\sum_{k+j \leq 1} \sup_{t \in [0,T]} \|\D_{t}^{k}\D_{x}^{j}u(t)\|^{2}_{L^{2}(I)}
+ \sum_{i=1}^{2}\int_{0}^{T} |(\D_{t} P_{i}u) (\tau,L_{i})|^{2} d\tau 
\leq N_{2}, 
\label{X2N2}\\
\sum_{k+j = 2} \sup_{t \in [0,T]} \|\D_{t}^{k}\D_{x}^{j}u(t)\|^{2}_{L^{2}(I)}
+ \sum_{i=1}^{2}\int_{0}^{T} |(\D_{t}^{2} P_{i}u) (\tau,L_{i})|^{2} d\tau 
\leq N_{3}, 
\label{X2N3}\\
\int_{0}^{T} |\D_{t}^{3}\sigma(\tau)|^{2} d\tau \leq N_{4},
\label{X2N4}
\end{gather}
\end{subequations}
where $T\leq \min\{1,T_{0}\}$, $\ve \leq \min\{1,\ve_{0}\}$, $N_{1} \geq 1$, $N_{2} \geq 1$, $N_{3} \geq 1$, and $N_{4} \geq 1$ are some positive constants. 

We are now in a position to state the time-local solvability of the initial--boundary value problem \eqref{gIBVP0}, which is the main theorem in Appendix \ref{SA}.
To this end, we introduce $\tilde{C}_{s}:=\max\{1,C_{s}\}$, where
$C_{s}$ is the best constant of Sobolve's inequality $\|\cdot\|_{L^\infty(I)} \leq C_{s} \|\cdot\|_{H^{1}(I)}$.
We also note that $|\mbox{``}\dot{\sigma}\mbox{''}|$ can be taken arbitrarily small by using the condition (viii) in Assumption \ref{ABasp1} and letting $\sup_{x\in I}|u_{0}(x)|+|\sigma_{0}|$ suitably small.

\begin{thm}\label{local} Let Assumption \ref{ABasp1} hold.
There exists positive constants $N_{1}^{*}, N_{2}^{*}, N_{3}^{*}, N_{4}^{*} \geq 1$ and $\ve^{*}, T^{*} \leq 1$ such that if 
the initial data $(u_{0},\sigma_{0}) \in  [H^{2}(I)]^{l} \times \mathbb R^{2}$
satisfies the compatibility conditions \eqref{compa0} and
\begin{gather}\label{inismall}
\|u_{0}\|_{H^{2}(I)}+|(\sigma_{0},\mbox{``}\dot{\sigma}\mbox{''})| < \ve^{*}/4\tilde{C}_{s},
\end{gather}
then the problem \eqref{gIBVP0} has a unique solution $(u,\sigma) \in X_2(T^{*}/3,\ve^{*},N_{1}^{*},$ $N_{2}^{*},N_{3}^{*},N_{4}^{*})$.
\end{thm}

Appendix \ref{SA} is organized as follows.
We begin with the study of the linearized problem associated with \eqref{gIBVP0} in subsection \ref{SecA1}.
Then Theorem \ref{local} is shown in subsection \ref{SecA2}.
In subsection \ref{SecA3}, we explain how to apply Theorem \ref{local} to the initial--boundary value  problems of the quasi-one-dimensional model \eqref{eq101}.

\subsection{Linearized problem}\label{SecA1}
To solve the nonlinear problem \eqref{gIBVP0}, we consider the following linearized problem:
\begin{subequations}\label{lIBVP0}
\begin{gather}
A^0(t,x,u,\sigma,\dot{\sigma})\D_t \hat{u} + A^1(t,x,u,\sigma,\dot{\sigma}) \D_{x} \hat{u} + B(t,x,u,\sigma,\dot{\sigma}) \hat{u} = F^{0}(t,x,\sigma,\dot{\sigma}) + F^{1}(t,x), 
\label{leq0}\\
\hat{u}(0,x)=u_0(x), 
\label{lID0}\\
Q_{i}({\rm id}-M_{i}(t,P_{i} u (t,L_{i}),\sigma_{i}) P_{i}) \hat{u} (t,L_{i}) =0, \quad  \ i=1,2
\label{lBD0}
\end{gather}
with ODEs
\begin{gather}\label{lODE0}
\dot{\hat{\sigma}}_{i} = \alpha_{i}(t,P_{i} u(t, L_{i}),\sigma_{i}) \hat{\sigma}_{i} + \beta_{i}(t,P_{i} u(t, L_{i})), \quad \hat{\sigma}_{i}(0)=\sigma_{i0},
\quad i=1,2.
\end{gather}
\end{subequations}
We suppose that $(u,\sigma)$ belongs to $X_{2}(T,\ve,N_{1},N_{2},N_{3},N_{4})$, and then see that the compatibility conditions of the zeroth and first order are the same as in \eqref{compa0}. 

In this section, we sometimes abbreviate $A^0(t,x,u,\sigma,\dot{\sigma})$,  $A^1(t,x,u,\sigma,\dot{\sigma})$,  $B(t,x,u,\sigma,\dot{\sigma})$,  
$F^{0}(t,x,\sigma,\dot{\sigma})$, $F^{1}(t,x)$,  $M_{i}(t,P_{i} u(t,L_{i}),\sigma_{i})$, $\alpha_{i}(t,P_{i} u(t,L_{i}),\sigma_{i})$,  and $\beta_{i}(t,P_{i} u(t,L_{i}))$ as $A^{0}$, $A^{1}$, $B$, $F^{0}$, $F^{1}$, $M_{i}$, $\alpha_{i}$, and $\beta_{i}$, respectively.

\subsubsection{Solvability}\label{solv}
In this section, we show the following lemma on the solvability of the linearized problem \eqref{lIBVP0}.
\begin{lem}\label{lexist1}
Suppose that $(u,\sigma) \in X_{2}(T,\ve,N_{1},N_{2},N_{3},N_{4})$ and $(u_{0},\sigma_{0}) \in [H^{2}(I)]^{l} \times \mathbb R^{2}$ satisfies the compatibility conditions \eqref{compa0}. Then the linearized problem \eqref{lIBVP0} has a unique solution $(\hat{u},\hat{\sigma})$ which satisfies 
\begin{gather*}
(\hat{u},\hat{\sigma}) \in \left[\bigcap_{k=0}^{2} C^{k}([0,T-\eta];H^{2-k}(I)) \right]^{l} \times \left[H^{3}(0,T)\right]^{2}, \quad
\hat{u}(\cdot,L_{i}) \in \left[H^{2}(0,T-\eta) \right]^{l}
\end{gather*}
for any $\eta \in (0,T)$ and $i=1,2$. 
\end{lem}
\begin{proof}
We first notice that \eqref{lODE0} and \eqref{leq0}--\eqref{lBD0} are decoupled in the problem \eqref{lIBVP0}.
A standard theory of ODEs ensures that \eqref{lODE0} admits a unique solution $\hat{\sigma} \in H^{3}(0,T)$. 
\footnote{It is straightforward to find smooth approximation of the matrices $M_{i}(t,P_{i} u(t,L_{i}),\sigma_{i})$ with the conditions (vi) and (vii)\mbox{''} in Assumption \ref{ABasp1}. Therefore, the methods in \cite{KST2013,Sc1986} are applicable.}In much the same way as in Appendix in \cite{KST2013} (see also \cite{RM1974,Sc1986}), 
we see that the initial--boundary value problem \eqref{leq0}--\eqref{lBD0} has a unique solution
\begin{gather}\label{reg-1}
\hat{u} \in \left[\bigcap_{k=0}^{2} C^{k}([0,T-\eta];H^{2-k}(I)) \right]^{l}
\end{gather}
for any $\eta \in (0,T)$.

What is left is to prove that
\begin{gather}\label{reg0}
\hat{u}(\cdot,L_{i}) \in \left[H^{2}(0,T-\eta) \right]^{l}, \quad i=1,2.
\end{gather}
To this end, we use the time mollifier
\begin{gather}\label{mol1}
J_\delta f (t,x):= \int_{0}^{\infty}  \frac{1}{\delta} \vphi \left(\frac{t-\tau}{\delta}+1\right) f(\tau,x)d\tau,
\end{gather}
where $\vphi(t)$ is a non-negative $C^\infty_0$ function
whose support lies compactly in $\{|t| \leq 1\}$, and  $\|\vphi\|_{L^1(\mathbb R)}=1$.
Let us mention several properties of this mollifier. We know that for any $k \in \mathbb N$ and $\delta \in (0,T_{0}/2)$,
\begin{align*}
\D_t^k J_\delta &: H^1(0,T_0) \to H^1(0,T_0-2\delta),
\\
\D_t^k J_\delta &: \bigcap_{k=0}^{2} C^{k}([0,T_{0}];H^{2-k}(I))  \to \bigcap_{k=0}^{2} C^{k}([0,T_{0}-2\delta];H^{2-k}(I)),
\\
\D_t^k J_\delta &: H^2((0,T_0)\times I) \to H^2((0,T_0-2\delta)\times I).
\end{align*}
Furthermore, the mollifier $J_\delta$ commutes with the differential operators,
i.e. $J_\delta \D_t^k  f =\D_t^k   J_\delta f$ 
for $f\in H^p(0,T_0)$ and $k \in \mathbb N$ as well as 
$J_\delta \D_t^k \D_x^l f =\D_t^k \D_x^l  J_\delta f$ 
for $f\in H^p((0,T_0)\times I)$ and $k+l \le p \in \mathbb N$.
For any $\eta \in (0,T_{0})$, the following convergences also hold as $\delta$ tends to zero:
\begin{subequations}\label{mp0}
\begin{align}
&\|J_\delta f - f  \|_{L^2(0,T_0-\eta)} + \left\| \D_t^k \left( [ J_\delta , a ] f \right) \right \|_{L^2(0,T_0-\eta)} \to 0
\quad \text{if} \  f \in L^2(0,T_0),
\label{mp5} \\
&\|J_\delta f - f  \|_{L^2((0,T_0-\eta)\times I)} + \left\| \D_t^k \left( [ J_\delta , b ] f \right) \right \|_{L^2((0,T_0-\eta)\times I)} \to 0
\quad \text{if} \  f \in L^2((0,T_0)\times I),
\label{mp4} \\
&\|J_\delta f(0) \! - \! f_0  \|_{L^2(I)} \to 0 
\quad \text{if} \  f \! \in \! L^\infty (0,T_0;L^2(I)) \text{ with }   \| f(t)\! -\! f_0\|_{L^2(I)} \to 0  \text{ as } t \to 0+,
\label{mp2} 
\end{align}
\end{subequations}
where $a \in L^\infty(0,T_0) \cap Lip([0,T_0])$, 
$b \in L^\infty((0,T_0)\times I) \cap Lip([0,T_0] \times \overline{I})$,  $k=0,1$, and $[\cal T,\cal S]:=\cal T\cal S-\cal S\cal T$ is a commutator.

From now on we show \eqref{reg0}.
Differentiate \eqref{leq0} in $t$, apply the mollifier above, and then differentiate it in $t$ again to obtain
\begin{gather}
A^0\D_{t}(\D_{tt} J_{\delta} \hat{u}) 
+ A^1\D_{x}(\D_{tt} J_{\delta} \hat{u})  
=G_{\delta}^{0}+G_{\delta}^{1}+J_{\delta} G,
\label{meq1}\\
\begin{aligned}
G_{\delta}^{0}:=&\D_{t} \left( [A^0,J_\delta ] \D_{tt} \hat{u} \right)
+\D_{t} \left( [A^1,J_\delta ] \D_{tx} \hat{u} \right),
\\
G_{\delta}^{1}:=&  [A^0,\D_{t} ]  J_\delta \D_{tt} \hat{u}
+ [A^1, \D_{t} ]  J_\delta \D_{tx} \hat{u},
\\
G:=&\D_{t} \left([A^0 ,\D_{t} ] \D_{t} \hat{u} \right)
+\D_{t} \left([A^1,\D_{t} ] \D_{x} \hat{u} \right)
+\D_{tt}\left\{ -B\hat{u} + F^{0} + F^{1} \right\}.
\end{aligned}
\notag
\end{gather}
Similarly, it follows from \eqref{lBD0} that 
\begin{gather}
Q_{i}({\rm id}-M_{i}P_{i}) \D_{tt} J_{\delta}  \hat{u} (t,L_{i}) =H_{i\delta}^{0} + H_{i\delta}^{1} + J_{\delta} H_{i},
\label{mBD1}
\\
\begin{aligned}
H_{i\delta}^{0}&:= \D_{t} \left( [J_\delta, Q_{i}M_{i}P_{i} ] \D_{t} \hat{u} \right)(t,L_{i}),
\quad
H_{i\delta}^{1}:= [\D_{t}, Q_{i}M_{i}P_{i} ] J_\delta \D_{t} \hat{u}(t,L_{i}),
\\
H_{i}&:=\D_{t} \left([\D_{t}, Q_{i}M_{i}P_{i} ] \hat{u} \right)(t,L_{i}).
\end{aligned}
\notag
\end{gather}
Taking a difference between \eqref{meq1} with $\delta=\delta'$ and $\delta=\delta''$,
multiplying the result by $2v:=2(\D_{tt} J_{\delta'}\hat{u}-\D_{tt} J_{\delta''}\hat{u})$, 
and integrating it over $(0,T-\eta)\times I$, we obtain 
\begin{align}
&\int_{I} \left\langle A^0 v, v \right\rangle (T-\eta,x) dx
+\int_{0}^{T-\eta} \left\langle A^1  v, v \right\rangle (t,L_{2}) dt 
-\int_{0}^{T-\eta} \left\langle A^1 v, v \right\rangle (t,L_{1}) dt 
\notag \\
&=\int_{I} \left\langle A^0 v, v \right\rangle (0,x) dx
+ \int_{0}^{T-\eta} \!\! \int_{I} \left\langle (\D_{t}  A^0 ) v, v \right\rangle (t,x) dxdt
+  \int_{0}^{T-\eta} \!\! \int_{I} \left\langle (\D_{x} A^1 ) v, v \right\rangle (t,x) dxdt
\notag \\
&\quad +2\int_{0}^{T-\eta} \!\! \int_{I}  \left\langle  G_{\delta'}^{0}+G_{\delta'}^{1}+J_{\delta'} G - G_{\delta''}^{0} - G_{\delta''}^{1} - J_{\delta''} G, v \right\rangle (t,x) dxdt
\notag \\
& \to 0 \quad \text{as $\delta',\delta'' \to 0$,} 
\end{align}
where we have used \eqref{reg-1}, \eqref{mp0}, and $(u,\sigma) \in X_{2}(T,\ve,N_{1},N_{2},N_{3},N_{4})$ in showing the convergence. 
The first term on the left hand side is nonnegative, since $A^{0}$ is positive definite.
By using the boundary condition \eqref{mBD1} and setting
$\overline{H}_{2}:=(H_{2\delta'}^{0}+H_{2\delta'}^{1}+J_{\delta'} H_{2}) - (H_{2\delta''}^{0}+H_{2\delta''}^{1}+J_{\delta''} H_{2})$, the second term is estimated from below as follows: 
\begin{align*}
&\int_{0}^{T-\eta} \left\langle A^1 v, v \right\rangle (t,L_{2}) dt 
\\
&=\int_{0}^{T-\eta} \left\langle A^1(P_{2} +Q_{2} )  v, (P_{2} +Q_{2}) v \right\rangle (t,L_{2}) dt 
\\
&=\int_{0}^{T-\eta} \left\langle A^1 (P_{2} +Q_{2}M_{2} P_{2})   v, (P_{2} + Q_{2}M_{2} P_{2}) v \right\rangle (t,L_{2}) dt 
\\
&\quad  + 2\int_{0}^{T-\eta} \left\langle A^1 \overline{H}_{2}, (P_{2} + Q_{2}M_{2}P_{2}) v \right\rangle (t,L_{2}) dt 
+ \int_{0}^{T-\eta} \left\langle A^1 \overline{H}_{2}, \overline{H}_{2} \right\rangle (t,L_{2}) dt 
\\
& \geq \frac{c_{2}}{2} \int_{0}^{T-\eta} |P_{2} v |^{2 } (t,L_{2})  dt 
- C(c_{2}) \int_{0}^{T-\eta} |\overline{H}_{2} |^{2}(t,L_{2})  dt,
\end{align*}
where we have used $(P_{2} +Q_{2}M_{2} P_{2})   v \in \ker Q_{2}({\rm id}-M_{2}P_{2})$, the condition (vii)' in Assumption \ref{ABasp1}, and Schwarz's inequality in deriving the last inequality. 
We also notice that the rightmost converges to zero as $\delta',\delta'' \to 0$.
Similarly, setting
$\overline{H}_{1}:=(H_{1\delta'}^{0}+H_{1\delta'}^{1}+J_{\delta'} H_{1}) - (H_{1\delta''}^{0}+H_{1\delta''}^{1}+J_{\delta''} H_{1})$, we have
\begin{align*}
-\int_{0}^{T-\eta} \left\langle A^1v, v \right\rangle (t,L_{1}) dt 
\geq  \frac{c_{2}}{2} \int_{0}^{T-\eta} |P_{1} v |^{2} (t,L_{1}) dt 
 - C(c_{2}) \int_{0}^{T-\eta} |\overline{H}_{1}|^{2} (t,L_{1}) dt,
\end{align*}
where the rightmost converges to zero as $\delta',\delta'' \to 0$.
From these, we arrive at
\begin{gather*}
\| P_{i} (\D_{tt} J_{\delta'}\hat{u}-\D_{tt} J_{\delta''}\hat{u}) (\cdot,L_{i}) \|_{L^{2}(0,T-\eta)} \to 0 \quad \text{as $\delta',\delta'' \to 0$},
\end{gather*}
where $i=1,2$.
Furthermore, it is seen from \eqref{reg-1} and \eqref{mp0} that 
\begin{gather*}
\| P_{i} ( J_{\delta}\hat{u}- \hat{u}) (\cdot,L_{i}) \|_{H^{1}(0,T-\eta)} \to 0 \quad \text{as $\delta \to 0$}.
\end{gather*}
These two convergences imply that $P_{i} \hat{u}(\cdot, L_{i}) \in H^{2}(0,T-\eta)$.
Now \eqref{reg0} follows from \eqref{lBD0} and $Q_{i} = {\rm id} - P_{i}$. The proof is complete.
\end{proof}

\subsubsection{Invariant set}

From now on we show the following property 
that the set $X_{2}(T,\ve,N_{1},N_{2},N_{3},N_{4})$ is invariant under the mapping 
$A:(u,\sigma)\mapsto(\hat{u},\hat{\sigma})$, i.e.
$A(X_{2}(T,\ve,N_{1},N_{2},N_{3},N_{4})) \subset X_{2}(T,\ve,N_{1},N_{2},N_{3},$ $N_{4})$.
For this purpose, let us define the constants $K$ and $\tilde{K}$ by
\begin{align}
K&:=1+|L_{2}-L_{1}| + c_{0}^{-1} + c_{1}^{-1} + c_{2}^{-1} + \| (A^{0},(A^{0})^{-1},A^{1},B,F^{0})\|_{{\cal B}^2([0,T_0]\times \overline{I} \times \overline{B_{2\ve_{0}}(0)})} 
\notag \\
& \quad + \|F^{1}\|_{H^{2}((0,T_{0})\times I)}+ \sum_{i=1}^{2} \| (M_{i},\alpha_{i},\beta_{i})\|_{{\cal B}^2([0,T_0]\times \overline{B_{2\ve_{0}}(0)})},
\label{K1} \\
\tilde{K}&:=1+ \| (\D_{t} r ,  \D_{tx} r) \|^{2}_{L^{2}(I)}, \ \  r=r(t,x):= \left(\! \mbox{``}\D_{t} u (L_{1})\mbox{''} \frac{L_{2}-x}{L_{2}-L_{1}} + \mbox{``}\D_{t} u (L_{2})\mbox{''} \frac{x-L_{1}}{L_{2}-L_{1}} \! \right)  t ,
\label{tK1}
\end{align}
where $c_{0}$, $c_{1}$, $c_{2}$ are the same constants as in Assumption \ref{ABasp1},
and $B_{2\ve_{0}}(0)$ is defined just above Assumption \ref{ABasp1}.
We remark that $K$ and $\tilde{K}$ are independent of $T$, $\ve$, $N_{1}$, $N_{2}$, $N_{3}$, and $N_{4}$.
In addition, $\tilde{K}$ is bounded by some positive constant $C_{K,N}$  depending only on $K$ and $N$ if $\|u_{0}\|_{H^{2}}+|\sigma_{0}| \leq N$, where we have used the fact \footnote{We have used $\|F^{1}\|_{{\cal X}_{1}(T_{0})} \leq C \|F^{1}\|_{H^{2}((0,T_{0})\times I)}$, where ${\cal X}_{1}(T_{0}):=C([0,T_{0}];H^{1}(I)) \cap C^{1}([0,T_{0}];L^{2}(I))$.}that
 \begin{gather}\label{compa1}
 | (\mbox{``}\D_{t} u(L_{1})\mbox{''},\mbox{``}\D_{t} u(L_{2})\mbox{''}) | + |\mbox{``}\dot{\sigma}\mbox{''}| \leq C_{K,N}.
 \end{gather}
The main result in this subsection is summarized in the next proposition.

\begin{pro}\label{invariant1}
For any $N\geq 1$ and $\ve \in (0,\ve_{0}]$, there exist positive constants $N_{1}^{*}=N_{1}^{*}(N)$, \footnote{We use $N_{2}^{*} \geq \tilde{K}$ and $N_{3}^{*} \geq \tilde{K}$ to solve the nonlinear problem \eqref{gIBVP0}.}$N_{2}^{*}=N_{2}^{*}(N) \geq \tilde{K}$, $N_{3}^{*}=N_{3}^{*}(N)\geq \tilde{K}$, $N_{4}^{*}=N_{4}^{*}(N)$, $\tilde{T}=\tilde{T}(\ve,N)$ such that 
if the initial data $(u_{0},\sigma_{0}) \in  [H^{2}(I)]^{l} \times {\mathbb R}^{2}$
satisfies the compatibility conditions \eqref{compa0}, $\|u_{0}\|_{H^{2}(I)}+|\sigma_{0}| \leq N$, and 
\begin{align}
\sup_{x \in I} |u_{0}(x)|
+|(\sigma_{0},\mbox{``}\dot{\sigma}\mbox{''})|
 < \frac14 \min\left\{\ve, \sqrt{N_{3}^{*}} \right\},
\label{asmp_ini1}
\end{align}
and the functions $(u,\sigma)$ belongs to $X_{2}(T,\ve,N_{1}^{*},N_{2}^{*},N_{3}^{*},N_{4}^{*})$ for some $T  \in (0,\tilde{T}]$, 
then the problem \eqref{lIBVP0} has a unique solution 
$(\hat{u},\hat{\sigma}) \in X_{2}(T-\eta,\ve,N_{1}^{*},N_{2}^{*},N_{3}^{*},N_{4}^{*})$
for any $\eta \in (0,T)$.
\end{pro}

We first derive the estimates of $\hat{\sigma}$ in order to find suitable $N_{1}^{*}$ and $N_{4}^{*}$.

\begin{lem}
Let $(\hat{u},\hat{\sigma})$ be the solution in Lemma \ref{lexist1}. For any $N \geq 1$, 
there exists a positive constant $T_{1}=T_{1}(\ve,N,N_{1},N_{2},N_{3},N_{4})$ such that if 
$\|u_{0}\|_{H^{2}(I)}+|\sigma_{0}| \leq N$ and $0<T \leq \min\{1,T_{1}\}$, then
\begin{gather}
\sum_{k=0}^{1}\sum_{i=1}^{2} \sup_{t \in [0,T]} |\D_{t}^{k} \hat{\sigma}_{i} (t)|^{2} 
\leq |\sigma_{0}|^{2}+ |\mbox{``}\dot{\sigma}\mbox{''}|^{2} +\frac{\ve^{2}}{16},
\label{essig1}\\
\sum_{i=1}^{2} \sup_{t \in [0,T]} |\D_{t}^{2} \hat{\sigma}_{i} (t)|^{2}
\leq D_{K,N}^{\natural},
\label{essig2}\\
\sum_{i=1}^{2} \int_{0}^{T} |\D_{t}^{3} \hat{\sigma}_{i} (t)|^{2} ds \leq D_{K,N}^{\#} N_{3},
\label{essig3}
\end{gather}
where $D_{K,N}^{\natural},D_{K,N}^{\#} \geq 1$ are constants depending only on $K$ and $N$. 
\end{lem}
\begin{proof}
Applying $\D_{t}^{k}$ to \eqref{lODE0} for $k\leq 1$, multiplying it by $2\D_{t}^{k} \hat{\sigma}_{i}$, 
and integrating it over $[0,T]$, we have
\begin{align*}
|\D_{t}^{k}\hat{\sigma}_{i} (t)|^{2} &= |\D_{t}^{k}\hat{\sigma}_{i} (0)|^{2}
+ 2 \int_{0}^{t} \left[ \D_{t}^{k} \left\{\alpha_{i}(\tau,P_{i} u(\tau, L_{i}),\sigma_{i}) \hat{\sigma}_{i} + \beta_{i}(\tau,P_{i} u(\tau,L_{i})) \right\}\right]\D_{t}^{k}\hat{\sigma}_{i} (\tau) d\tau
\\
& \leq  |\D_{t}^{k}\hat{\sigma}_{i} (0)|^{2} + \tilde{C}T 
+\tilde{C}T \sum_{k=0,1} \sup_{t \in [0,T]}   |\D_{t}^{k}\hat{\sigma}_{i} (t)|^{2}
\end{align*}
for $t \in [0,T]$,
where we have also used \eqref{K1} and $(u,\sigma) \in X_{2}(T,\ve,N_{1},N_{2},N_{3},N_{4})$, and the positive constant $\tilde{C}$ depends on $K$, $N$, $N_{1}$, $N_{2}$, $N_{3}$, and $N_{4}$ but independent of $\ve$ and $T$.
The first term of the right hand side can be written as
\begin{gather*}
|\hat{\sigma}_{i} (0)|^{2} = |\sigma_{i0}|^{2}, 
\quad 
|\D_{t}\hat{\sigma}_{i} (0)|^{2} = |\alpha_{i} (0,P_{i} u_{0}(L_{i}),\sigma_{i0}) {\sigma}_{i0} + \beta_{i} (0,P_{i} u_{0}(L_{i}))|^{2} = |\mbox{``}\dot{\sigma}_{i}\mbox{''}|^{2}.
\end{gather*}
Using these and taking $T_{1}$ small,  we conclude \eqref{essig1}.

Similarly as above, we observe that
\begin{align*}
|\D_{t}^{2}\hat{\sigma}_{i} (t)|^{2} &= |\D_{t}^{2}\hat{\sigma}_{i} (0)|^{2}
+ 2\int_{0}^{t} \left[\D_{t}^{2} \left\{ \alpha_{i}(\tau,P_{i} u(\tau, L_{i}),\sigma_{i}) \hat{\sigma}_{i} + \beta_{i}(\tau,P_{i} u(\tau,L_{i})) \right\} \right]\D_{t}^{2}\hat{\sigma}_{i} (\tau) d\tau
\\
& \leq  |\D_{t}^{2} \hat{\sigma}_{i} (0)|^{2} + \tilde{C}\sqrt{T} 
+\tilde{C}\sqrt{T} \sum_{k=0}^{2} \ \sup_{t \in [0,T]}   |\D_{t}^{k}\hat{\sigma}_{i} (t)|^{2}.
\end{align*}
The first term of the right hand side can be estimated from above as
\begin{align*}
|\D_{t}^{2} \hat{\sigma}_{i} (0)|^{2} 
&= |\D_{t} \{ \alpha_{i} (t,P_{i} u(t,L_{i}),\sigma_{i}) \hat{\sigma}_{i} + \beta_{i}(t,P_{i} u(t,L_{i}))\} |_{t=0} |^{2}
\\
&\leq C_{K,N} \left\{1+ | (\mbox{``}\D_{t} u(L_{1})\mbox{''},\mbox{``}\D_{t} u(L_{2})\mbox{''}) | +|\mbox{``}\dot{\sigma}\mbox{''}| \right\}^{2} |\sigma_{0}|^{2} 
+C_{K,N} |\mbox{``}\dot{\sigma}\mbox{''}|^{2} 
\\
&\quad +C_{K,N} \{1+ | (\mbox{``}\D_{t} u(L_{1})\mbox{''},\mbox{``}\D_{t} u(L_{2})\mbox{''} ) | \}^{2}
\\
& \leq C_{K,N},
\end{align*}
where we have also used \eqref{compa1}, and $C_{K,N}$ is a positive constant depending only on $K$ and $N$. 
Consequently, we obtain \eqref{essig2} by taking $T_{1}$ small.

Applying $\D_{t}^{2}$ to \eqref{lODE0} and taking the $L^{2}$ norm, we have
\begin{align*}
\int_{0}^{T} |\D_{t}^{3} \hat{\sigma}_{i}(t)|^{2} dt 
&\leq C_{K,N} \sup_{t \in [0,T]} |\hat{\sigma}_{i} (t)|^{2} \int_{0}^{T} |\D_{t}^{2}  P_{i} u(t,L_{i})|^{2} dt
+\tilde{C}T\sum_{k=0}^{2} \sup_{t \in [0,T]} |\D_{t}^{k} \hat{\sigma}_{i} (t)|^{2} 
\\
&\quad +C_{K,N} \int_{0}^{T} |\D_{t}^{2}  P_{i} u(t,L_{i})|^{2} dt +\tilde{C}T
\\
& \leq C_{K,N}N_{3}+\tilde{C}T,
\end{align*}
where we have also used \eqref{compa1} and \eqref{essig1}--\eqref{essig2} with $\ve \leq \ve_{0} \leq 1$ in deriving the last inequality.
Consequently, we obtain \eqref{essig3} by taking $T_{1}$ small.
\end{proof}


Next we derive the estimates of $\hat{u}$ in order to find suitable $N_{2}^{*}$ and $N_{3}^{*}$.
Owing to the condition (v) in Assumption \ref{ABasp1}, we can find some $\theta>0$ such that
\begin{gather}\label{detA1}
|\det A^1| \geq \frac{c_1}{2}, \quad  (t,x,u,\sigma,\dot{\sigma}) \in [0,T_0]\times ([L_{1},L_{1}+\theta] \cup [L_{2}-\theta,L_{2}] ) \times  \overline{B_{2\ve_{0}}(0)},
\end{gather}
which means that the inverse of $A^{1}$ exists around the boundaries.
We first consider the interior estimates.
To do so, we employ a cut-off function $\chi \in C_{0}^{\infty}(I)$ with $0 \leq \chi(x) \leq 1$ and
\begin{gather}\label{chi1}
\chi(x) = \left\{
\begin{array}{ll}
1 & \text{if } L_{1} \leq x \leq  L_{1}+\theta/2,
\\
0 & \text{if } L_{1}+\theta \leq x \leq L_{2}-\theta,
\\
1 &  \text{if } L_{2}-\theta/2 \leq x \leq L_{2}.
\end{array}
\right.
\end{gather}

\begin{lem}
Let $N \geq 1$. If $\|u_{0}\|_{H^{2}(I)}+|\sigma_{0}| \leq N$,
the solution $(\hat{u},\hat{\sigma})$ in Lemma \ref{lexist1} satisfies the following inequality:
\begin{align}
&\sup_{t \in [0,T]} \|(1-\chi)\D_{x}^{j}\hat{u}(t)\|^{2}_{L^{2}(I)}
\notag \\
&\leq C_{K, N} + \tilde{C} \sqrt{T} \sum_{m=0}^{j} \sup_{t \in [0,T]}\|\D_{t}^{m}\hat{u}(t)\|_{H^{j-m}(I)}^{2}
+\tilde{C} \sqrt{T}, \quad j=0,1,2,
\label{ues0}
\end{align}
where the positive constant $C_{K,N}$ depends only on $K$ and $N$;
the positive constant $\tilde{C}$ depends on $K$, $N$, $N_{1}$, $N_{2}$, $N_{3}$, and $N_{4}$ but independent of $\ve$ and $T$.
\end{lem}
\begin{proof}
Applying $\D_{x}^{j}$ to \eqref{leq0} for $j=0,1,2$, we have the equation
\begin{align}
A^0\D_{t} \D_{x}^{j} \hat{u}
+A^1 \D_{x} \D_{x}^{j} \hat{u}
=G_{j}:=\D_{x}^{j} (F^{0}+F^{1}-B\hat{u})
+[A^{0},\D_{x}^{j}] \D_{t} \hat{u}
+[A^1,\D_{x}^{j}]\D_x \hat{u}.
\label{ueq1}
\end{align}
Then we multiply \eqref{ueq1} by $2(1-\chi)^{2} \D_{x}^{j} \hat{u}$,
integrate it over $[0,t]\times I$ for $t \in [0,T]$, and use $(1-\chi)(L_{i})=0$ for $i=1,2$ to obtain
\begin{align}
&\int_{I} \left\langle A^{0} \left((1-\chi) \D_{x}^{j} \hat{u}\right), (1-\chi) \D_{x}^{j} \hat{u}\right\rangle (t,x) dx
\notag \\
&=
\int_{I} \left\langle A^{0} \left((1-\chi) \D_{x}^{j} \hat{u}\right), (1-\chi) \D_{x}^{j} \hat{u}\right\rangle (0,x) dx
\notag \\
&\quad + \int_{0}^{t} \!\! \int_{I}  \left\langle  (1-\chi)^{2} (\D_{t} A^{0}) \D_{x}^{j} \hat{u} + \D_{x} ((1-\chi)^{2} A^1) \D_{x}^{j} \hat{u}+2 (1-\chi)^{2}G_{j}, \D_{x}^{j} \hat{u}\right\rangle (\tau,x)  dxd\tau 
\notag \\
&\leq K N^{2} +\tilde{C} \sqrt{T} \sum_{m=0}^{j} \sup_{t \in [0,T]}\|\D_{t}^{m}\hat{u}(t)\|_{H^{j-m}(I)}^{2} +\tilde{C}\sqrt{T},
\label{ues1}
\end{align}
where we have also used \eqref{K1}, $(u,\sigma) \in X_{2}(T,\ve,N_{1},N_{2},N_{3},N_{4})$, and $\|u_{0}\|_{H^{2}(I)} \leq N$ in deriving the last inequality.
The first term on the left hand side is estimated by using the condition (iv) in Assumption \ref{ABasp1} as follows:
\begin{align*}
\int_{I} \left\langle A^{0} \left( (1-\chi) \D_{x}^{j} \hat{u} \right), (1-\chi) \D_{x}^{j} \hat{u}\right\rangle (t,x) dx
&\geq c_{0} \|(1-\chi) \D_{x}^{j} \hat{u}(t)\|^{2}_{L^{2}(I)}
\\
&\geq \frac1K \|(1-\chi) \D_{x}^{j} \hat{u}(t)\|^{2}_{L^{2}(I)}.
\end{align*}
Substituting this into \eqref{ues1} yields \eqref{ues0}.
\end{proof}

Next we consider the estimates of time-derivatives.

\begin{lem}\label{DK1}
Let $N \geq 1$. If $\|u_{0}\|_{H^{2}}+|\sigma_{0}| \leq N$, the solution $(\hat{u},\hat{\sigma})$  in Lemma \ref{lexist1} satisfies the following inequality:
\begin{align}
&\sup_{t \in [0,T]} \| \D_{t}^{k}\hat{u}(t)\|^{2}_{L^{2}(I)} 
+ \sum_{i=1}^{2} \int_{0}^{T} | \D_{t}^{k} P_{i} \hat{u}(t,L_{i}) |^{2} dt 
\notag \\
& \leq C_{K,N}
+\max\{0,k-1\} \left(C_{K,N}(N_{1}+N_{2}) + C_{K,N} N_{3} \sum_{i=1}^{2} \sup_{t \in [0,T]}|P_{i} \hat{u}|^{2}(t,L_{i}) \right)
\notag \\
& \quad +\tilde{C} \sqrt{T} \sum_{m=0}^{k} \sup_{t \in [0,T]}\|\D_{t}^{m}\hat{u}(t)\|_{H^{k-m}(I)}^{2} + \tilde{C} \sqrt{T}, \quad k=0,1,2,
\label{ues2}
\end{align}
where the positive constant $C_{K,N}$ depends only on $K$ and $N$, and
the positive constant $\tilde{C}$ depends on $K$, $N$, $N_{1}$, $N_{2}$, $N_{3}$, and $N_{4}$ but independent of $\ve$ and $T$.
\end{lem}
\begin{proof}
Applying $\D_{t}^{k}$  to \eqref{leq0} and \eqref{lBD0} for $k=0,1,2$,
we have the equation
\begin{align}
&A^0\D_{t} \D_{t}^{k} \hat{u} + A^1 \D_{x} \D_{t}^{k} \hat{u} 
=H_{k}:=\D_{t}^{k} (F^{0} + F^{1} -B\hat{u}) +[A^{0},\D_{t}^{k}]\D_{t}\hat{u}
+[A^1,\D_{t}^{k}]\D_x\hat{u}
\label{ueq2}
\end{align}
and the boundary conditions
\begin{gather}
Q_{i}({\rm id}-M_{i}P_{i}) \D_{t}^{k} \hat{u} (t,L_{i}) =Q_{i} [\D_{t}^{k},M_{i} ]P_{i} \hat{u} (t,L_{i}), \quad  \ i=1,2.
\label{ubc1}
\end{gather}
Then we multiply \eqref{ueq2} by $2\D_{t}^{k} \hat{u}$
and integrate it over $[0,t]\times I$ for $t \in [0,T]$ to obtain
\begin{align}
&\int_{I} \left\langle A^0  \D_{t}^{k} \hat{u}, \D_{t}^{k} \hat{u}\right\rangle (t,x) dx
+ \int_{0}^{t}  \left\langle A^1 \D_{t}^{k} \hat{u},  \D_{t}^{k} \hat{u}\right\rangle  (t,L_{2})  d\tau
- \int_{0}^{t}  \left\langle A^1 \D_{t}^{k} \hat{u},  \D_{t}^{k} \hat{u}\right\rangle  (t,L_{1}) d\tau
\notag \\
&=
\int_{I} \left\langle A^0  \D_{t}^{k} \hat{u}, \D_{t}^{k} \hat{u}\right\rangle (0,x) dx
+ \int_{0}^{t}\!\! \int_{I}  \left\langle (\D_{t} A^{0}) \D_{t}^{k} \hat{u} + (\D_{x} A^1)  \D_{t}^{k} \hat{u} + 2H_{k},  \D_{t}^{k} \hat{u}\right\rangle (\tau,x) dxd\tau 
\notag \\
&\leq K \|\D_{t}^{k} \hat{u}(0)\|^{2}_{L^{2}(I)}
+\tilde{C}\sqrt{T} \sum_{m=0}^{k} \sup_{t \in [0,T]}\|\D_{t}^{m}\hat{u}(t)\|_{H^{k-m}(I)}^{2}
+\tilde{C}\sqrt{T}, 
\label{ues3}
\end{align}
where we have also used \eqref{K1} and $(u,\sigma) \in X_{2}(T,\ve,N_{1},N_{2},N_{3},N_{4})$ in deriving the last inequality.
The first term on the left hand side is estimated by using the condition (iv) in Assumption \ref{ABasp1} as follows:
\begin{gather*}
\int_{I} \left\langle A^0  \D_{t}^{k} \hat{u}, \D_{t}^{k} \hat{u}\right\rangle (t,x) dx
\geq c_{0} \|\D_{t}^{k} \hat{u}(t)\|^{2}_{L^{2}(I)}
\geq \frac1K \|\D_{t}^{k} \hat{u}(t)\|^{2}_{L^{2}(I)}.
\end{gather*}
For the second term on the left hand side, we observe by using \eqref{ubc1},  $(P_{2} +Q_{2}M_{2} P_{2}) \D_{t}^{k} \hat{u} \in \ker Q_{2}({\rm id}-M_{2} P_{2})$,  and the condition (vii)' in Assumption \ref{ABasp1} that
\begin{align*}
&\int_{0}^{t} \left\langle A^{1} \D_{t}^{k} \hat{u}, \D_{t}^{k} \hat{u}\right\rangle (\tau,L_{2}) d\tau
\\
&=\int_{0}^{t} \left\langle A^{1} (P_{2}+Q_{2}) \D_{t}^{k} \hat{u}, (P_{2}+Q_{2}) \D_{t}^{k} \hat{u}\right\rangle (\tau,L_{2}) d\tau
\\
&=\int_{0}^{t} \left\langle A^{1} (P_{2}+Q_{2}M_{2}P_{2}) \D_{t}^{k} \hat{u}, (P_{2}+Q_{2}M_{2}P_{2})\D_{t}^{k} \hat{u}\right\rangle (\tau,L_{2}) d\tau
\\
&\quad +2\int_{0}^{t} \left\langle A^{1} Q_{2}[\D_{t}^{k},M_{2}]P_{2}  \hat{u}, (P_{2}+Q_{2}M_{2}P_{2}) \D_{t}^{k} \hat{u}\right\rangle (\tau,L_{2})  d\tau
\\
&\quad +\int_{0}^{t} \left\langle A^{1}Q_{2}[\D_{t}^{k},M_{2}]P_{2} \hat{u}, Q_{2}[\D_{t}^{k},M_{2}]P_{2} \hat{u}\right\rangle (\tau,L_{2})  d\tau
\\
&\geq \frac{1}{2K} \int_{0}^{t} |P_{2}  \D_{t}^{k} \hat{u}(\tau,L_{2}) |^{2} d\tau
-\tilde{C} T \sum_{m=0}^{k} \sup_{t \in [0,T]}\|\D_{t}^{m}\hat{u}(t)\|_{H^{k-m}(I)}^{2}
\\
& \quad -\max\{0,k-1\} C_{K,N} \int_{0}^{t} |P_{2}  \D_{t}^{2} u|^{2} |P_{2}  \hat{u}|^{2}(\tau,L_{2})d\tau
\\
&\geq \frac{1}{2K} \int_{0}^{t} |P_{2}  \D_{t}^{k} \hat{u}(\tau,L_{2}) |^{2} d\tau
-\tilde{C} T \sum_{m=0}^{k} \sup_{t \in [0,T]}\|\D_{t}^{m}\hat{u}(t)\|_{H^{k-m}(I)}^{2}
\\
&\quad -\max\{0,k-1\} C_{K,N} N_{3} \sup_{t \in [0,T]}|P_{2} \hat{u}|^{2}(t,L_{2}),
\end{align*}
where we have also used Schwarz's inequality in deriving the first inequality, and the positive constant $C_{K,N}$ depends only on $K$.
Similarly, 
\begin{align*}
-\int_{0}^{t} \left\langle A^{1} \D_{t}^{k} \hat{u}, \D_{t}^{k} \hat{u}\right\rangle (\tau,L_{1}) d\tau
&\geq \frac{1}{2K} \int_{0}^{t} |P_{1}  \D_{t}^{k} \hat{u}(\tau,L_{1}) |^{2} d\tau
-\tilde{C} T \sum_{m=0}^{k} \sup_{t \in [0,T]}\|\D_{t}^{m}\hat{u}(t)\|_{H^{k-m}(I)}^{2}
\\
&\quad -\max\{0,k-1\} C_{K,N} N_{3} \sup_{t \in [0,T]}|P_{1} \hat{u}|^{2}(t,L_{1}).
\end{align*}

Let us handle $\|\D_{t}^{k} \hat{u}(0)\|^{2}_{L^{2}(I)}$.
Recall that ${\cal X}_{1}(T_{0})=C([0,T_{0}];H^{1}(I)) \cap C^{1}([0,T_{0}];L^{2}(I))$.
We multiply \eqref{ueq1} with $j=0,1$ by $(A^{0})^{-1}$ and estimate those $L^{2}$-norm by using $\| F^{1} \|_{{\cal X}_{1}(T_{0})} \leq C \| F^{1} \|_{H^{2}((0,T_{0})\times I)}$ and Sobolev's inequality to obtain
\begin{align*}
\|\D_{t} \hat{u}(0)\|^{2}_{L^{2}(I)}+\|\D_{tx} \hat{u}(0)\|^{2}_{L^{2}(I)}
\leq C_{K,N},
\end{align*}
where the positive constant $C_{K,N}$ depends only on $K$ and $N$.
We also multiply \eqref{ueq2} with $k=1$ by $(A^{0})^{-1}$ and then estimate the $L^{2}$-norm by using Sobolev's inequality and  $(u,\sigma) \in X_{2}(T,\ve,N_{1},N_{2},N_{3},N_{4})$ to obtain
\begin{align*}
&\|\D_{tt} \hat{u}(0)\|^{2}_{L^{2}(I)}
\\
&\leq C_{K,N}(N_{1}+1) +C_{K,N}\left(\|\D_{t} {u}(0)\|^{2}_{L^{2}(I)}+N_{1}+1\right)\| u_{0} \|_{L^{\infty}(I)}^{2} +C_{K,N}\|\D_{t} \hat{u}(0)\|^{2}_{L^{2}(I)}
\\
&\quad +C_{K,N}\left(\|\D_{t} {u}(0)\|^{2}_{L^{2}(I)}+N_{1}+1\right) \left(\| \D_{t} \hat{u} (0) \|_{L^{\infty}(I)}^{2} + \| \D_{x} u_{0} \|_{L^{\infty}(I)}^{2}\right)
 +C_{K,N}\|\D_{tx} \hat{u}(0)\|^{2}_{L^{2}(I)}
\\
&\leq C_{K,N}(N_{1}+N_{2}+1).
\end{align*}
Substituting these into \eqref{ues3} yields \eqref{ues2}.
\end{proof}

Now we complete the estimates of $\hat{u}$.

\begin{lem}
Let $(\hat{u},\hat{\sigma})$ be the solution in Lemma \ref{lexist1}.
For any $N \geq 1$, there exists a positive constant $T_{2}=T_{2}(N,N_{1},N_{2},N_{3},N_{4})$ 
such that if $\|u_{0}\|_{H^{2}(I)}+|\sigma_{0}| \leq N$, 
$|P_{i} u_{0}(L_{i})| \leq  N_{3}^{-1/2}$, and
$0<T\leq \min\{1,T_{2}\}$, then
\begin{gather}
\sum_{k+j\leq 1} \sup_{t \in [0,T]} \| \D_{t}^{k}\D_{x}^{j}\hat{u}(t)\|^{2}_{L^{2}(I)} 
+ \sum_{i=1}^{2}\int_{0}^{T} |(\D_{t} P_{i}\hat{u}) (t,L_{i}) |^{2} dt
\leq D_{K,N}^{\dagger},
\label{ues4} \\
\sum_{k+j=2} \sup_{t \in [0,T]} \| \D_{t}^{k}\D_{x}^{j}\hat{u}(t)\|^{2}_{L^{2}(I)} 
+ \sum_{i=1}^{2} \int_{0}^{T} | \D_{t}^{2} P_{i} \hat{u}(t,L_{i}) |^{2} dt
\leq D_{K,N}^{\ddagger}(N_{1}+N_{2}+1)^{3},
\label{ues5}
\end{gather}
where $D_{K,N}^{\dagger},D_{K,N}^{\ddagger} \geq \tilde{K}$ are constants depending only on $K$ and $N$.
\end{lem}
\begin{proof}

We recall \eqref{detA1}, multiply \eqref{leq0} by $\chi(A^1)^{-1}$, and then take the $L^{2}$-norm to obtain
\begin{align}
\|\chi\D_{x} \hat{u} (t)\|^{2}_{L^{2}(I)}
&\leq C_{K,N}\|\chi\D_{t} \hat{u} (t)\|^{2}_{L^{2}(I)}+C_{K,N}\|\chi \hat{u} (t)\|^{2}_{L^{2}(I)} +C_{K,N}
\notag\\
&\leq C_{K,N}
+\tilde{C} \sqrt{T} \sum_{m=0}^{1} \sup_{t \in [0,T]}\|\D_{t}^{m}\hat{u}(t)\|_{H^{1-m}(I)}^{2} + \tilde{C} \sqrt{T},
\label{ues7}
\end{align}
where we have used \eqref{ues2} with $k=1$ in deriving the inequality.
Adding \eqref{ues2}  with $k=0$, \eqref{ues0} with $j=1$, \eqref{ues2}  with $k=1$, and \eqref{ues7},
and then taking $T_{1}$ small enough, we conclude \eqref{ues4}. 

Let us show \eqref{ues5}. First it is seen by using \eqref{ues4} that 
\begin{gather*}
\sup_{t \in [0,T]}|P_{i} \hat{u}(t,L_{i})|
\leq |P_{i} u_{0}(L_{i})| + \sqrt{T} \left( \int_{0}^{T} |\D_{t} P_{i} \hat{u}(t,L_{i}) |^{2} dt \right)^{1/2}
\leq |P_{i} u_{0}(L_{i})| + C_{K,N}\sqrt{T}.
\end{gather*}
Plugging this into \eqref{ues2} with $k=2$ and using the assumption $|P_{i} u_{0}(L_{i})| \leq N_{3}^{-1/2}$, we see that
\begin{align}
&\sup_{t \in [0,T]} \| \D_{t}^{2}\hat{u}(t)\|^{2}_{L^{2}(I)} 
+ \sum_{i=1}^{2} \int_{0}^{T} | \D_{t}^{2} P_{i} \hat{u}(t,L_{i}) |^{2} dt 
\notag \\
& \leq C_{K,N} (N_{1}+N_{2}+1)
+\tilde{C} \sqrt{T} \sum_{m=0}^{2} \sup_{t \in [0,T]}\|\D_{t}^{m}\hat{u}(t)\|_{H^{2-m}(I)}^{2} + \tilde{C} \sqrt{T}.
\label{ues8}
\end{align}
We multiply \eqref{leq0} by $(A^0)^{-1}$, differentiate the result in $x$, multiply it by $(1-\chi)$, and take the $L^{2}$-norm.
Then it follows from $\|F^{1}\|_{{\cal X}_{1}(T_{0})} \leq C \|F^{1}\|_{H^{2}((0,T_{0})\times I)}$, Sobolev's inequality, \eqref{ues0}, \eqref{ues4}, and $(u,\sigma) \in X_{2}(T,\ve,N_{1},N_{2},N_{3},N_{4})$ that 
\begin{align}
&\|(1-\chi)\D_{tx} \hat{u}(t)\|^{2}_{L^{2}(I)}
\notag\\
&\leq C_{K,N} \|(1-\chi)\D_{xx} \hat{u}(t)\|^{2}_{L^{2}(I)} + C_{K,N}\left(\|\D_{x} u(t)\|^{2}_{L^{2}(I)}+1\right)  \|(1-\chi)\D_{x} \hat{u}(t)\|_{L^{\infty}(I)}^{2} 
\notag\\
&\quad  +C_{K,N} \|\D_{x} \hat{u}(t)\|^{2}_{L^{2}(I)} + C_{K,N}\left(\|\D_{x} u(t)\|^{2}_{L^{2}(I)}+1\right)  \|\hat{u}(t)\|_{L^{\infty}(I)}^{2}+C_{K,N}\left(\|\D_{x} u(t)\|^{2}_{L^{2}(I)}+1\right) 
\notag\\
&\leq C_{K,N}(N_{2}+1)
+\tilde{C}\sqrt{T}  \sum_{m=0}^{2} \sup_{t \in [0,T]}\|\D_{t}^{m}\hat{u}(t)\|_{H^{2-m}(I)}^{2} + \tilde{C}\sqrt{T}.
\label{ues9}
\end{align}
Furthermore, we recall \eqref{detA1}, multiply \eqref{leq0} by $\chi (A^1)^{-1}$, 
differentiate the result in $t$, and take the $L^{2}$-norm.
Then estimating the norm by using Gagliardo--Nirenberg's inequality, i.e. $\|\cdot\|_{L^{\infty}(I)} \leq C (\|\cdot\|^{1/2}_{L^{2}(I)}\|\D_{x}\cdot\|^{1/2}_{L^{2}(I)} + \|\cdot\|_{L^{2}(I)})$, Schwarz's inequality, \eqref{ues4}, \eqref{ues8},
and $(u,\sigma) \in X_{2}(T,\ve,N_{1},N_{2},N_{3},N_{4})$, we deduce that
\begin{align}
&\|\chi\D_{xt} \hat{u}(t)\|^{2}_{L^{2}(I)}
\notag\\
&\leq C_{K,N} \|\chi\D_{tt} \hat{u}(t)\|^{2}_{L^{2}(I)}
+C_{K,N}\left(\|\D_{t} {u}(t)\|^{2}_{L^{2}(I)}+N_{1}+1\right) \|\chi\D_{t} \hat{u}(t)\|_{L^{\infty}(I)}^{2}
\notag\\
&\quad + C_{K,N} \|\D_{t} \hat{u}(t)\|^{2}_{L^{2}(I)}
+C_{K,N}\left(\|\D_{t} {u}(t)\|^{2}_{L^{2}(I)}+N_{1}+1\right) \|\hat{u}(t)\|_{L^{\infty}(I)}^{2}
\notag\\
&\quad +C_{K,N}\left(\|\D_{t} {u}(t)\|^{2}_{L^{2}(I)}+N_{1}+1\right)
\notag\\
&\leq \frac{1}{2} \|\chi\D_{xt} \hat{u}(t)\|^{2}_{L^{2}(I)} 
+ C_{K,N}(N_{1} + N_{2}+1)^{2}
 +\tilde{C} \sqrt{T} \sum_{m=0}^{2} \sup_{t \in [0,T]}\|\D_{t}^{m}\hat{u}(t)\|_{H^{2-m}(I)}^{2} + \tilde{C} \sqrt{T}.
\label{ues11}
\end{align}
We multiply \eqref{leq0} by $\chi (A^1)^{-1}$, 
differentiate the result in $x$,  and then estimate the $L^{2}$-norm by using \eqref{ues4}, \eqref{ues11}, and $(u,\sigma) \in X_{2}(T,\ve,N_{1},N_{2},N_{3},N_{4})$ to obtain
\begin{align}
&\|\chi\D_{xx} \hat{u}(t)\|^{2}_{L^{2}(I)}
\notag\\
&\leq  C_{K,N} \|\chi\D_{xt} \hat{u}(t)\|^{2}_{L^{2}(I)} + C_{K,N}\left(\|\D_{x} u(t)\|^{2}_{L^{2}(I)}+1\right) \|\chi\D_{t} \hat{u}(t)\|_{L^{\infty}(I)}^{2} + C_{K,N} \|\D_{t} \hat{u}(t)\|^{2}_{L^{2}(I)} 
\notag\\
&\quad + C_{K,N} \|\D_{x} \hat{u}(t)\|^{2}_{L^{2}(I)} + C_{K,N}\left(\|\D_{x} u(t)\|^{2}_{L^{2}(I)}+1\right) \|\hat{u}(t)\|_{L^{\infty}(I)}^{2} + C_{K,N}\left(\|\D_{x} u(t)\|^{2}_{L^{2}(I)}+1\right)
\notag\\
&\leq C_{K,N}(N_{1} + N_{2}+1)^{3} 
+\tilde{C} \sqrt{T} \sum_{m=0}^{2} \sup_{t \in [0,T]}\|\D_{t}^{m}\hat{u}(t)\|_{H^{2-m}(I)}^{2} + \tilde{C} \sqrt{T}.
\label{ues12}
\end{align}
Adding \eqref{ues0} with $j=2$ and  \eqref{ues8}--\eqref{ues12},
and then taking $T_{2}$ small enough, we conclude \eqref{ues5}. 
\end{proof}

We are now in a position to prove Proposition \ref{invariant1}.
\begin{proof}[Proof of Proposition \ref{invariant1}]
First we take $T_{3}$ so small that
\begin{align}
\sup_{x \in I}|\hat{u}(t,x)| 
&\leq \sup_{x \in I}|{u}_{0}(x)| 
+C\int_{0}^{T_{3}}\|\D_{t}\hat{u}(t)\|_{H^{1}(I)} dt
\notag\\
&\leq \sup_{x \in I}|{u}_{0}(x)|  +\tilde{C}T_{3}
\leq \sup_{x \in I}|{u}_{0}(x)|  +\frac{\ve}{4},
\label{ues13}
\end{align}
and set $N_{1}^{*}$, $N_{2}^{*}$, $N_{3}^{*}$, $N_{4}^{*}$, and $\tilde{T}$ as follows:
\begin{gather*} 
N_{1}^{*}:=D_{K,N}^{\natural}, \quad
N_{2}^{*}:=D_{K,N}^{\dagger}, \quad
N_{3}^{*}:=D_{K,N}^{\ddagger}(N_{1}^{*}+N_{2}^{*}+1)^{3}, \quad
N_{4}^{*}:=D_{K,N}^{\#} N_{3}^{*}, 
\\
\tilde{T}:=\min\{1,T_{0},T_{1},T_{2},T_{3}\}.
\end{gather*}
We complete the proof by showing that $(\hat{u},\hat{\sigma}) \in X_{2}(T-\eta,\ve,N_{1}^{*},N_{2}^{*},N_{3}^{*},N_{4}^{*})$ under the conditions $({u},{\sigma}) \in X_{2}(T,\ve,N_{1}^{*},N_{2}^{*},N_{3}^{*},N_{4}^{*})$, and $T \leq \tilde{T}$.
It is clear that \eqref{X2r}--\eqref{X2i} in the definition of the set $X_{2}$ hold.
Using \eqref{asmp_ini1}, \eqref{essig1}, and \eqref{ues13} yields \eqref{X2ep}.
Furthermore, \eqref{X2N1}, \eqref{X2N2}, \eqref{X2N3}, and \eqref{X2N4} follow from \eqref{essig2}, \eqref{ues4}, \eqref{ues5}, and \eqref{essig3}, respectively.
The proof is complete.
\end{proof}

\subsection{Nonlinear problem}\label{SecA2}
This section is devoted to the proof of Theorem \ref{local}.
\begin{proof}[Proof of Theorem \ref{local}]
Let $N_{1}^{*}$, $N_{2}^{*}$, $N_{3}^{*}$, $N_{4}^{*}$, and $\tilde{T}$ be the same constants as in Proposition \ref{invariant1} with $N=1$ and $\ve=\ve^{*}:=\min\{\ve_{0}/2,\sqrt{N_{3}^{*}},(2K)^{-4}\}$.
Suppose that \eqref{inismall} holds, i.e. $\|u_{0}\|_{H^{2}(I)}+|(\sigma_{0},\mbox{``}\dot{\sigma}\mbox{''})| \leq \ve^{*}/(4 \tilde{C}_{s})$.
We define a successive approximation $\{(u^{k},\sigma^{k})\}_{k=0}^\infty$ by for $k=0$, 
\begin{gather*}
u^{0}(t,x)= u_{0}(x) +r(t,x), \quad 
\sigma^{0}(t)=\sigma_{0}+\mbox{``}\dot{\sigma}\mbox{''}t
\end{gather*}
and for $k \geq 1$,
\begin{subequations}\label{kIBVP0}
\begin{gather}
A^{0,k-1} \D_{t} u^{k} + A^{1,k-1} \D_{x} u^{k} + B^{k-1} u^{k} = F^{0,k-1} + F^{1},
\lb{keq0}\\
u^{k}(0,x)=u_{0}(x), 
\lb{kID0}\\  
Q_{i}({\rm id}-M_{i}^{k-1} P_{i}) u^{k} (t,L_{i}) =0, \quad i=1,2,
\lb{kBD0}\\
\dot{\sigma}_{i}^{k} = \alpha_{i}^{k-1} \sigma_{i}^{k} + \beta_{i}^{k-1}, \quad \sigma_{i}^{k}(0)=\sigma_{i0}, \quad i=1,2,
\label{kODE}
\end{gather}
\end{subequations}
where $r$ is defined in \eqref{tK1} and
\begin{gather*}
\begin{array}{lll}
A^{j,k}:=A^j(t,x,u^{k},\sigma^{k},\dot{\sigma}^{k}), \quad 
&B^{k}:=B(t,x,u^{k},\sigma^{k},\dot{\sigma}^{k}), \quad 
& F^{0,k}:=F^{0}(t,x,\sigma^{k},\dot{\sigma}^{k}), \quad 
\\
M_{i}^{k}:=M_{i}(t,P_{i} u^{k}(t, L_{i}),\sigma_{i}^{k}), \quad
&\alpha_{i}^{k}:=\alpha_{i}(t,P_{i} u^{k}(t, L_{i}),\sigma_{i}^{k}), \quad
&\beta_{i}^{k}:=\beta_{i}(t,P_{i} u^{k}(t, L_{i})). 
\end{array}
\end{gather*}
It follows from $N_{2}^{*},N_{3}^{*} \geq \tilde{K}$ and \eqref{tK1}
that $(u^{0},\sigma^{0})$ belongs to $X_2(T^{*},\ve^{*},N_{1}^{*},N_{2}^{*},N_{3}^{*},N_{4}^{*})$ for some $T^{*} \in (0, \tilde{T}]$.
Now Proposition~\ref{invariant1} together with \eqref{inismall} ensures that
$(u^{k},\sigma^{k}) \in X_2(T^{*}-\eta_k,\ve^{*},N_{1}^{*},N_{2}^{*},N_{3}^{*},N_{4}^{*})$ is well-defined for all $k \geq 1$, 
where $\eta_k:=T^{*}\sum_{i=1}^k 2^{-(i+1)}$.

We claim that there exists
$(u,\sigma) \in [\mathfrak{X}_{1}(T^{*}/2)]^{l} \times [H^{2}(0,T^{*}/2)]^{2}$
such that $(u^{k},\sigma^{k})$ converges to 
$(u,\sigma) \in [\mathfrak{X}_{1}(T^{*}/2)]^{l} \times [H^{2}(0,T^{*}/2)]^{2}$
by taking $T^{*}$ appropriately small if necessary.
Using this claim, the Banach--Alaoglu theorem, and $(u^{k},\sigma^{k}) \in X_2(T^{*}/2,\ve^{*},N_{1}^{*},N_{2}^{*},N_{3}^{*},N_{4}^{*})$, we conclude that
\begin{subequations}\label{unique1}
\begin{gather}
(u,\sigma) \in  \left[\bigcap_{j=0}^2 W^{j,\infty}(0,T^{*}/2;H^{2-j}(I))\right]^{l}\times [H^{3}(0,T^{*}/2)]^{2},
\quad
u(\cdot,L_{i}) \in \left[H^{2}(0,T^{*}/2) \right]^{l}, 
\\
\sup_{(t,x)\in [0,T]\times \overline{I}} | u (t,x)| +\sup_{t\in [0,T] } |(\sigma,\dot{\sigma})(t)|  \leq  \ve^{*}.
\end{gather}
\end{subequations}
It is easy to see that $(u,\sigma)$ is a \footnote{We can show the uniqueness by the energy method used in the proof of \eqref{des1} with some modification.}unique solution of the problem \eqref{gIBVP0} with \eqref{unique1}.
Furthermore, $(u,\sigma) \in X_2(T^{*}/3,\ve^{*},N_{1}^{*},N_{2}^{*},N_{3}^{*},N_{4}^{*})$ follows from the same method using mollifiers as in the proof of Theorem A.4 in \cite{KST2013}.

What is left is to show the above claim. It suffices to show that
\begin{gather}\label{des1}
E^{k}(T^{*}/2) \leq \frac12 E^{k-1}(T^{*}/2),
\end{gather}
where
\begin{gather*}
U^{k}:=u^{k}-u^{k-1}, \quad
\Sigma^{k}:=\sigma^{k}-\sigma^{k-1},
\\
E^{k}(T):=\sup_{t \in [0,T]} \|U^{k}(t)\|^{2}_{L^{2}(I)}
+ \sup_{t \in [0,T]}  |\Sigma^{k}(t)|^{2}
+ \sum_{i=1}^{2} \int_{0}^{T}  |P_{i} U^{k}(t,L_{i}) |^{2} dt + \frac{1}{8K^{2}} \int_{0}^{T} |\dot{\Sigma}^{k}(t)|^{2}dt.
\end{gather*}
First we see from \eqref{kIBVP0} that $(U^{k},\Sigma^{k})$ for $k\geq 2$ solves
\begin{subequations}
\begin{gather}
A^{0,k-1} \D_{t} U^{k} + A^{1,k-1} \D_{x} U^{k} + B^{k-1} U^{k} = \overline{F}^{k},
\lb{kkeq0}\\
U^{k}(0,x)=0, 
\lb{kkID0}\\
Q_{i}({\rm id}-M_{i}^{k-1} P_{i}) U^{k} (t,L_{i})=\overline{M}_{i}^{k}, \quad i=1,2,
\lb{kkBD0}\\
\dot{\Sigma}_{i}^{k} = \alpha_{i}^{k-1} \Sigma_{i}^{k} + \overline{\beta}_{i}^{k} , \quad  \Sigma^{k}(0)=0, \quad i=1,2,
\lb{kkODE}
\end{gather}
\end{subequations}
where 
\begin{align*}
\overline{F}^{k}&:= -(A^{0,k-1}-A^{0,k-2}) \partial_{t} u^{k-1}  - (A^{1,k-1}-A^{1,k-2}) \partial_{x} u^{k-1}  
\\
&\quad - (B^{k-1}-B^{k-2}) u^{k-1}  + (F^{0,k-1} - F^{0,k-2}),
\\
\overline{M}_{i}^{k}&:=Q_{i}(M_{i}^{k-1} -M_{i}^{k-2}) P_{i}  u^{k-1} (t,L_{i}),
\\
\overline{\beta}_{i}^{k}&:= ({\alpha}_{i}^{k-1}-{\alpha}_{i}^{k-2}) \sigma_{i}^{k-1} +({\beta}_{i}^{k-1}-{\beta}_{i}^{k-2}).
\end{align*}
It is straightforward to show by using $(u^{k},\sigma^{k}) \in X_2(T^{*}/2,\ve^{*},N_{1}^{*},N_{2}^{*},N_{3}^{*},N_{4}^{*})$ that  
\begin{align}
|\overline{F}^{k}| & \leq \tilde{C} (|U^{k-1}|+|\Sigma^{k-1}|+|\dot{\Sigma}^{k-1}|),
\label{kes1}\\
|\overline{M}_{i}^{k}| & \leq K \ve^{*} (  |P_{i} U^{k-1}(t,L_{i})|+|\Sigma_{i}^{k-1}(t)|) ,
\label{kes2}\\
|\overline{\beta}_{i}^{k}| & \leq K (  |P_{i} U^{k-1}(t,L_{i})|+|\Sigma_{i}^{k-1}(t)|),
\label{kes3}
\end{align}
where the positive constant $\tilde{C}$ depends on $K$, $N_{1}^{*}$, $N_{2}^{*}$, $N_{3}^{*}$, and $N_{4}^{*}$ but independent of $T$.

Multiplying \eqref{kkODE} by $2\Sigma^{k}_{i}$, integrating it over $[0,t]$, and then using \eqref{kes3} and Schwarz's  inequality, we have
\begin{align}
|\Sigma^{k}(t)|^{2} 
&\leq  \tilde{C} \sqrt{T}  \sup_{t \in [0,T]}   |({\Sigma^{k}},{\Sigma^{k-1}})(t)|^{2} 
+ \tilde{C} \sqrt{T} \sum_{i=1}^{2} \int_{0}^{T}  |P_{i} U^{k-1}(t,L_{i})|^{2} dt
\notag \\
&\leq \tilde{C} \sqrt{T} \{E^{k}(T) + E^{k-1}(T)\}
\label{kes5}
\end{align}
for any $T\in (0, T_{*}]$ and $t \in [0,T]$.
Furthermore, we take the $L^{2}$ norm of \eqref{kkODE} and use \eqref{kes3} and Schwarz's inequality to obtain
\begin{align}
\int_{0}^{T} |\dot{\Sigma}^{k}(t)|^{2} dt
&\leq   \tilde{C} \sqrt{T} \sup_{t \in [0,T]}   |({\Sigma^{k}},{\Sigma^{k-1}})(t)|^{2} 
+ (K^{2} + \tilde{C} \sqrt{T} )  \sum_{i=1}^{2} \int_{0}^{T} |P_{i} U^{k-1}(t,L_{i})|^{2} dt
\notag \\
&\leq K^{2} \sum_{i=1}^{2} \int_{0}^{T}  |P_{i} U^{k-1}(t,L_{i})|^{2} dt + \tilde{C} \sqrt{T} \{E^{k}(T) + E^{k-1}(T)\}.
\label{kes6}
\end{align}
Multiplying \eqref{kkeq0} by $2 U^{k}$, integrating the result over $[0,t]\times I$, and using \eqref{kkID0}, we obtain
\begin{align}
&\int_{I} \left\langle A^{0,k-1}   U^{k},  U^{k}\right\rangle (t,x) dx
+ \int_{0}^{t} \left\langle A^{1,k-1}  U^{k},   U^{k}\right\rangle  (\tau,L_{2}) d\tau
- \int_{0}^{t} \left\langle A^{1,k-1}  U^{k},   U^{k}\right\rangle  (\tau,L_{1}) d\tau
\notag \\
&=
 \int_{0}^{t}\!\! \int_{I}  \left\langle (\D_{t} A^{0,k-1})  U^{k} + (\D_{x} A^{1,k-1})   U^{k} - 2 B^{k-1} U^{k} + 2\overline{F}^{k},   U^{k}\right\rangle (\tau,x) dxd\tau 
\notag \\
&\leq \tilde{C}\sqrt{T} \left(\sup_{t \in [0,T]} \|(U^{k},U^{k-1})(t)\|^{2}_{L^{2}(I)}  + \sup_{t \in [0,T]}  |{\Sigma}^{k-1}(t)|^{2} + \int_{0}^{T} |\dot{\Sigma}^{k-1}(t)|^{2}dt \right)
\notag \\
&\leq \tilde{C} \sqrt{T} \{E^{k}(T) + E^{k-1}(T)\},
\label{kes7}
\end{align}
where we have also used \eqref{kes1} and $(u^{k},\sigma^{k}) \in X_2(T^{*}/2,\ve^{*},N_{1}^{*},N_{2}^{*},N_{3}^{*},N_{4}^{*})$ in deriving the inequalities.
The first term on the left hand side is estimated by using the condition (iv) in Assumption \ref{ABasp1} as follows:
\begin{gather*}
\int_{I} \left\langle A^0_{k}   U^{k},  U^{k}\right\rangle (t,x) dx
\geq c_{0} \| U^{k}(t)\|^{2}_{L^{2}(I)}
\geq \frac1K \| U^{k}(t)\|^{2}_{L^{2}(I)}.
\end{gather*}
For the second and third terms, we observe by using \eqref{kkBD0},  $(P_{i} +Q_{i} M_{i}^{k-1} P_{i} )  U^{k} \in \ker Q_{i}({\rm id}-M_{i}^{k-1} P_{i})$,  the condition (vii)' in Assumption \ref{ABasp1} that
\begin{align*}
&\int_{0}^{t} \left\langle (-1)^{i} A^{1,k-1}  U^{k},  U^{k}\right\rangle (\tau,L_{i}) d\tau
\\
&=\int_{0}^{t} \left\langle (-1)^{i} A^{1,k-1} (P_{i}+Q_{i})  U^{k}, (P_{i}+Q_{i})  U^{k}\right\rangle (\tau,L_{i}) d\tau
\\
&=\int_{0}^{t} \left\langle (-1)^{i} A^{1,k-1} (P_{i}+Q_{i} M_{i}^{k-1} P_{i})  U^{k}, (P_{i} +Q_{i} M_{i}^{k-1} P_{i}) U^{k}\right\rangle (\tau,L_{i}) d\tau
\\
&\quad +2\int_{0}^{t} \left\langle (-1)^{i} A^{1,k-1} \overline{M}_{i}^{k}, (P_{i} + Q_{i} M_{i}^{k-1}P_{i})  U^{k}\right\rangle (\tau,L_{i})  d\tau
+\int_{0}^{t} \left\langle (-1)^{i} A^{1} \overline{M}_{i}^{k}, \overline{M}_{i}^{k} \right\rangle (\tau,L_{i})  d\tau
\\
&\geq \frac{1}{2K} \int_{0}^{t} |P_{i}   U^{k}(\tau,L_{i}) |^{2} d\tau
- 2^{5}(\ve^{*})^{2}K^{6} \int_{0}^{T} |P_{i} U^{k-1}(t,L_{i})|^{2}+|\Sigma_{i}^{k-1}(t)|^{2}  dt,
\end{align*}
where we have also used \eqref{K1}, \eqref{kes2}, and Schwarz's inequality in deriving the inequality. 
Plugging these into \eqref{kes7}, multiplying it by $2K$, and using $2^{6}(\ve^{*})^{2}K^{7}\leq 1/4$, we have
\begin{align}
& 2\| U^{k}(t)\|^{2}_{L^{2}(I)}+ \sum_{i=1}^{2} \int_{0}^{t}  |P_{i}   U^{k}(\tau,L_{i}) |^{2} d\tau
\notag\\
&\leq \frac{1}{4} \sum_{i=1}^{2}  \int_{0}^{t} |P_{i}   U^{k-1}(\tau,L_{i}) |^{2} d\tau
+ \tilde{C} \sqrt{T} \{E^{k}(T) + E^{k-1}(T)\}.
\label{kes8}
\end{align}
We multiply \eqref{kes6} by $(8K^{2})^{-1}$ and sum up the result, $\eqref{kes5}$, and $\eqref{kes8}$.
Then taking $T^{*}$ small enough leads to \eqref{des1}.
\end{proof}

\subsection{Application to the IBVPs of the quasi-one-dimensional model}\label{SecA3}

This subsection provides the proofs of the time-local solvability of the initial--boundary value problems \eqref{eq311*},  \eqref{eq411*}, and \eqref{eq701*},  which are stated in Lemmas \ref{local3}, \ref{local5}, and \ref{local7}.
We remark that classical theorems such as \cite[Theorem A2]{Sc1986} are enough to show the solvability of the problems \eqref{eq311*} and \eqref{eq411*}, since the boundary conditions are linear. We need Theorem \ref{local} only to show the solvability of the problem \eqref{eq701*} in Lemma \ref{local7}.
To make this paper self-contained as possible, we show the solvability of all problems above using only Theorem \ref{local}.

\begin{proof}[Proof of the time-local solvability in Lemmas \ref{local3}]
To prove the time-local solvability of the initial--boundary value problem \eqref{eq311*} using Theorem \ref{local},
we use the fundamental values $(\rho,u)$ and adopt the system \eqref{eq512} 
for the unknown functions $f_{p}=(\phi,\psi)^{\top}=(\rho - \bar{\rho}, u - \bar{u})^{\top}$.
Namely,  we consider the following problem equivalent to \eqref{eq311*}:
\begin{gather*}
A^0_p \partial_t f_{p} + A^1_p \partial_x f_{p} + B_p f_{p} = 0, \quad t>0,  \ x \in I=(x_{1},x_{2}),
\\
f_{p}(0,x) = f_{p}^0(x), 
\\
\phi(t, x_1) = -  \frac{\bar{\rho}}{\bar{u}+\psi} (t,x_1) \psi(t, x_1), \quad \phi(t,x_{2})=0, \quad t>0,
\end{gather*}
where
\begin{align*} 
A^0_p=A^0_p(x,\rho,u) &=
A(x) \begin{bmatrix}
\dfrac{ p'(\rho)}{\rho}   & 0
\\[10pt]
0 & \rho
\end{bmatrix},
\quad
A^1_p=A^1_p(x,\rho,u) = 
A(x)\begin{bmatrix}
\dfrac{ p'(\rho)u}{\rho}   &  p'(\rho)
\\[10pt]
 p'(\rho) & \rho u
\end{bmatrix},
\\
{B}_p={B}_p(x,\rho,u) &= A^0_p
\begin{bmatrix}
\bar{u}_x(x) + \dfrac{A_x(x)}{A(x)} u   & \bar{\rho}_x(x) +  \dfrac{A_x(x)}{A(x)}\bar{\rho}(x)
\\  
b_{p} & \bar{u}_x(x)
\end{bmatrix},
\\
b_{p}=b_{p}(x,\rho,u)  &= \bar{\rho}_x(x)
\int_0^1 \left(\dfrac{d}{d \rho} \dfrac{p'(\rho)}{\rho} \right)\Big|_{\rho = \bar{\rho}(x) + \theta (\rho-\bar{\rho}(x)) } d \theta.
\notag
\end{align*}

To apply Theorem \ref{local} to this problem, we set
\begin{gather*}
L_{1}=x_{1}, \quad L_{2}=x_{2}, \quad {\cal O}_{u,\sigma,\dot{\sigma}}=B_{\delta}(0),
\quad u=u(t,x)=f_{p}(t,x), \quad \sigma=\sigma(t)=(0,0), 
\end{gather*}
where $B_{\delta}(0) \subset \mathbb R^{6}$ is a ball centered at the origin with radius $\delta>0$.
We also choose $A^0$, $A^1$, $B$, $P_{1}$, $Q_{1}$, ${M}_{1}$, $P_{2}$, $Q_{2}$, ${M}_{2}$, $F$, $\alpha_{1}$, $\alpha_{2}$, $\beta_{1}$, $\beta_{2}$ in Theorem \ref{local} as
\begin{align*}
A^0 &= A_{p}^{0}, \quad A^1 =A_{p}^{1}, \quad  B = B_{p}, \quad F=\alpha_{1}=\alpha_{2}=\beta_{1}=\beta_{2}=0,
\\
P_{1}&=P_{2}=
\begin{bmatrix}
0 & 0 \\
0 & 1
\end{bmatrix},
\quad 
Q_{1}=Q_{2}=
\begin{bmatrix}
1 & 0 \\
0 & 0
\end{bmatrix},
\quad 
M_{1}=
\begin{bmatrix}
0 & \displaystyle -  \frac{\bar{\rho}}{\bar{u}+\psi} (t,x_1)  \\
0 & 0
\end{bmatrix},
\quad 
M_{2}=
\begin{bmatrix}
0 & 0 \\
0 & 0
\end{bmatrix}.
\end{align*}
Note that this choice does not include $\sigma$, and so the hyperbolic system is not coupled with $\sigma$.
In addition, $\sigma=(0,0)$ obviously solves \eqref{lODE0}.
Thus we can ignore $\sigma$ in the application of Theorem \ref{local}.
Recall that $(\bar{\rho},\bar{m})$ satisfies \eqref{spro1}--\eqref{spro2}.
By taking $\delta$ small enough, it is easy to see that the above choice satisfies Conditions (i)--(vi) and (viii) in Assumption \ref{ABasp1}.
What is left is to verify Condition (vii). 
We observe at the boundary $x=x_{1}$ that for all $\eta=(\eta_{1},\eta_{2})\in \ker Q_{1}({\rm id}-M_{1}P_{1})$ with $|\eta|=1$,
\begin{align*}
& -\langle A^1 \eta, \eta \rangle
=A(x_{1})\left( (p'(\bar{\rho})- \bar{u}^{2}) \frac{\bar{u}}{\bar{\rho}}(x_{1}) +O(\delta) \right) \eta_{1}^{2}>0,
\end{align*}
where we have used $\eta_{1} \neq 0$. 
Similarly, it is seen at the other boundary $x=x_{2}$ that for all $\eta=(\eta_{1},\eta_{2})\in \ker Q_{2}({\rm id}-M_{2}P_{2})$ with $|\eta|=1$,
\begin{align*}
& \langle A^1 \eta, \eta \rangle
=A(x_{2})\left( \bar{\rho}\bar{u}(x_{2}) +O(\delta) \right) \eta_{2}^{2}>0,
\end{align*}
where we have used $(\eta_{1},\eta_{2})=(0,1)$.
From these, Condition (vii) holds.
Thus Theorem \ref{local} ensures the time-local solvability of the problem \eqref{eq311*}.
\end{proof}

\begin{proof}[Proof of the time-local solvability in Lemma \ref{local5}]
The solvability in Lemmas \ref{local5} follows from the same manner as in the proof of Lemma \ref{local3}, but with the matrices $P_{2}$, $Q_{2}$, and $M_{2}$ replaced by 
\begin{gather*}
P_{2}={\rm id}, \quad  Q_{2}=M_{2}=
\begin{bmatrix}
0 & 0 \\
0 & 0
\end{bmatrix},
\end{gather*}
and \eqref{spro51}--\eqref{spro52} instead of \eqref{spro1}--\eqref{spro2}.
\end{proof}

\begin{proof}[Proof of the time-local solvability in Lemma \ref{local7}]
To prove the time-local solvability of the initial--boundary value problem \eqref{eq701*} using Theorem \ref{local},
we use the fundamental values $(\rho_{1},u_{1})$ and adopt the system \eqref{eq704} for the unknown function $f_{1p}=(\phi_{1},\psi_{1})^{\top}=(\rho_{1} - \bar{\rho}_{1}, u_{1} - \bar{u}_{1})^{\top}$ in $I_{1}$, 
and use the fundamental values $(\rho_{2},m_{2})$ and adopt the system \eqref{eq703} for the unknown function $f_{2}=(\phi_{2},\omega_{2})^{\top}=(\rho_{2} - \bar{\rho}_{2}, m_{2} - \bar{m}_{2})^{\top}$ in $I_{2}$.
Namely,  we consider the following problem equivalent to \eqref{eq701*}:
\begin{subequations}\label{eqA}
\begin{gather}
A^0_ {1p} \partial_s f_{1p} + A^1_{1p} \partial_y f_{1p} + B_{1p} f_{1p} = 0, \quad s>0, \ y \in I_{1}=(x_{1},\bar{d}),
\\
A^0_ {2} \partial_s f_{2} + A^1_{2} \partial_y f_{2} + B_{2} f_{2} = 0, \quad s>0, \ y \in I_{2}=(\bar{d},x_{2}),
\\
 \dot{\sigma}(s) \jump{ \rho }(s,\bar{d}) 
 = \jump{m}(s,\bar{d}), \quad  s>0,
 \label{eqA06a}
\end{gather}
with
\begin{gather} 
f_{1p}(0,y)=f_{1p}^{0}(y)=
\begin{bmatrix}
\phi_{1p}^{0}(y)
\\[3pt]
\psi_{1p}^{0}(y)
\end{bmatrix},  
\quad 
f_{2}(0,y)=f_{2}^{0}(y):=
\begin{bmatrix}
\phi_{2}^{0}(y)
\\[3pt]
\omega_{2}^{0}(y)
\end{bmatrix},  
\quad 
\sigma(0)=\sigma^{0},
\\
\phi_{1}(s, x_1) = -  \frac{\bar{\rho}_{1}}{\bar{u}_{1}+\psi_{1}} (s,x_1) \psi_{1}(s, x_1), \quad \phi_{2}(s,x_{2})=0, \quad s>0,
\\
\jump{ \rho } (s, \bar{d}) \Bigl[\!\Bigl[ \dfrac{m^2}{\rho} + p(\rho) \Bigr]\!\Bigr] (s, \bar{d})
= \jump{ m } (s, \bar{d})^2, \quad s>0,
\label{eqA06c}
\end{gather}
\end{subequations}
where we have used $\bar{m}_{1}(\bar{d})=\bar{m}_{2}(\bar{d})$. Here 
\begin{align*}
\hat{x}_{1}&=\hat{x}_{1}(\sigma,y) = y + \sigma \dfrac{y - x_1}{\bar{d} - x_1}, \quad
A^0_{1p}=A^0_{1p}(y,\rho_{1},u_{1},\sigma) =
A(\hat{x}_{1})\begin{bmatrix}
\dfrac{ p'(\rho_{1})}{\rho_{1}}   & 0
\\[10pt]
0 & \rho_{1}
\end{bmatrix},
\\
A^1_{1p}&=A^1_{1p}(y,\rho_{1},u_{1},\sigma,\dot{\sigma}) = 
A(\hat{x}_{1}) \frac{\bar{d} - x_1}{\bar{d}+\sigma - x_1} 
 \begin{bmatrix}
\dfrac{ p'(\rho_{1})u_{1}}{\rho_{1}}   &  p'(\rho_{1})
\\[10pt]
 p'(\rho_{1}) & \rho_{1} u_{1}
\end{bmatrix}
- \dot{\sigma} \frac{y - x_1}{\bar{d}+\sigma - x_1} A^0_{1p},
\\
{B}_{1p}&={B}_{1p}(y,\rho_{1},u_{1},\sigma) = A^0_{1p}
\begin{bmatrix}
\bar{u}_{1x}(\hat{x}_{1}) + \dfrac{A_x(\hat{x}_{1})}{A(\hat{x}_{1})} u_{1}   & \bar{\rho}_{1x}(\hat{x}_{1}) +  \dfrac{A_x(\hat{x}_{1})}{A(\hat{x}_{1})}\bar{\rho}_{1}(\hat{x}_{1})
\\  
b_{1p} & \bar{u}_{1x}(\hat{x}_{1})
\end{bmatrix},
\\
b_{1p}&:=b_{p}(\hat{x}_{1},\rho_{1},u_{1})  = \bar{\rho}_{1x}(\hat{x}_{1})
\int_0^1 \left(\dfrac{d}{d \rho} \dfrac{p'(\rho)}{\rho} \right)\Big|_{\rho = \bar{\rho}_{1}(\hat{x}_{1}) + \theta (\rho-\bar{\rho}_{1}(\hat{x}_{1})) } d \theta
\end{align*}
and
\begin{align*}
\hat{x}_{2}&=\hat{x}_{2}(\sigma,y) = y + \sigma \dfrac{y - x_2}{\bar{d} - x_2}, \quad
A^0_{2}=A^0_{2}(y,\rho_{2},m_{2},\sigma) =
\frac{A(\hat{x}_{2})}{\bar{\rho}_{2}(\hat{x}_{2})}
\begin{bmatrix}
S_{2}  & 0
\\
0 & 1
\end{bmatrix},
\\
A^1_{2}&=A^1_{2}(y,\rho_{2},m_{2},\sigma,\dot{\sigma}) = 
\frac{A(\hat{x}_{2})}{\bar{\rho}_{2}(\hat{x}_{2})}
\frac{\bar{d} - x_2}{\bar{d}+\sigma - x_2} 
\begin{bmatrix}
0  &  S_{2}
\\[5pt]
S_{2} & 2\dfrac{m_{2}}{\rho_{2}}
\end{bmatrix}
- \dot{\sigma} \frac{y - x_2}{\bar{d}+\sigma - x_2} A^0_{2},
\\
S_{2}&:=S(\rho_{2},m_{2}) = p'(\rho_{2}) - \frac{m_{2}^{2}}{\rho_{2}^{2}}, 
\quad
{B}_{2}={B}_{2}(y,\rho_{2},m_{2},\sigma) = A^0_{2}
\begin{bmatrix}
0 & \dfrac{A_x(\hat{x}_{2})}{A(\hat{x}_{2})}
\\  
b_{221} & b_{222}
\end{bmatrix},
\\
b_{221}&:=b_{21}(\hat{x}_{2},\rho_{2},m_{2})=
\frac{\bar{m}_{2}(\hat{x}_{2})^{2}}{\rho_{2} \bar{\rho}_{2}(\hat{x}_{2})} \frac{A_x(\hat{x}_{2})}{A(\hat{x}_{2})}
+ \frac{\bar{m}_{2}(\hat{x}_{2})^2}{\rho_{2}^2 \bar{\rho}_{2}(\hat{x}_{2})^2} ( \rho_{2} + \bar{\rho}_{2}(\hat{x}_{2}) ) \bar{\rho}_{2x}(\hat{x}_{2})
\\
& \qquad \qquad \qquad \qquad \quad + \bar{\rho}_{2x}(\hat{x}_{2}) \int_0^1 p'' ( \bar{\rho}_{2}(\hat{x}_{2}) + \theta (\rho_{2}-\bar{\rho}_{2}(\hat{x}_{2})) ) d \theta,
\\
b_{222}&:=b_{22}(\hat{x}_{2},\rho_{2},m_{2})=
\frac{m_{2}-\bar{m}_{2}(\hat{x}_{2})}{\rho_{2}} \frac{A_x(\hat{x}_{2})}{A(\hat{x}_{2})}
- \frac{ m_{2} + \bar{m}_{2}(\hat{x}_{2}) }{\rho_{2}^2} \bar{\rho}_{2x}(\hat{x}_{2}).
\end{align*}

If the perturbation $f_{1p}$, $f_{2}$, $\sigma$ are small enough, 
the boundary value \eqref{eqA06c} can be written in an abstract form by using the implicit function theorem as
\begin{align*}
\phi_{2}(s,\bar{d})&= \sA(\phi_{1}(s,\bar{d}),\psi_{1}(s,\bar{d}),\omega_{2}(s,\bar{d}),\sigma(s)) 
\\
&= \sB^{\phi_{1}} \phi_{1}(s,\bar{d}) +  \sB^{\psi_{1}} \psi_{1}(s,\bar{d}) + \sB^{\omega_{2}} \omega_{2}(s,\bar{d}) + \sA(0,0,0,\sigma(s)), 
\end{align*}
where
\begin{align*}
\sB^{g}&=\sB^{g}(\phi_{1}(s,\bar{d}),\psi_{1}(s,\bar{d}),\omega_{2}(s,\bar{d}),\sigma(s)) 
\\
&:= \int_{0}^{1} \partial_{g} \sA (\theta\phi_{1}(s,\bar{d}),\theta\psi_{1}(s,\bar{d}),\theta\omega_{2}(s,\bar{d}),\sigma(s)) d\theta \quad \text{for $g=\phi_{1},\psi_{1},\omega_{2}$}.
\end{align*}
We see from \eqref{eq732} that 
\begin{gather}\label{Bomega2}
\sA(0,0,0,0)=0, \quad 
\sB^{\omega_{2}}(0,0,0,0)=-2\frac{\bar{u}_{2}}{\bar{S}_{2}}(\bar{d}). 
\end{gather}
Similarly, the ODE \eqref{eqA06a} can be written in an abstract form as
\begin{align*}
\dot{\sigma}(s) &= \sC (\phi_{1}(s,\bar{d}),\psi_{1}(s,\bar{d}),\omega_{2}(s,\bar{d}),\sigma(s)) 
\\
&= \sB^{\sigma} \sigma(s) + \sC(\phi_{1}(s,\bar{d}),\psi_{1}(s,\bar{d}),\omega_{2}(s,\bar{d}),0),
\end{align*}
where
\begin{align*}
\sB^{\sigma} &=\sB^{\sigma} (\phi_{1}(s,\bar{d}),\psi_{1}(s,\bar{d}),\omega_{2}(s,\bar{d}),\sigma(s))
\\
&:= \int_{0}^{1} \partial_{\sigma} \sC (\phi_{1}(s,\bar{d}),\psi_{1}(s,\bar{d}),\omega_{2}(s,\bar{d}),\theta\sigma(s)) d\theta.
\end{align*}
From \eqref{eq732}, it is seen that
\begin{gather}\label{Asigma1}
\sC(0,0,0,0)=0.
\end{gather}

To rewrite the problem \eqref{eqA} to an initial--boundary value problem only in $I_{1}$,
we perform a change of variables
\begin{align*}
z= y \quad \text{for $y \in I_{1}$},
\quad
z= \frac{x_{1}-\bar{d}}{x_{2}-\bar{d}} (y-\bar{d}) +\bar{d} \quad \text{for $y \in I_{2}$}.
\end{align*}
The latter one maps $I_{2}$ onto $I_{1}$. Conversely, the original variable is represented as
\begin{align*}
y= z \quad \text{for $z \in I_{1}$}, \quad
y= \frac{x_{2}-\bar{d}}{x_{1}-\bar{d}}(z-\bar{d}) + \bar{d} \quad \text{for $z \in I_{1}$}.
\end{align*}
Using this, we set
\begin{gather*}
L_{1}=x_{1}, \quad L_{2}=\bar{d}, \quad {\cal O}_{u,\sigma,\dot{\sigma}}=B_{\delta^{2}}(0), \quad 
(\sigma_{1},\sigma_{2})=(\sigma_{1}(s),\sigma_{2}(s))=(0,\sigma(s)),
\\
u=u(s,z)=\begin{bmatrix}
\phi_{1}(s,y(z))
\\
\psi_{1}(s,y(z))
\\
\delta\{\phi_{2}(s,y(z))-\chi(z,\sigma_{2}(s))\}
\\
\delta \omega_{2}(s,y(z))
\end{bmatrix}, \quad
\chi=\chi(z,\sigma_{2})=\frac{z-L_{1}}{L_{2}-L_{1}} \sA(0,0,0,\sigma_{2}),
\end{gather*}
where $\delta \in (0,1)$ is a constant to be determined later, and $B_{\delta^{2}}(0) \subset \mathbb R^{8}$ is a ball centered at the origin with radius $\delta^{2}$.
Then we see that $(u,\sigma_{1},\sigma_{2})$ satisfies the following problem equivalent to \eqref{eqA}:
\begin{subequations}\label{eqAAA}
\begin{gather}
A^{0}\partial_{s} u + A^{1} \partial_{z} u + B u = \delta F, \quad s>0, \ z \in (L_{1},L_{2}),
\\
u(0,z)=u^{0}(z):=\begin{bmatrix}
\phi_{1}^{0}(y(z))
\\[3pt]
\psi_{1}^{0}(y(z))
\\
\delta\{\phi_{2}^{0}(y(z))-\chi(z,\sigma^{0})\}
\\
\delta \omega_{2}^{0}(y(z))
\end{bmatrix},
\\
Q_{i}({\rm id}-M_{i} P_{i}) u (s,L_{i}) =0, \quad s>0, \ i=1,2
\end{gather}
with
\begin{gather}
\dot{\sigma}_{1} = 0, \quad \sigma_{1}(0)=0,
\\
\dot{\sigma}_{2} = \alpha_{2}\sigma_{2} + \beta_{2}, \quad {\sigma}_{2}(0)=\sigma^{0},
\end{gather}
\end{subequations}
where
\begin{align*}
A^{0}&=
\begin{bmatrix}
A^{0}_{1p} & O
\\
O & A^{0}_{2}
\end{bmatrix},
\quad 
A^{1}=
\begin{bmatrix}
A^{1}_{1p} & O
\\
O & \dfrac{x_{1}-\bar{d}}{x_{2}-\bar{d}} \, A^{1}_{2}
\end{bmatrix},
\quad 
B=
\begin{bmatrix}
B_{1p} & O
\\
O & B_{2}+\tilde{B}_{2}
\end{bmatrix},
\\
\tilde{B}_{2}&= 
\int_{0}^{1}  \nabla_{(\rho_{2},m_{2})} A^0_{2}\bigl(y(z),\theta \rho_{2} + (1-\theta) (\bar{\rho}_{2}+\chi),\theta m_{2} + (1-\theta)\bar{m}_{2},\sigma_{2}\bigr) d\theta
\begin{bmatrix}
(\partial_{\sigma_{2}}\chi) \dot{\sigma}_{2}
\\
0
\end{bmatrix}
\\
&\quad + \dfrac{x_{1}-\bar{d}}{x_{2}-\bar{d}}  \int_{0}^{1}  \nabla_{(\rho_{2},m_{2})} A^{1}_{2}\bigl(y(z),\theta \rho_{2} + (1-\theta) (\bar{\rho}_{2}+\chi),\theta m_{2} + (1-\theta)\bar{m}_{2},\sigma_{2}\bigr)d\theta
\begin{bmatrix}
\partial_{z}\chi
\\
0
\end{bmatrix}
\\
&\quad + \int_{0}^{1}  \nabla_{(\rho_{2},m_{2})} B_{2}\bigl(y(z),\theta \rho_{2} + (1-\theta) (\bar{\rho}_{2}+\chi),\theta m_{2} + (1-\theta)\bar{m}_{2},\sigma_{2}\bigr)d\theta
\begin{bmatrix}
\chi
\\
0
\end{bmatrix},
\\
F&=
-\begin{bmatrix}
O & O
\\
O & A^0_{2}(y(z),\bar{\rho}_{2}+\chi,\bar{u}_{2},\sigma_{2})
\end{bmatrix}\!\!
\begin{bmatrix}
0
\\
0
\\
(\partial_{\sigma_{2}} \chi)  \dot{\sigma}_{2}
\\
0
\end{bmatrix}
\\
& \quad  -
\begin{bmatrix}
O & O
\\
O & \dfrac{x_{1}-\bar{d}}{x_{2}-\bar{d}} A^{1}_{2}(y(z),\bar{\rho}_{2}+\chi,\bar{m}_{2},\sigma_{2})
\end{bmatrix}\!\!
\begin{bmatrix}
0
\\
0
\\
\partial_{z} \chi
\\
0
\end{bmatrix}
-
\begin{bmatrix}
O & O
\\
O & B_{2}(y(z),\bar{\rho}_{2}+\chi,\bar{m}_{2},\sigma_{2})
\end{bmatrix}\!\!
\begin{bmatrix}
0
\\
0
\\
\chi
\\
0
\end{bmatrix},
\\
P_{1}&={\rm diag}(0,1,0,1), \quad P_{2}={\rm diag}(1,1,0,1), \quad
Q_{1}={\rm diag}(1,0,1,0), \quad Q_{2}={\rm diag}(0,0,1,0),
\\
M_{1}&=
\begin{bmatrix}
0 & \displaystyle -  \frac{\bar{\rho}_{1}}{\bar{u}_{1}+\psi_{1}} (s,L_1) & 0 & 0
\\
0 & 0 & 0 & 0
\\
0 & 0 & 0 & 0
\\
0 & 0 & 0 & 0
\end{bmatrix},
\quad
M_{2}=
\begin{bmatrix}
0 & 0 & 0 & 0
\\
0 & 0 & 0 & 0
\\
\delta\sB^{\phi_{1}}  & \delta\sB^{\psi_{1}} & 0 & \sB^{\omega_{2}}
\\
0 & 0 & 0 & 0
\end{bmatrix},
\\
\alpha_{2}&=\sB^{\sigma}, \quad \beta_{2}=\sC(\phi_{1}(s,\bar{d}),\psi_{1}(s,\bar{d}),\omega_{2}(s,\bar{d}),0).
\end{align*}

Now we can apply Theorem \ref{local} to the problem \eqref{eqAAA} as follows.
Recall that $(\bar{\rho},\bar{m})$ satisfies \eqref{spro71}--\eqref{spro72} and that
$(\bar{\rho}_1, \bar{m}_1)$ and $(\bar{\rho}_2, \bar{m}_2)$ denote
the restriction of the stationary solution {$(\bar{\rho},\bar{m})$}
on $[x_1, \bar{d})$ and $(\bar{d}, x_2]$, respectively (see also the opening of Section \ref{Sec7}). 
By taking $\delta$ small enough and using \eqref{Bomega2} and \eqref{Asigma1}, it is easy to see that Conditions (i)--(vi) and (viii) in Assumption \ref{ABasp1} are verified for the problem \eqref{eqAAA}.

What is left is to verify Condition (vii). 
We observe at the boundary $z=L_{1}$ that for all $\eta=(\eta_{1},\eta_{2},\eta_{3},\eta_{4})\in \ker Q_{1}({\rm id}-M_{1}P_{1})$ with $|\eta|=1$,
\begin{align*}
& -\langle A^1 \eta, \eta\rangle
=\left( A(p'(\bar{\rho}_{1})- \bar{u}_{1}^{2}) \frac{\bar{u}_{1}}{\bar{\rho}_{1}}(x_{1}) +O(\delta) \right)\eta_{1}^{2}
+ \dfrac{\bar{d}-x_{1}}{x_{2}-\bar{d}} \left( 2 \frac{A\bar{m}_{2}}{\bar{\rho}_{2}^{2}}(x_{2}) +O(\delta) \right)\eta_{4}^{2} >0,
\end{align*}
where we have used $(\eta_{1},\eta_{2},\eta_{3},\eta_{4})=(\eta_{1},-\frac{\bar{u}_{1}+\psi_{1}}{\bar{\rho}_{1}} \eta_{1}, 0,\eta_{4})$ and $(\eta_{1},\eta_{4}) \neq (0,0)$. 
Similarly, it is seen at another boundary $z=L_{2}$ that for all $\eta=(\eta_{1},\eta_{2},\eta_{3},\eta_{4})\in \ker Q_{2}({\rm id}-M_{2}P_{2})$ with $|\eta|=1$,
\begin{align*}
\langle A^1 \eta, \eta\rangle &= \left\{ A(\bar{d}) 
\left(
\frac{p'(\bar{\rho}_{1})\bar{u}_{1}}{\bar{\rho}_{1}}(\bar{d}) \eta_{1}^{2}  + 2 p'(\bar{\rho}_{1})(\bar{d}) \eta_{1}\eta_{2}
+ \bar{\rho}_{1} \bar{u}_{1}(\bar{d})\eta_{2}^{2}
 \right) 
 +O(\delta)|(\eta_{1},\eta_{2})|^{2}  \right\}
\\
& \quad + \dfrac{\bar{d}-x_{1}}{x_{2}-\bar{d}} \left\{ 
2\frac{A\bar{m}_{2}}{\bar{\rho}_{2}^{2}}(\bar{d})  \eta_{4}^{2}
 +O(\delta)|(\eta_{1},\eta_{2},\eta_{4})|^{2}  \right\}
\\
& \geq c |(\eta_{1},\eta_{2},\eta_{4})|^{2}>0,
\end{align*}
where $c$ is a positive constant, and we have used \eqref{Bomega2} and $(\eta_{1},\eta_{2},\eta_{4}) \neq (0,0,0)$. 
From these, Condition (vii) holds.
Thus Theorem \ref{local} ensures the time-local solvability of the problem \eqref{eqAAA},
which is equivalent to the problem \eqref{eq701*}.
\end{proof}

\section{Trace of solutions of hyperbolic system}

In this section, we study the trace of the solution $u=u(t,x)$ of the following linear hyperbolic system:
\begin{gather}\label{eqB1}
\D_t u + A^1(t,x) \D_{x} u  =  F(t,x), \quad t \in [0,T], \ x \in  [\ell_{1},\ell_{2}],
\end{gather}
Here $u$ and $F$ are vector functions with $l$ components,
and $A^1$ is a $l \times l$ matrix. 
They are assumed to satisfy the following conditions on conditions and regularities:
\begin{asp}\label{aspB1} The following conditions hold.
\begin{enumerate}[\rm (i)]
\item The $l \times l$ matrix $A^1$ is regular and diagonalizable. So there exists a regular matrix $P$ such that 
\begin{gather*}
P^{-1}A^{1}P=\Lambda:= {\rm diag}(\lambda_{1}^{-},\cdots,\lambda_{l_{-}}^{-},\lambda_{1}^{+},\cdots,\lambda_{l_{+}}^{+}),
\end{gather*}
where $\lambda_{m}^{-}<0$ and $\lambda_{n}^{+}>0$ hold for all $(t,x) \in [0,T] \times  [\ell_{1},\ell_{2}]$, $m=1,\ldots,l_{-}$, and  $n=1,\ldots,l_{+}${\rm;} 
\item $A^{1}, (A^{1})^{-1}, P, P^{-1}, \Lambda \in C^{1}([0,T] \times  [\ell_{1},\ell_{2}])${\rm;} 
\item $A^{1}, P^{-1}, F \in H^{2}((0,T) \times  (\ell_{1},\ell_{2}))$.
\end{enumerate}
\end{asp}

We show Theorem \ref{thmB1} on the trace of the solutions $u$. Its application to the initial--boundary value  problems of the quasi-one-dimensional model \eqref{eq101} will be discussed in subsection \ref{secB1}.

\begin{thm}\label{thmB1}
Let Assumption \ref{aspB1} hold.
Suppose that $u$ solves the system \eqref{eqB1} and satisfies
\begin{gather*}
u \in \left[\bigcap_{k=0}^{2} C^{k}([0,T];H^{2-k}((\ell_{1},\ell_{2}))) \right]^{l}.
\end{gather*}
For any $y \in [\ell_{1},\ell_{2}]$ and $j,k=0,1,2$ with $j+k=2$, the trace $\partial_{t}^{j}\partial_{x}^{k} u(\cdot, y)$ at $x=y$ belongs to $L^{2}(0,T)$.

\begin{proof}
We may suppose $y \in [\ell_{1},\ell_{2})$, since the other case $y=\ell_{2}$ can be treated in a similar way.
We differentiate \eqref{eqB1} in $t$, apply the mollifier $J_{\delta}$ in \eqref{mol1},
and then differentiate the resultant equality in $t$ again to obtain 
\begin{gather*}
\D_{t}(\D_{tt} J_{\delta } u) + A^1 \D_{x} (\D_{tt} J_{\delta } u)  = G_{\delta}^{0} +  G_{\delta}^{1} + J_{\delta }  G, 
\\
G_{\delta}^{0}:=\D_{t} \left([A^{1},J_{\delta}] \D_{xt} u \right),
\quad
G_{\delta}^{1}:= (\D_{t} A^{1}) J_{\delta} \D_{xt} u,
\quad
G:=\D_{t} \{ (\D_{t} A^1) \D_{x} u + \D_{t} F \},
\notag
\end{gather*}
where $[\cal T,\cal S]:=\cal T\cal S-\cal S\cal T$ is a commutator.
Multiplying this by $P^{-1}$ from the left, we arrive at 
\begin{gather}\label{eqB2}
\D_{t} r_{\delta} + \Lambda \D_{x} r_{\delta}  
= I_{\delta} := (\D_{t} P^{-1})\D_{tt} J_{\delta } u + \Lambda (\D_{x} P^{-1})\D_{tt} J_{\delta } u + P^{-1}(G_{\delta}^{0} +  G_{\delta}^{1} + J_{\delta }  G),
\end{gather}
where
\begin{gather*}
r_{\delta}=
\begin{bmatrix}
r_{\delta}^{-} \\ r_{\delta}^{+}
\end{bmatrix}
:=P^{-1}\D_{tt} J_{\delta } u,
\quad
r_{\delta}^{-}:=
\begin{bmatrix}
r_{1\delta}^{-}
\\
\vdots
\\
r_{l_{-}\delta}^{-}
\end{bmatrix},
\quad
r_{\delta}^{-}:=
\begin{bmatrix}
r_{1\delta}^{+}
\\
\vdots
\\
r_{l_{+}\delta}^{+}
\end{bmatrix}.
\end{gather*}
Let us set
\begin{gather*}
\lambda_{*}^{-}:=  \min_{m=1,\ldots,l_{-}} \left( \min_{(t,x)\in [0,T]\times[\ell_{1},\ell_{2}]}\lambda_{m}^{-}(t,x) \right) <0, \quad
\lambda_{*}^{+}:= \max_{n=1,\ldots,l_{+}} \left( \max_{(t,x)\in [0,T]\times[\ell_{1},\ell_{2}]}\lambda_{n}^{+}(t,x) \right) >0
\end{gather*}
and 
\begin{align*}
\Omega_{T_{1},T_{2}}^{-}&:=\{ (t,x) \, ; \, T_{1} \leq t \leq T_{2}, 
\ y \leq x\leq y+\lambda_{*}^{-}(t-T_{2})\},
\\
\Omega_{T_{1},T_{2}}^{+}&:=\{ (t,x) \, ; \, T_{1} \leq t \leq T_{2}, 
\ y \leq x \leq y+\lambda_{*}^{+}(t-T_{1})\}
\end{align*}
for any $T_{1},T_{2}\geq 0$ with $0<T_{2}-T_{1}\leq\min\{T-\eta,|y-\ell_{2}|/|\lambda_{*}^{-}|,|y-\ell_{2}|/\lambda_{*}^{+} \}$.

Now we take a difference between \eqref{eqB2} with $\delta=\delta'$ and $\delta=\delta''$,
multiply the result by $2(r_{\delta'}^{-}-r_{\delta''}^{-},0)$ from the left, 
integrate it over $\Omega_{T_{1},T_{2}}^{-}$, and apply Green's theorem to obtain
\begin{align*}
& \sum_{m=1}^{l_{-}} \int_{T_{1}}^{T_{2}} |\lambda_{m}^{-}| |r_{m\delta'}^{-}-r_{m\delta''}^{-}|^{2}(t,y) dt
+ \sum_{m=1}^{l_{-}}  \int_{T_{1}}^{T_{2}} ( |\lambda_{*}^{-}| - |\lambda_{m}^{-}| )|r_{m\delta'}^{-}-r_{m\delta''}^{-}|^{2}(t,y+\lambda_{*}^{-}(t-T_{2})) dt
\\
&= \int_{y}^{y+\lambda_{*}^{-}(T_{1}-T_{2})} |r_{\delta'}^{-}-r_{\delta''}^{-}|^{2} (T_{1},x) dx
+ \sum_{m=1}^{l_{-}}  \int\!\!\int_{\Omega_{T_{1},T_{2}}} 
  (\D_{x} \lambda_{m}^{-}) |r_{m\delta'}^{-}-r_{m\delta''}^{-}|^{2} (t,x) dtdx
  \\
  & \quad
+ 2  \int\!\!\int_{\Omega_{T_{1},T_{2}}}  \left\langle \begin{bmatrix} r_{\delta'}^{-}-r_{\delta''}^{-} \\ 0 \end{bmatrix},
 I_{\delta'}-I_{\delta''}\right\rangle (t,x) dtdx
\\
& \to 0 \quad \text{as $\delta',\delta'' \to 0$,}
\end{align*}
where we have used \eqref{mp0} in showing the convergence.
We note that the second term on the left hand side is nonnegative owing to the definition of $\lambda_{*}^{-}$. 
Therefore using arbitrariness of $T_{1}$ and $T_{2}$, we deduce that for $\eta \in (0,T/2)$,
\begin{gather}\label{cvB1}
\int_{0}^{T-\eta} |r_{\delta'}^{-}-r_{\delta''}^{-}|^{2} (t,y) dt
\to 0 \quad \text{as $\delta',\delta'' \to 0$.}
\end{gather}

Similarly, we take a difference between \eqref{eqB2} with $\delta=\delta'$ and $\delta=\delta''$,
multiply the result by $2(0,r_{\delta'}^{+}-r_{\delta''}^{+})$ from the left, integrate it over $\Omega_{T_{1},T_{2}}$, and apply  Green's theorem to obtain
\begin{align*}
& \sum_{n=1}^{l_{+}} \int_{T_{1}}^{T_{2}} \lambda_{n}^{+} |r_{n\delta'}^{+}-r_{n\delta''}^{+}|^{2}(t,y) dt
+ \sum_{n=1}^{l_{+}}  \int_{T_{1}}^{T_{2}} ( \lambda_{*}^{+} - \lambda_{n}^{+} )|r_{n\delta'}^{+}-r_{n\delta''}^{+}|^{2}(t,y+\lambda_{*}^{+}(t-T_{1})) dt
\\
&= \int_{y}^{y+\lambda_{*}^{+}(T_{2}-T_{1})} |r_{\delta'}^{+}-r_{\delta''}^{+}|^{2} (T_{2},x) dx
+ \sum_{n=1}^{l_{+}}  \int\!\!\int_{\Omega_{T_{1},T_{2}}} 
  (\D_{x} \lambda_{n}^{+}) |r_{n\delta'}^{+}-r_{n\delta''}^{+}|^{2} (t,x) dtdx
  \\
  & \quad
+2  \int\!\!\int_{\Omega_{T_{1},T_{2}}}  \left\langle \begin{bmatrix} 0 \\ r_{\delta'}^{+}-r_{\delta''}^{+} \end{bmatrix},
 I_{\delta'}-I_{\delta''}\right\rangle (t,x) dtdx
\\
& \to 0 \quad \text{as $\delta',\delta'' \to 0$.}
\end{align*}
Note that the second term on the left hand side is nonnegative.
Then it follows from arbitrariness of $T_{1}$ and $T_{2}$ that for $\eta \in (0,T/2)$,
\begin{gather}\label{cvB2}
\int_{0}^{T-\eta} |r_{\delta'}^{+}-r_{\delta''}^{+}|^{2} (t,y) dt
\to 0 \quad \text{as $\delta',\delta'' \to 0$.}
\end{gather}

From \eqref{mp0}, \eqref{cvB1}, and \eqref{cvB2}, it is seen that
$\{J_{\delta}u(\cdot,y)\}$ is a Cauchy sequence in $[H^{2}(0,T-\eta)]^{l}$.
Furthermore, $J_{\delta}u(\cdot,y) \to u(\cdot,y)$ in $[H^{1}(0,T-\eta)]^{l}$.
From these, it follows that $\partial_{t}^{2} u(\cdot, y) \in L^{2}(0,T-\eta)$.
Multiplying \eqref{eqB1} by $(A^{1})^{-1}$ from the left gives $\D_{x} u  = (A^1)^{-1}(F-\D_t u)$.
From this, we see that 
$\partial_{t}^{j}\partial_{x}^{k} u(\cdot, y) \in L^{2}(0,T-\eta)$ for $\eta \in (0,T/2)$ and $j,k=0,1,2$ with $j+k=2$.

Using another mollifier which is replaced $(t-\tau)/\delta + 1$ by $(t-\tau)/\delta - 1$ in \eqref{mol1},
and then repeating the above method,  we can also conclude that 
$\partial_{t}^{j}\partial_{x}^{k} u(\cdot, y) \in L^{2}(\eta,T)$ for $\eta \in (0,T/2)$ and  $j,k=0,1,2$ with $j+k=2$. 
Thus $\partial_{t}^{j}\partial_{x}^{k} u(\cdot, y) \in L^{2}(0,T)$ holds for $j,k=0,1,2$ with $j+k=2$.
\end{proof}

\end{thm}

\subsection{Application to the IBVPs of the quasi-one-dimensional model}\label{secB1}

In this subsection, we discuss the trace of the solutions of the initial--boundary value problems of the quasi-one-dimensional model, which are stated in Lemmas \ref{local3}, \ref{local5}, and \ref{local7}.

\begin{proof}[Proof of the property of the trace in Lemma \ref{local3}]
We claim that if the solution $f \in \mathfrak{X} (T)$ solves the problem \eqref{eq311*} and the flow is subsonic at $x=x_{2}$ for all $t \in [0,T]$, then the following holds that for any $j,k=0,1,2$ with $j+k=2$:
\begin{gather}\label{Btrace1}
\partial_{t}^{j}\partial_{x}^{k} f(\cdot, x_{2}) \in L^{2}(0,T).
\end{gather}
Recall that $f$ satisfies the system \eqref{eq317}, i.e.
\begin{gather}\label{eqB3}
\partial_t f + \hat{A}^1 \partial_x f + \hat{B} f = 0,
\end{gather}
where
\begin{align*}
\hat{A}^{1}&=\hat{A}^{1}  (\rho,m) =
\begin{bmatrix}
0  &  1
\\[5pt]
S & \dfrac{2m}{\rho}
\end{bmatrix},
\quad
\hat{B}=\hat{B}(x,\rho,m) =
\begin{bmatrix}
0 & \dfrac{A_x(x)}{A(x)}
\\  
b_{21} & b_{22}
\end{bmatrix},
\\
S&=S(\rho,m) = p'(\rho) - \frac{m^{2}}{\rho^{2}}, 
\quad
b_{22}=b_{22}(x,\rho,m)=
\frac{m-\bar{m}(x)}{\rho} \frac{A_x(x)}{A(x)}
- \frac{ m + \bar{m}(x) }{\rho^2} \bar{\rho}_x(x),
\\
b_{21}&=b_{21}(x,\rho,m)=
\frac{\bar{m}(x)^{2}}{\rho \bar{\rho}(x)} \frac{A_x(x)}{A(x)}
+ \frac{\bar{m}(x)^2}{\rho^2 \bar{\rho}(x)^2} ( \rho + \bar{\rho}(x) ) \bar{\rho}_x(x)
\\
& \qquad \qquad \qquad \quad + \bar{\rho}_x(x) \int_0^1 p'' ( \bar{\rho}(x) + \theta (\rho-\bar{\rho}(x)) ) d \theta.
\end{align*}

We apply Theorem \ref{thmB1} to the solution $f \in \mathfrak{X} (T)$ as follows.
Since the flow is subsonic at $x=x_{2}$,
it is seen that $(m/\rho - \sqrt{p'(\rho)})(t,x_{2})<0<(m/\rho + \sqrt{p'(\rho)})(t,x_{2})$ for all $t \in [0,T]$. 
Hence there exists $\delta>0$ such that $(m/\rho - \sqrt{p'(\rho)})(t,x)<0<(m/\rho + \sqrt{p'(\rho)})(t,x)$ for $(t,x) \in [0,T] \times [x_{2}-\delta,x_{2}]$.
Now we set
\begin{gather*}
\ell_{1}=x_{2}-\delta, \quad \ell_{2}=x_{2}, \quad u=f.
\end{gather*}
We also choose $A^1$, $F$, $P$, and $\Lambda$ in Theorem \ref{thmB1} as
\begin{align*}
A^1 = \hat{A}^{1}, \ \
F =-\hat{B}f, \ \
P = \begin{bmatrix}
1 &  1 
\\
\dfrac{m}{\rho} - \sqrt{p'(\rho)} & \dfrac{m}{\rho} + \sqrt{p'(\rho)}
\end{bmatrix}, \ \
\Lambda= 
\begin{bmatrix}
\dfrac{m}{\rho} - \sqrt{p'(\rho)} & 0
\\
0 & \dfrac{m}{\rho} + \sqrt{p'(\rho)}
\end{bmatrix},
\end{align*}
which we regard as given functions.
It is straightforward to see that this choice satisfies all Conditions (i)--(iii) in Assumption \ref{aspB1}.
Thus Theorem \ref{thmB1} concludes \eqref{Btrace1}.
\end{proof}

\begin{proof}[Proof of  the property of the trace in Lemma \ref{local5}]
We suppose that $f \in \mathfrak{X} (T)$ is a solution of \eqref{eq411*}, and there exists some $x \in \bar{I}=[x_{1},x_{2}]$ such that none of $(u \pm \sqrt{p'(\rho)})(t,x)$, the characteristic speed of the system \eqref{eq417}, is zero for all $t \in [0,T]$.
Note that the system \eqref{eq417} is the same as \eqref{eqB3}. 
In much the same way as in the proof of \eqref{Btrace1}, we can conclude $\partial_{t}^{j}\partial_{x}^{k} f(\cdot, x) \in L^{2}(0,T)$ for any $j,k=0,1,2$ with $j+k=2$. The proof is complete.
\end{proof}

\begin{proof}[Proof of the property of the trace in Lemma \ref{local7}]
We suppose that $(f_{1},f_{2},\sigma) \in \mathfrak{X}_{1} (T) \times \mathfrak{X}_{2} (T) \times H^{3}(0,T)$ is a solution of  the problem \eqref{eq701*}, and there exists some $y_{i} \in \bar{I}_{i}$ such that none of $(\frac{\bar{d} - x_i}{\bar{d}+\sigma - x_i} (u_{i} \pm \sqrt{p'(\rho_{i})})-\dot{\sigma} \frac{y - x_i}{\bar{d} - x_i})(s,y_{i})$, the characteristic speed of the system \eqref{eq701}, is zero for all $s \in [0,T]$.
Then it suffices to show that for any $j,k=0,1,2$ with $j+k=2$,
\begin{gather}\label{Btrace3}
\partial_{s}^{j}\partial_{y}^{k} f_{i}(\cdot, y_{i}) \in L^{2}(0,T).
\end{gather}
Recall that $f_{i}$ satisfies the system \eqref{eq701}, i.e.
\begin{gather*}
\partial_s f_{i} + \hat{A}_{i}^1 \partial_y f_{i} + \hat{B}_{i} f_{i} = 0,
\end{gather*}
where
\begin{align*}
\hat{x}_{i}&=\hat{x}_{i}(\sigma,y) = y + \sigma \dfrac{y - x_i}{\bar{d} - x_i}, 
\\
\hat{A}^1_{i}&=\hat{A}^1_{i}(y,\rho_{i}, m_{i},\sigma,\dot{\sigma}) = 
\frac{\bar{d} - x_i}{\bar{d}+\sigma - x_i} 
\begin{bmatrix}
0  &  1
\\[5pt]
S_{i} & 2\dfrac{m_{i}}{\rho_{i}}
\end{bmatrix}
- \dot{\sigma} \frac{y - x_i}{\bar{d}+\sigma - x_i} {\rm id},
\\
S_{i}&:= S(\rho_{i}, m_{i}) = p'(\rho_{i}) - \frac{m_{i}^{2}}{\rho_{i}^{2}}, 
\quad
\hat{B}_{i}=\hat{B}_{i} (y,\rho_{i}, m_{i},\sigma) =
\begin{bmatrix}
0 & \dfrac{A_x(\hat{x}_{i})}{A(\hat{x}_{i})}
\\  
b_{i21} & b_{i22}
\end{bmatrix},
\\
b_{i22}&:= b_{22}(\hat{x}_{i},\rho_{i}, m_{i})=
\frac{m_{i}-\bar{m}_{i} (\hat{x}_{i})}{\rho_{i}} \frac{A_x(\hat{x}_{i})}{A(\hat{x}_{i})}
- \frac{ m_{i} + \bar{m}_{i} (\hat{x}_{i}) }{\rho_{i}^2} \bar{\rho}_{ix} (\hat{x}_{i}),
\\
b_{i21}&:= b_{21}(\hat{x}_{i},\rho_{i}, m_{i})=
\frac{\bar{m}_{i} (\hat{x}_{i})^{2}}{\rho_{i} \bar{\rho}_{i} (\hat{x}_{i})} \frac{A_x(\hat{x}_{i})}{A(\hat{x}_{i})}
+ \frac{\bar{m}_{i} (\hat{x}_{i})^2}{\rho_{i}^2 \bar{\rho}_{i} (\hat{x}_{i})^2} ( \rho_{i} + \bar{\rho}_{i} (\hat{x}_{i}) ) \bar{\rho}_{i x}(\hat{x}_{i})
\\
& \qquad \qquad \qquad \qquad \quad + \bar{\rho}_{ix} (\hat{x}_{i}) \int_0^1 p'' ( \bar{\rho}_{i} (\hat{x}_{i}) + \theta (\rho_{i}-\bar{\rho}_{i} (\hat{x}_{i})) ) d \theta.
\end{align*}

We apply Theorem \ref{thmB1} to the solution $(f_{1},f_{2},\sigma) \in \mathfrak{X}_{1} (T) \times \mathfrak{X}_{2} (T) \times H^{3}(0,T)$ as follows.
It is seen by choosing $\delta>0$ small  that $(\frac{\bar{d} - x_i}{\bar{d}+\sigma - x_i} (u_{i} \pm \sqrt{p'(\rho_{i})})-\dot{\sigma} \frac{y - x_i}{\bar{d} - x_i})(s,y)\neq 0$ holds for any $(s,y) \in
[0,T] \times [y_{i}-\delta, y_{i}]$ or $(s,y) \in [0,T] \times [y_{i}, y_{i}+\delta]$.
Let us only consider the first case $(s,y) \in [0,T] \times [y_{i}-\delta, y_{i}]$, since the other case can be treated similarly. 
We set
\begin{gather*}
\ell_{1}=y_{i} -\delta, \quad \ell_{2}=y_{i}, \quad u=f_{i}.
\end{gather*}
We also choose $A^1$, $F$, $P$, and $\Lambda$ in Theorem \ref{thmB1} as
\begin{gather*}
A^1 = \hat{A}^{1}_{i}, \ \
F =-\hat{B}_{i} f_{i}, 
\\
P = \begin{bmatrix}
1 &  1 
\\
\dfrac{\bar{d} - x_i}{\bar{d}+\sigma - x_i} (u_{i} - \sqrt{p'(\rho_{i})})-\dot{\sigma} \dfrac{y - x_i}{\bar{d} - x_i} &\dfrac{\bar{d} - x_i}{\bar{d}+\sigma - x_i} (u_{i} + \sqrt{p'(\rho_{i})})-\dot{\sigma} \dfrac{y - x_i}{\bar{d} - x_i}
\end{bmatrix}, 
\\[3pt]
\Lambda= 
\begin{bmatrix}
\dfrac{\bar{d} - x_i}{\bar{d}+\sigma - x_i} (u_{i} - \sqrt{p'(\rho_{i})})-\dot{\sigma} \dfrac{y - x_i}{\bar{d} - x_i} & 0
\\
0 & \dfrac{\bar{d} - x_i}{\bar{d}+\sigma - x_i} (u_{i} + \sqrt{p'(\rho_{i})}) -\dot{\sigma} \dfrac{y - x_i}{\bar{d} - x_i}
\end{bmatrix},
\end{gather*}
which we regard as given functions.
It is straightforward to see that this choice satisfies all Conditions (i)--(iii) in Assumption \ref{aspB1}.
Thus Theorem \ref{thmB1} concludes \eqref{Btrace3}.
\end{proof}

\end{appendix}

\section*{Statements and Declarations}

\subsection*{Competing Interests}
The authors have no relevant financial or non-financial interests to disclose.

\subsection*{Funding}
MO and MS were supported in part by JSPS KAKENHI Grant Number 17K05313 and 21K03305.
MO and MS are currently supported in part by JSPS KAKENHI Grant Number 25K07079 and 25K00915, respectively.

\subsection*{Data Availability}
No datasets were generated or analyzed in the present study.

\end{document}